\documentclass[10pt]{amsart}
\usepackage{amssymb,amsmath,amsthm,amscd,bm}
\usepackage{leftidx}
\usepackage{graphicx}
\usepackage{color}
\usepackage{microtype}
\usepackage{hyperref}
\usepackage[usenames,dvipsnames,svgnames,table]{xcolor}
\usepackage{tikz}
\usepackage{comment}
\usepackage{xfrac,faktor}
\usepackage{epstopdf} 
\usepackage{float}
\usepackage{tikz-cd}
\usepackage[all]{xy}
\newtheorem{lemma}{Lemma}[section]
\newtheorem{theorem}[lemma]{Theorem}
\newtheorem{corollary}[lemma]{Corollary}
\newtheorem{proposition}[lemma]{Proposition}

\theoremstyle{definition}
\newtheorem{definition}[lemma]{Definition}
\theoremstyle{remark}
\newtheorem{remark}{Remark}
\newcommand{\C}{\mathbb{C}}
\newcommand{\D}{\mathbb{D}}
\newcommand{\N}{\mathbb{N}}

\newcommand{\R}{\mathbb{R}}
\newcommand{\Z}{\mathbb{Z}}

\newcommand{\cC}{\mathcal{C}}

\DeclareMathOperator{\re}{Re}
\DeclareMathOperator{\im}{Im}
\DeclareMathOperator{\Int}{int}

\renewcommand{\epsilon}{\varepsilon}
\renewcommand{\phi}{\varphi}

\DeclareMathOperator{\Deg}{deg}

\DeclareMathOperator{\ind}{ind}

\date{\today}

\begin{document}

\title[Schwarz reflections and correspondences]{Schwarz reflections and anti-holomorphic correspondences}

\author[S.-Y. Lee]{Seung-Yeop Lee}
\address{Department of Mathematics and Statistics, University of South Florida, Tampa, FL 33620, USA}
\email{lees3@usf.edu}

\author[M. Lyubich]{Mikhail Lyubich}
\address{Institute for Mathematical Sciences, Stony Brook University, NY, 11794, USA}
\email{mlyubich@math.stonybrook.edu}

\author[N. G. Makarov]{Nikolai G. Makarov}
\address{Department of Mathematics, California Institute of Technology, Pasadena, California 91125, USA}
\email{makarov@caltech.edu}

\author[S. Mukherjee]{Sabyasachi Mukherjee}
\address{School of Mathematics, Tata Institute of Fundamental Research, 1 Homi Bhabha Road, Mumbai 400005, India}
\email{sabya@math.tifr.res.in}

\maketitle

\begin{abstract}
In this paper, we continue exploration of the dynamical and parameter planes of one-parameter families of Schwarz reflections that was initiated in \cite{LLMM1,LLMM2}. Namely, we consider a family of quadrature domains obtained by restricting the Chebyshev  cubic polynomial to various univalent discs. Then we perform a quasiconformal surgery that turns these reflections to parabolic rational maps (which is the crucial technical ingredient of our theory). It induces a straightening map between the parameter plane of Schwarz reflections and the parabolic Tricorn. We describe various properties of this straightening highlighting the issues related to its anti-holomorphic nature. We complete the discussion by comparing our family with the classical Bullett-Penrose family of matings between groups and rational maps induced by holomorphic correspondences. More precisely, we show that the Schwarz reflections give rise to anti-holomorphic correspondences that are matings of parabolic anti-rational maps with the abstract modular group. We further illustrate our mating framework by studying the correspondence associated with the Schwarz reflection map of a deltoid.
\end{abstract}

\setcounter{tocdepth}{1}

\tableofcontents

\section{Introduction}\label{intro}

As we know from classical geometric function theory, any analytic curve
(and some piecewise analytic curves) in $\mathbb{C}$ can serve as a mirror
for a local anti-holomorphic reflection map, called the \emph{Schwarz reflection map}. If the curve bounds some domain
$\Omega $ then it can happen that this reflection can be extended to a
(generally non-invertible) anti-holomorphic map in $\Omega $. Such a domain
is called a \emph{quadrature domain}.

Quadrature domains first appeared in the work of Aharonov and Shapiro
\cite{AS1,AS2,AS}, and since then have been extensively studied in connection
with various problems of complex analysis, field theory, and fluid dynamics
(see \cite[\S 1.1]{LLMM1} for more references). Dynamics of Schwarz reflections
was first used in \cite{LM} to address some questions of interest in statistical
physics concerning topology and singular points of quadrature domains.
In \cite{LLMM1,LLMM2}, a systematic exploration of Schwarz dynamics was
launched. These works demonstrated that Schwarz dynamics can combine features of dynamics
of rational maps and Fuchsian groups, and provided explicit models for
the corresponding dynamical and parameter loci.

The phenomenon of ``mating" of rational maps with Fuchsian groups was discovered in the 1990s by Bullett and Penrose in the context of iterated algebraic correspondences \cite{BP}. It appears that Schwarz dynamical systems provide a general framework for this phenomenon. Indeed, in the current paper we will describe one more family of Schwarz reflections
with similar features, which gives rise to anti-holomorphic versions of the Bullett-Penrose correspondences. Note that together with the families considered in \cite{LLMM1,LLMM2}, this family will exhaust the list of one-parameter families of quadratic Schwarz reflection maps.

To put the contents of the present paper in perspective, let us briefly recall the principal results of \cite{LLMM1,LLMM2}. In these papers, we carried out a detailed dynamical study of Schwarz reflection with respect to a deltoid, and a one-parameter family of Schwarz reflections with respect to a cardioid and a family of circumscribing circles (referred to as the C$\&$C family). The simplest example of the aforementioned mating phenomenon comes from the Schwarz reflection map of a deltoid, which produces a conformal mating of the anti-holomorphic polynomial $\overline{z}^2$ and the ideal triangle reflection group \cite[\S 5]{LLMM1}. The main results on the C$\&$C family include a description of the geometrically finite maps in this family as conformal matings of quadratic anti-holomorphic polynomials and the ideal triangle group, and the construction of a homeomorphism between the combinatorial models of the connectedness locus of the C$\&$C family and the basilica limb of the Tricorn \cite[Theorems~1.1, 1.4]{LLMM2}.

In this paper, we focus on a family $\mathcal{S}$ of Schwarz reflections associated with simply connected bounded quadrature domains that appear as univalent images of round disks under a fixed cubic polynomial $f$. For such a (maximal) round disk centered at $a$, the corresponding Schwarz reflection map is denoted by $\sigma_a$. 

In accordance with the general dynamical decomposition of Schwarz reflections indicated in \cite[\S 1.5]{LLMM1}, the dynamical plane of each $\sigma_a\in\mathcal{S}$ can be partitioned into two invariant sets: the \emph{tiling set} and the \emph{non-escaping set} (see Subsection~\ref{dyn_plane}). A careful study of the mapping properties of the members of $\mathcal{S}$ shows that each Schwarz reflection map $\sigma_a$ gives rise to a \emph{pinched anti-quadratic-like map} (see Definition~\ref{pinched_def}), which can be thought of as a degenerate version of (anti-)quadratic-like maps. This pinched anti-quadratic-like map completely captures the dynamics of $\sigma_a$ on its non-escaping set. Since every (anti-)polynomial-like map is hybrid conjugate to an actual (anti-)polynomial \cite{DH2}, it is natural to expect that an analogous statement should hold true for pinched anti-quadratic-like maps. However, the existence of the pinching point (which results in the fundamental domain of a pinched anti-quadratic-like map being a degenerate annulus) adds significant subtlety to the situation. At the technical heart of the paper lies a ``straightening theorem'' for pinched anti-quadratic-like maps that allows one to find a quadratic anti-holomorphic rational map with a parabolic fixed point that is hybrid conjugate to a pinched anti-quadratic-like map (see Lemma~\ref{straightening_lemma} and~Theorem~\ref{straightening_schwarz}). We should mention that in \cite{Lu}, Lomonaco proved a related straightening theorem for \emph{parabolic-like maps}. However, since the maps $\sigma_a$ do not have an (anti-)holomorphic extension in a neighborhood of the pinching point, the main result of \cite{Lu} cannot be directly applied to our setting. On the contrary, our definition of pinched anti-quadratic-like maps does not require a local (anti-)holomorphic extension in a neighborhood of the pinching point, and the corresponding straightening theorem handles this difficulty by an analysis of boundary behavior of conformal maps which allows for quasiconformal interpolation in infinite strips.

At the level of parameter spaces, this defines a \emph{straightening map} $\chi$ from the connectedness locus $\cC(\mathcal{S})$ of the family $\mathcal{S}$ (i.e., the set of maps in $\mathcal{S}$ with connected non-escaping set, see Subsection~\ref{conn_locus_def}) to the \emph{parabolic Tricorn}, which is the connectedness locus $\cC(\mathfrak{L}_0)$ of a suitable slice $\mathfrak{L}_0$ of quadratic anti-holomorphic rational maps with a (persistent) parabolic fixed point (see Appendix~\ref{anti_rational_parabolic} for details on the persistently parabolic family $\mathfrak{L}_0$). It is well-known that continuity of straightening maps is a ``miracle'' that one typically cannot expect in general parameter spaces. This is indeed the case in our setting. However, a thorough analysis of the continuity and surjectivity properties of the straightening map (defined above) permits us to prove that the straightening map induces a homeomorphism between the locally connected models of the above two connectedness loci.

\begin{theorem}[Combinatorial Model of Connectedness Locus]\label{abstract_homeo}
The abstract connectedness locus $\widetilde{\cC(\mathcal{S})}$ (of the family $\mathcal{S}$) is homeomorphic to the abstract parabolic Tricorn $\widetilde{\cC(\mathfrak{L}_0)}$.
\end{theorem}

A few words on the proof of the above theorem are in order. Since every member of $\cC(\mathcal{S})$ has the same ``external map'' (i.e., they are conformally conjugate to each other on the tiling set, see Proposition~\ref{schwarz_group}), it is easy to adapt the classical proof of injectivity of straightening maps for the current setting (Proposition~\ref{chi_injective_prop}). The first step towards ``almost surjectivity'' of the straightening map $\chi$ is to construct a uniformization of the \emph{escape locus} of $\mathcal{S}$ (i.e., the complement of the connectedness locus in the parameter space) in terms of the conformal position of the escaping critical point, which is carried out in Theorem~\ref{escape_unif_thm}. This uniformization (more precisely, the parameter rays coming from it) is then used to show that certain critically pre-periodic maps in the parabolic Tricorn $\cC(\mathfrak{L}_0)$ lie in the image of $\chi$ (Proposition~\ref{onto_dyadic}). Subsequently, approximating parabolic parameters by these critically pre-periodic ones, and using our knowledge of the hyperbolic components and their bifurcation structure (see Section~\ref{hyp_comp_sec}), we conclude that the image of $\chi$ contains the closure of all ``hyperbolic parameters'' in $\cC(\mathfrak{L}_0)$. 

As mentioned earlier, the map $\chi$ is not everywhere continuous on $\cC(\mathcal{S})$ (see Subsection~\ref{chi_discont_subsec}). On the other hand, in Subsection~\ref{chi_cont_subsec}, we use classical arguments to show that $\chi$ is continuous at all hyperbolic and quasiconformally rigid parameters. Finally, defining the abstract connectedness loci $\widetilde{\cC(\mathcal{S})}$ and $\widetilde{\cC(\mathfrak{L}_0)}$ as locally connected combinatorial models of the corresponding connectedness loci, we show that $\chi$ descends to a homeomorphism between $\widetilde{\cC(\mathcal{S})}$ and $\widetilde{\cC(\mathfrak{L}_0)}$. Intuitively speaking, passing to the abstract connectedness loci `tames' the straightening map $\chi$.

To obtain our desired mating description, we define a $2:2$ anti-holomorphic correspondence $\widetilde{\sigma_a}^*$ on the Riemann sphere $\widehat{\C}$ (for every $a\in\cC(\mathcal{S})$) by lifting the action of $\sigma_a$ by the cubic polynomial $f$ (see Section~\ref{sec_corr}). The correspondences $\widetilde{\sigma_a}^*$ can be viewed as anti-holomorphic analogues of Bullett-Penrose correspondences, and have several similarities with their holomorphic counterparts (see \cite{BP} and the recent work of Bullett and Lomonaco \cite{BuLo1} where it is shown that each Bullett-Penrose correspondence is a mating between the modular group and a parabolic quadratic rational map). In particular, the two branches of the correspondence $\widetilde{\sigma_a}^*$ (respectively, of a Bullett-Penrose correspondence) are given by compositions of the (non-trivial) ``deck maps'' of $f$ with an anti-holomorphic involution of $\widehat{\C}$ (respectively, a holomorphic involution of $\widehat{\C}$). A key difference between the two settings is that our correspondences $\widetilde{\sigma_a}^*$ naturally arise from the maps $\sigma_a$, and hence can be profitably studied by looking at the dynamics of $\sigma_a$. This allows us to apply the above straightening theorem to suitable branches of the correspondences in question. This, combined with a careful analysis of the ``deck maps'' of the cubic polynomial $f$, yields the following mating theorem, an expanded version of which is proved in Theorem~\ref{group_rational_mating_thm}.

\begin{theorem}[Anti-holomorphic Correspondences as Matings]\label{group_rational_mating_thm_intro}
For each $a\in\cC(\mathcal{S})$, the Riemann sphere $\widehat{\C}$ can be decomposed into two $\widetilde{\sigma_a}^*$-invariant subsets; namely, the lifted tiling set and the lifted non-escaping set. On the lifted tiling set, the dynamics of the correspondence $\widetilde{\sigma_a}^*$ is equivalent to the action of the abstract modular group $\Z/2\Z\ast\Z/3\Z$, and on a suitable subset of the lifted non-escaping set, a forward branch of the correspondence is conjugate to the anti-rational map $R_{\chi(a)}$.
\end{theorem}

Finally, our knowledge of the image of the straightening map $\chi$ (Corollary~\ref{onto_hyperbolic_closure}) combined with the mating description of the correspondence $\widetilde{\sigma_a}^*$ given in Theorem~\ref{group_rational_mating_thm} readily imply the following.

\begin{theorem}[Realizing Matings as Correspondences]\label{almost_all_maps_mated}
For every $(\alpha,A)\in\cC(\mathfrak{L}_0)$ that lies in the closure of hyperbolic parameters, there exists a unique $a\in\cC(\mathcal{S})$ such that the correspondence $\widetilde{\sigma_a}^*$ is a mating of the rational map $R_{\alpha,A}$ and the abstract modular group $\Z/2\Z\ast\Z/3\Z$.
\end{theorem}

We end the paper with the simplest examples of $d:d$ anti-holomorphic correspondences that are matings of anti-polynomials and groups. These correspondences arise from univalent restrictions of suitable rational maps of degree $d+1$, and realize matings of the anti-polynomial $\overline{z}^d$ with the abstract Hecke group $\Z/2\Z\ast\Z/(d+1)\Z$.

Let us now elaborate on the organization of the paper. In Section~\ref{schwarz_background}, we recall some basic definitions and results on quadrature domains, and give a classification of one-parameter families of quadratic Schwarz reflection maps. This shows that from the point of view of mating rational maps with groups, the only one-parameter families (of quadratic Schwarz reflections) of interest are the C$\&$C family and the family $\mathcal{S}$ studied in this paper. Section~\ref{family_schwarz} is devoted to the definition of the family $\mathcal{S}$. More precisely, in Subsection~\ref{univalence_sec}, we give a complete description of the set of parameters $a$ (in the right half-plane) for which the cubic polynomial $f(u)=u^3-3u$ is univalent on the round disk $\Delta_a:=B(a,\vert a-1\vert)$. For such parameters, the image of the disk $\Delta_a$ under $f$ is a simply connected quadrature domain $\Omega_a$ with a cusp point on its boundary (the corresponding Schwarz reflection map is denoted by $\sigma_a$). In Subsection~\ref{setup}, we study some basic mapping properties of the maps $\sigma_a$, which allow us to define the family $\mathcal{S}$. An important feature of the maps $\sigma_a$ in $\mathcal{S}$ is that $\sigma_a:\sigma_a^{-1}(\Omega_a)\to\Omega_a$ is a $2:1$ branched covering branched at a unique simple critical point (this mapping behavior and the existence of a fixed cusp on the boundary of $\Omega_a$ are precursors to the anti-quadratic-like restriction of $\sigma_a$). In Section~\ref{schwarz_family}, we first describe the decomposition of the dynamical plane of $\sigma_a$ into non-escaping and tiling sets. In Subsection~\ref{dyn_near_cusp}, we compute the asymptotic development of $\sigma_a$ near the cusp point on $\partial\Omega_a$. Subsequently, in Subsection~\ref{conn_locus_def}, we define the connectedness locus $\cC(\mathcal{S})$ of the family $\mathcal{S}$, and study some of its elementary properties. We conclude Section~\ref{schwarz_family} by giving a dynamical uniformization of the tiling set of the maps $\sigma_a$ in $\mathcal{S}$. Among other things, it is shown here that all maps in $\cC(\mathcal{S})$ are conformally conjugate on their tiling sets. Section~\ref{sec_straightening} contains a general straightening theorem for pinched anti-quadratic-like maps (which is introduced in Definition~\ref{pinched_def}), and its application to the maps $\sigma_a$ in $\mathcal{S}$. The asymptotics of $\sigma_a$ near the cusp point on $\partial\Omega_a$ (obtained in Subsection~\ref{dyn_near_cusp}) are of fundamental importance in the proof of the Straightening Theorem~\ref{straightening_schwarz}. In Section~\ref{hyp_comp_sec}, we study (the closures of) the hyperbolic components in $\mathcal{S}$ and their bifurcation structure. The next Section~\ref{parameter_tessellation} uses the dynamical uniformization of the tiling set of $\sigma_a$ (given in Subsection~\ref{dyn_unif_tiling_sec}) to furnish a uniformization of the exterior of the connectedness locus of $\mathcal{S}$. Section~\ref{chi_prop_sec} is dedicated to a detailed study of the straightening map $\chi$ (defined in Section~\ref{sec_straightening}) from $\cC(\mathcal{S})$ to the parabolic Tricorn $\cC(\mathfrak{L}_0)$, which is the connectedness locus of a suitable slice $\mathfrak{L}_0$ of quadratic anti-holomorphic rational maps with a persistent parabolic fixed point (the family $\mathfrak{L}_0$ is introduced and studied in Appendix~\ref{anti_rational_parabolic}). In Subsection~\ref{continuity}, we analyze (dis)continuity properties of the map $\chi$ at various parameters. After proving continuity of $\chi$ at the hyperbolic and quasiconformally rigid parameters of $\cC(\mathcal{S})$, we show that discontinuity of $\chi$ may occur on quasiconformally deformable parabolic parameters. To conclude our analysis of the straightening map $\chi$, we study its surjectivity properties in Subsection~\ref{almost_surjective} which culminates in the statement that the image of $\chi$ contains the closure of all hyperbolic parameters in $\cC(\mathfrak{L}_0)$. The proof of this fact uses the results of Sections~\ref{hyp_comp_sec} and~\ref{parameter_tessellation} in an essential way. In Section~\ref{model_homeo}, we construct the abstract connectedness locus $\widetilde{\cC(\mathcal{S})}$ as a locally connected topological model of $\cC(\mathcal{S})$, and use the results of Section~\ref{chi_prop_sec} to show that $\chi$ induces a homeomorphism between $\widetilde{\cC(\mathcal{S})}$ and the abstract connectedness locus of $\widetilde{\cC(\mathfrak{L}_0)}$ of $\cC(\mathfrak{L}_0)$ (the construction of $\widetilde{\cC(\mathfrak{L}_0)}$ is carried out in Appendix~\ref{anti_rational_parabolic}). This completes the proof of Theorem~\ref{abstract_homeo}. In the final Section~\ref{sec_corr}, we define the anti-holomorphic counterpart of Bullett-Penrose correspondences by taking all possible lifts of $\sigma_a^{\pm 1}$ under the cubic polynomial $f$ (and throwing away the anti-diagonal). This is followed by a meticulous study of the dynamics of these correspondences on the two dynamically invariant subsets (i.e., the lifted tiling set and the lifted non-escaping set), which requires a good understanding of the deck maps of $f$ on suitable regions. Combining this with the Straightening Theorem~\ref{straightening_schwarz}, we give a proof of our Mating Theorem~\ref{group_rational_mating_thm_intro} (in fact, we prove a more detailed version of this theorem in Theorem~\ref{group_rational_mating_thm}). Putting together the Mating Theorem and the surjectivity results on $\chi$ (from Subsection~\ref{almost_surjective}), the proof of Theorem~\ref{almost_all_maps_mated} follows immediately. Appendix~\ref{deltoid_corr} contains the construction of a $d:d$ anti-holomorphic correspondence that is a mating of the anti-polynomial $\overline{z}^d$ with the abstract Hecke group $\Z/2\Z\ast\Z/(d+1)\Z$.
\bigskip

\noindent\textbf{Acknowledgements.} The second author was partially supported by NSF grants DMS-1600519 and 1901357, and a fellowship from the Hagler Institute for Advanced Study. The fourth author was supported by the Institute for Mathematical Sciences at Stony Brook University, an endowment from Infosys Foundation, and SERB grant SRG/2020/000018 during parts of the work on this project. 

\section{Quadrature Domains, and Schwarz Reflection Maps}\label{schwarz_background}

Although we will deal with explicit quadrature domains and Schwarz reflection maps in this paper, we would like to remind the readers the general definitions of these objects. For a more detailed exposition on quadrature domains and Schwarz reflection maps, and their connection with various areas of complex analysis and statistical physics, we refer the readers to \cite[\S 1,~\S 4]{LLMM1} and the references therein.

\subsection{Basic definitions}\label{background_def}

\begin{definition}[Schwarz Function]
Let $\Omega\subsetneq\widehat{\C}$ be a domain such that $\infty\notin\partial\Omega$ and $\Int{\overline{\Omega}}=\Omega$. A \emph{Schwarz function} of $\Omega$ is a meromorphic extension of $\overline{z}\vert_{\partial\Omega}$ to all of $\Omega$. More precisely, a continuous function $S:\overline{\Omega}\to\widehat{\C}$ is called a Schwarz function of $\Omega$ if it satisfies the following two properties:
\begin{enumerate}
\item $S$ is meromorphic on $\Omega$,

\item $S(z)=\overline{z}$ on $\partial \Omega$.
\end{enumerate}
\end{definition}

It is easy to see from the definition that a Schwarz function of a domain (if it exists) is unique. 

\begin{definition}[Quadrature Domains]
A domain $\Omega\subsetneq\widehat{\C}$ with $\infty\notin\partial\Omega$ and $\Int{\overline{\Omega}}=\Omega$ is called a \emph{quadrature domain} if $\Omega$ admits a Schwarz function.
\end{definition}

Therefore, for a quadrature domain $\Omega$, the map $\sigma:\overline{\Omega}\to\widehat{\C},\ z\mapsto\overline{S(z)}$ is an anti-meromorphic extension of the Schwarz reflection map with respect to $\partial \Omega$ (the reflection map fixes $\partial\Omega$ pointwise). We will call $\sigma$ the \emph{Schwarz reflection map of} $\Omega$.

Simply connected quadrature domains are of particular interest, and these admit a simple characterization (see \cite[Theorem~1]{AS}).

\begin{proposition}[Simply Connected Quadrature Domains]\label{simp_conn_quad}
A simply connected domain $\Omega\subsetneq\widehat{\C}$ with $\infty\notin\partial\Omega$ and $\Int{\overline{\Omega}}=\Omega$ is a quadrature domain if and only if the Riemann uniformization $\phi:\mathbb{D}\to\Omega$ extends to a rational map on $\widehat{\C}$. 

In this case, the Schwarz reflection map $\sigma$ of $\Omega$ is given by $\phi\circ(1/\overline{z})\circ(\phi\vert_{\D})^{-1}$. Moreover, if the degree of the rational map $\phi$ is $d$, then $\sigma:\sigma^{-1}(\Omega)\to\Omega$ is a (branched) covering of degree $(d-1)$, and $\sigma:\sigma^{-1}(\Int{\Omega^c})\to\Int{\Omega^c}$ is a (branched) covering of degree $d$.
\end{proposition}

\begin{figure}[ht!]
\centering
\includegraphics[scale=0.28]{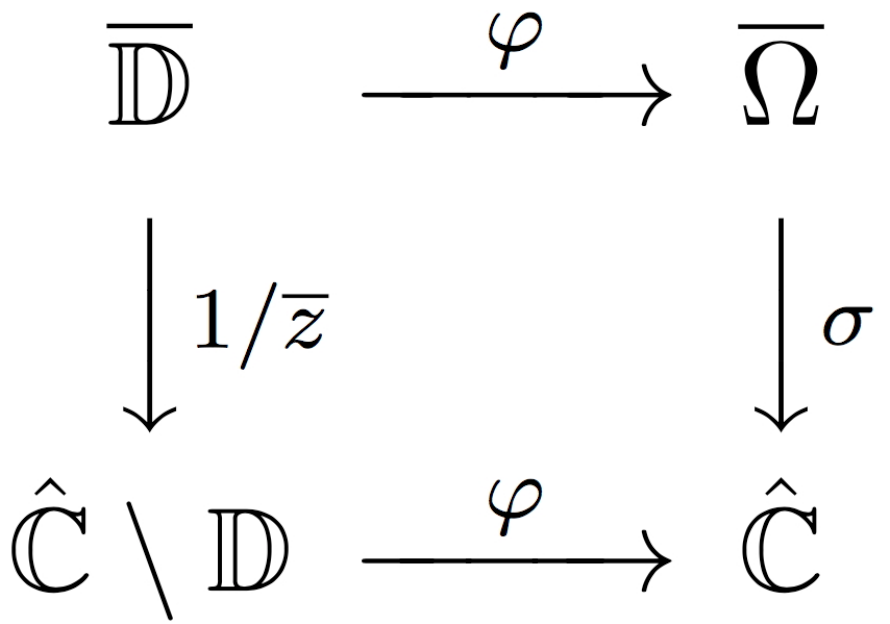}
\caption{The rational map $\phi$ semi-conjugates the reflection map $1/\overline{z}$ of $\D$ to the Schwarz reflection map 
$\sigma$ of $\Omega$ .}
\label{comm_diag_schwarz}
\end{figure}

\begin{remark}\label{prelim_quad}
1) For a simply connected quadrature domain, the Riemann map $\phi$ semi-conjugates the reflection map $1/\overline{z}$ of the unit disk to the Schwarz reflection map $\sigma$ of $\Omega$ (see Figure~\ref{comm_diag_schwarz}). This yields an explicit description of $\sigma$.

2) If $\Omega$ is a simply connected quadrature domain with associated Schwarz reflection map $\sigma$, and $M$ is a M{\"o}bius transformation, then $M(\Omega)$ is also a quadrature domain with Schwarz reflection map $M\circ\sigma\circ M^{-1}$.
\end{remark}

\subsection{Classification of One-parameter Families of Quadratic Schwarz Reflection Maps}\label{classi_subsec}
In this subsection, we will classify one-parameter families of Schwarz reflection maps of ``degree" two associated with disjoint unions of simply connected quadrature domains. This will lead to the family of Schwarz reflection maps that is the principal object of study of this paper. 

Consider a finite collection of disjoint simply connected quadrature domains $\Omega_j(\subsetneq\widehat{\C})$ ($j=1,\cdots,k$) with associated Schwarz reflection maps $\sigma_j$. We define $\displaystyle\Omega:=\displaystyle\bigsqcup_{j=1}^k\Omega_j$, and the Schwarz reflection map 
$$
\sigma:\overline{\Omega}\to\widehat{\C},\ w \mapsto \begin{array}{ll}
                    \sigma_j(w) & \mbox{if}\ w\in\overline{\Omega}_j
                                          \end{array}. 
$$

Let us also set $T_j:=\widehat{\C}\setminus\Omega_j$, $T:=\widehat{\C}\setminus\Omega$, and $T^0:=T\setminus\{$ Singular points on $\partial T\}$. We define the \emph{tiling set} $T^\infty$ of $\sigma$ as the set of all points that eventually land in $T^0$; i.e., 
$$
T^\infty:=\displaystyle\bigcup_{n=0}^\infty \sigma^{-n}(T^0).
$$ 
The \emph{non-escaping set} of $\sigma$ is its complement $\widehat{\C}\setminus T^\infty$.

Let $\phi_j:\overline{\D}\to\overline{\Omega}_j$ be the Riemann uniformizations of the simply connected quadrature domains $\Omega_j$ such that each $\phi_j$ extends as a rational map of $\widehat{\C}$ of degree $d_j$. It follows that $\sigma_j:\sigma_j^{-1}(\Omega_j)\to\Omega_j$ is a branched covering of degree $(d_j-1)$, and $\sigma_j:\Int{\sigma_j^{-1}(T_j)}\to\Int{T_j}$ is a branched covering of degree $d_j$. Therefore, $\sigma:\sigma^{-1}(\Omega)\to\Omega$ is a (possibly branched) covering of degree $\displaystyle(\sum_{j=1}^k d_j-1)$.

We will now focus on the case when $\sigma:\sigma^{-1}(\Omega)\to\Omega$ has degree $2$; i.e., 
$$
\displaystyle\sum_{j=1}^k d_j-1=2\implies\displaystyle\sum_{j=1}^k d_j=3.
$$ 
Since each $d_j\geq 1$, we have that $k\leq 3$. We also restrict our attention to families of Schwarz reflection maps for which $\sigma:\sigma^{-1}(\Omega)\to\Omega$ has at least one critical point, and such that the connectedness loci of the families are non-empty.
\bigskip

\noindent\textbf{Case 1: $k=3$.} In this case, each $\phi_j$ is a M{\"o}bius map, and hence each $\Omega_j$ is a round disk. In particular, each $\sigma_j$ is the reflection in a round circle (thus, $\sigma$ has no critical point), and the resulting dynamics of $\sigma$ is completely understood.
\vspace{2mm}

\noindent\textbf{Case 2: $k=2$.} We can assume that $\Deg{\phi_1}=2$, and $\Deg{\phi_2}=1$. Note that post-composing $\phi_1$ and $\phi_2$ with a (common) M{\"o}bius map does not alter the conformal conjugacy class of $\sigma$ (see Remark~\ref{prelim_quad}). Hence, after a non-dynamical change of coordinates (i.e., different change of coordinates in the domain and the range), we can assume that $\phi_1(w)=w^2$, $\Omega_1$ is the univalent image of a round disk of the form $B(1,r)$ (where $r\in(0,1]$) under $\phi_1$, and $\Omega_2$ is a round disk in $\widehat{\C}$. If the boundaries of the domains $\Omega_1$ and $\Omega_2$ are disjoint, then the non-escaping set of $\sigma$ is necessarily disconnected. To avoid this, we will assume that $\partial\Omega_1\cap\partial\Omega_2$ is a singleton.
\vspace{2mm}

\noindent\textbf{Subcase 2.1: $r\in(0,1).$} The moduli space of this family has real dimension $3$.
\vspace{2mm}

\noindent\textbf{Subcase 2.2: $r=1$.} In this case, $\partial\Omega_1$ has a singularity (which is a simple cusp at $0$), and $\Omega_1$ is a cardioid. Now if $\Omega_2\subset\C$, then $\sigma:\sigma^{-1}(\Omega)\to\Omega$ has no critical point (the only critical point of $\sigma$ in this case is $\infty$), and the non-escaping set of $\sigma$ is necessarily disconnected (more precisely, $\sigma^{-1}(\Omega_2)$ is the disjoint union of two domains each of which maps univalently onto $\Omega_2$, and these two domains disconnect the non-escaping set). So we may assume that $\Omega_2$ is an exterior disk; i.e., $\partial\Omega_2$ is a circumcircle of $\partial\Omega_1$ with a single point of intersection. Therefore, up to conformal conjugacy, the moduli space of such maps is obtained by fixing a cardioid as $\Omega_1$, and varying the center of the exterior disk $\Omega_2$ that touches $\Omega_1$ at a unique point. This one-parameter family has been studied in our earlier work \cite{LLMM1,LLMM2}.
\vspace{1mm}

\noindent\textbf{Case 3: $k=1$.} In this case, $\Omega=\Omega_1$ is a singe quadrature domain that is the univalent image of $\D$ under some cubic rational map. We need to consider three cases here.
\vspace{2mm}

\noindent\textbf{Subcase 3.1.} Suppose that the rational map $\phi_1$ has two double critical points. Then under a non-dynamical change of coordinates, we can assume that $\phi_1(w)=w^3$. Remark~\ref{prelim_quad} now implies that up to conformal equivalence, $\sigma$ is the Schwarz reflection map of $\Omega$, where $\Omega$ is the univalent image of some round disk under $\phi_1$. It now easily follows from the commutative diagram in Figure~\ref{comm_diag_schwarz} that in this case, $\sigma^{-1}(\Omega)$ is the disjoint union of two topological disks compactly contained in $\Omega$, and $\sigma$ maps each of these two disks univalently onto $\Omega$ (in particular, $\sigma:\sigma^{-1}(\Omega)\to\Omega$ has no critical point). Moreover, $\sigma:\sigma^{-1}(\Omega)\to\Omega$ is an expanding map, and the corresponding non-escaping set is a Cantor set.
\vspace{2mm}

\noindent\textbf{Subcase 3.2.} Now suppose that the rational map $\phi_1$ has a unique double critical point. Then under a non-dynamical change of coordinates, we can assume that $\phi_1(w)=w^3-3w$; and up to conformal equivalence, $\sigma$ is the Schwarz reflection map of $\Omega$, where $\Omega$ is the univalent image of a round disk under $\phi_1$. Suitably varying the round disk (whose image under $\phi_1$ is $\Omega$) now leads to a one-parameter family of Schwarz reflection maps. The current paper is dedicated to the study of this family.
\vspace{1mm}

\noindent\textbf{Subcase 3.3.} In the last remaining case, $\phi_1$ is a rational map with four simple critical points. A specific example of this type of quadrature domains is the exterior of a deltoid, whose associated Schwarz reflection map was studied in \cite[\S 5]{LLMM1}. The full moduli space of such Schwarz reflection maps arises from all possible univalent images of round disks in $\widehat{\C}$ under a one-parameter family of cubic rational maps with fixed critical points at $0, 1$, and $\infty$. Thus, the moduli space has real dimension $5$.

\section{A Family of Schwarz Reflections}\label{family_schwarz}

The main goal of this paper is to study the dynamics and parameter plane of the family of Schwarz reflection maps arising from Subcase 3.2 above.

Let $f(u)=u^3-3u$. Note that the map $f$ is the cubic Chebychev polynomial. In particular, the finite critical points $\pm1$ of $f$ map to the repelling fixed points $\mp2$ in one iterate; i.e., the critical orbits are given by $\pm 1\rightarrow \mp 2\circlearrowleft$. Moreover, $f:\left[-2,2\right]\to\left[-2,2\right]$ is a triple branched cover such that each of the intervals $\left[-2,-1\right], \left[-1,1\right],$ and $\left[1,2\right]$ map monotonically onto $\left[-2,2\right]$. It follows that the Julia set of $f$ is $\left[-2,2\right]$. These properties of the map $f$ will be useful in what follows.

\subsection{Univalence Properties of The Cubic Chebychev Polynomial}\label{univalence_sec}

We will consider suitable maximal disks on which $f$ is univalent. Clearly, such a disk must not contain any critical point of $f$. We will focus on the case where the disk has exactly one critical point on the boundary. 

To this end, consider $a\in\C$ with $\re(a)>0,\ a\neq1$, and set $\Delta_a:=B(a,\vert a-1\vert)$. We define 
$$
\widehat{S}:=\{a\in\C: \re(a)>0,\ a\neq1,\ \mathrm{and}\ f(\partial \Delta_a)\ \mathrm{is\ a\ Jordan\ curve}\}.
$$ 
Note that when $\re(a)>0$, none of the critical points of $f$ lies in $\Delta_a$, and hence, the requirement that $f(\partial\Delta_a)$ is a Jordan curve implies that $f$ is univalent on $\Delta_a$. For $a\in\widehat{S}$, we set $\Omega_a:=f(\Delta_a)$. Since $\partial\Delta_a$ contains the simple critical point $1$ of $f$, its image $\partial\Omega_a$ has a conformal cusp at $f(1)=-2$. 

\begin{proposition}\label{univalent_disk}
$$
\widehat{S}\subset\{a\in\C:0<\re(a)\leq4,\ a\neq1\}.
$$
\end{proposition}
\begin{proof}
Let us fix $a\in\widehat{S}$, and set $a=1+\vert a-1\vert e^{i\theta_0}$, for some $\theta_0\in(-\pi,\pi]$. Then, 
$$
f(1+\epsilon e^{i\theta_0})=-2+\epsilon^2 e^{2i\theta_0}(3+\epsilon e^{i\theta_0})\in\Omega_a,
$$
for $\epsilon>0$ sufficiently small. It follows that $(-2+\delta e^{2i\theta_0})\in\Omega_a$, for $\delta>0$ sufficiently small.

We parametrize the circle $\partial\Delta_a$ as $\{z(t):=a+(1-a)e^{it}:-\pi\leq t\leq\pi\}$. Then, $z(0)=1$, and $z(\pi)=z(-\pi)=2a-1$. A straightforward computation shows that 
\begin{align*}
f(z(t))
&=z(t)^3-3z(t)\\
&=(a^3-3a)-3e^{it}(1+a)(a-1)^2+3ae^{2it}(a-1)^2-e^{3it}(a-1)^3\\
&=-2-3e^{2i\theta_0}\vert a-1\vert^2t^2+ie^{2i\theta_0}\vert a-1\vert^2(a-4)t^3+O(t^4).
\end{align*}

\begin{figure}[ht!]
\begin{center}
\includegraphics[scale=0.215]{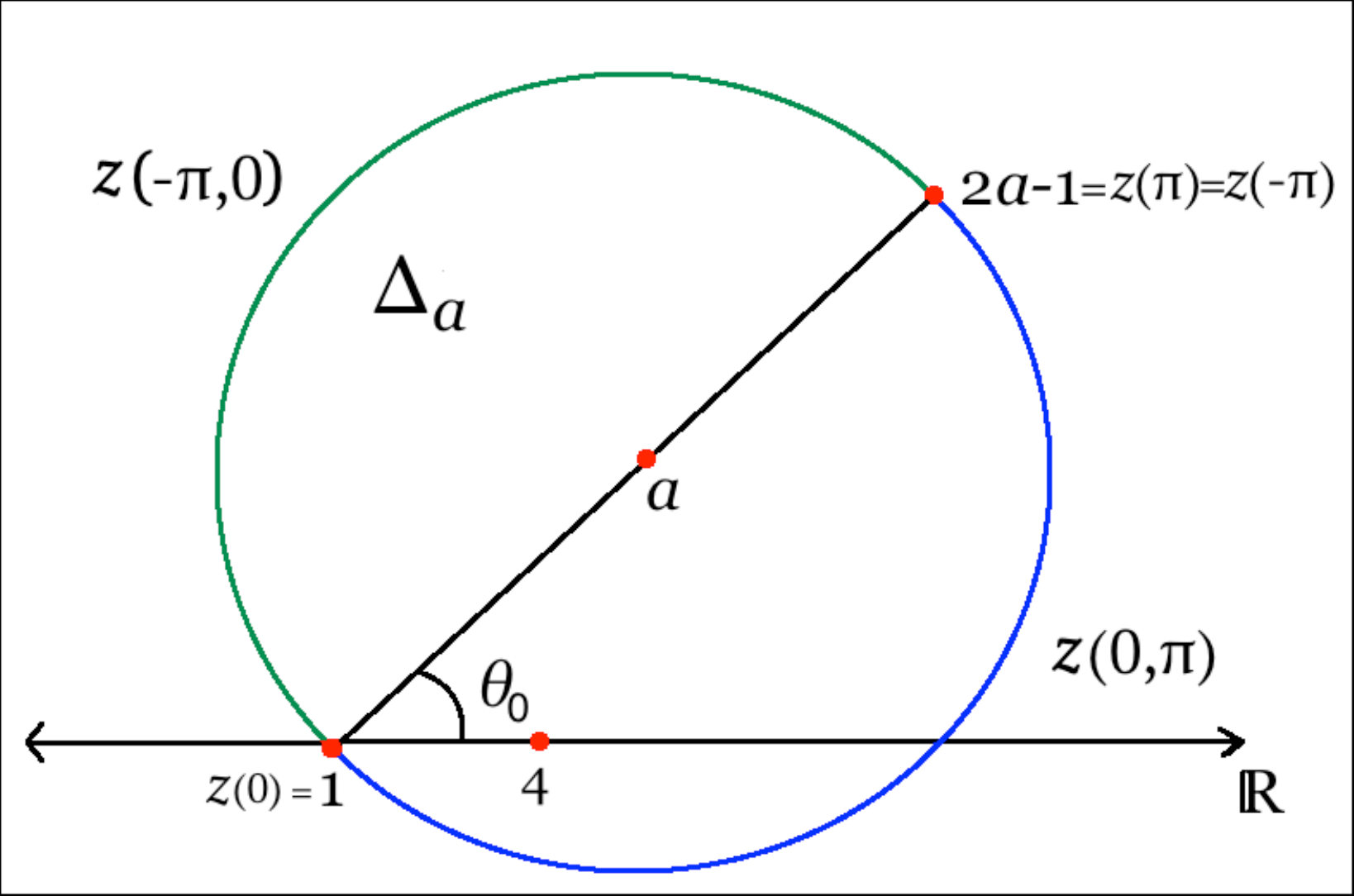}\ \includegraphics[scale=0.12]{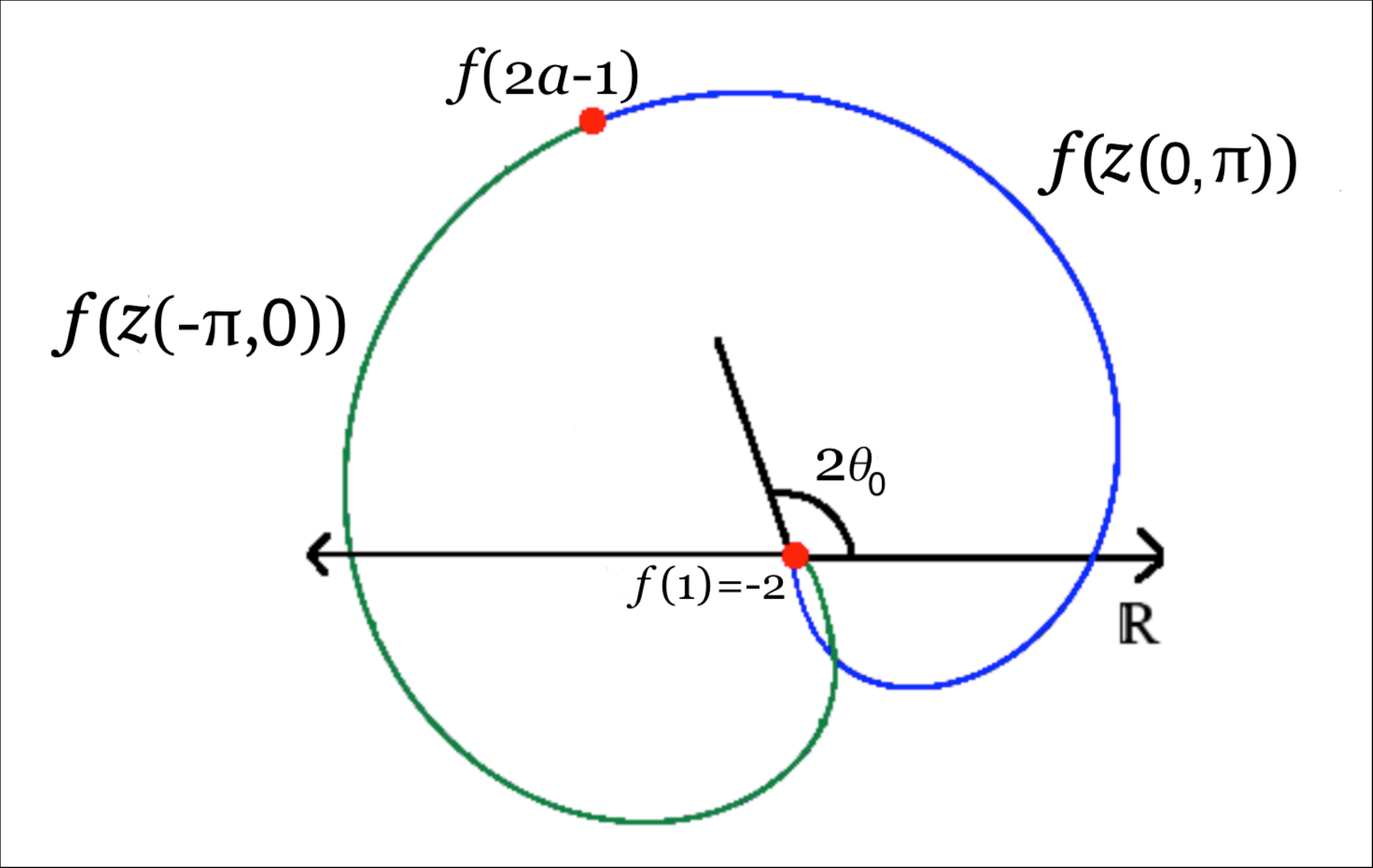}
\end{center}
\noindent\caption{Left: The disk $\Delta_a$ for some parameter $a$ with $\re(a)>4$ is shown. The diameter connecting $1$ and $2a-1$ makes an angle $\theta_0$ with the positive real axis. The green (respectively, blue) part of the boundary circle $\partial\Delta_a$ is $z(-\pi,0)$ (respectively, $z(0,\pi)$. Right: The image of $\Delta_a$ under $f$ with the image of $z(-\pi,0)$ (respectively, $z(0,\pi)$) shaded in green (respectively, blue). Near the cusp point $-2$, these two smooth branches of $f(\partial\Delta_a)\setminus\{f(2a-1)\}$ cross.}
\label{crossing_pic}
\end{figure}

By way of contradiction, let us assume that $\re(a)>4$; i.e., $a=4+p+iq$, for some $p>0$, and $q\in\R$. The above computation yields that 
\begin{equation}
f(z(t))=-2-e^{2i\theta_0}\vert a-1\vert^2(3t^2+qt^3)+ipe^{2i\theta_0}\vert a-1\vert^2t^3+O(t^4).
\label{p_positive}
\end{equation}

For the remainder of the proof, we will choose the argument of a complex number $z$ in the interval $[2\theta_0, 2\theta_0+2\pi)$, and denote it by $\arg{z}$. Relation~\ref{p_positive} shows that for $t>0$ small enough, $\arg{(f(z(t))+2)}\in(2\theta_0,2\theta_0+\pi)$; while for $t<0$ small enough, $\arg{(f(z(t))+2)}\in(2\theta_0+\pi,2\theta_0+2\pi)$. If $f\vert_{\Delta_a}$ is univalent with $f(\partial\Delta_a)$ a Jordan curve, then $f$ will be orientation-preserving on $\partial\Delta_a$. But this implies that $f(\partial\Delta_a)$ must have a self-crossing; i.e., $f(z(-\pi,0))$ and $f(z(0,\pi))$ must intersect (see Figure~\ref{crossing_pic}). This contradicts the fact that $f(\partial\Delta_a)$ is a Jordan curve, and completes the proof.
\end{proof}

\begin{proposition}\label{univalence_non_empty}
$\widehat{S}\cap\R=(0,4]\setminus\{1\}$. In particular, $\widehat{S}\neq\emptyset$.
\end{proposition}
\begin{figure}[ht!]
\begin{center}
\includegraphics[scale=0.1]{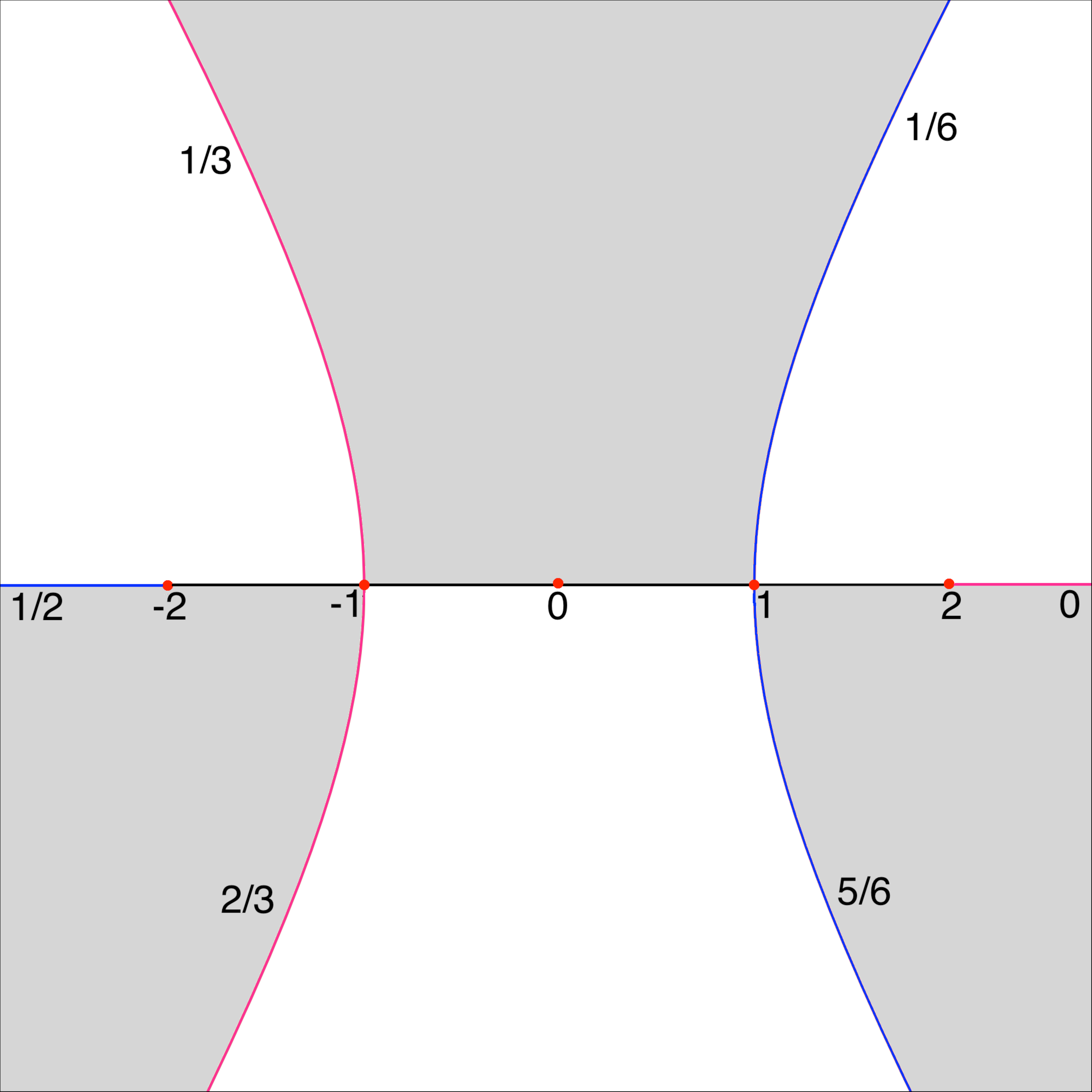}
\end{center}
\noindent\caption{The dynamical plane of $f$ with the critical points $\pm 1$ and the critical values $\pm 2$ marked on the Julia set. The dynamical rays at angles $1/6, 5/6$ (which land at $1$), and $1/2$ (which lands at $-2$) are shown in blue; while the dynamical rays at angles $1/3, 2/3$ (which land at $-1$), and $0$ (which lands at $2$) are shown in red. The closure of every white (respectively, gray) component maps univalently onto the closed upper (respectively, lower) half-plane under $f$.}
\label{f_dyn_pic}
\end{figure}
\begin{proof}
Note that the map $h(w)=w+\frac1w$ is a conformal isomorphism from the basin of infinity $\widehat{\C}\setminus\overline{\D}$ of the polynomial $g(w)=w^3$ onto the basin of infinity $\widehat{\C}\setminus\left[-2,2\right]$ of $f(u)=u^3-3u$. Moreover, $h$ conjugates $g$ to $f$ on their respective basins of infinity. In other words, $h$ is the B{\"o}ttcher coordinate for the basin of infinity of $f$, normalized to be tangent to the identity at $\infty$. We define the \emph{external ray} of $f$ at angle $\theta$ to be the image of the radial ray at angle $\theta$ in $\C\setminus\overline{\D}$ under $h$.

It clearly follows from the definition of external rays that the $0$-ray (respectively, the $1/2$-ray) of $f$ lands at $2$ (respectively, at $-2$). Pulling back the $0$ and $1/2$-rays under $f$, it is easy to see that the $1/3$ and $2/3$-rays (respectively, the $1/6$ and $5/6$-rays) of $f$ land at $-1$ (respectively, at $1$). Moreover, the explicit description of the external rays (of $f$) given in the previous paragraph also shows that the union of the (closures of the) $1/6$ and $5/6$-rays of $f$ are given by the branch of the hyperbola $x^2-\frac{y^2}{3}=1$ in the right half-plane.

There are six unbounded complementary components of the union of the fixed and pre-fixed rays of $f$ and its Julia set $\left[-2,2\right]$. It is easy to see from the mapping properties of $f$ that $f$ is univalent on each such unbounded component. More precisely, the closure of every white (respectively, gray) component shown in Figure~\ref{f_dyn_pic} maps univalently onto the closed upper (respectively, lower) half-plane under $f$.

A simple computation now shows that for $a\in(0,4]\setminus\{1\}$, the closed disk $\overline{\Delta_a}$ does not intersect the $1/6$ and $5/6$-rays of $f$. It follows that $f$ is injective on $\overline{\Delta_a}$; i.e., $a\in\widehat{S}$. This completes the proof.
\end{proof}

Note that the boundary of $\widehat{S}$ in $\{0<\re(a)\leq4, a\neq1\}$ is contained in the set of parameters for which either $f(\partial\Delta_a)$ has a tangential self-intersection, or $-2$ is a higher order cusp (i.e., not a simple cusp) of $\partial\Omega_a$.

A direct (but tedious) computation shows that the locus of parameters in $\{0<\re(a)\leq 4, a\neq1\}$ for which $f(\partial\Delta_a)$ has a tangential self-intersection is the union of two real-symmetric arcs $\mathfrak{T}^\pm$ that are contained in the upper (respectively, lower) half-plane. In fact, we have 
$$
\mathfrak{T}^+=\{a=x+iy: x>0, y>\sqrt{3}, R(x,y)=0\},
$$ 
where $R(x,y)=x^6+3x^4y^2+3x^2y^4+y^6-6x^5-12x^3y^2-6xy^4+6x^4-6y^4+16x^3-30xy^2-12x^2-15y^2-24x-8$.\footnote{We would like to thank Bernhard Reinke for helping us with this computation.} 

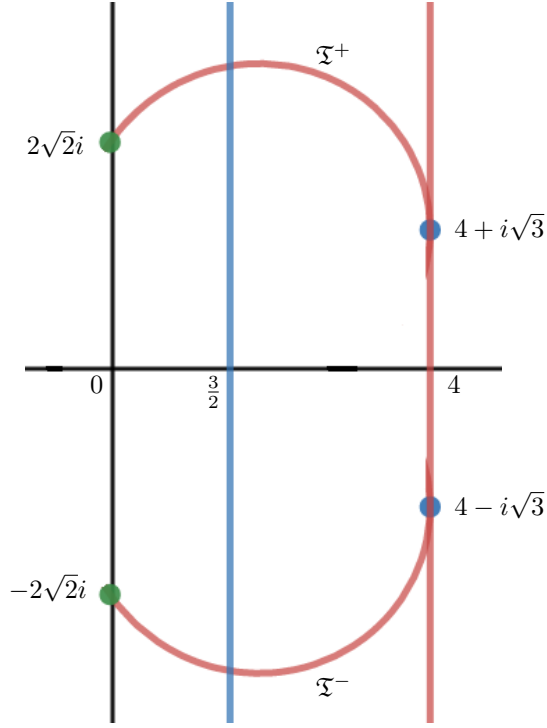
\begin{figure}[ht!]
\begin{tikzpicture}
  \node[anchor=south west,inner sep=0] at (2.5,0) {\includegraphics[width=0.5\textwidth]{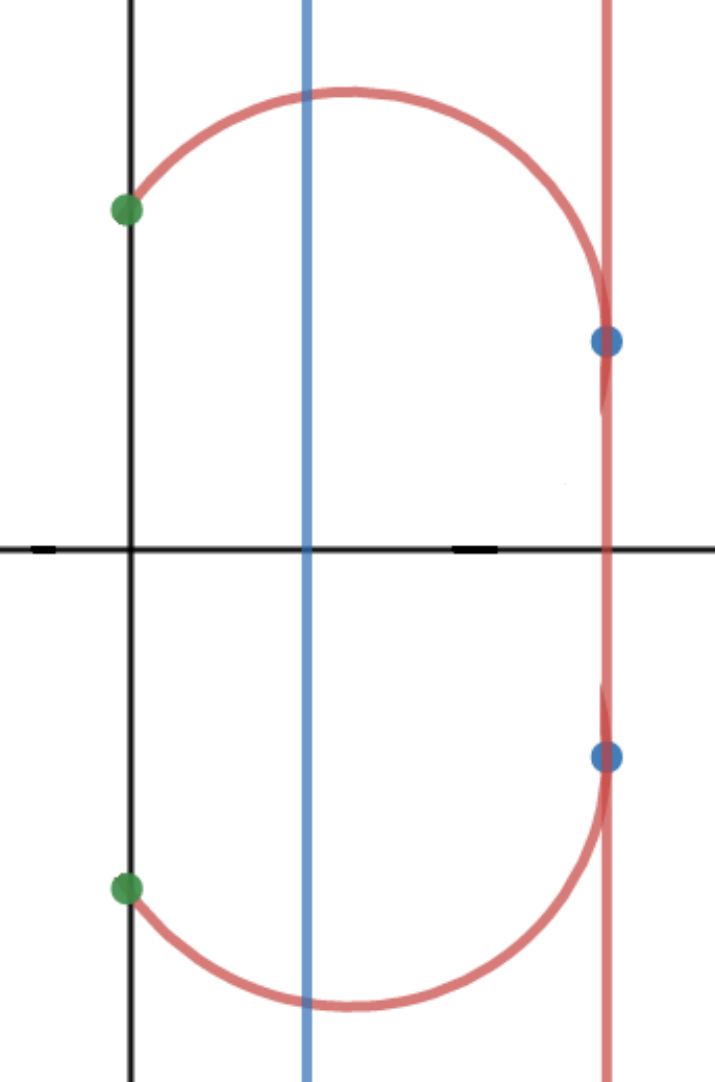}};
  \node at (3.45,4.5) {$0$};
  \node at (5,4.4) {$\frac32$};
  \node at (8.2,4.5) {$4$};
  \node at (8.8,6.6) {$4+i\sqrt{3}$};
  \node at (8.8,2.9) {$4-i\sqrt{3}$};
  \node at (2.9,7.7) {$2\sqrt{2}i$};
  \node at (2.8,1.8) {$-2\sqrt{2}i$};
  \node at (6.6,8.9) {$\mathfrak{T}^+$};
  \node at (6.6,0.5) {$\mathfrak{T}^-$}; 
  \end{tikzpicture}
\caption{The interior of $\widehat{S}$ is the open region bounded by the vertical line segments $\{\re(a)=0, \vert\im(a)\vert\leq 2\sqrt{2}\}$, $\{\re(a)=4,\vert\im(a)\vert\leq\sqrt{3}\}$, and the curves $\mathfrak{T}^\pm$ (with a puncture at the point $1$). $\widehat{S}$ is the union of $\Int{\widehat{S}}$ and the vertical line segment $\{\re(a)=4,\vert\im(a)\vert\leq\sqrt{3}\}$. Our parameter space $S$ is the open subset of $\widehat{S}$ lying between the vertical lines $\{\re(a)=\frac32\}$ and $\{\re(a)=4\}$.}
\label{fig:para_space_boundary}
\end{figure}

It is easy to see that $\mathfrak{T}^+$ is a real-algebraic arc connecting $2\sqrt{2}i$ and $4+ i\sqrt{3}$. In fact, $\mathfrak{T}^+$ is completely contained in $\{0<\re(a)<4, \im(a)>\sqrt{3}\}$ (see Figure~\ref{fig:para_space_boundary}).

The proof of Proposition~\ref{univalent_disk} implies that the set of parameters in $\{0<\re(a)\leq4,\ a\neq1\}$ for which $-2$ is a higher order cusp of $\partial\Omega_a$ is $\{\re(a)=4\}$. It follows that
\begin{equation}
\partial\widehat{S}=\{\re(a)=0, \vert\im(a)\vert\leq 2\sqrt{2}\}\cup\mathfrak{T}^\pm\cup\{\re(a)=4,\vert\im(a)\vert\leq\sqrt{3}\}.
\label{para_space_boundary}
\end{equation}
Also, $\mathfrak{T}^\pm\cap\widehat{S}=\emptyset$, and $\{\re(a)=4, \vert\im(a)\vert\leq\sqrt{3}\}\subset\widehat{S}$.

\textbf{Notation}: We will denote the closed interval $\{\re(a)=4,\vert\im(a)\vert\leq\sqrt{3}\}$ by $\mathcal{I}$, and the open interval $\{\re(a)=4,\vert\im(a)\vert<\sqrt{3}\}$ by $\mathcal{I}^0$.

\subsection{The Corresponding Schwarz Reflections and Their Critical Points}\label{setup}

We will now assume that $a\in\widehat{S}$.

Let $T_a=\widehat{\C}\setminus\Omega_a$. Since 
$$
f\vert_{\Delta_a}:=f:\Delta_a\to\Omega_a
$$ 
is a biholomorphism, Proposition~\ref{simp_conn_quad} implies that $\Omega_a$ is a quadrature domain.

We will denote the reflection in the circle $\partial\Delta_a$ by $\iota_a$.

It will be convenient to introduce the new coordinate $z=\frac{u-a}{1-a}$ to study the Schwarz reflection map of $\Omega_a$. The disk $B(a,\vert 1-a\vert)$ in the $u$-coordinate becomes the unit disk $\D$ in the $z$-coordinate.

Define $f_a(z):=f(u)=f(a+(1-a)z)$. Note that $f_a(0)=f(a)$, and $f_a(1)=-2$. Moreover, the critical points of $f_a$ are $1, \frac{a+1}{a-1}$, and $\infty$ with associated critical values $-2, 2$, and $\infty$ respectively. By our choice of $a$, the critical point $\frac{a+1}{a-1}$ of $f_a$ lies outside the unit disk. Since $f_a:\D\to\Omega_a$ is univalent, the Schwarz reflection map $\sigma_a$ of $\Omega_a$ is given by $f\circ\iota_a\circ \left(f\vert_{\overline{\Delta}_a}\right)^{-1}=f_a\circ\iota\circ \left(f_a\vert_{\overline{\D}}\right)^{-1}$, where $\iota$ is the reflection in the unit circle, and $\iota_a$ is the reflection in the circle $\partial\Delta_a$.

\begin{figure}[ht!]
\begin{tikzpicture}
  \node[anchor=south west,inner sep=0] at (0,0) {\includegraphics[width=0.4\textwidth]{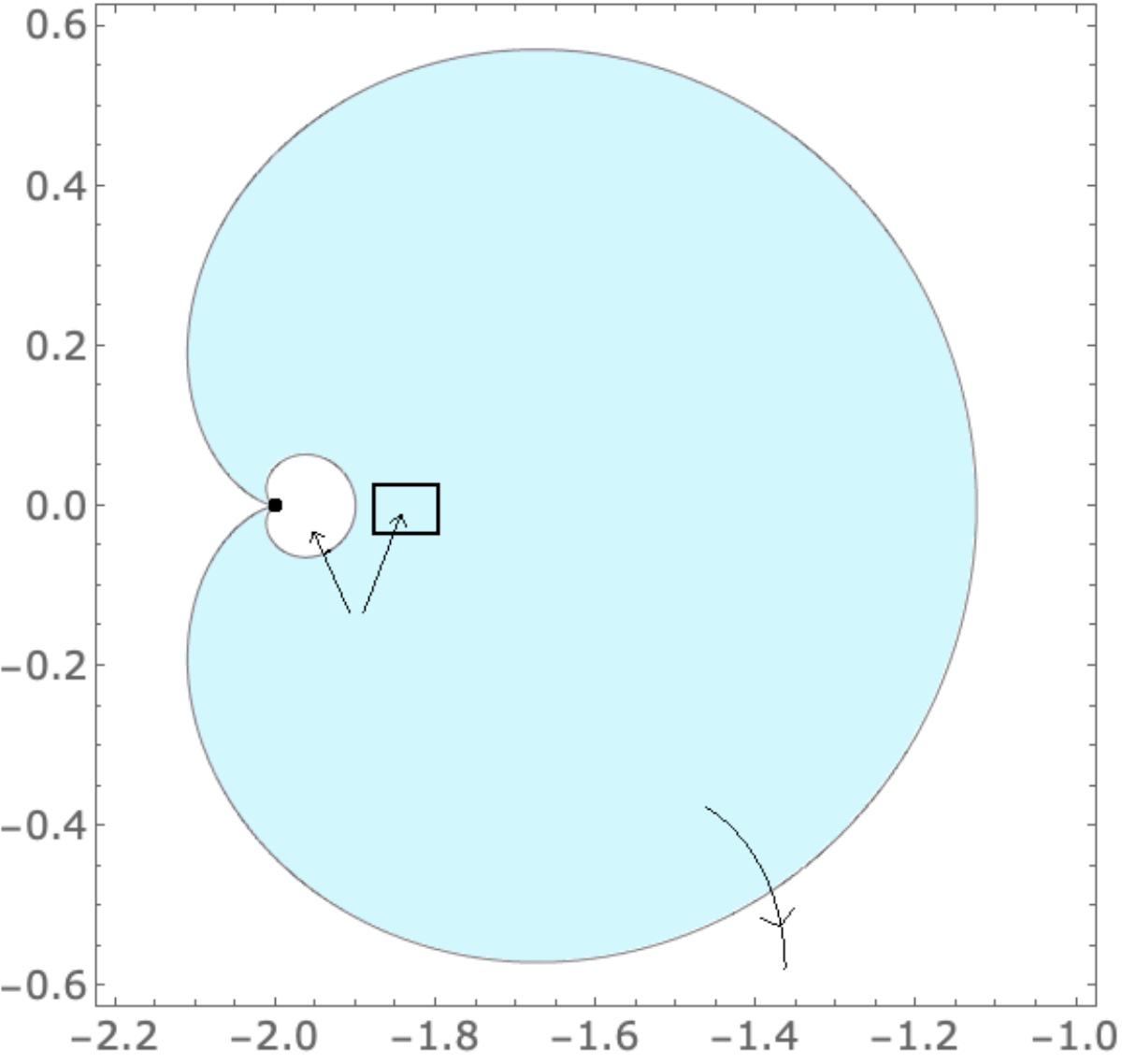}};
\node[anchor=south west,inner sep=0] at (5.5,0) {\includegraphics[width=0.58\textwidth]{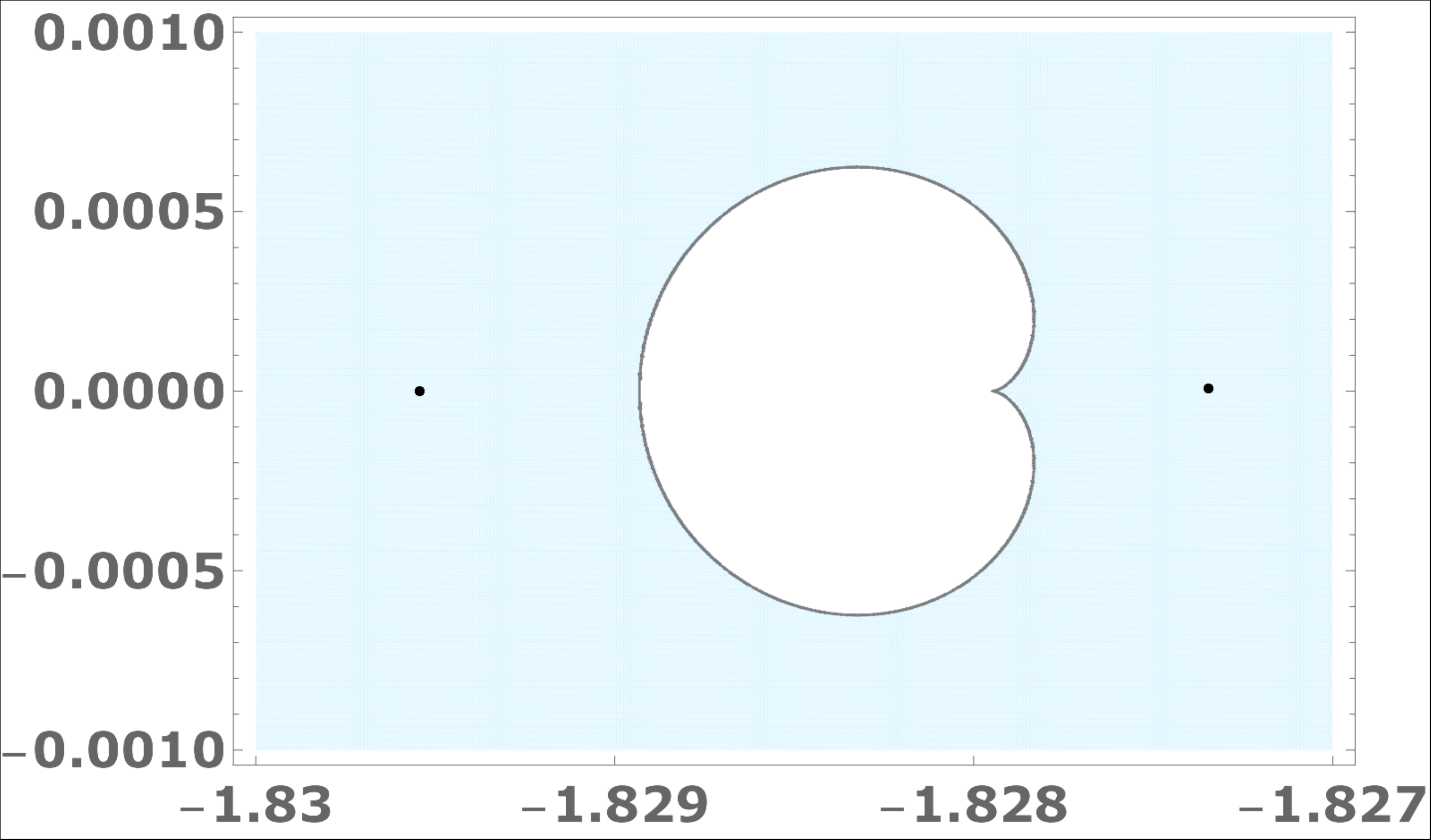}};
  \node at (2.8,3.6) {$\sigma_a^{-1}(\Int{T_a})$};
   \node at (2,1.8) {\begin{small}$\sigma_a^{-1}(\Omega_a)$\end{small}};
  \node at (4.4,4.5) {$\Int{T_a}$};
    \node at (4.2,0.9) {$3:1$};
  \node at (7.7,2.1) {$c_a$};
  \node at (11.8,2.1) {$c_a^*$};
  \node at (0.8,2.5) {\begin{tiny}$-2$\end{tiny}};
  \node at (9.8,2.2) {\begin{tiny}$\sigma_a^{-1}(\Omega_a)$\end{tiny}};
  \end{tikzpicture}
\noindent\caption{Left: For parameters $a\in\widehat{S}$ with $0<\re(a)<\frac32$, the unique tile of rank one (in blue) is a conformal annulus which contains the double critical point $c_a^\ast$ and the simple critical point $c_a$ (of $\sigma_a$). The Schwarz reflection map $\sigma_a$ maps $\sigma_a^{-1}(\Int{T_a})$ as a $3:1$ branched covering onto $\Int{T_a}$. On the other hand, $\Omega_a'=\sigma_a^{-1}(\Omega_a)$ consists of two simply connected components each of which is mapped univalently onto $\Omega_a$ by $\sigma_a$. One of these components, which is visible in the picture, has the cusp point $-2$ on its boundary. The other component, which is too small to be seen in this scale, lies in the box shown. Right: A blow-up of the box shows the other component of $\Omega_a'$.}
\label{schwarz_dyn_pic_less_3_2}
\end{figure}
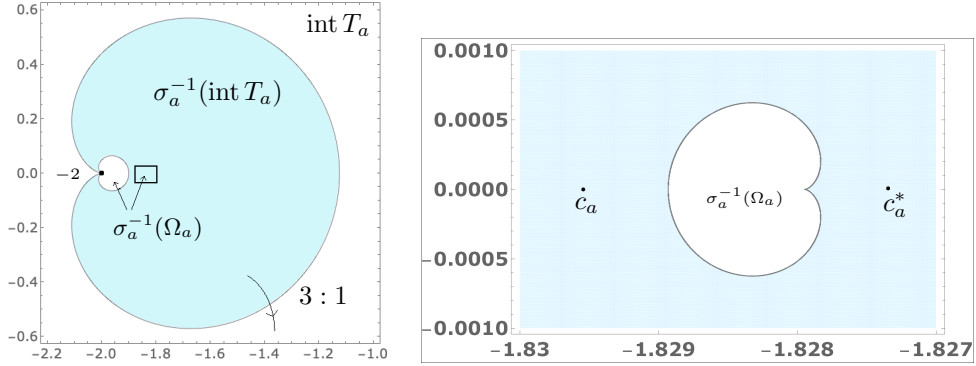

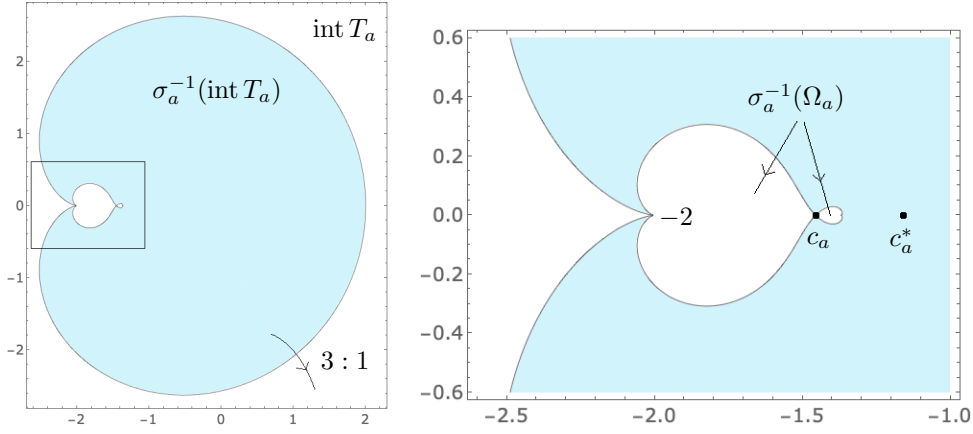
\begin{figure}[ht!]
\begin{tikzpicture}
  \node[anchor=south west,inner sep=0] at (0,0) {\includegraphics[width=0.4\textwidth]{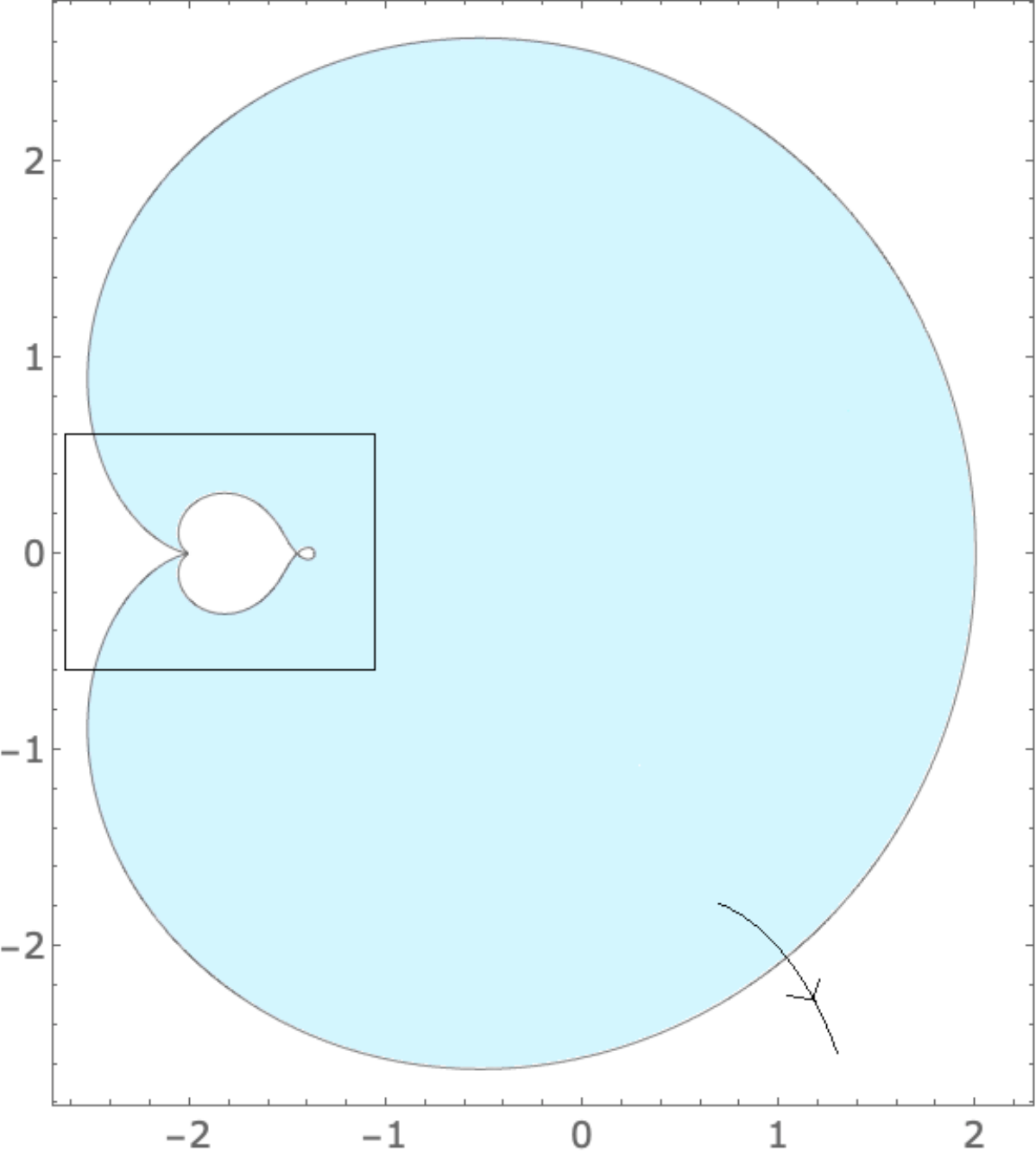}};
\node[anchor=south west,inner sep=0] at (5.5,0) {\includegraphics[width=0.58\textwidth]{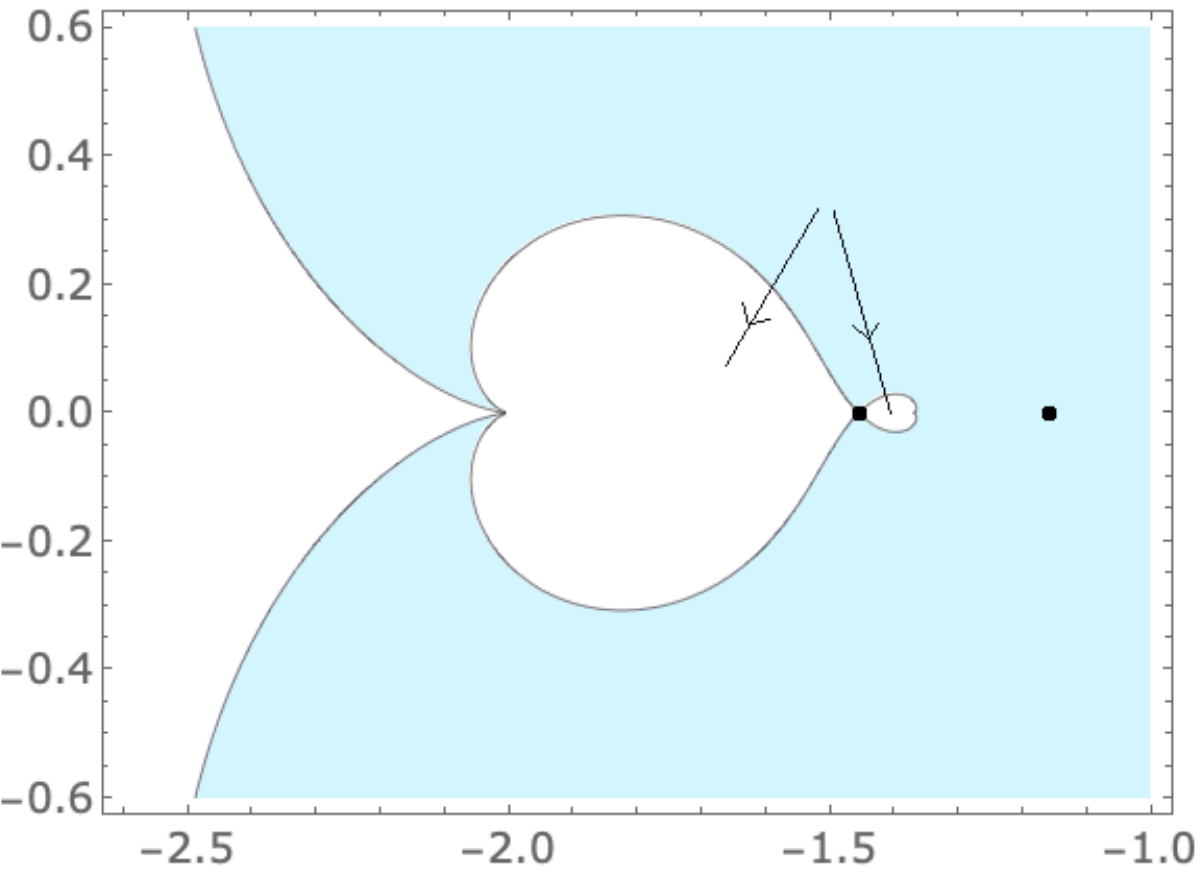}};
  \node at (2.8,4.5) {$\sigma_a^{-1}(\Int{T_a})$};
  \node at (4.5,5.3) {$\Int{T_a}$};
    \node at (4.5,0.9) {$3:1$};
  \node at (10.8,2.5) {$c_a$};
  \node at (11.9,2.5) {$c_a^*$};
  \node at (8.9,2.8) {$-2$};
  \node at (10.5,4.4) {$\sigma_a^{-1}(\Omega_a)$};
  \end{tikzpicture}
\noindent\caption{For parameters $a\in\widehat{S}$ with $\re(a)=\frac32$, the boundary of the unique tile of rank one (in blue) is pinched at the simple critical point $c_a$ (of $\sigma_a$). In particular, this tile is not simply connected (although its interior is so). On the other hand, $\Omega_a'=\sigma_a^{-1}(\Omega_a)$  consists of two simply connected domains, each of which is mapped univalently onto $\Omega_a$. In particular, $\Omega_a'$ does not contain any critical point of $\sigma_a$. (The figure on the right is a blow-up of the left figure around $\Omega_a'$.)}
\label{schwarz_dyn_pic_3_2}
\end{figure}

It follows from the definition of $\sigma_a$ that the map $\sigma_a$ has two distinct critical points; namely, $c=c_a:=f(\iota_a(-1))=f_a(\frac{\overline{a}-1}{\overline{a}+1})$ and $c^\ast=c_a^\ast:=f(\iota_a(\infty))=f_a(0)$. In fact, $c_a$ is a simple critical point, while $c_a^\ast$ is a double critical point. Moreover, 
$$
\sigma_a(c_a)=2,\ \mathrm{and}\ \sigma_a(c_a^\ast)=\infty.
$$ 

Since $f$ is a degree $3$ polynomial which sends $\Delta_a$ univalently onto $\Omega_a$, it follows that $\sigma_a:\Omega_a':=\sigma_a^{-1}(\Omega_a)\to\Omega_a$ is a two-to-one (possibly branched) covering. On the other hand, $\sigma_a:\sigma_a^{-1}(\Int{T_a})\to \Int{T_a}$ is a branched covering of degree three. (In the language of \cite[Lemma 4.1]{LM}, $\Omega_a$ is a bounded quadrature domain of order three.) 

Note that for each parameter $a\in\widehat{S}$, the set $\sigma_a^{-1}(\Int{T_a})$ contains the critical point $c_a^\ast$ of $\sigma_a$. Let us denote the set of parameters $a\in\widehat{S}$ for which the free critical point $c_a$ lies in $\Omega_a'$ by $X$. We will now give a precise description of $X$.
\begin{align*}
X
&=\{a\in\widehat{S}: c_a\in\Omega_a'\}=\{a\in\widehat{S}: \sigma_a\left(c_a\right)\in\Omega_a\}\\
&=\{a\in\widehat{S}: 2\in\Omega_a\}=\{a\in\widehat{S}: 2\in\Delta_a\}\\
&=\{a\in\widehat{S}: \vert a-2\vert<\vert a-1\vert\}=\{a\in\widehat{S}: \re(a)>\frac32\},
\end{align*}
where we used the fact that $f^{-1}(2)=\{-1,2\}$, and $-1\notin\Delta_a$ (in the second line of the above chain of equalities).

It now follows by the Riemann-Hurwitz formula that for parameters $a\in\widehat{S}$ with $\re(a)>\frac32$, the set $\Omega_a'$ is a simply connected domain which maps onto $\Omega_a$ as a two-to-one branched cover (branched only at $c_a$) under $\sigma_a$ (see Figure~\ref{schwarz_dyn_pic}(left)). Moreover for such parameters, $\sigma_a^{-1}(\Int{T_a})$ is a simply connected domain which maps (under $\sigma_a$) onto $\Int{T_a}$ as a three-to-one branched cover branched only at $c_a^\ast$.

\begin{remark}
For $a\in\widehat{S}$ with $\re(a)\leq\frac32$, the set $\Omega_a'$ contains no critical point of $\sigma_a$, and hence is a union of two disjoint simply connected domains each of which maps univalently onto $\Omega_a$ under $\sigma_a$ (see Figures~\ref{schwarz_dyn_pic_less_3_2} and~\ref{schwarz_dyn_pic_3_2}).
\end{remark}

\subsection{Quasiconformal Deformations}\label{qcdef_sec} We will now prove a lemma that will allow us to talk about quasiconformal deformations of the Schwarz reflection maps under consideration.

\begin{proposition}[Quasiconformal Deformation of Schwarz Reflections]\label{schwarz_pre_qcdef}
Let $a\in\widehat{S}$ with $\re(a)>\frac32$, $\mu$ be a $\sigma_a$-invariant Beltrami coefficient on $\widehat{\C}$, and $\Phi$ be a quasiconformal map satisfying $\Phi_{\overline{z}}/\Phi_{z}=\mu$ a.e. such that $\Phi$ fixes $\pm 2$ and $\infty$. Then, there exists $b\in\widehat{S}$ with $\re(b)>\frac32$ such that $\Phi(\Omega_a)=\Omega_b$, and $\Phi\circ\sigma_a\circ\Phi^{-1}=\sigma_b$ on $\Omega_b$.
\end{proposition}
\begin{proof}
The assumption that $\mu$ is $\sigma_a$-invariant implies that $\Phi\circ\sigma_a\circ\Phi^{-1}$ is anti-meromorphic on $\check{\Omega}:=\Phi(\Omega_a)$ that continuously extends to the identity map on $\Phi(\partial\Omega_a)=\partial\check{\Omega}$. Since $\Phi$ fixes $\infty$, it follows that $\check{\Omega}$ is a bounded simply connected quadrature domain with Schwarz reflection map $\check{\sigma}:=\Phi\circ\sigma_a\circ\Phi^{-1}$. 

The Schwarz reflection map $\check{\sigma}$ of $\check{\Omega}$ has a double critical point at $\Phi(c_a^\ast)$, and a simple critical point at $\Phi(c_a)$. Moreover, by our normalization of $\Phi$, the corresponding critical values are $\infty$ and $2$ respectively. Also note that $\check{\sigma}$ maps $\check{\sigma}^{-1}(\check{\Omega})=\Phi(\Omega_a')$ onto $\check{\Omega}$ as a two-to-one branched cover.

By Proposition~\ref{simp_conn_quad}, there exists a rational map $R$ of degree $3$ on $\widehat{\C}$ such that $R:\D\to\check{\Omega}$ is a Riemann uniformization of $\check{\Omega}$. We can assume that $R(0)=\Phi(c_a^\ast)$, and $R(1)=-2$. 

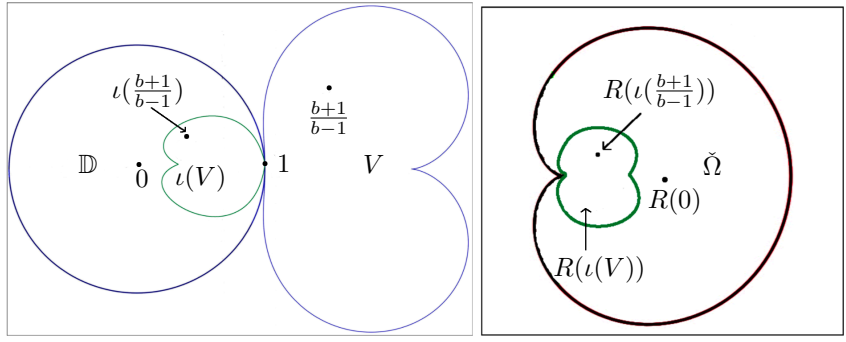
\begin{figure}[ht!]
\begin{tikzpicture}
  \node[anchor=south west,inner sep=0] at (0,0) {\includegraphics[width=0.9\textwidth]{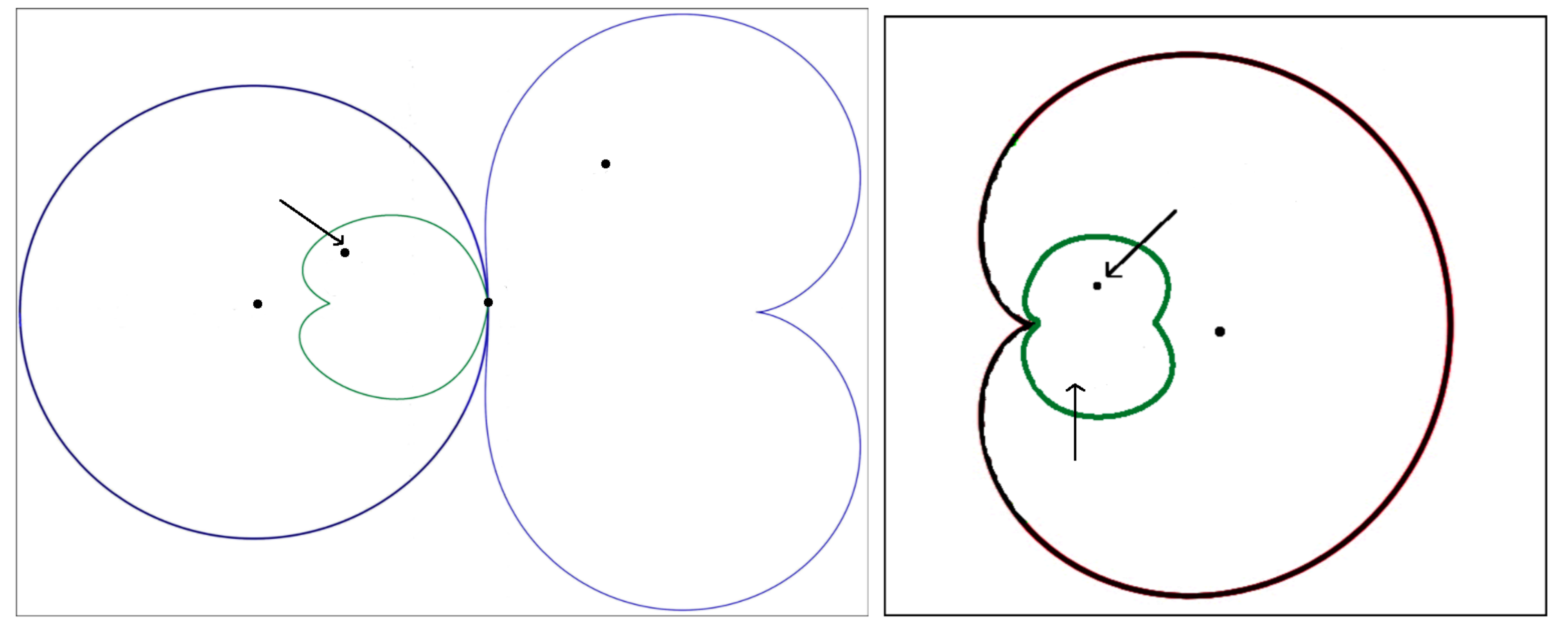}};
  \node at (1.2,2.4) {$\D$};
  \node at (1.92,2.2) {$0$};
  \node at (5,2.4) {$V$};
  \node at (3.8,2.4) {$1$};
  \node at (2.7,2.2) {$\iota(V)$};
  \node at (2,3.36) {$\iota(\frac{b+1}{b-1})$};
  \node at (4.4,3) {$\frac{b+1}{b-1}$};
  \node at (9.5,2.4) {$\check{\Omega}$};
  \node at (9,1.9) {$R(0)$};
  \node at (8,1) {$\tiny{R(\iota(V))}$};
  \node at (8.8,3.36) {$R(\iota(\frac{b+1}{b-1}))$};
  \end{tikzpicture}
\caption{Left: The pre-image of $\check{\Omega}$ under the cubic polynomial $R$ is $\D\cup V$. The set $\D$ is mapped univalently onto $\check{\Omega}$ by $R$. On the other hand, $V$ contains a simple critical point $\frac{b+1}{b-1}$ (where $\re(b)>0$) of $R$, and hence is simply connected. $V$ is mapped as a two-to-one branched covering onto $\check{\Omega}$. Right: $\check{\sigma}$ is a two-to-one branched covering from $\check{\sigma}^{-1}(\check{\Omega})=R(\iota(V))$ onto $\check{\Omega}$ branched only at the simple critical point $R(\iota(\frac{b+1}{b-1}))$ of $\check{\sigma}$.}
\label{pre_image_under_poly}
\end{figure}

Since $\partial\Omega_a$ has a cusp at $-2$, it follows that $\check{\Omega}$ has a cusp at $-2$ as well. Hence, $R'(1)=0$. Moreover, by the commutative diagram in Figure~\ref{comm_diag_schwarz}, we deduce that $R$ has a double critical point at $\infty$ with $R(\infty)=\infty$. Hence, $R$ is a polynomial of degree $3$.

The assumption that $\re(a)>\frac32$ implies that $c_a\in\Omega_a'$, and hence the simple critical point $\Phi(c_a)$ of $\check{\sigma}$ lies in $\check{\sigma}^{-1}(\check{\Omega})$. Hence, $R$ has a simple critical point in $\left(\C\setminus\overline{\D}\right)\cap R^{-1}(\check{\Omega})$ with corresponding critical value $2$. Let us denote this critical point by $\frac{b+1}{b-1}$, for some $b$ with $\re(b)>0$, and $b\neq1$ (see Figure~\ref{pre_image_under_poly}).

We will now bring the critical points of $R$ to $\pm1$ by pre-composing it with an affine map. To this end, let us set $u=b+(1-b)z$, and define $\check{f}(u)=R(z)=R(\frac{u-b}{1-b})$. The finite critical points of $\check{f}$ are easily seen to be $\pm 1$ with corresponding critical values $\mp2$ (respectively). Therefore, $\check{f}(u)=u^3-3u$, and $R\equiv f_b$. 

It follows that $\check{\Omega}=f_b(\D)=\Omega_b$, and $\check{\sigma}\equiv\sigma_b$ on $\Omega_b$. As $\partial\Omega_b$ is a Jordan curve, it follows that $b\in\widehat{S}$. Moreover, the fact that the simple critical point of $\sigma_b$ lies in $\sigma_b^{-1}(\Omega_b)$ implies that $\re(b)>\frac32$. 

This completes the proof.
\end{proof}

\section{The Family $\mathcal{S}$}\label{schwarz_family}

For $a\in\widehat{S}$, we define $T_a^0:=T_a\setminus\{-2\}$, and call it the \emph{desingularized droplet} or the \emph{fundamental tile}. 

One of the main goals of this paper is to demonstrate that suitable Schwarz reflection maps defined in Subsection~\ref{setup} produce matings of the abstract modular group $\Z/2\Z\ast\Z/3\Z$ with certain anti-holomorphic rational maps (anti-rational maps for short). The group structure will essentially appear in the iterated pre-images of $T_a^0$. For this, we will need $\sigma_a^{-1}(T_a^0)$ to be a topological triangle (with its vertices removed). Equivalently, the ``rational map" feature will essentially come from $\sigma_a:\Omega_a'\to\Omega_a$, and hence we need to focus on parameters for which $\Omega_a'$ contains a critical point (note that the dynamics of a rational map is dictated by that of its critical points). Therefore, we will only focus on parameters $a\in\widehat{S}$ with $\re(a)>\frac32$.

This leads to the parameter space; i.e., the set of parameters $a\in\widehat{S}$ that we will be interested in:
$$
S:=\lbrace a\in\widehat{S}:\frac32<\re(a)<4\rbrace.
$$ 
(See Figure~\ref{fig:para_space_boundary}.)

\begin{remark}
The reason for not considering parameters on $\mathcal{I}\subset\widehat{S}$ will become apparent in Subsection~\ref{dyn_near_cusp}.  
\end{remark}

\begin{definition}
The family $\mathcal{S}$ of Schwarz reflection maps is defined as 
$$
\mathcal{S}:=\{\sigma_a:\Omega_a\to\widehat{\C}\ \vert\ a\in S\}.
$$
\end{definition}

\subsection{Invariant Partition of The Dynamical Plane}\label{dyn_plane} 

Since the critical point $c_a^\ast$ is mapped to $\infty$ in one iterate for all $a\in S$, this critical point is ``passive'' throughout the parameter space. Thus, the dynamics of $\sigma_a$ is largely controlled by the unique free critical point $c_a$.

\begin{figure}[ht!]
\begin{center}
\includegraphics[scale=0.128]{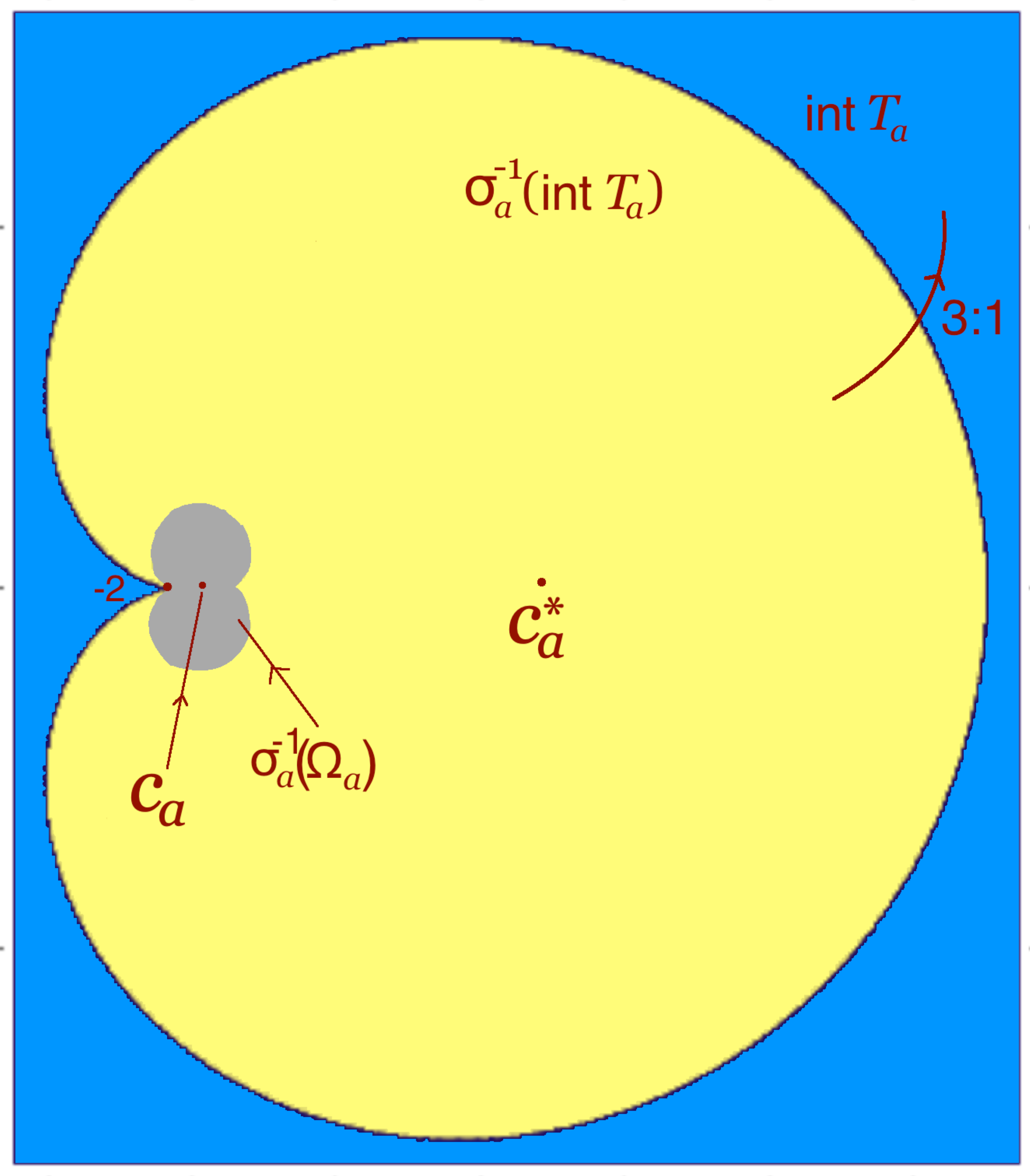}\ \includegraphics[scale=0.125]{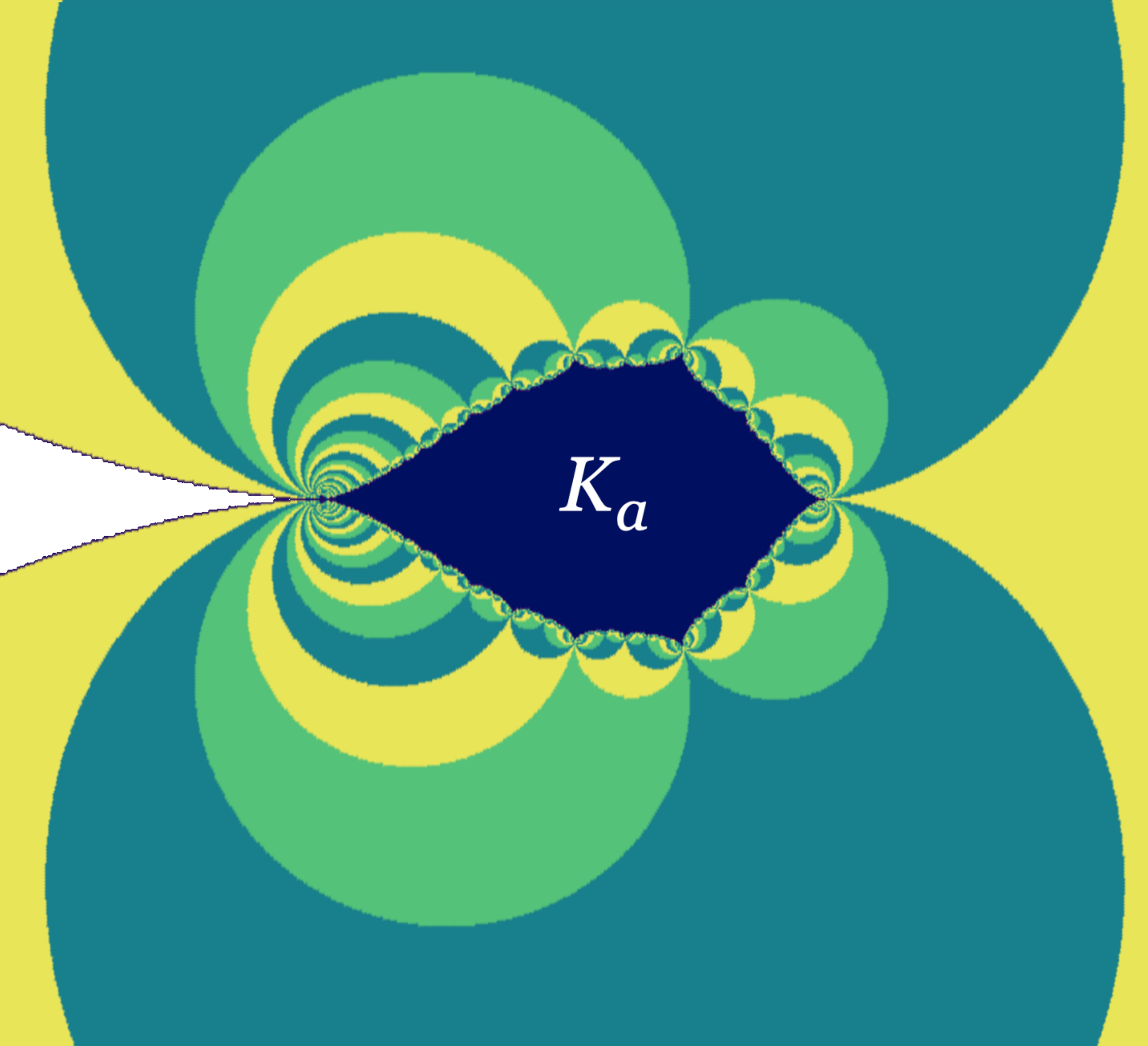}
\caption{Left: For parameters $a$ with $\re(a)>\frac32$, the unique tile of rank one (in yellow) contains the double critical point $c_a^\ast$ of $\sigma_a$. The Schwarz reflection map $\sigma_a$ maps $\sigma_a^{-1}(\Int{T_a})$ (respectively, the rank one tile) as a $3:1$ branched covering onto $\Int{T_a}$ (respectively, onto $T_a^0$) branched only at $c_a^\ast$. On the other hand, $\Omega_a'$ (in gray) is mapped as a $2:1$ branched cover onto $\Omega_a$ by $\sigma_a$ branched only at $c_a$. Right: A zoom of the non-escaping set $K_a$, which is contained in $\sigma_a^{-1}(\Omega_a\cup\{-2\})$. The fixed point $-2$ is on the boundary of $K_a$.}
\label{schwarz_dyn_pic}
\end{center}
\end{figure}

\begin{definition}[Tiles, Tiling set, and Non-escaping Set]\label{non_escaping_def}
\noindent\begin{enumerate}\upshape
\item Connected components of $\sigma_a^{-n}(T_a^0)$ are called \emph{tiles} of $\sigma_a$ of rank $n$.

\item The \emph{tiling set} $T_a^\infty$ is defined as $\displaystyle\bigcup_{n\geq0}\sigma_a^{-n}(T_a^0)$; i.e., it is the union of tiles of all ranks.

\item The \emph{non-escaping set} $K_a$ is defined as the complement of $T_a^\infty$ in the Riemann sphere. In other words, 
$$
K_a=\{z\in\Omega_a\cup\{-2\}: \sigma_a^{\circ n}(z)\in \Omega_a\cup\{-2\}\ \forall\ n\geq 0\}.
$$ 
Connected components of $\Int{K_a}$ are called \emph{Fatou components} of $\sigma_a$. 
\end{enumerate}
\end{definition}

\begin{remark}
Note that the critical point $c_a^\ast$ lies in the tiling set, more precisely in the tile of rank one, for all $a\in S$. In fact, $\sigma_a$ is a three-to-one branched cover from the tile of rank one onto the fundamental tile branched only at $c_a^\ast$.
\end{remark}

\begin{proposition}[Tiling Set is Open and Connected]\label{tiling_connected}
For each $a\in S$, the tiling set $T_a^\infty$ is an open, connected set, and the non-escaping set $K_a$ is a full, compact subset of $\C$.
\end{proposition} 
\begin{proof}
Let us denote the union of the tiles of rank $0$ through $k$ by $E_a^k$. Since every tile of rank $k\geq1$ is attached to a tile of rank $(k-1)$ along a boundary curve and $\Int{E_a^0}$ is connected, it follows that $\Int{E_a^k}$ is connected. Moreover, $\Int{E_a^k}\subsetneq\Int{E_a^{k+1}}$ for each $k\geq 0$. 

Also note that if $z\in T_a^\infty$ belongs to the boundary of a tile of rank $k$, then it lies in the interior of $E_a^{k+1}$. Hence, 
$$
T_a^\infty=\bigcup_{k\geq 0}\Int{E_a^k}.
$$ 
Thus, $T_a^\infty$ is an increasing union of open, connected sets, and hence itself is such. Finally, since $\infty\in\Int{T_a^\infty}$, it follows that the complement $K_a$ of $T_a^\infty$ is a full, compact subset of $\C$.
\end{proof}

Note that the non-escaping set of each $a\in S$ has the fixed point $-2$ on its boundary (see Figure~\ref{schwarz_dyn_pic}(right)).

\subsection{Dynamics Near Cusp Point}\label{dyn_near_cusp}
Since $\partial\Omega_a\setminus\{-2\}$ is a real-analytic curve, it follows that the Schwarz reflection map $\sigma_a$ is anti-holomorphic in a neighborhood of $\partial\Omega_a\setminus\{-2\}$. We will now analyze the behavior of $\sigma_a$ near the cusp point $-2$.

Recall that for $a\in\widehat{S}$, we have $\sigma_a=f_a\circ\iota\circ\left(f_a\vert_{\overline{\D}}\right)^{-1}=f\circ\iota_a\circ \left(f\vert_{\overline{\Delta}_a}\right)^{-1}$, where $\iota_a$ is the reflection in the circle $\partial\Delta_a$. Near $1$, $f$ has a Taylor series 
\begin{equation}
f(1+\epsilon)=-2+3\epsilon^2+\epsilon^3.
\label{riemann_cusp_expansion}
\end{equation}
It follows that near $f(1)=-2$, the inverse function $\left(f\vert_{\overline{\Delta}_a}\right)^{-1}$ can be expanded as a Puiseux series 
\begin{equation}
\left(f\vert_{\overline{\Delta}_a}\right)^{-1}(-2+\delta)=1+\frac{\delta^{\frac12}}{\sqrt{3}}-\frac{\delta}{18}+\frac{5}{216\sqrt{3}}\delta^{\frac32}-\frac{1}{243}\delta^2+\frac{77}{31104\sqrt{3}}\delta^{\frac52}+O(\delta^3)
\label{f_inverse_cusp}
\end{equation}
where $\delta\approx 0$, $(-2+\delta)\in\overline{\Omega}_a$, and the branch of square root has been chosen so that $\left(f\vert_{\overline{\Delta}_a}\right)^{-1}(-2+\delta)\in\overline{\Delta_a}$. 

\subsubsection{Case I: $a\in\widehat{S}, \frac32<\re(a)<4$}\label{cusp_asymp_sec_1}
By Relation~(\ref{f_inverse_cusp}), and the definition of $\sigma_a$, we have that 
\begin{align*}
\sigma_a(-2+\delta)
&=  f\left(\iota_a\left(1+\frac{\delta^{\frac12}}{\sqrt{3}}-\frac{\delta}{18}+\frac{5}{216\sqrt{3}}\delta^{\frac32}-\frac{1}{243}\delta^2+\frac{77}{31104\sqrt{3}}\delta^{\frac52}+O(\delta^3)\right)\right)
\\
&= -2+\left(\frac{a-1}{\overline{a}-1}\right)^2\overline{\delta}+\frac{2}{3\sqrt{3}}\frac{\left(a-1\right)^2(4-\re(a))}{\left(\overline{a}-1\right)^3}\overline{\delta}^{\frac32}+O(\overline{\delta}^2) \;.
\end{align*}

Let $\arg(a-1)=\theta_0\in(-\pi,\pi]$. Since $\re(a)>\frac32$, we have that $\vert\theta_0\vert<\frac{\pi}{2}$. Since $f(1+\epsilon e^{i\theta_0})=-2+\epsilon^2 e^{2i\theta_0}(3+\epsilon e^{i\theta_0}),$ it follows that $\arg(f(1+\epsilon e^{i\theta_0})+2)\approx2\theta_0$, for $\epsilon>0$ small enough. In particular, $(-2+\delta e^{2i\theta_0})\in\Omega_a$, for $\delta=\delta(\theta_0)>0$ sufficiently small.
Now, 
\begin{equation}
\sigma_a(-2+\delta e^{2i\theta_0})=-2+\delta e^{2i\theta_0}+\frac{2}{3\sqrt{3}}\frac{(4-\re(a))}{\vert a-1\vert}\delta^{\frac32}e^{2i\theta_0}+O(\delta^2),
\label{schwarz_cusp_expansion}
\end{equation}
for sufficiently small positive $\delta$.

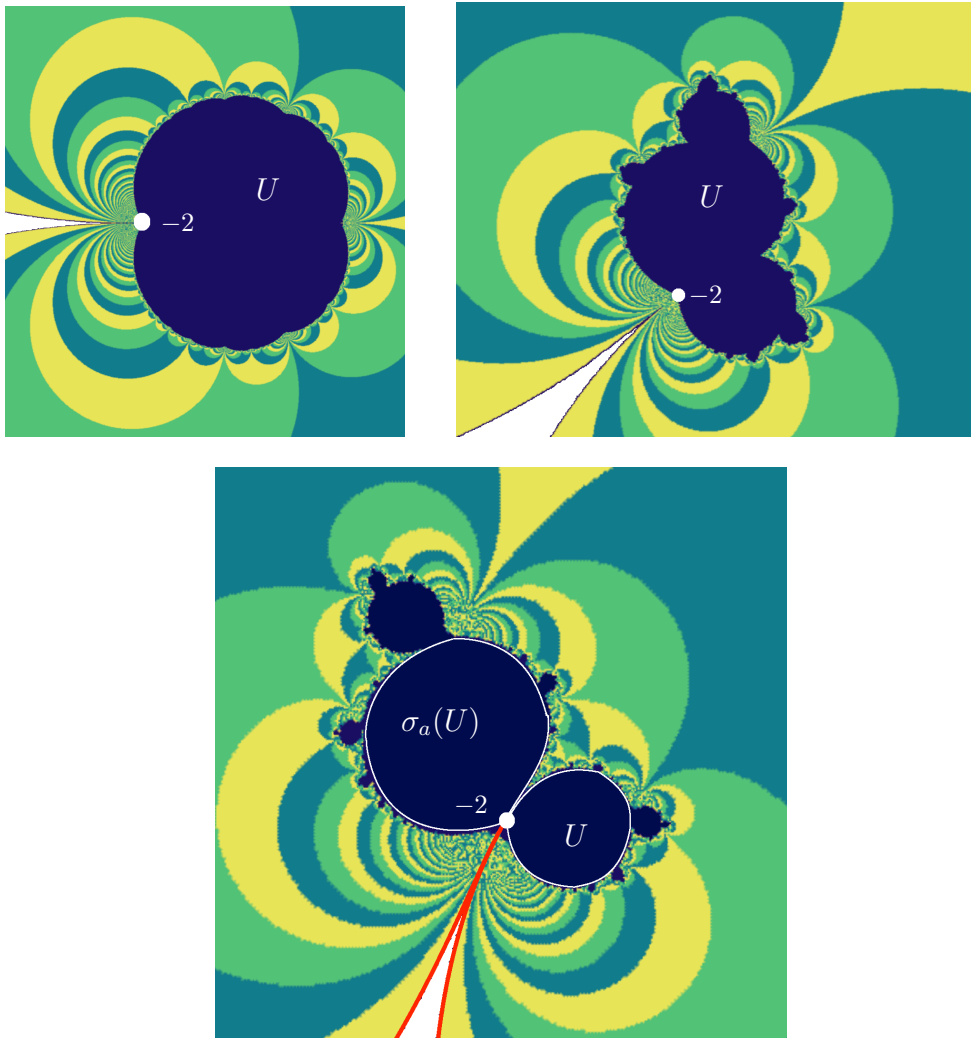
\begin{figure}[ht!]
\begin{tikzpicture}
\node[anchor=south west,inner sep=0] at (0,0) {\includegraphics[width=0.42\textwidth]{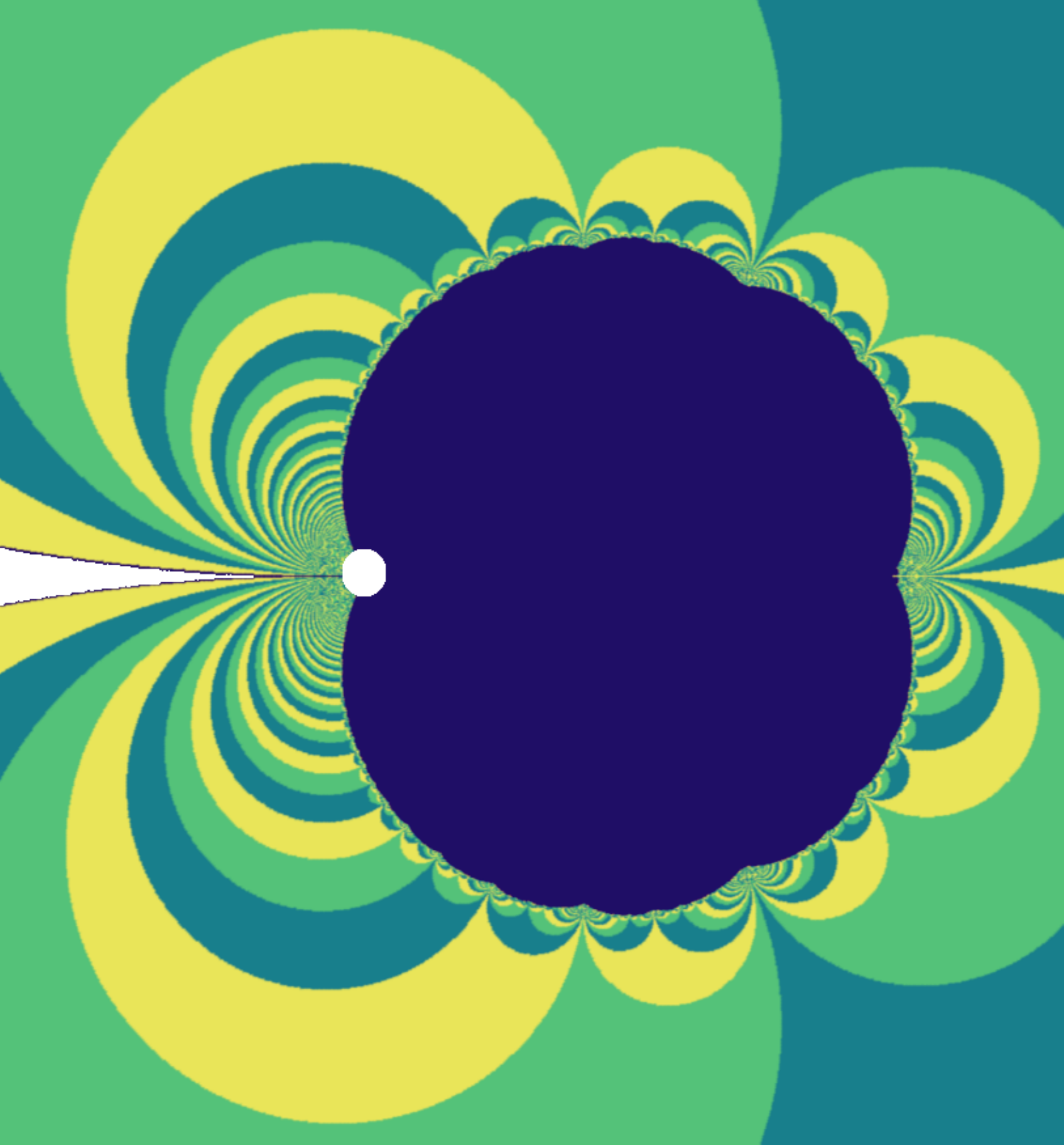}};
\node[anchor=south west,inner sep=0] at (6,0) {\includegraphics[width=0.54\textwidth]{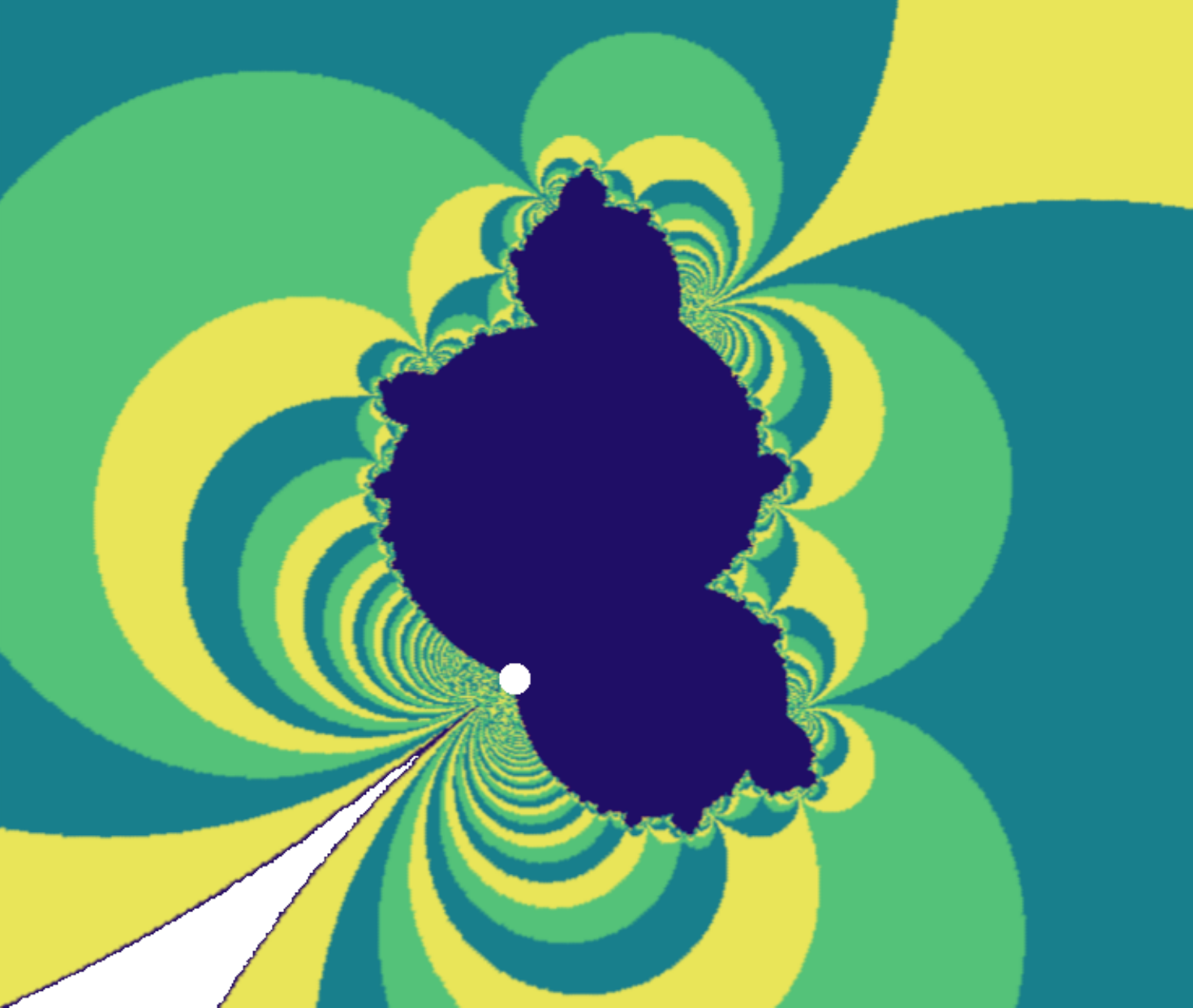}};
\node[anchor=south west,inner sep=0] at (2.8,-8) {\includegraphics[width=0.6\textwidth]{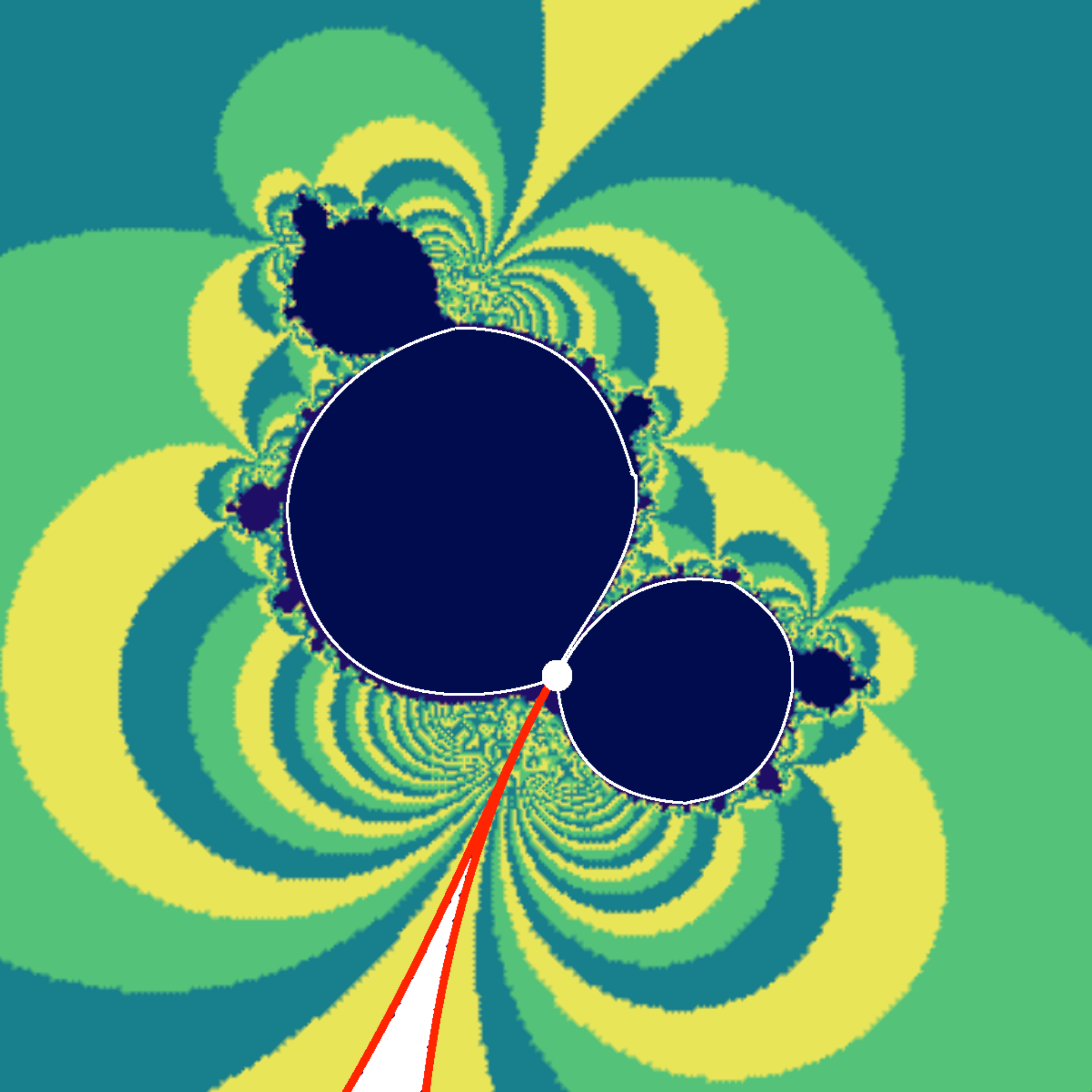}};
\node at (3.5,3.3) {\begin{Large}{\textcolor{white}{$U$}}\end{Large}};
\node at (2.3,2.84) {\textcolor{white}{$-2$}};
\node at (9.4,3.2) {\begin{Large}{\textcolor{white}{$U$}}\end{Large}};
\node at (9.32,1.88) {\textcolor{white}{$-2$}};
\node at (7.6,-5.3) {\begin{Large}{\textcolor{white}{$U$}}\end{Large}};
\node at (5.8,-3.8) {\begin{Large}{\textcolor{white}{$\sigma_a(U)$}}\end{Large}};
\node at (6.2,-4.9) {\textcolor{white}{$-2$}};
\end{tikzpicture}
\caption{Top: The dynamical planes of $\sigma_a$, where $a=4$ (left) and $a=4+i$ (right), show the unique invariant Fatou component $U$ with $-2\in\partial U$. The forward orbit of every point in $U$ converges to the fixed point $-2$. The critical value $2$ lies in $U$. For $a=4$, the critical {\'E}calle height of $\sigma_a$ is $0$; while for $a=4+i$, the critical {\'E}calle height of $\sigma_a$ is some large positive real number. Bottom: The dynamical plane of $\sigma_a$, where $a=4+i\sqrt{3}$, shows the $2$-cycle of Fatou components $\{U, \sigma_a(U)\}$ with $-2\in\partial U\cap\partial\sigma_a(U)$. The forward orbit of every point in $U\cup\sigma_a(U)$ converges to the fixed point $-2$. The critical value $2$ lies in $U$.}
\label{special_arc_dynamics}
\end{figure}

For $a\in S$, it follows that $-2$ repels nearby points in the direction $2\theta_0$ under application of $\sigma_a$. In particular, points in $K_a$ close to $-2$ are repelled from $-2$ under iterates of $\sigma_a$.

\subsubsection{Case II: $a\in\mathcal{I}\subset\widehat{S}$}\label{cusp_asymp_sec_2}
Here we will explain the reason for not considering parameters lying on the segment $\{a:\re(a)=4,\vert\im(a)\vert\leq\sqrt{3}\}$.

First let $a\in\mathcal{I}^0$; i.e., $\re(a)=4$, and $\vert\im(a)\vert<\sqrt{3}$. Set $\arg(a-1)=\theta_0$. Then, 
\begin{equation}
\sigma_a^{\circ 2}(-2+\delta e^{2i\theta_0})=-2+\delta e^{2i\theta_0}-\frac{2(3-\im(a)^2)}{9\sqrt{3}\vert a-1\vert^3}\delta^{\frac52}e^{2i\theta_0}+O(\delta^3),
\label{schwarz_cusp_expansion_1}
\end{equation}
for sufficiently small positive $\delta$.

Hence, for parameters $a\in\mathcal{I}^0$, the fixed point $-2$ attracts nearby points in the direction $2\theta_0$ under application of $\sigma_a^{\circ 2}$; i.e., there is a forward invariant Fatou component for $\sigma_a$ such that the forward orbit of every point in this component converges to the fixed point $-2$ on its boundary (see Figure~\ref{special_arc_dynamics}). 

Now suppose that $a=4+i\sqrt{3}$. Then,
\begin{equation}
\sigma_a^{\circ 2}(-2+\delta)=-2+\delta-\frac{(\sqrt{3}+i)}{972}\delta^{\frac72}+O(\delta^4),
\label{schwarz_cusp_expansion_2}
\end{equation}
for $\delta$ with sufficiently small absolute value such that $(-2+\delta)\in\Omega_a'$ (where the chosen branch of square root sends the radial line at angle $2\theta_0$ to the radial line at angle $\theta_0$). It follows that $-2$ has two attracting directions, which are permuted by $\sigma_a$. Therefore, for $a=4+i\sqrt{3}$, there is a $2$-cycle of Fatou components with the fixed point $-2$ on their common boundary such that the forward orbit of every point in these components converges to $-2$ (see Figure~\ref{special_arc_dynamics}).

A completely analogous analysis shows that for $a=4-i\sqrt{3}$, there is a $2$-cycle of Fatou components with the fixed point $-2$ on their common boundary such that the forward orbit of every point in these components converges to $-2$.

To sum up, we ignore parameters $a$ with $\re(a)=4$ and $\vert\im(a)\vert\leq\sqrt{3}$ so that all maps in the parameter space have the same qualitative dynamics near the fixed point $-2$; namely, they repel nearby points in $K_a$.

We end this subsection with a couple of consequences of the above analysis on the local dynamics of $\sigma_a$ near $-2$.

\begin{proposition}\label{schwarz_qcdef}
1) Let $a\in S$, $\mu$ be a $\sigma_a$-invariant Beltrami coefficient on $\widehat{\C}$, and $\Phi$ be a quasiconformal map satisfying $\Phi_{\overline{z}}/\Phi_{z}=\mu$ a.e. such that $\Phi$ fixes $\pm 2$ and $\infty$. Then, there exists $b\in S$ such that $\Phi(\Omega_a)=\Omega_b$, and $\Phi\circ\sigma_a\circ\Phi^{-1}=\sigma_b$ on $\Omega_b$.

2) Let $a\in\widehat{S}$ with $\re(a)=4$ and $\vert\im(a)\vert<\sqrt{3}$. Let $\mu$ be a $\sigma_a$-invariant Beltrami coefficient on $\widehat{\C}$, and $\Phi$ be a quasiconformal map satisfying $\Phi_{\overline{z}}/\Phi_{z}=\mu$ a.e. such that $\Phi$ fixes $\pm 2$ and $\infty$. Then, there exists $b\in\widehat{S}$ with $\re(b)=4$ and $\vert\im(b)\vert<\sqrt{3}$ such that $\Phi(\Omega_a)=\Omega_b$, and $\Phi\circ\sigma_a\circ\Phi^{-1}=\sigma_b$ on $\Omega_b$.
\end{proposition}
\begin{proof}
1) By Proposition~\ref{schwarz_pre_qcdef}, there exists some $b\in\widehat{S}$ with $\re(b)>\frac32$ satisfying the desired properties. 

Since $a\in S$, the fixed point $-2$ repels nearby points in $K_a$ (under application of $\sigma_a$). This property is preserved under the quasiconformal conjugacy $\Phi$, and hence the same is true for the map $\sigma_b$. It now follows that $b\in S$.

2) Once again by Proposition~\ref{schwarz_pre_qcdef}, there exists some $b\in\widehat{S}$ with $\re(b)>\frac32$ satisfying the desired properties. 

Since $\re(a)=4$ and $\vert\im(a)\vert<\sqrt{3}$, the map $\sigma_a$ has a forward invariant Fatou component such that the forward orbit of every point in this component converges to the fixed point $-2$ on its boundary. The quasiconformal conjugacy $\Phi$ preserves this property, and hence the same is true for the map $\sigma_b$. It now follows that $\re(b)=4$ and $\vert\im(b)\vert<\sqrt{3}$.
\end{proof}

\begin{proposition}\label{special_neutral_critical}
For all maps $\sigma_a$ with $a\in\mathcal{I}\subset\widehat{S}$, the forward critical orbit $\{\sigma_a^{\circ n}(2)\}_n$ converges to $-2$.
\end{proposition}
\begin{proof}
Let us fix some $a\in\mathcal{I}$. Note that the second iterate $\sigma_a^{\circ 2}$ has at least one attracting direction at the fixed point $-2$. Let $U\subset\Omega_a$ be an immediate basin of attraction of $-2$; i.e., a connected component of the set of all points converging to $-2$ (under iterates of $\sigma_a^{\circ 2}$) having $-2$ on its boundary. Using the asymptotics of $\sigma_a^{\circ 2}$ near the cusp point $-2$ (see Subsection~\ref{cusp_asymp_sec_2}), one can easily adapt the proof of existence of attracting Fatou coordinates at parabolic fixed points (for instance, see \cite[Theorem~10.9]{M1new}) to show that there exists a $\sigma_a^{\circ 2}$-invariant open set $\mathcal{P}\subset U$ with $-2\in\partial \mathcal{P}$ and a conformal map from $\mathcal{P}$ onto some right half-plane that conjugates $\sigma_a^{\circ 2}$ to the translation $\zeta\mapsto\zeta+1$ (this conformal coordinate is unique up to addition by a complex constant). One can now argue as in \cite[Theorem~10.15]{M1new} to conclude that the boundary of the maximal domain of definition of this conjugacy contains a critical point of $\sigma_a^{\circ 2}$. Hence, $U$ contains an infinite critical orbit of $\sigma_a^{\circ 2}$.

Clearly, the sequence $\{\sigma_a^{\circ n}(2)\}_n$ is the union of two infinite critical orbits of $\sigma_a^{\circ 2}$. Moreover, these are the only infinite critical orbits of $\sigma_a^{\circ 2}$, and these orbits are related by $\sigma_a$. It follows that if $a\in\mathcal{I}^0$; i.e., if $\sigma_a^{\circ 2}$ has exactly one attracting direction at the fixed point $-2$, then $\sigma_a(U)=U$, and hence, $\{\sigma_a^{\circ n}(2)\}_n\subset U$. On the other hand, if $a=4\pm i\sqrt{3}$, then $\sigma_a^{\circ 2}$ has exactly two attracting directions at the fixed point $-2$, and these two attracting directions are permuted by $\sigma_a$. In this case, the sequence $\{\sigma_a^{\circ n}(2)\}_n$ is contained in the immediate basin of attraction $U\cup\sigma_a(U)$ of $-2$ (see Figure~\ref{special_arc_dynamics}). Hence, in both cases, the critical orbit $\{\sigma_a^{\circ n}(2)\}_n$ converges to $-2$. 
\end{proof}

\subsection{The Connectedness Locus}\label{conn_locus_def} 

\begin{proposition}\label{connected_critical}
For $a\in S$, $K_a$ is connected if and only if $2\in K_a$.
\end{proposition}
\begin{proof}
Let $E_a^k$ be the union of the tiles of rank $\leq k$. Note that although the tile of rank one is mapped to $T_a^0$ under $\sigma_a$ as a ramified cover, $\widehat{\C}\setminus \Int{E_a^1}$ is a full continuum. 

If $2\in T_a^\infty$, then the tile containing the free critical point $c_a$ is ramified and it disconnects $K_a$. On the other hand, if $c_a$ does not escape to $T_a^0$ under iterates of $\sigma_a$, then every tile of rank $\geq 2$ is unramified, and $\widehat{\C}\setminus\Int{E_a^k}$ is a full continuum for each $k\geq 0$. Therefore, 
$$
K_a=\bigcap_{k\geq 0}\left(\widehat{\C}\setminus\Int{E_a^k}\right)
$$ 
is a nested intersection of full continua, and hence is a full continuum itself. 
\end{proof}

\begin{figure}[ht!]
\begin{center}
\includegraphics[scale=0.8]{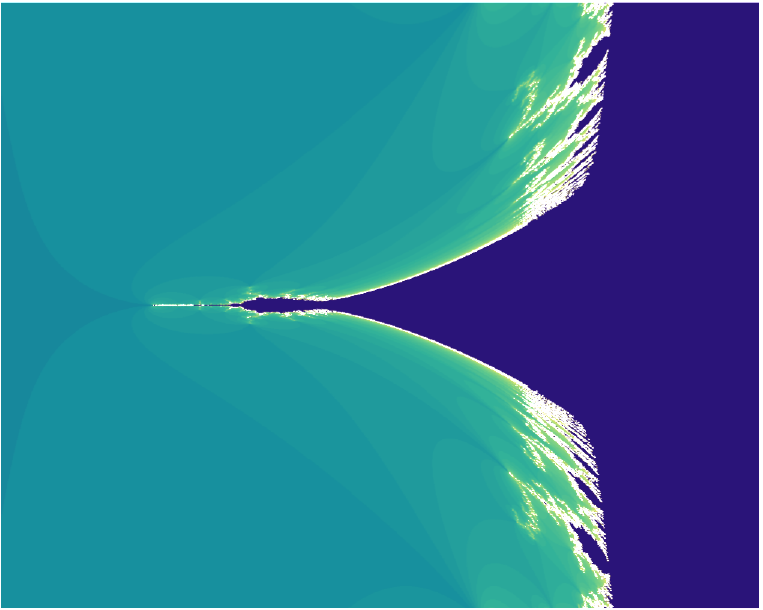}
\caption{A part of the connectedness locus $\cC(\mathcal{S})$ is shown in blue.}
\label{conn_loc_pic}
\end{center}
\end{figure}

Proposition~\ref{connected_critical} leads to the following definition.

\begin{definition}[Connectedness Locus and Escape Locus]\label{conn_escape_def}
The connectedness locus of the family $\mathcal{S}$ is defined as 
$$
\cC(\mathcal{S})=\{a\in S: 2\notin T_a^\infty\}=\{a\in S: K_a\ \textrm{is\ connected}\}.
$$ 
The complement of the connectedness locus in the parameter space is called the \emph{escape locus}. (See Figure~\ref{conn_loc_pic} for a part of the connectedness locus.)
\end{definition}

We now record some basic properties of the connectedness and escape loci.

\begin{proposition}\label{escape_locus_non_empty}
$(\frac32,\frac52)\subset S\setminus\cC(\mathcal{S})$. In particular, $S\setminus\cC(\mathcal{S})\neq\emptyset$.
\end{proposition}
\begin{proof}
By Proposition~\ref{univalence_non_empty} and the definition of $S$, we know that $(\frac32,\frac52)\subset S$. 

Note that for $a>\frac32$, the critical value $2$ lies in $\Omega_a$, and hence $2\in\Delta_a$. Also, $-1\notin\Delta_a$. It follows by the commutative diagram in Figure~\ref{comm_diag_schwarz} that $\sigma_a(2)=f(\iota_a(2))=f\left(\frac{1}{2-a}\right)$.

\noindent\textbf{Case 1: $a\in(2,\frac52)$.} In this case, we have
\begin{align*}
\iota_a(2)+2=\frac{5-2a}{2-a}<0
&\implies \iota_a(2)<-2\\
&\implies f(\iota_a(2))<f(-2)=-2\\
&\implies \sigma_a(2)<-2.
\end{align*}

By the proof of Proposition~\ref{univalence_non_empty}, we see that $(-\infty,-2)\cap\Omega_a=\emptyset$. Therefore, $\sigma_a(2)\in T_a^0\subset T_a^\infty$. It follows that $(2,\frac52)\in S\setminus\cC(\mathcal{S})$.

\noindent\textbf{Case 2: $a=2$.} For this parameter, we have $\sigma_a(2) =f(\infty)=\infty\in T_a^0$. Hence, $2\in S\setminus\cC(\mathcal{S})$.

\noindent\textbf{Case 3: $a\in(\frac32,2)$.} In this case, we have 
\begin{align*}
\frac{1}{2-a}-(2a-1)=\frac{2(a-1)(a-\frac32)}{2-a}>0
&\implies \frac{1}{2-a}>(2a-1)>2\\
&\implies f\left(\frac{1}{2-a}\right)>f(2a-1)>f(2)\\
&\implies \sigma_a(2)>f(2a-1)>2.
\end{align*}

The proof of Proposition~\ref{univalence_non_empty} shows that $(f(2a-1),+\infty)\cap\Omega_a=\emptyset$, and hence, $\sigma_a(2)\in T_a^0$. Therefore, $a$ lies in the escape locus. It follows that $(\frac32,2)\subset S\setminus\cC(\mathcal{S})$.

This completes the proof.
\end{proof}

\begin{proposition}\label{conn_locus_non_empty}
$\cC(\mathcal{S})\cap\R=\left[\frac52,4\right)$. In particular, $\cC(\mathcal{S})\neq\emptyset$.
\end{proposition}
\begin{proof}
In view of Proposition~\ref{escape_locus_non_empty}, it suffices to show that $\left[\frac52,4\right)\subset\cC(\mathcal{S})$. 

Let us now fix $a\in\left[\frac52,4\right)$. 

Note that $-2\notin\Delta_a$, and hence $\iota_a(-2)=\frac{4a-1}{a+2}\in\Delta_a$. A direct computation using the commutative diagram in Figure~\ref{comm_diag_schwarz} now shows that 
$$
f(\iota_a(-2))=f\left(\frac{4a-1}{a+2}\right)\in\Omega_a,\ \mathrm{and}\ \sigma_a\left(f\left(\frac{4a-1}{a+2}\right)\right)=-2.
$$

We also have 
$$
1<\frac{3a-1}{a+1}<\frac{4a-1}{a+2}<a,
$$ 
and hence, 
$$
-2=f(1)<f\left(\frac{3a-1}{a+1}\right)=c_a<f\left(\frac{4a-1}{a+2}\right)<f(a)=c_a^\ast.
$$

Moreover, $\sigma_a$ is a monotone increasing function from $\left[-2,c_a\right]$ onto $\left[-2,2\right]$, and a monotone decreasing function from $\left[c_a,f\left(\frac{4a-1}{a+2}\right)\right]$ onto $\left[-2,2\right]$.

The assumption that $a\geq\frac52$ implies that 
\begin{align*}
\frac{4a-1}{a+2}\geq 2
\implies f\left(\frac{4a-1}{a+2}\right)\geq f(2)=2.
\end{align*}

Therefore, the interval $\left[-2,f\left(\frac{4a-1}{a+2}\right)\right]$ contains $2$, and is invariant under $\sigma_a$. It follows that the critical value $2$ (of $\sigma_a$) does not escape to $T_a^\infty$ under the iterates of $\sigma_a$; i.e., $a\in\cC(\mathcal{S})$. Therefore, $\left[\frac52,4\right)\subset\cC(\mathcal{S})$.
\end{proof}

\begin{corollary}\label{hyp_per_one}
For $a\in\left[3,4\right)$, the critical orbit $\{\sigma_a^{\circ n}(2)\}_n$ converges to a fixed point different from $-2$.
\end{corollary}
\begin{proof}
This is a simple extension of the arguments of Proposition~\ref{conn_locus_non_empty}. 

For $a\in\left[3,4\right)$, we have that 
$$
\frac{3a-1}{a+1}\geq2\implies c_a=f\left(\frac{3a-1}{a+1}\right)\geq f(2)=2.
$$ 
Since $\sigma_a$ is monotone increasing on $\left[-2,c_a\right]$, it follows that the sequence $\{\sigma_a^{\circ n}(2)\}_n\subset\left[-2,2\right]$ is monotonically decreasing and hence converges to a (real) fixed point of $\sigma_a$. By Subsection~\ref{cusp_asymp_sec_1}, the fixed point $-2$ repels real points on its right side; and hence the fixed point that the critical orbit $\{\sigma_a^{\circ n}(2)\}_n$ converges to is different from $-2$.
\end{proof}

\begin{figure}[ht!]
\begin{center}
\includegraphics[scale=0.16]{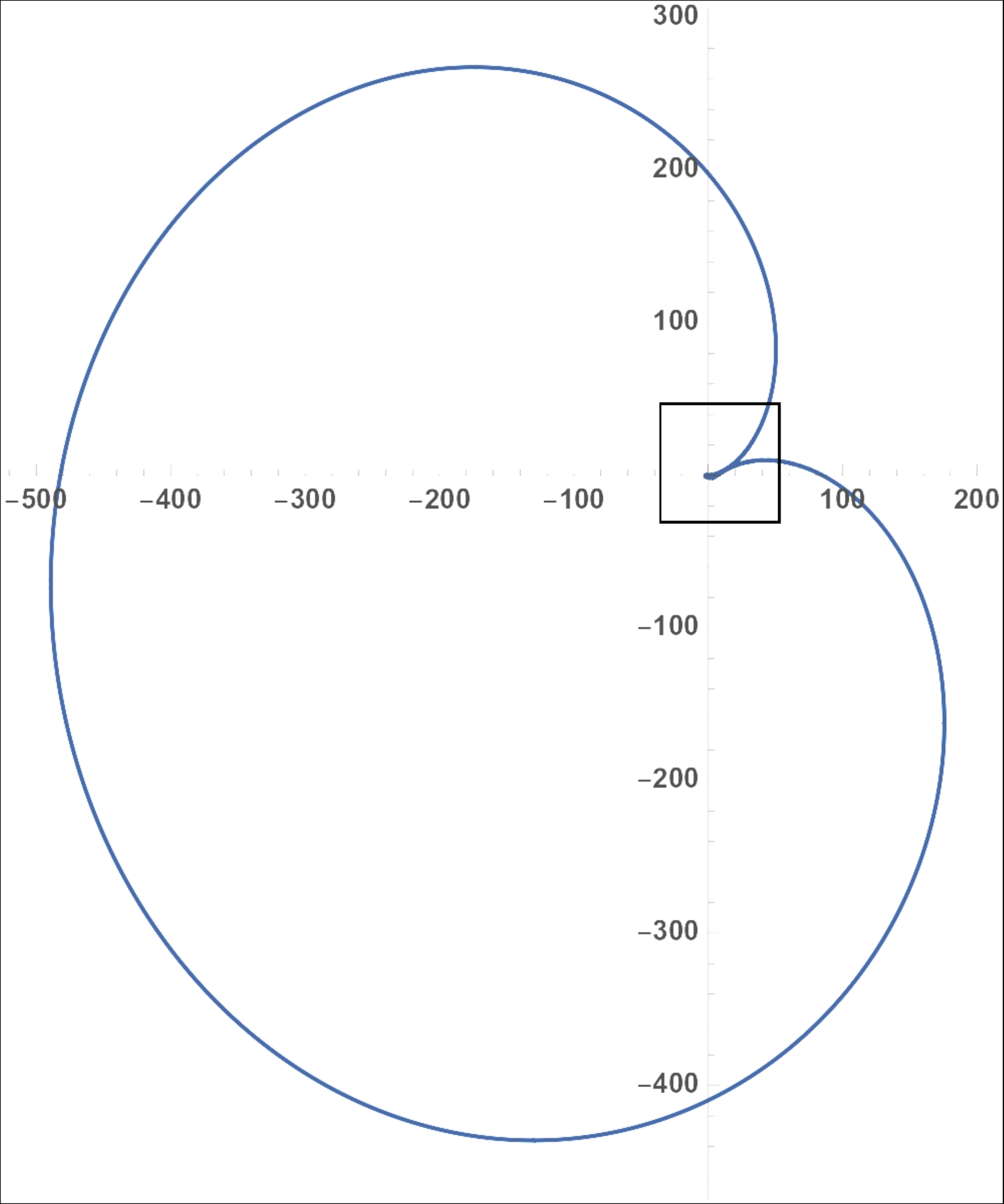}\ \includegraphics[scale=0.16]{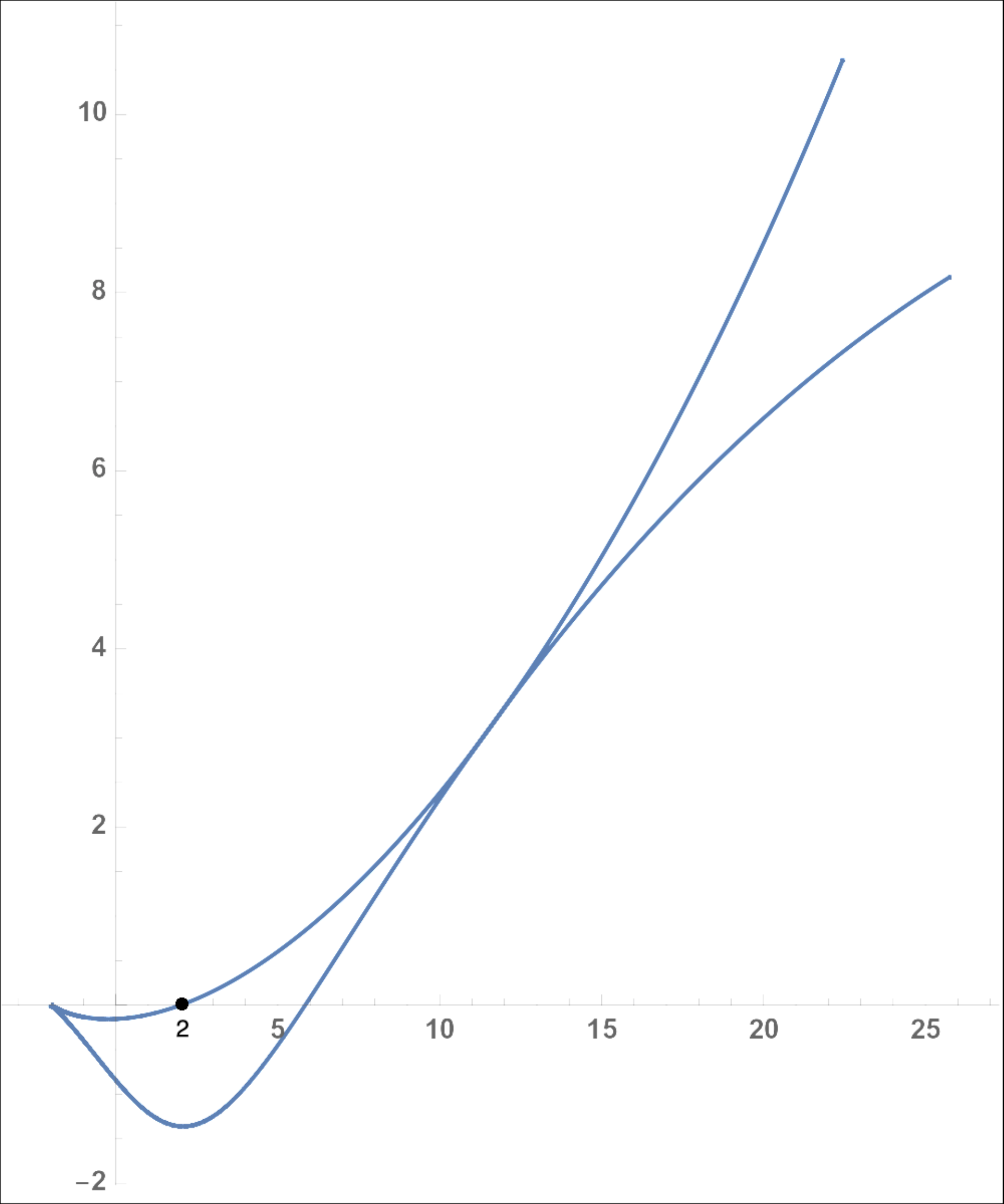}
\end{center}
\caption{Left: The quadrature domain $\Omega_{a_0}$ for the unique intersection point $a_0=\frac32+i\frac{\sqrt{17+18\sqrt{5}}}{2}$ of the curve $\mathfrak{T}^+$ and the vertical line $\{\re(a)=\frac32\}$. Right: A zoom of $\Omega_{a_0}$ around the cusp $-2$ shows a double point on $\partial\Omega_{a_0}$. Moreover, the critical value $2$ of the Schwarz reflection map $\sigma_{a_0}$ lies on the boundary of the bounded connected component $^{b}T_{a_0}^0$ of the desingularized droplet.}
\label{fig:double_3_2}
\end{figure}

\begin{proposition}\label{conn_locus_almost_closed}
$\cC(\mathcal{S})$ is closed in the parameter space $S$.
\end{proposition}
\begin{proof}
Note that the fundamental tile $T_a^0$ varies continuously with the parameter as $a$ runs over $S$. Now let $a_0\in S$ be a parameter outside the connectedness locus $\cC(\mathcal{S})$. Then there exists some integer $n_0\geq0$ such that $\sigma_{a_0}^{\circ n_0}(2)\in T_{a_0}^0$. It follows that for all $a\in S$ sufficiently close to $a_0$, we have $\sigma_{a}^{\circ n_0}(2)\in T_a^0$ or $\sigma_{a}^{\circ (n_0+1)}(2)\in T_a^0$. Hence, $\cC(\mathcal{S})$ is closed in $S$.
\end{proof}

\begin{proposition}\label{limit_point_conn_locus}
Every limit point of $\cC(\mathcal{S})$ outside $S$ must lie on $\mathcal{I}$.
\end{proposition}
\begin{proof}
First note that for parameters on the line $\{\re(a)=\frac32\}$, the critical value $2$ lies on $\partial\Omega_a$. Hence, for parameters with real part greater than but sufficiently close to $\frac32$, the critical value $2$ lies in the rank one tile; i.e., such parameters belong to the escape locus $S\setminus\cC(\mathcal{S})$. Therefore, $\cC(\mathcal{S})$ has no limit point on the line $\{\re(a)=\frac32\}$.

To show that $\cC(\mathcal{S})$ does not accumulate on $\mathfrak{T}^+\cap\partial S$, we need to analyze the structure of $\mathfrak{T}^+\cap\partial S$ in greater detail. To this end, first note that for each $a\in\mathfrak{T}^+\cap\partial S$ (see Figure~\ref{fig:para_space_boundary}), the cubic polynomial $f$ is univalent on $\Delta_{a}$, but $\partial\Omega_{a}=f(\partial\Delta_{a})$ is not a Jordan curve. More precisely, there exists a double point (i.e., a point of tangential self-intersection) on $\partial\Omega_{a}$. Hence for such a parameter $a$, $\Omega_{a}=f(\Delta_{a})$ is a simply connected quadrature domain, and the corresponding desingularized droplet (i.e., the set obtained by removing the cusp $-2$ and the point of tangential self-intersection from the droplet $T_{a}=\widehat{\C}\setminus\Omega_{a}$) has two connected components. We denote the corresponding bounded component by $^{b}T_{a}^0$.

\begin{lemma}\label{arc_escape_time_lemma}
$\mathfrak{T}^+\cap\partial S=\bigcup_{n\geq0}\gamma_n$,  where 
$$
\gamma_n=\{a\in\mathfrak{T}^+\cap\partial S: \sigma_a^{\circ n}(2)\in\ ^{b}T_a^0\}.
$$
\end{lemma}
\begin{proof}[Proof of Lemma]
The curve $\mathfrak{T}^+\cap\partial S$ intersects the vertical line $\{\re(a)=\frac32\}$ at $a_0=\frac32+i\frac{\sqrt{17+18\sqrt{5}}}{2}$. Since $\re(a_0)=\frac32$, it follows that $2\in\partial\Delta_{a_0}$, and hence, $2=f(2)\in\partial\Omega_{a_0}=\partial T_{a_0}^0$. In fact, we claim that $2$ lies on the boundary of $^{b}T_{a_0}^0$. To see this, let us denote the point of tangential self-intersection on $\partial T_{a_0}^0$ by $p$, and set $f^{-1}(p)\cap\partial\Delta_{a_0}=\{p_1,p_2\}$. Then, an explicit computation using the value of $a_0$ shows that the points $1$ and $2$ lie in the shorter connected component of $\partial\Delta_{a_0}\setminus\{p_1,p_2\}$. As $f$ is an orientation-preserving continuous map from $\partial\Delta_{a_0}$ onto $\partial T_{a_0}^0$ with $f(p_1)=f(p_2)=p$ (and injective elsewhere), it carries the shorter component of $\partial\Delta_{a_0}\setminus\{p_1,p_2\}$ onto the boundary of $^{b}T_{a_0}^0$, and the longer one onto the boundary of $T_{a_0}\setminus\overline{^{b}T_{a_0}^0}$. It follows that both $f(1)=-2$ and $f(2)=2$ lie on the boundary of $^{b}T_{a_0}^0$ (see Figure~\ref{fig:double_3_2}). Moreover, since $2$ is different from the singular points $-2$ and $p$ on $\partial T_{a_0}^0$, it follows that $2\in\ ^{b}T_{a_0}^0$. In particular, $a_0\in\gamma_0$.

Let us choose a continuous parametrization $\gamma:[0,1]\to\overline{\mathfrak{T}^+}\cap\partial S$ such that $\gamma(0)=a_0$, and $\gamma(1)=4+i\sqrt{3}$. As the desingularized droplet and the Schwarz reflection map vary continuously with the parameter, it follows that for $s>0$ sufficiently small, either $2\in\ ^{b}T_{\gamma(s)}^0$ or $2\in \sigma_{\gamma(s)}^{-1}(^{b}T_{\gamma(s)}^0)$. However, since $\re(\gamma(s))>\frac32$ for $s>0$, we know that $2$ must lie in the quadrature domain $\Omega_{\gamma(s)}$; i.e., $2\notin\ ^{b}T_{\gamma(s)}^0$, and hence, $\sigma_{\gamma(s)}(2)\in\ ^{b}T_{\gamma(s)}^0$ (for $s>0$ sufficiently small). We set $s_0=0$, and define 
$$
s_1:=\sup \{s>s_0: \sigma_{\gamma(t)}(2)\in\ ^{b}T_{\gamma(t)}^0\ \forall\ t\in\left(0,s\right)\}.
$$
It is now easy to see that $\gamma((s_0,s_1])\subset\gamma_1$, and $\sigma_{\gamma(s_1)}(2)\in \partial\ ^{b}T_{\gamma(s_1)}^0$. Once again, since the desingularized droplet and the Schwarz reflection map vary continuously with the parameter, it follows that for $s>s_1$ sufficiently close to $s_1$, either $2\in \sigma_{\gamma(s)}^{-1}(^{b}T_{\gamma(s)}^0)$, or $2\in \sigma_{\gamma(s)}^{-2}(^{b}T_{\gamma(s)}^0)$. By definition of $s_1$, we must have that $\sigma_{\gamma(s)}^{\circ 2}(2)\in\ ^{b}T_{\gamma(s)}^0$ (for $s>s_1$ sufficiently close to $s_1$). Defining 
$$
s_2:=\sup \{s>s_1: \sigma_{\gamma(t)}^{\circ 2}(2)\in\ ^{b}T_{\gamma(t)}^0\ \forall\ t\in\left(s_1,s\right)\},
$$
we now see that $\gamma((s_1,s_2])\subset\gamma_2$, and $\sigma_{\gamma(s_2)}^{\circ 2}(2)\in \partial\ ^{b}T_{\gamma(s_2)}^0$. This way, we inductively define $s_n$ once $s_{n-1}$ has already been defined. If some $s_n$ is equal to $1$, then $\mathfrak{T}^+\cap\partial S$ is the union of finitely many $\gamma_n$, and we are done. Otherwise, we obtain a strictly increasing infinite sequence $\{s_n\}_{n\geq 0}\subset[0,1)$ such that $\gamma((s_n,s_{n+1}])\subset\gamma_{n+1}$, and $\sigma_{\gamma(s_{n+1})}^{\circ (n+1)}(2)\in \partial\ ^{b}T_{\gamma(s_{n+1})}^0$. Since $\{s_n\}_{n\geq 0}\subset[0,1)$ is an increasing sequence, it converges to some $s_\infty\in(0,1]$. As the critical value $2$ of $\sigma_{\gamma(s_{n+1})}$ escapes to $\partial\Omega_{\gamma(s_{n+1})}$ in $n+1$ iterates, it follows that in the dynamical plane of $\sigma_{\gamma(s_\infty)}$, the infinite forward orbit of the critical value $2$ is well-defined, and this orbit accumulates on the boundary of $\Omega_{\gamma(s_\infty)}$. By Subsection~\ref{dyn_near_cusp}, this can happen only if the cusp point $-2$ has an attracting direction (for $\sigma_{\gamma(s_\infty)}$), and if $2$ lies in one of its attracting petals. In particular, we must have that $\gamma(s_\infty)\in\mathcal{I}$. As $\overline{\mathfrak{T}^+}$ intersects $\mathcal{I}$ only at $4+i\sqrt{3}$, we conclude that $\gamma(s_\infty)=4+i\sqrt{3}$; i.e., $s_\infty=1$. Therefore, $\mathfrak{T}^+\cap\partial{S}=\gamma(s_0)\bigcup_{n\geq0}\gamma((s_n,s_{n+1}])$. Since $\gamma(s_0)\in\gamma_0$, and $\gamma((s_n,s_{n+1}])\subset\gamma_{n+1}$, the proof is now complete.
\end{proof}

By Lemma~\ref{arc_escape_time_lemma}, for all $a'\in\mathfrak{T}^+\cap\partial{S}$, the critical value $2$ lies in the tiling set $T_{a'}^\infty$. The arguments of Proposition~\ref{conn_locus_almost_closed} now apply verbatim to show that for parameters $a\in S$ sufficiently close to $\mathfrak{T}^+$, the critical value $2$ lies in the tiling set $T^\infty_a$. Hence, $\cC(\mathcal{S})$ does not accumulate on $\mathfrak{T}^+\cap\partial S$. A completely analogous argument shows that $\cC(\mathcal{S})$ does not accumulate on $\mathfrak{T}^-\cap\partial S$.

It follows from the above and the description of the boundary of $\widehat{S}$ given in (\ref{para_space_boundary}) (compare Figure~\ref{fig:para_space_boundary}) that any limit point of $\cC(\mathcal{S})$ outside $S$ must lie on $\mathcal{I}$. 
\end{proof}

\subsection{Dynamics on The Tiling set}\label{dyn_unif_tiling_sec}

The goal of this subsection is to construct a conformal model of $\sigma_a$ on a suitable subset of its tiling set. To this end, we first need to define a reflection map on a suitable simply connected domain.

\subsubsection{The Reflection Map $\rho$}\label{reflection_subsec}

Consider the open unit disk $\D$ in the complex plane. Let $C_1$, $C_2$, $C_3$ be the circles with centers at $(1,\sqrt{3}), (-2,0),$ and $(1,-\sqrt{3})$ of radius $\sqrt{3}$ each. We denote the intersection of $\D$ and $C_i$ by $\widetilde{C}_i$. Then $\widetilde{C}_1$, $\widetilde{C}_2$, and $\widetilde{C}_3$ are hyperbolic geodesics in $\D$, and they form an ideal triangle which we call $\widetilde{T}$ (see Figure~\ref{ideal_triangle_pic}). They bound a closed (in the topology of $\D$) region $\Pi$.

\begin{figure}[ht!]
\begin{center}
\includegraphics[scale=0.33]{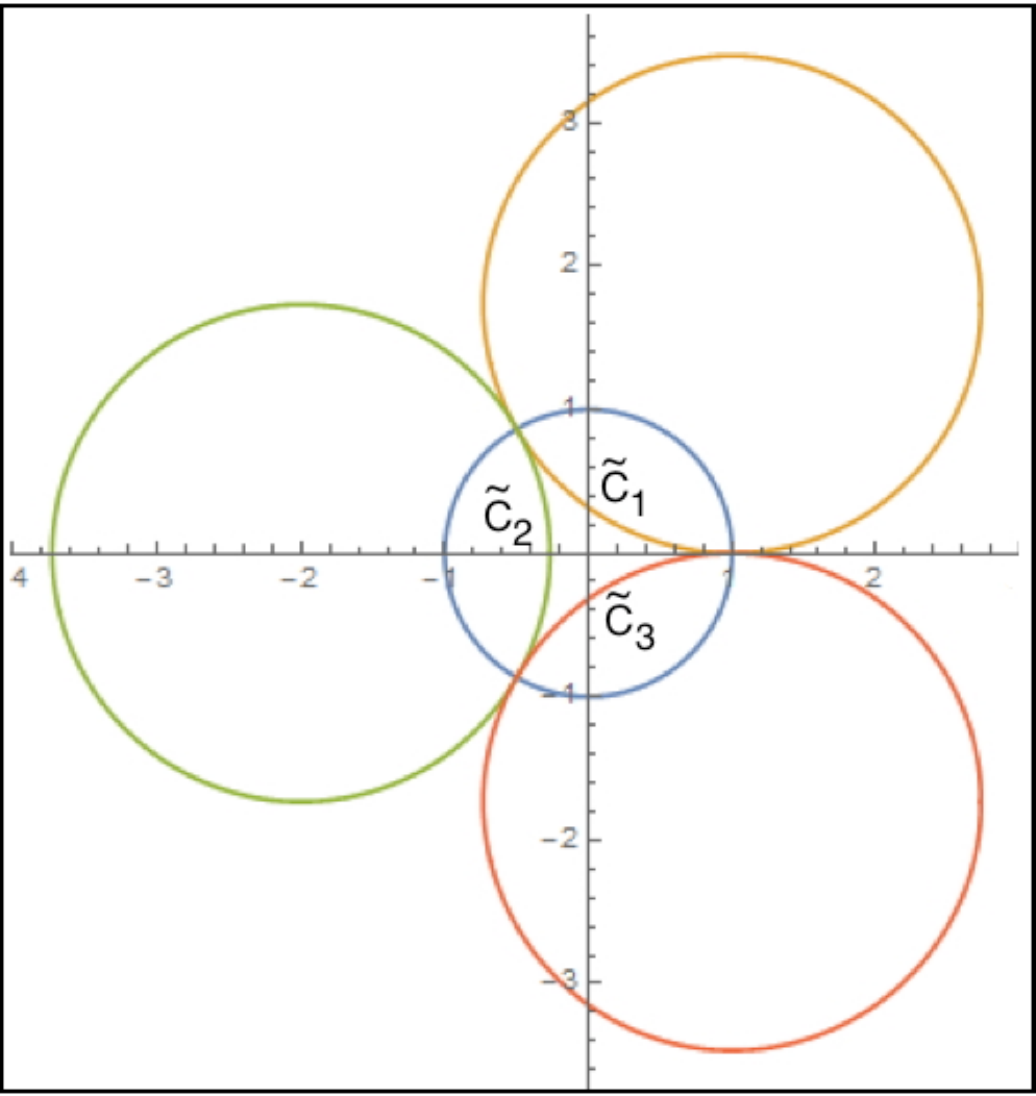}\ \includegraphics[scale=0.09]{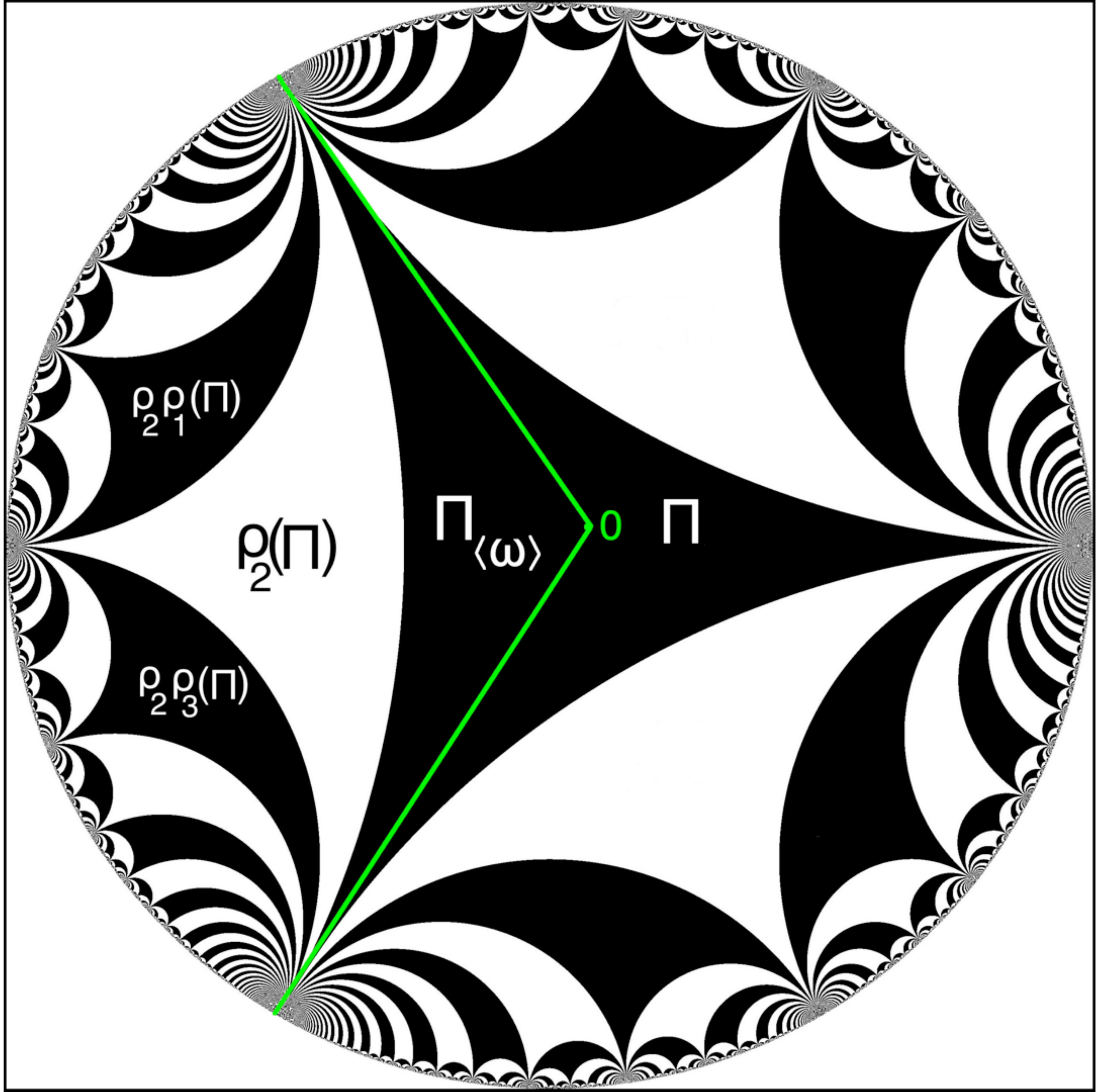}
\end{center}
\caption{Left: The hyperbolic geodesics $\widetilde{C}_1$, $\widetilde{C}_2$ and $\widetilde{C}_3$, which are sub-arcs of the circles $C_1$, $C_2$ and $C_3$ respectively, form an ideal triangle in $\D$. Right: The image of $\Pi$ under $\rho_2$, and the fundamental domain $\Pi_{\langle\omega\rangle}$ are shown.}
\label{ideal_triangle_pic}
\end{figure}

Let $\rho_i$ be the reflection with respect to the circle $C_i$, and $\D_i$ be the connected component of $\D\setminus \Pi$ containing $\Int{\rho_i(\Pi)}$. The maps $\rho_1$, $\rho_2$, and $\rho_3$ generate a subgroup $\mathcal{G}$ of $\mathrm{Aut}(\D)$. The group $\mathcal{G}$ is called the \emph{ideal triangle group}. As an abstract group, it is given by the generators and relations 
$$
\langle\rho_1, \rho_2, \rho_3: \rho_1^2=\rho_2^2=\rho_3^2=\mathrm{id}\rangle.
$$

We consider the Riemann surface $\mathcal{Q}:=\faktor{\D}{\langle\omega\rangle}$, where $\omega=e^{\frac{2\pi i}{3}}$. Note that a fundamental domain of $\D$ under the action of $\langle\omega\rangle$ is given by 
$$
\D_{\langle\omega\rangle}:=\{\vert z\vert<1,\ \frac{2\pi}{3}\leq\arg{z}<\frac{4\pi}{3}\}\cup\{0\},
$$ 
and hence $\mathcal{Q}$ is biholomorphic to the surface obtained from $\D_{\langle\omega\rangle}$ by identifying the radial line segments $\{re^{\frac{2\pi i}{3}}:0<r<1\}$ and $\{re^{\frac{4\pi i}{3}}:0<r<1\}$ by $z\mapsto\omega z$. This endows $\mathcal{Q}$ with a preferred choice of conformal coordinates. In these coordinates, the identity map is an embedding of the surface $\D_2\cup\widetilde{C}_2$ into $\mathcal{Q}$.

The map $\rho_2$ induces a map $\rho:\D_2\cup\widetilde{C}_2\to\mathcal{Q}$. Connected components of $\rho^{-n}(\rho_2(\Pi))$ are called \emph{tiles of rank $n$} of $\D_2$, and each such component is of the form $\rho_2\circ\cdots\circ\rho_i(\Pi)$. 

Clearly, $\rho$ extends continuously as an orientation-reversing double covering of $\partial\mathcal{Q}$ with three neutral fixed points. Note that $\partial\mathcal{Q}$ is obtained by gluing the end-points $\omega$ and $\omega^2$ of the arc $\{e^{i\theta}:\frac{2\pi}{3}\leq\theta\leq\frac{4\pi}{3}\}$. Moreover, the map $\rho:\partial\mathcal{Q}\to\partial\mathcal{Q}$ admits a Markov partition $\partial\mathcal{Q}=\{e^{i\theta}:\frac{2\pi}{3}\leq\theta\leq\pi\}\cup\{e^{i\theta}:\pi\leq\theta\leq\frac{4\pi}{3}\}$ with transition matrix 
$$
M:=\begin{bmatrix} 1 & 1\\ 1 & 1 \end{bmatrix}.
$$  

Incidentally, the map \begin{equation} B:\mathbb{S}^1\to\mathbb{S}^1, B(z)=\frac{3\overline{z}^2+1}{3+\overline{z}^2}, \label{blaschke_map}\end{equation} which models the dynamics of maps in $\cC(\mathfrak{L}_0)$ on their Julia sets (see Appendix~\ref{anti_rational_parabolic}), admits a Markov partition $\mathbb{S}^1=\{e^{i\theta}:0\leq\theta\leq\pi\}\cup\{e^{i\theta}:\pi\leq\theta\leq 2\pi\}$ with the same transition matrix $M$ as above. Using expansiveness of the maps $\rho\vert_{\partial\mathcal{Q}}$ and $B\vert_{\mathbb{S}^1}$, one now easily sees that $\rho\vert_{\partial\mathcal{Q}}$ and $B\vert_{\mathbb{S}^1}$ are topologically conjugate by a homeomorphism $\mathcal{E}:\partial\mathcal{Q}\to\mathbb{S}^1$ which maps $\omega, -1\in\partial\mathcal{Q}$ to $1,-1\in\mathbb{S}^1$ respectively.

We will denote the set of all angles in $\partial\mathcal{Q}\cong\faktor{\left[\frac13,\frac23\right]}{\{\frac13\sim\frac23\}}$ (respectively, in $\R/\Z$) that are pre-periodic under $\rho$ (respectively, under $B$) by $\mathrm{Per}(\rho)$ (respectively, by $\mathrm{Per}(B)$). Clearly, $\mathcal{E}$ maps $\mathrm{Per}(\rho)$ onto $\mathrm{Per}(B)$.

Note that in \cite[\S~3.1]{LLMM1}, we defined $\mathcal{G}$-rays of $\D$. $\mathcal{G}$-rays of $\D$ with angles in $\left[\frac13,\frac23\right]$ yield rays in $\mathcal{Q}$ such that the image of the ray at angle $\theta$ under $\rho$ is the ray at angle $\rho(\theta)$.

The map $\rho$ will be used below to describe a conformal model of $\sigma_a$ on its tiling set.

\subsubsection{Dynamical Uniformization of The Tiling set}\label{schwarz_group_sec}

\begin{definition}[Depth]\label{def_depth}
For any $a$ in the escape locus of $\mathcal{S}$, the \emph{smallest} positive integer $n(a)$ such that $\sigma_a^{\circ n(a)}(2)\in T_a^0$ is called the \emph{depth} of $a$.
\end{definition} 

Let us denote the sub-surface $\faktor{\Pi}{\langle\omega\rangle}$ of $\mathcal{Q}$ by $\mathcal{Q}_1$. Note that $\faktor{\Pi}{\langle\omega\rangle}$ is homeomorphic to the surface obtained from $\Pi_{\langle\omega\rangle}:=\{z\in\Pi: 2\pi/3\leq\arg{z}\leq4\pi/3\}\cup\{0\}$ by identifying the radial line segments $\{re^{2\pi i/3}:0<r<1\}$ and $\{re^{4\pi i/3}:0<r<1\}$ under $z\mapsto\omega z$ (see Figure~\ref{ideal_triangle_pic}). 

\begin{proposition}\label{schwarz_group}
1) For $a\in\cC(\mathcal{S})$, the map $\sigma_a:T_a^\infty\setminus\Int{T_a^0}\to T_a^\infty$ is conformally conjugate to $\rho:\D_2\cup\widetilde{C}_2\to\mathcal{Q}$.

2) For $a\in S\setminus\cC(\mathcal{S})$, 
$$
\sigma_a:\displaystyle\bigcup_{n=1}^{n(a)} \sigma_a^{-n}(T_a^0)\to\displaystyle\bigcup_{n=0}^{n(a)-1} \sigma_a^{-n}(T_a^0)
$$ 
is conformally conjugate to 
$$
\rho:\displaystyle\bigcup_{n=1}^{n(a)} \rho^{-n}(\mathcal{Q}_1)\to\displaystyle\bigcup_{n=0}^{n(a)-1}\rho^{-n}(\mathcal{Q}_1).
$$
\end{proposition}
\begin{proof}
Since $\mathcal{Q}_1$ is simply connected, we can choose a homeomorphism $\psi_a$ between $T_a^0$ and $\mathcal{Q}_1$ such that it is conformal on the interior. We can further assume that $\psi_a(\infty)=0$, and its continuous extension sends the cusp point $-2\in\partial T_a^0$ to the point $\omega$ on $\partial\mathcal{Q}_1$.

Note that $\sigma_a:\sigma_a^{-1}(T_a^0)\to T_a^0$ is a three-to-one branched cover branched only at $c_a^\ast$, and $\rho:\rho_2(\Pi)\to\mathcal{Q}_1$ is a three-to-one branched cover branched only at $\rho_2(0)$. Moreover, $\sigma_a$ fixes $\partial T_a^0$ pointwise, and $\rho$ fixes $\widetilde{C}_2\cup\{\omega\}\cong\partial\mathcal{Q}_1$ pointwise.

This allows one to lift $\psi_a$ to a conformal isomorphism from $\sigma_a^{-1}(T_a^0)$ onto $\rho_2(\Pi)$ such that the lifted map sends $c_a^\ast$ to $\rho_2(0)$, and continuously matches with the initial map $\psi_a$ on $\partial T_a^0$. We denote this extended conformal isomorphism by $\psi_a$. By construction, $\psi_a$ is equivariant with respect to the actions of $\sigma_a$ and $\rho$ on $\partial\sigma_a^{-1}(T_a^0)$. 

1) If $a\in\cC(\mathcal{S})$, then every tile of $T_a^\infty$ (of rank greater than one) maps diffeomorphically onto $\sigma_a^{-1}(T_a^0)$ under some iterate of $\sigma_a$, and each tile of $\D_2$ (of rank greater than one) maps diffeomorphically onto $\rho_2(\Pi)$ under some iterate of $\rho$. This fact, along with the equivariance property of $\psi_a$ mentioned above, enables us to lift $\psi_a$ to all tiles using the iterates of $\sigma_a$ and $\rho$. This produces the desired biholomorphism $\psi_a$ between $T_a^\infty$ and $\mathcal{Q}$ which conjugates $\sigma_a$ to $\rho$.

2) For $a\in S\setminus\cC(\mathcal{S})$, the above construction of $\psi_a$ can be carried out on the tiles of $T_a^\infty$ that map diffeomorphically onto $\sigma_a^{-1}(T_a^0)$, which includes all tiles of rank up to $n(a)$. This completes the proof.
\end{proof}

\begin{remark}\label{external_conj_rem}
Since $\psi_a$ (obtained in Proposition~\ref{schwarz_group}) conjugates $\sigma_a$ to the model map $\rho$ ``outside" the non-escaping set, the conjugacy $\psi_a$ is referred to as the ``external conjugacy'', and the model map $\rho$ is called the ``external map" of $\sigma_a$.
\end{remark}

\begin{remark}\label{thrice_punc_sphere_rem}
$\rho$ maps each of the two connected components of $\D_2\setminus\rho_2(\Pi)$ univalently onto $\D_2$. These two univalent restrictions of $\rho$ as well as their inverses act on $\D_2\setminus\rho_2(\Pi)$ and generate a partially defined dynamical system. A fundamental domain of $\D_2\setminus\rho_2(\Pi)$ under the action of the conformal maps (i.e., under words of even length) of this dynamical system can be identified with $\Int{\rho_2\rho_1(\Pi)}\cup\rho_2\rho_1\rho_2(\Pi)$. The quotient of $\D_2\setminus\rho_2(\Pi)$ by this conformal dynamical system is a thrice punctured sphere. For $a\in \cC(\mathcal{S})$, the Schwarz reflection map $\sigma_a$ induces an anti-conformal involution on the thrice punctured sphere fixing the punctures.
\end{remark}

We can use the map $\psi_a$ to define dynamical rays for the maps $\sigma_a$.

\begin{definition}[Dynamical Rays of $\sigma_a$]\label{dyn_ray_schwarz}
The pre-image of a ray at angle $\theta$ in $\mathcal{Q}$ under the map $\psi_a$ is called a $\theta$-dynamical ray of $\sigma_a$.
\end{definition}

Clearly, the image of a dynamical $\theta$-ray under $\sigma_a$ is a dynamical ray angle $\rho(\theta)$.

\begin{proposition}[Landing of Pre-periodic Rays]\label{per_rays_land}
Let $a\in\cC(\mathcal{S})$, and $\theta\in\mathrm{Per}(\rho)$. Then the following statements hold true.

1) The dynamical $\theta$-ray of $\sigma_a$ lands on $\partial T_a^\infty$. 

2) The $\frac13=\frac23$-ray of $\sigma_a$ lands at $-2$, and no other ray lands at $-2$. The iterated pre-images of the $\frac13$-ray land at the iterated pre-images of $-2$ (under $\sigma_a$).

3) Let $\theta\in\mathrm{Per}(\rho)\setminus\displaystyle\bigcup_{n\geq0}\rho^{-n}\left(\left\{\frac{1}{3}\right\}\right)$. Then, the dynamical ray of $\sigma_a$ at angle $\theta$ lands at a repelling or parabolic (pre-)periodic point on $\partial T_a^\infty$. 
\end{proposition}
\begin{proof}
The proof of \cite[Proposition~6.34]{LLMM1} applies mutatis mutandis to the present setting.
\end{proof}

Let us also state a converse which can be proved following \cite[Theorem~24.5, Theorem~24.6]{L6}.

\begin{proposition}[Repelling and Parabolic Points are Landing Points of Rays]\label{rep_para_landing_point}
Let $a\in\cC(\mathcal{S})$. Then, every repelling and parabolic periodic point of $\sigma_a$ is the landing point of finitely many (at least one) dynamical rays. Moreover, all these rays have the same period under $\sigma_a^{\circ 2}$.
\end{proposition}

\section{A Straightening Theorem}\label{sec_straightening}

The goal of this section is to prove a straightening theorem that will allow us to show that the dynamics of $\sigma_a$ on its non-escaping set is topologically equivalent to a suitable anti-rational map. We will prove our straightening theorem for a class of \emph{pinched anti-quadratic-like maps} with controlled geometry and prescribed asymptotics at the pinching point. 

\begin{definition}[Pinched Anti-quadratic-like Map]\label{pinched_def}
A continuous map $\mathbf{F}:\left(\overline{\mathbf{U}},\infty\right)\to\left(\overline{\mathbf{V}},\infty\right)$ of degree $2$ is called a \emph{pinched anti-quadratic-like map} if
\begin{enumerate}
\item $\mathbf{F}(\partial\mathbf{U})=\partial\mathbf{V}$,
 
\item $\mathbf{F}$ is anti-holomorphic on $\mathbf{U}$,

\item $\mathbf{U}\subset\mathbf{V}(\subset\widehat{\C})$ are Jordan domains, and $\left(-\infty,-x\right)\subset\mathbf{U}$ for some $x>0$, 

\item there exists $M>0$ such that $\partial\mathbf{V}\cap\{\vert z\vert\geq M\}=\{m e^{\pm\frac{2\pi i}{3}}: m\geq M\}$,

\item $\partial\mathbf{U}\cap\partial\mathbf{V}=\{\infty\}$,

\item $\mathbf{F}(z)= \overline{z}+\frac12+O(\frac{1}{\overline{z}})$ as $z\to\infty$,

\item $\partial\mathbf{V}$ is smooth except at $\infty$, and $\partial\mathbf{U}$ is smooth except at $\mathbf{F}^{-1}(\infty)$.
\end{enumerate}
\end{definition}
(see Figure~\ref{fig:anti_quad_like}).

\begin{remark}
1) It follows from the definition that $\mathbf{F}^{-1}(\infty)\subset\partial\mathbf{U}$, and $\partial\mathbf{U}$ meets the circle at infinity $\{\infty\cdot e^{2\pi i\theta}:\theta\in\R/\Z\}$ at $\infty\cdot e^{\pm\frac{2\pi i}{3}}$. Moreover, near $\infty$, the boundaries $\partial\mathbf{U}$ and $\partial\mathbf{V}$ bound two infinite strips each of asymptotic width $\frac{\sqrt{3}}{4}$.

2) The assumption that $\partial\mathbf{V}$ contains two infinite rays at angles $\pm\frac{2\pi}{3}$ is not a serious restriction. In fact, any unbounded Jordan domain (containing the negative real axis) whose boundary meets the circle at infinity at $\infty\cdot e^{\pm\frac{2\pi i}{3}}$ can be mapped to a domain $\mathbf{V}$ of the above type by a conformal map that is asymptotically linear near $\infty$.

3) The negative real axis is a repelling direction of $\mathbf{F}$ at $\infty$; i.e., points on the negative real axis with large absolute value are repelled away from $\infty$ under the action of $\mathbf{F}$.
\end{remark}

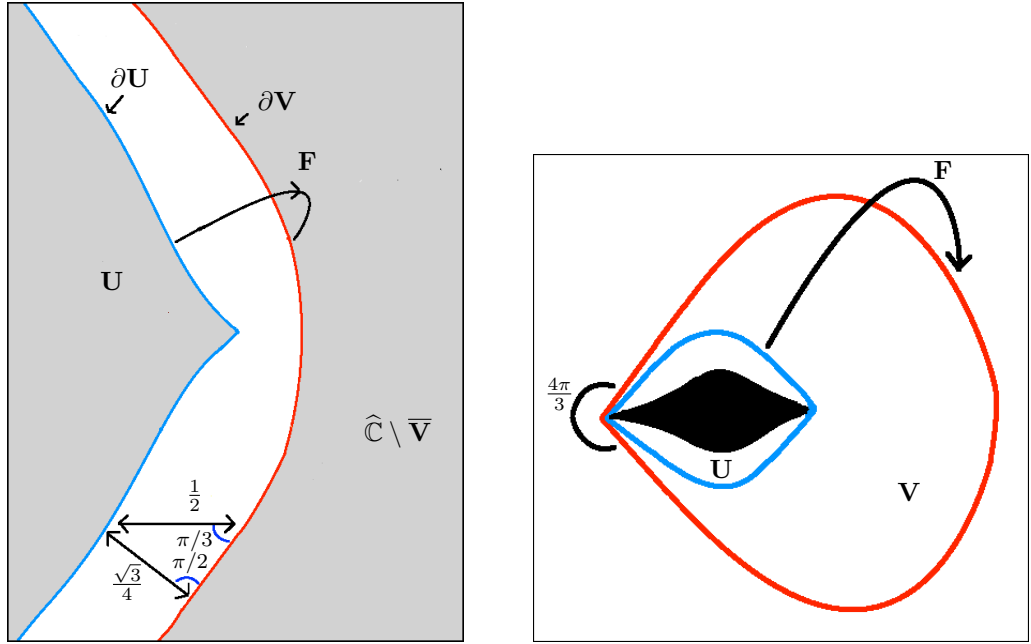
\begin{figure}[ht!]
\begin{tikzpicture}
\node[anchor=south west,inner sep=0] at (0,0) {\includegraphics[width=0.48\textwidth]{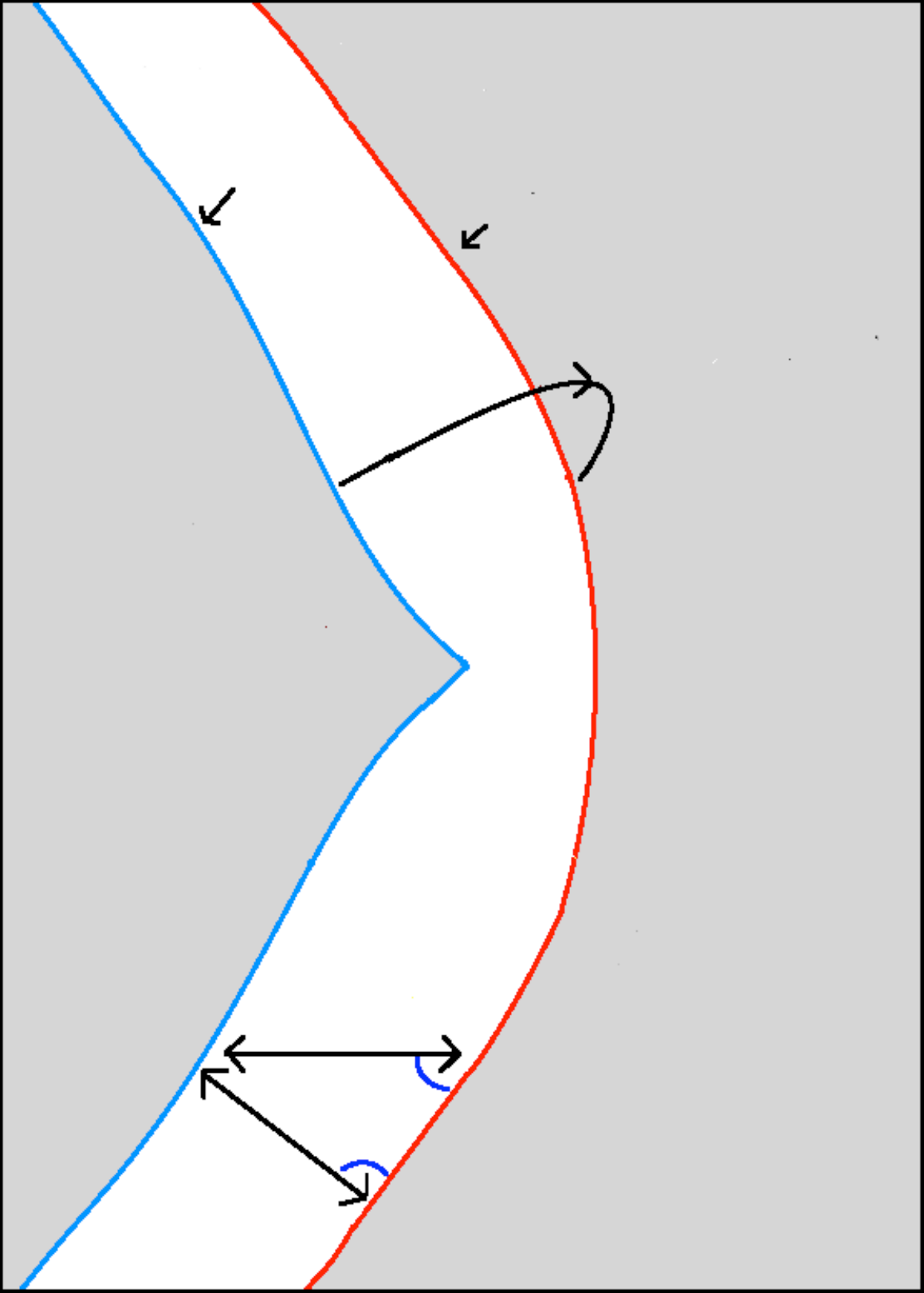}};
\node[anchor=south west,inner sep=0] at (7,0) {\includegraphics[width=0.52\textwidth]{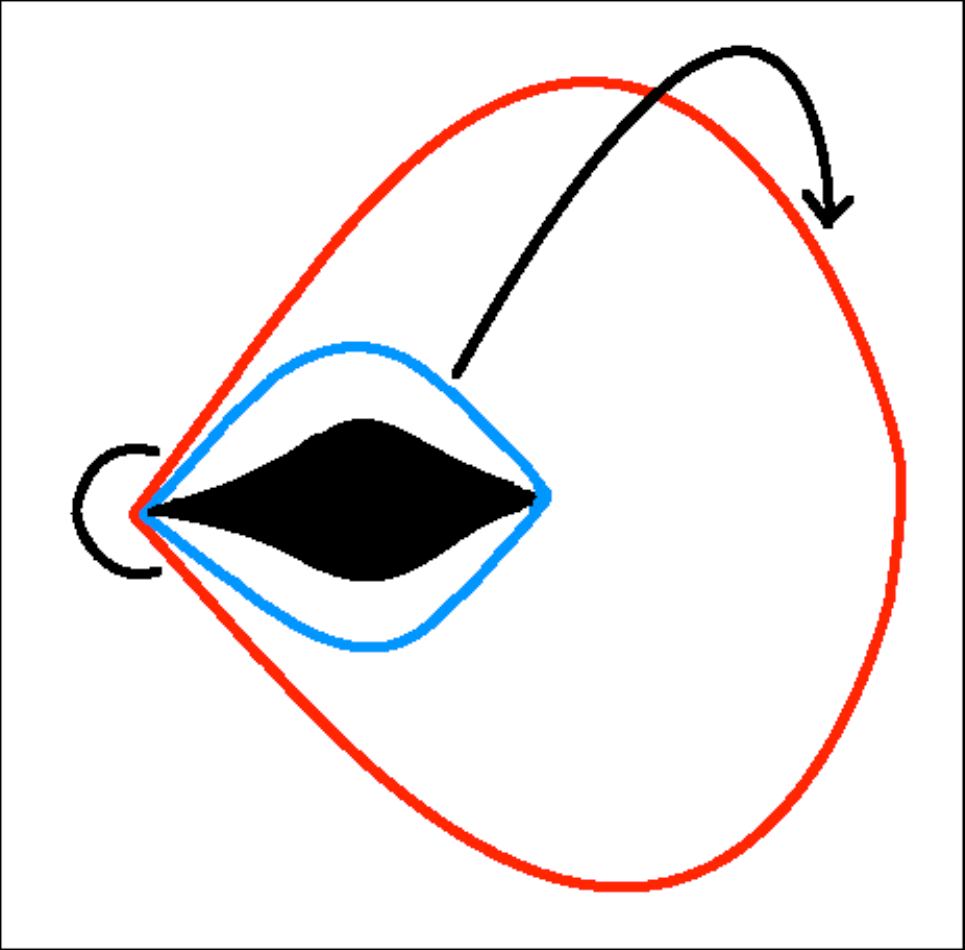}};
\node at (1.64,7.5) {$\mathbf{\partial\mathbf{U}}$};
\node at (3.6,7.2) {$\mathbf{\partial\mathbf{V}}$};
\node at (4,6.4) {$\mathbf{F}$};
\node at (1.4,4.8) {$\mathbf{U}$};
\node at (5.2,2.8) {$\widehat{\C}\setminus\overline{\mathbf{V}}$};
\node at (2.44,1.06) {\begin{footnotesize}$\pi/2$\end{footnotesize}};
\node at (2.5,1.35) {\begin{footnotesize}$\pi/3$\end{footnotesize}};
\node at (2.5,1.95) {$\frac12$};
\node at (1.6,0.8) {$\frac{\sqrt{3}}{4}$};
\node at (8.4,3) {$\infty$};
\node at (9.5,2.3) {$\mathbf{U}$};
\node at (12,2) {$\mathbf{V}$};
\node at (7.35,3.3) {$\frac{4\pi}{3}$};
\node at (12.45,6.28) {$\mathbf{F}$};
\end{tikzpicture}
\caption{Left: A pinched anti-quadratic-like map $\mathbf{F}$. Right: The same map in different coordinates, where the pinching point corresponds to $\infty$. The filled Julia set (schematic) is shown in black.}
\label{fig:anti_quad_like}
\end{figure}

\begin{definition}[Filled Julia Set, and Hybrid Conjugacy]\label{hybrid_def}
1) The \emph{filled Julia set} $K_\mathbf{F}$ of a pinched anti-quadratic-like map $\mathbf{F}:\overline{\mathbf{U}}\to\overline{\mathbf{V}}$ is defined as 
$$
K_\mathbf{F}:=\{z\in\overline{\mathbf{U}}:\mathbf{F}^{\circ n}(z)\in\overline{\mathbf{U}}\ \forall\ n\geq 0\}.
$$
 
2) Two pinched anti-quadratic-like maps $\mathbf{F}_i:\overline{\mathbf{U}_i}\to\overline{\mathbf{V}_i}$, $i\in\{1,2\}$, are said to be \emph{hybrid conjugate} if there exists a homeomorphism $\Phi:\overline{\mathbf{V}_1}\to\overline{\mathbf{V}_2}$, quasiconformal on $\mathbf{V}_1$ and $\overline{\partial}\Phi=0$ a.e. on $K_{\mathbf{F}_1}$, that conjugates $\mathbf{F}_1$ to $\mathbf{F}_2$.
\end{definition}

We now prove one of our key results that allows us to ``straighten'' pinched anti-quadratic-like maps to quadratic anti-rational maps with a simple parabolic fixed point (see Appendix~\ref{anti_rational_parabolic} for a detailed description of the family of such maps).

\begin{lemma}\label{straightening_lemma}
Every pinched anti-quadratic-like map $\mathbf{F}$ is hybrid conjugate to a pinched anti-quadratic-like restriction of some member of the family $\mathfrak{L}_0$.
\end{lemma}
\begin{proof}
We will glue a suitably chosen attracting petal of the model parabolic map $q(z)=\overline{z}+\overline{z}^2$ outside $\mathbf{U}$ (note that $q$ restricted to its parabolic basin is conformally conjugate to the anti-Blaschke product $B$ on $\D$, see (\ref{blaschke_map})). The change of coordinates $\eta: z\mapsto-\frac{1}{2z}$ conjugates $q$ to $z\mapsto \overline{z}+\frac12+O(\frac{1}{\overline{z}})$ as $z\to\infty$. Let us choose an attracting petal $P$ that subtends an angle $\frac{4\pi}{3}$ at the parabolic fixed point $0$ of $q$ and such that $\partial P$ contains the critical point $-\frac12$ (see Figure \ref{petal_pic}). We can also require that $\partial P$ is smooth except at $0$, and $q^{-1}(P)$ is simply connected. Then, we have that $q:q^{-1}(P)\to P$ is a two-to-one branched covering. Moreover, $P$ can be chosen so that the change of coordinate $\eta$ maps $P$ to an unbounded domain containing a part of the positive real axis such that the domain is bounded by the infinite rays $\{re^{\pm\frac{2\pi i}{3}}:r\geq M_0\}$ (for some $M_0>0$) and some smooth curve connecting $M_0 e^{\pm\frac{2\pi i}{3}}$. 

\begin{figure}[ht!]
\begin{tikzpicture}
\node[anchor=south west,inner sep=0] at (0,0) {\includegraphics[width=0.6\textwidth]{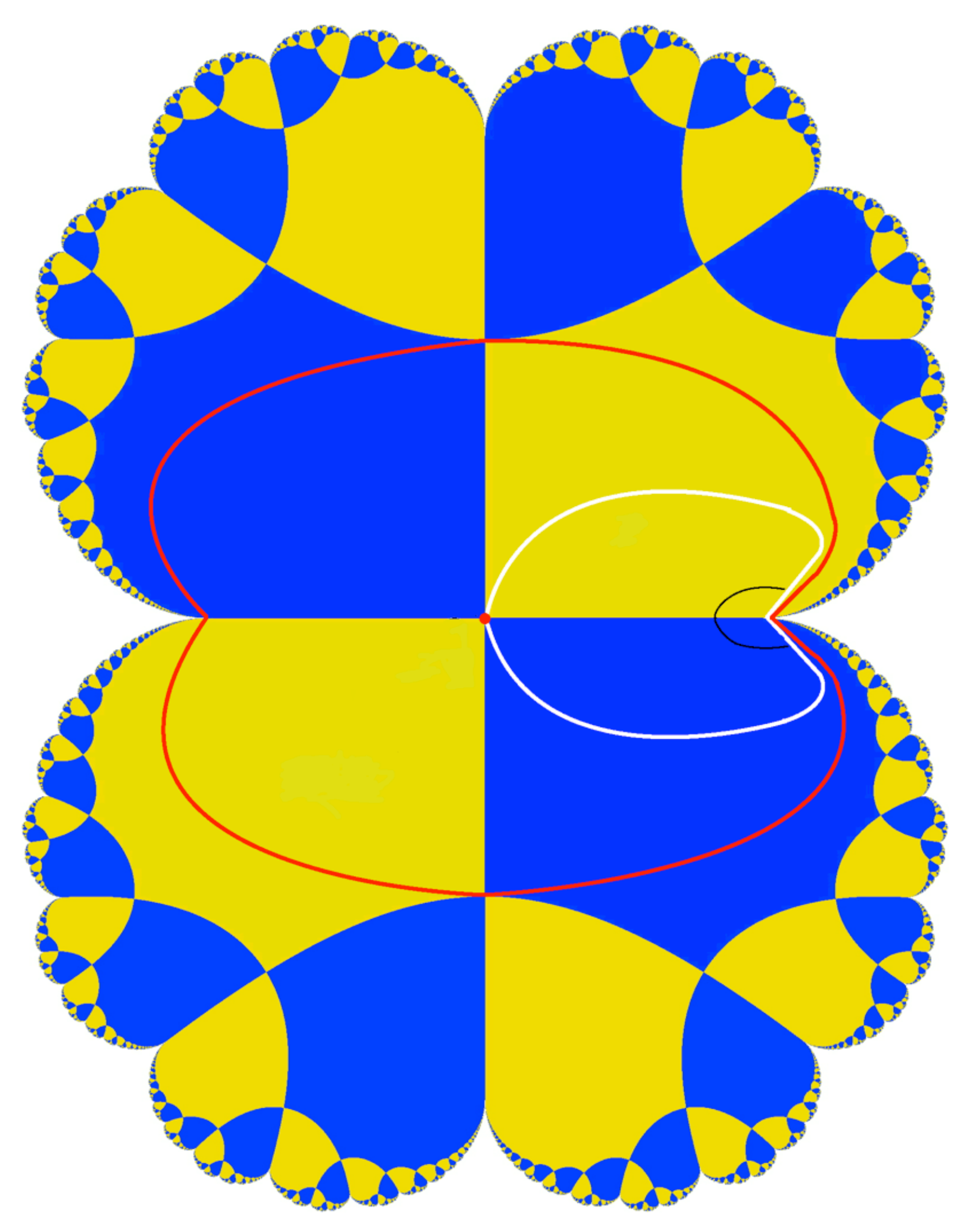}};
\node at (5,5.25) {$P$};
\node at (2.8,3.5) {$q^{-1}(P)$};
\node at (3.35,4.45) {$-\frac12$};
\node at (5.4,4.4) {\textcolor{white}{$\frac{4\pi}{3}$}};
\end{tikzpicture}
\caption{The filled Julia set of $q(z)=\overline{z}+\overline{z}^2$ is shown. The attracting petal $P$ subtends an angle $\frac{4\pi}{3}$ at the parabolic fixed point $0$. The critical point $-\frac12$ (of $q$) lies on the boundary of the petal $P$. The pre-image of $P$ (under $q$) is a simply connected domain, which maps two-to-one onto $P$ branched only at $-\frac12$.}
\label{petal_pic}
\end{figure}
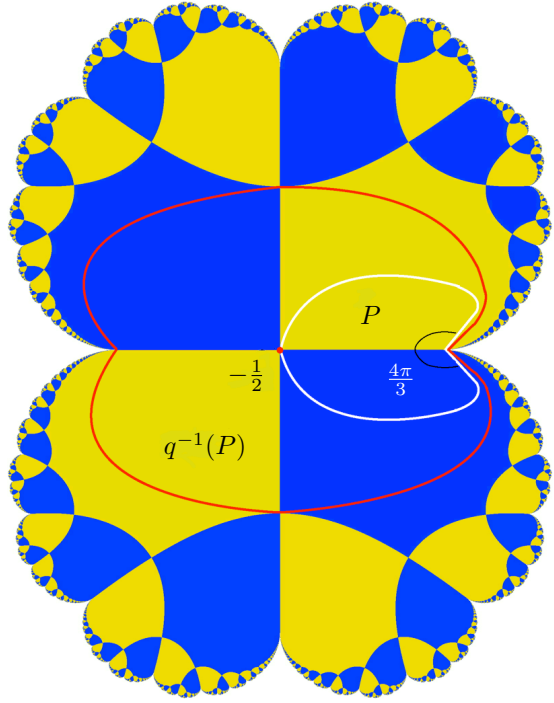

Let us choose a conformal map $\xi$ from $\widehat{\C}\setminus\overline{\mathbf{V}}$ onto $\mathfrak{P}:=\eta(P)$ which sends $\infty$ to $\infty$ (see Figure~\ref{interpolation}). Then, $\xi$ extends as a homeomorphism between $\partial\mathbf{V}$ and $\partial \mathfrak{P}$. Moreover, since the boundaries $\partial\mathbf{V}$ and $\partial \mathfrak{P}$ of both the domains (here, the boundaries are taken in the Riemann sphere $\widehat{\C}$) make a corner angle $\frac{4\pi}{3}$ at $\infty$, the extended map $\xi$ is approximately linear near $\infty$; i.e., $\xi(z)=\lambda_0 z+o(z)$ (for some $\lambda_0>0$) as $z\in\partial\mathbf{V}$ and $\vert\im(z)\vert\to+\infty$.

We take the pre-image of the curve $\partial\mathbf{V}$ under $\mathbf{F}$, and the pre-image of $\partial \mathfrak{P}$ under $\mathfrak{q}:=\eta\circ q\circ\eta^{-1}$. The homeomorphism $\xi:\partial\mathbf{V}\to\partial \mathfrak{P}$ can be lifted to obtain a homeomorphism $\xi:\partial\mathbf{U}\to \mathfrak{q}^{-1}(\partial \mathfrak{P})$ fixing $\infty$ (see Figure~\ref{interpolation}). Since the covering maps $\mathbf{F}:\partial\mathbf{U}\to\partial\mathbf{V}$ and $\mathfrak{q}: \mathfrak{q}^{-1}(\partial \mathfrak{P})\to\partial \mathfrak{P}$ are tangent to $\overline{z}$ near $\infty$, it follows that $\xi(z)$ is of the form $\lambda_0 z+o(z)$ where $z\in\partial\mathbf{U}$ and $\vert\im(z)\vert\to+\infty$.

\begin{figure}[ht!]
\begin{tikzpicture}
\node[anchor=south west,inner sep=0] at (0,0) {\includegraphics[width=\textwidth]{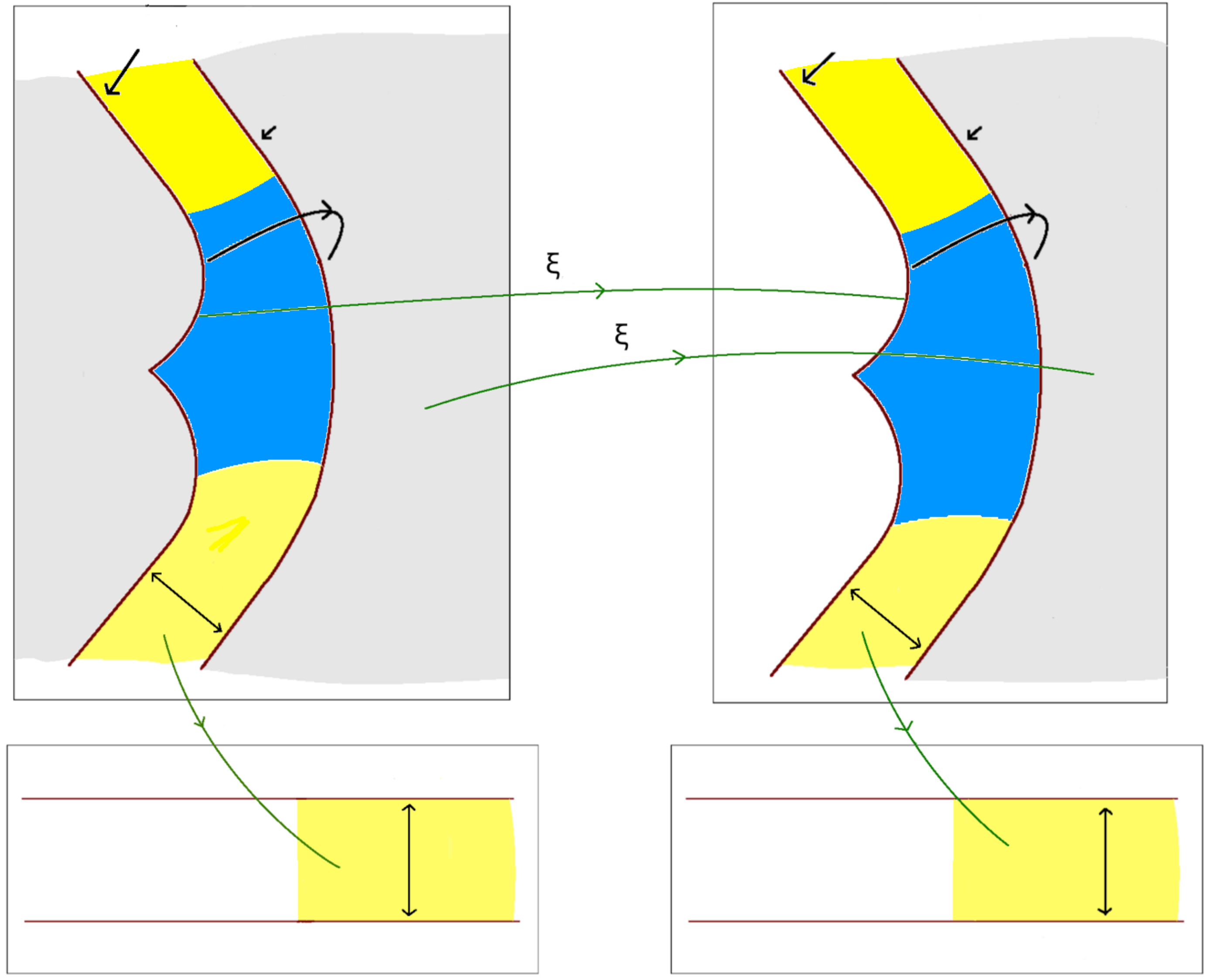}};
\node at (1.64,10) {$\partial\mathbf{U}$};
\node at (3.1,9.1) {$\partial\mathbf{V}$};
\node at (0.5,7.5) {$\mathbf{U}$};
\node at (2,8.5) {$\mathbf{V}$};
\node at (3.6,4) {$\widehat{\C}\setminus\overline{\mathbf{V}}$};
\node at (3.6,8.32) {$\mathbf{F}$};
\node at (3.8,2.2) {$y=\frac{\sqrt{3}}{8}$};
\node at (3.8,0.35) {$y=-\frac{\sqrt{3}}{8}$};
\node at (2.2,4.25) {$\frac{\sqrt{3}}{4}$};
\node at (9.5,4) {$\frac{\sqrt{3}}{4}$};
\node at (4.6,1.2) {$\frac{\sqrt{3}}{4}$};
\node at (10.8,2.2) {$y=\frac{\sqrt{3}}{8}$};
\node at (10.8,0.35) {$y=-\frac{\sqrt{3}}{8}$};
\node at (12,1.2) {$\frac{\sqrt{3}}{4}$};
\node at (11,4) {$\mathfrak{P}$};
\node at (2.5,2.7) {$\beta_1$};
\node at (9.8,2.7) {$\beta_2$};
\node at (10.5,9.08) {$\partial \mathfrak{P}$};
\node at (8.8,9.88) {$\mathfrak{q}^{-1}(\partial \mathfrak{P})$};
\node at (11.2,8.2) {$\mathfrak{q}$};
\end{tikzpicture}
\caption{The map $\xi$ is quasiconformally interpolated on the strip $\mathbf{S}$ bounded by $\partial\mathbf{U}$ and $\partial\mathbf{V}$ such that it maps $\mathbf{S}$ onto the strip $\mathfrak{S}$ bounded by $\partial \mathfrak{P}$ and $\mathfrak{q}^{-1}(\partial \mathfrak{P})$. The interpolation is done by mapping the top and bottom accesses of $\mathbf{S}$ and $\mathfrak{S}$ (to $\infty$) conformally onto the right-half of the horizontal strip $\{x+iy:\vert y\vert<\frac{\sqrt{3}}{8}\}$.}
\label{interpolation}
\end{figure}
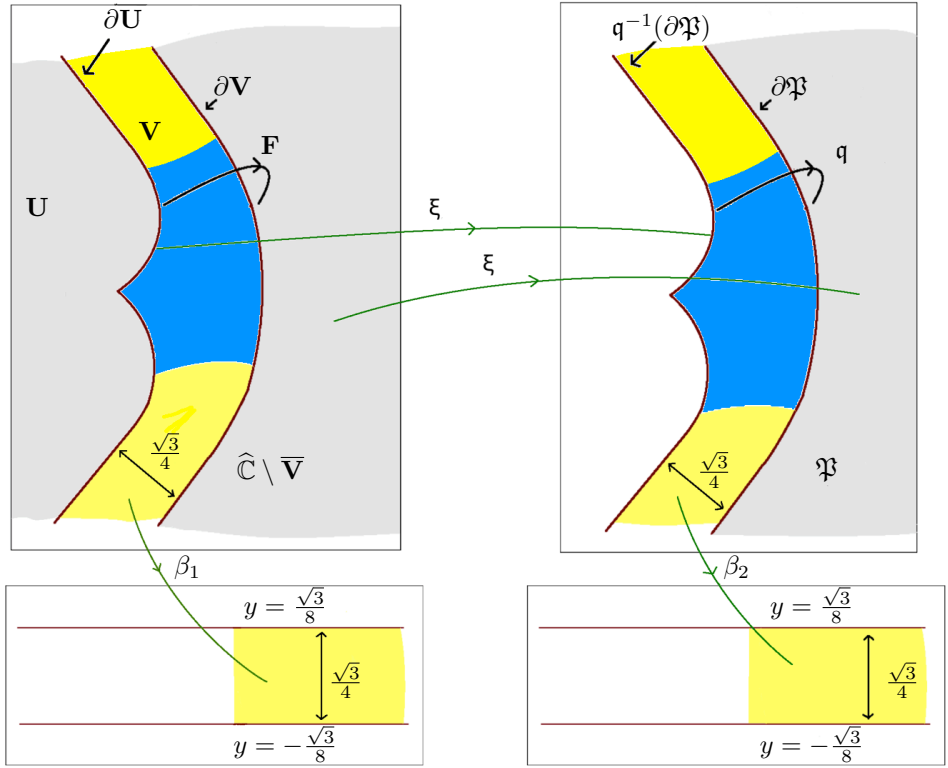

Let us denote the strip between $\partial\mathbf{U}$ and $\partial\mathbf{V}$ by $\mathbf{S}$. We claim that $\xi$ can be quasiconformally interpolated on $\mathbf{S}$ so that the image of $\mathbf{S}$ under the interpolating map is the strip $\mathfrak{S}$ bounded by $\partial \mathfrak{P}$ and $\mathfrak{q}^{-1}(\partial \mathfrak{P})$. 

We will first justify the existence of such an interpolating map in two accesses of $\mathbf{S}$ to $\infty$ (shaded in yellow in Figure~\ref{interpolation}). To this end, let us consider the bottom access of $\mathbf{S}$ to $\infty$. Rotating the bottom access of $\mathbf{S}$ to $\infty$ anti-clockwise by an angle $\frac{2\pi}{3}$, we obtain a horizontal strip bounded by suitable right halves of the curves $y=0$, and $y=-\frac{\sqrt{3}}{4}+O(\frac1x)$ as $x\to+\infty$ (this follows from Property (6) of a pinched anti-quadratic-like map). It now follows from \cite{War} that this horizontal strip can be mapped onto the right-half of the horizontal strip $\{x+iy:\vert y\vert<\frac{\sqrt{3}}{8}\}$ by a conformal map $\widehat{\beta}_1$ such that $\widehat{\beta}_1(w)=w+o(w)$ as $\re(w)\to+\infty$. Therefore, the conformal map $\beta_1(z)=\widehat{\beta}_1(\omega z)$ sends the bottom access of $\mathbf{S}$ to $\infty$ onto the right-half of the horizontal strip $\{x+iy:\vert y\vert<\frac{\sqrt{3}}{8}\}$ such that $\beta_1(z)=\omega z+o(z)$ as $z\to\infty$.

The same is true for the bottom access of $\mathfrak{S}$ to $\infty$. More precisely, there exists a conformal map $\beta_2$ from the bottom access of $\mathfrak{S}$ to $\infty$ onto the right-half of the horizontal strip $\{x+iy:\vert y\vert<\frac{\sqrt{3}}{8}\}$ such that $\beta_2(z)=\omega z+o(z)$ as $z\to\infty$.

Thus, we obtain two maps between pairs of horizontal rays 
$$
\beta_2\circ\xi\circ\beta_1^{-1}:\{x\pm i\frac{\sqrt{3}}{8}:x>0\}\to\{x\pm i\frac{\sqrt{3}}{8}:x>0\}.
$$ 
Clearly, the two maps are of the form $x\pm i\frac{\sqrt{3}}{8}\mapsto g_{\pm}(x)\pm i\frac{\sqrt{3}}{8}$. Moreover, it follows from our analysis of the asymptotics of $\xi, \beta_1$, and $\beta_2$ that 
$$
g_{\pm}(x)=g(x)+O(1),\ \textrm{where}\ g(x)=\lambda_0x+o(x)\ \textrm{as}\ x\to+\infty.
$$ 
Therefore, we can linearly interpolate between these two maps to obtain a quasiconformal homeomorphism 
$$
x+iy\mapsto \left(\left(\frac12-\frac{4\sqrt{3}}{3}y\right)g_{-}(x)+\left(\frac12+\frac{4\sqrt{3}}{3}y\right)g_+(x)\right)+iy
$$ 
on the right-half of the horizontal strip $\{x+iy:\vert y\vert<\frac{\sqrt{3}}{8}\}$.

Going back by the change of coordinates $\beta_1$ and $\beta_2$, we obtain our desired quasiconformal map defined on the bottom access of $\mathbf{S}$ to $\infty$. This proves that $\xi$ can be quasiconformally interpolated on the bottom access to $\infty$ of $\mathbf{S}$ so that the image of the interpolating map is the bottom access to $\infty$ of the strip $\mathfrak{S}$. A completely analogous argument shows that a similar interpolating quasiconformal homeomorphism exists on the top access to $\infty$ of $\mathbf{S}$. These extensions define a quasisymmetric map $\xi$ on the boundary of a bounded Jordan domain with piecewise smooth boundary and no cusp (the blue region in Figure~\ref{interpolation}). The quasiconformal extension of $\xi$ to this bounded part of $\mathbf{S}$ now follows from the Ahlfors-Beurling extension theorem on quasidisks \cite[\S 2.3, Proposition~2.30]{BrFa}. This completes the proof of the claim that $\xi$ can be quasiconformally interpolated on $\mathbf{S}$ so that the image of $\mathbf{S}$ under the interpolating map is the strip $\mathfrak{S}$. Moreover, this quasiconformal map, which we denote by $\xi$, is equivariant on the boundary (with respect to $\mathbf{F}$ and $\mathfrak{q}$).

We now define a quasiregular map $\mathbf{G}$ of degree two on $\widehat{\C}$ as follows:

\begin{equation}
\mathbf{G}:=\left\{\begin{array}{ll}
                    \mathbf{F} & \mbox{on}\ \mathbf{U}, \\
                     \xi^{-1}\circ \mathfrak{q}\circ\xi & \mbox{on}\ \widehat{\C}\setminus\mathbf{U}.
                                          \end{array}\right. 
\label{quasi_regular_def}
\end{equation}

\noindent The fact that $\mathbf{F}$ and $\xi^{-1}\circ \mathfrak{q}\circ\xi$ match on the boundary of their domains of definition follows from the equivariance property of $\xi$ mentioned above. 

Let us now define an ellipse field $\mu$ (i.e., a Beltrami form) on $\widehat{\C}$. On $\left(\widehat{\C}\setminus\mathbf{V}\right)\cup K_\mathbf{F}$, we define $\mu$ as circles. On $\mathbf{V}\setminus K_\mathbf{F}$, we define $\mu$ in a $\mathbf{G}$-invariant manner. Since $\mathbf{G}$ is conformal outside of $\overline{\mathbf{V}}\setminus\mathbf{U}$, it follows that $\mu$ is a $\mathbf{G}$-invariant Beltrami form with $\vert\vert\mu\vert\vert_\infty<k<1$.

By the Measurable Riemann Mapping Theorem, we get a quasiconformal map $\Phi$ which straightens $\mu$. We can normalize $\Phi$ so that it fixes $\infty$, sends the unique finite pole of $\mathbf{G}$ to $0$, and sends the critical point $\xi^{-1}(1)$ of $\mathbf{G}$ to $+1$. Then, $\Phi$ conjugates $\mathbf{G}$ to a degree two anti-rational map $R$ with a fixed point at $\infty$. By construction, $\mathbf{G}$ has a unique attracting direction (coming from the attracting direction of $\mathfrak{q}$ at $\infty$), and a unique repelling direction (coming from the repelling direction of $F$ at $\infty$) at $\infty$ each of which is invariant under $\mathbf{G}$. Hence, the same is true for $R$. Therefore, $R^{\circ 2}$ has a simple parabolic fixed point of multiplier $1$ at $\infty$. Moreover, $R$ has a pole at $0$, and a critical point at $+1$. Hence, $R\in\mathcal{F}$ (see Subsection~\ref{appendix_subsec_1}). Finally, since the critical {\'E}calle height of the map $q(z)=\overline{z}+\overline{z}^2$ is $0$ (see Appendix~\ref{appendix_subsec_3} for the definition of critical {\'E}calle height), the {\'E}calle height of the critical value $R(1)$ (which lies in an attracting petal of the neutral fixed point $\infty$ of $R$) is $0$ as well. Therefore, $R\in\mathfrak{L}_0$.

Note that by our construction, $\overline{\partial}\Phi\equiv0$ a.e. on $K_\mathbf{F}$. Therefore, $\Phi$ is the desired hybrid conjugacy between $\mathbf{F}:\overline{\mathbf{U}}\to\overline{\mathbf{V}}$ and a pinched anti-quadratic-like restriction of $R$.
\end{proof}

We will now apply Lemma~\ref{straightening_lemma} to extract parabolic quadratic anti-rational maps from the Schwarz reflection maps in the family $\mathcal{S}$.

\begin{theorem}[Straightening Schwarz Reflections]\label{straightening_schwarz}
1) For $a\in S$, there exists $V_a\subset\Omega_a$ and a univalent map $\eta_a$ on $\overline{V_a}$ sending $-2$ to $\infty$ such that with the notations 
$$
\pmb{\sigma}_a:=\eta_a\circ\sigma_a\circ\eta_a^{-1},\ \mathbf{V}_a:=\eta_a(V_a),\ \textrm{and}\ \mathbf{U}_a:=\eta_a(\sigma_a^{-1}(V_a)),
$$ 
the map 
$$
\pmb{\sigma}_a:\left(\overline{\mathbf{U}_a},\infty\right)\to\left(\overline{\mathbf{V}_a},\infty\right)
$$ 
is a pinched anti-quadratic-like map with filled Julia set $\eta_a(K_a)$. Hence, $\sigma_a$ is hybrid conjugate to a pinched anti-quadratic-like restriction of some member of the family $\mathfrak{L}_0$.

2) If $a\in\cC(\mathcal{S})$, then $\sigma_a:\overline{\sigma_a^{-1}(V_a)}\to\overline{V_a}$ is hybrid conjugate to a unique member $R_{\alpha,A}$ of the parabolic Tricorn $\cC(\mathfrak{L}_0)$. This unique map in $\cC(\mathfrak{L}_0)$ (equivalently, this unique parameter) is called \emph{the straightening} of $\sigma_a$.
\end{theorem}
\begin{proof}
1) Note that $\partial\Omega_a$ has a cusp at $-2$. We will now create a wedge on the boundary of $\Omega_a$ which will produce the desired pinched anti-quadratic-like restriction of $\sigma_a$. 

Let $\arg{(a-1)}=\theta_0\in(-\pi,\pi)$. We consider a closed curve $\gamma_a$ that is the union of some curve $\gamma'_a\subset\partial\Omega_a$, the line segments $L^\pm:=\{-2+\delta e^{i(2\theta_0\pm\frac{2\pi}{3})}:\delta\in\left[0,\delta_0\right)\}$ (for some $\delta_0>0$), and a pair of curves joining the end-points of $L^\pm$ to the end-points of $\gamma'_a$. Let us denote the bounded complementary component of $\gamma_a$ by $V_a$. We can choose $\gamma_a$ such that $K_a$ is contained in $V_a\cup B(-2,\delta_0)$, and $\gamma_a$ is smooth except at $-2$ (see Figure \ref{schwarz_wedge_pic}).

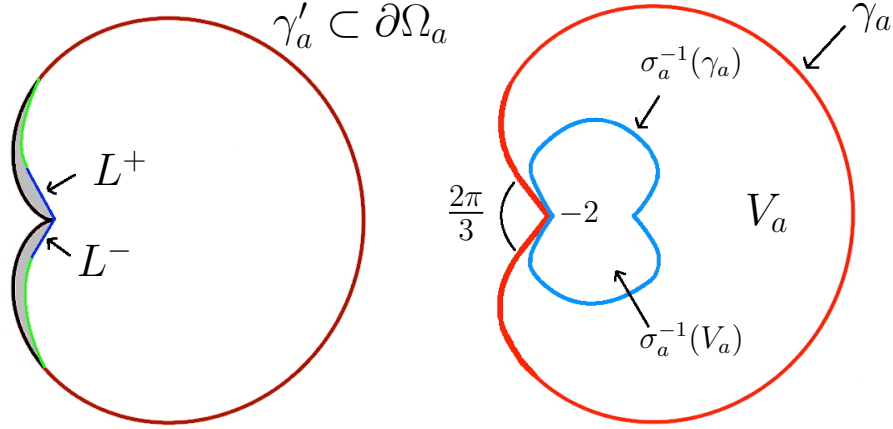
\begin{figure}[ht!]
\begin{tikzpicture}
  \node[anchor=south west,inner sep=0] at (-1,0) {\includegraphics[width=0.52\textwidth]{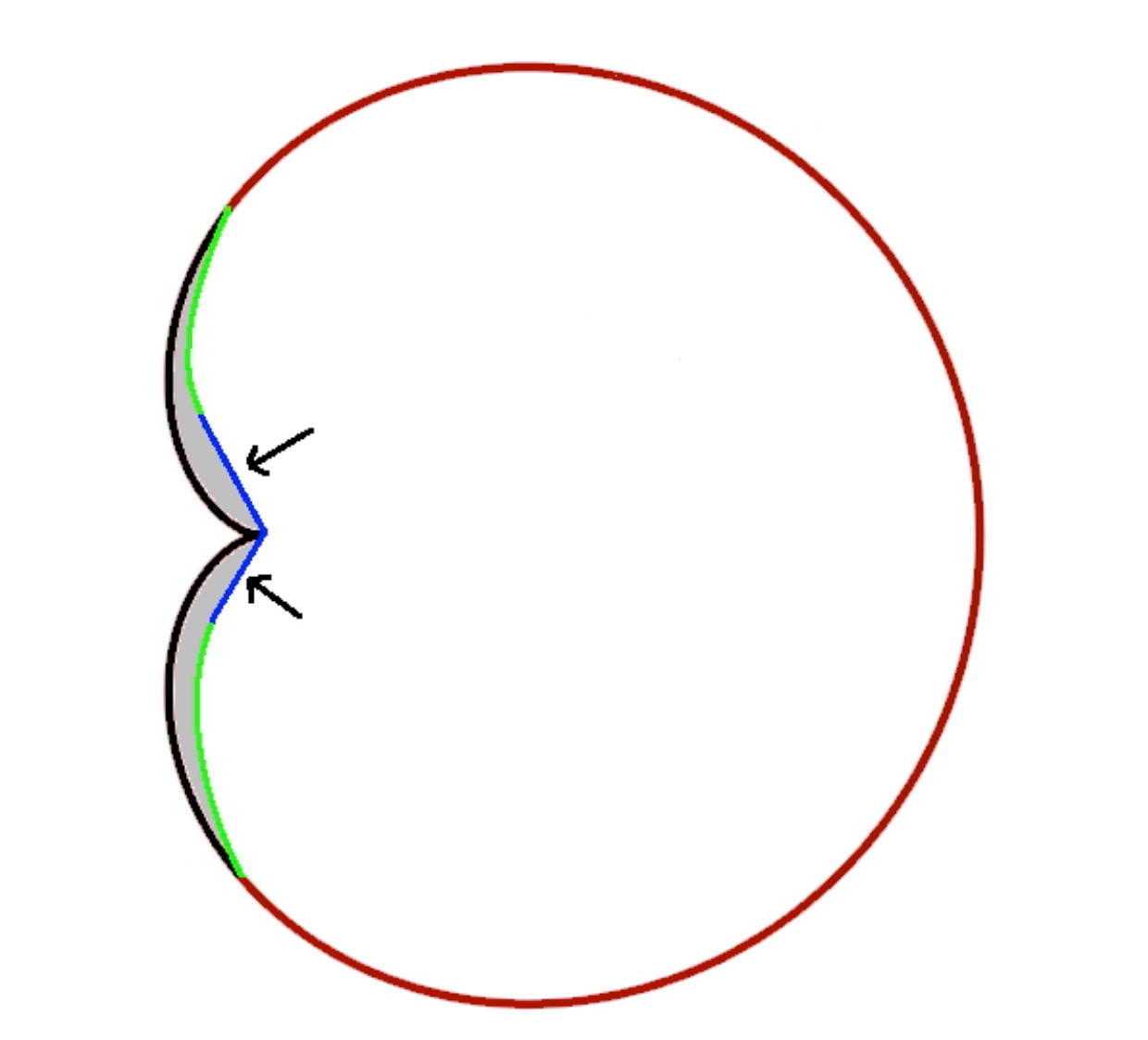}};
  \node[anchor=south west,inner sep=0] at (5.5,-0.1) {\includegraphics[width=0.48\textwidth]{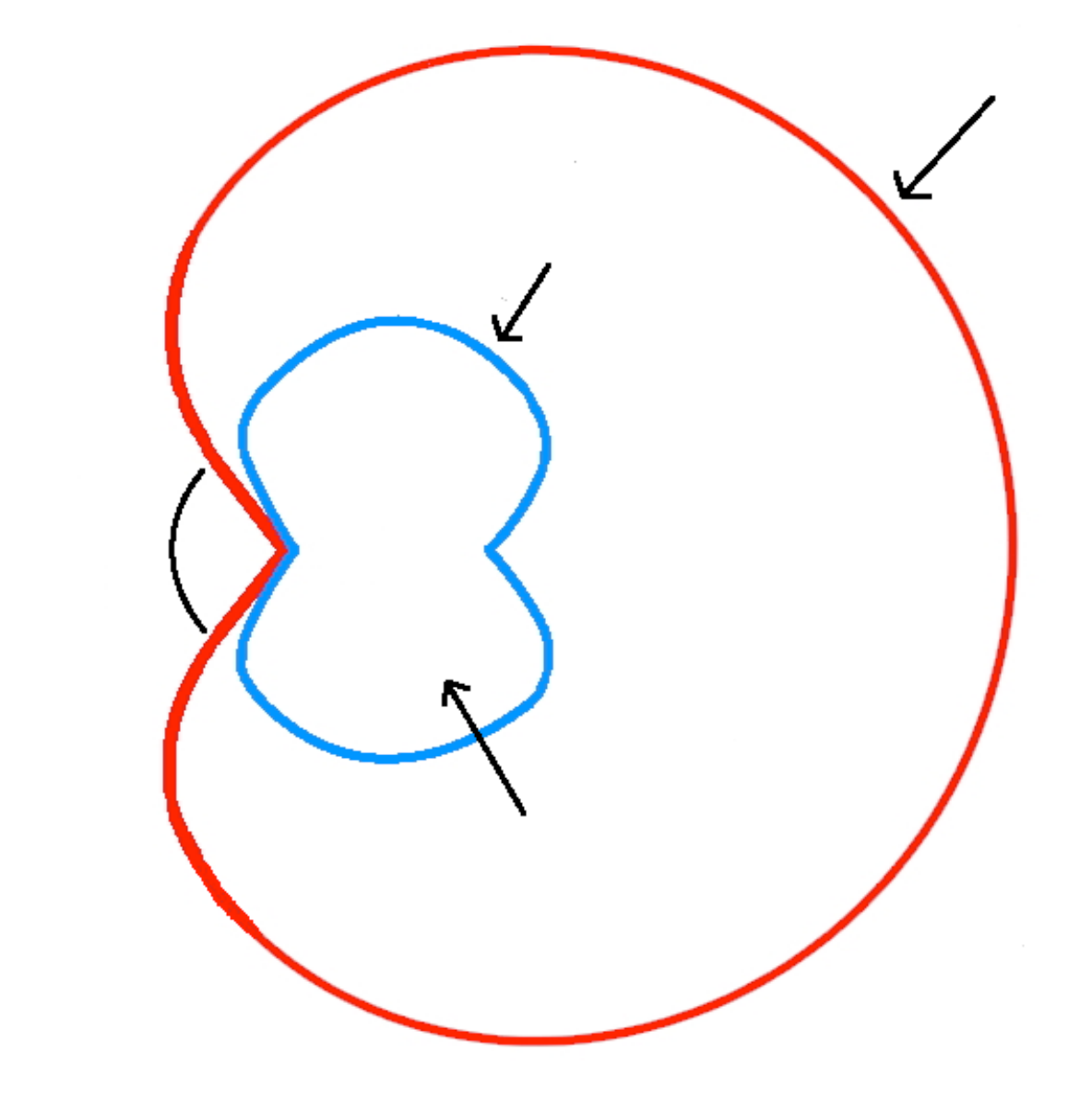}};
  \node at (1.4,3.6) {\begin{huge}$L^+$\end{huge}};
  \node at (1.2,2.4) {\begin{huge}$L^-$\end{huge}};
  \node at (4.6,5.5) {\begin{huge}$\gamma_a'\subset\partial\Omega_a$\end{huge}};
  \node at (6,3) {\begin{huge}$\frac{2\pi}{3}$\end{huge}};
   \node at (7.5,3) {\begin{Large}$-2$\end{Large}};
    \node at (10,3) {\begin{huge}$V_a$\end{huge}};
    \node at (11.4,5.6) {\begin{huge}$\gamma_a$\end{huge}};
   \node at (9,1.32) {\begin{Large}$\sigma_a^{-1}(V_a)$\end{Large}};
   \node at (9,5) {\begin{Large}$\sigma_a^{-1}(\gamma_a)$\end{Large}};
\end{tikzpicture}
\caption{Left: The brown curve is $\gamma_a'$, the blue line segments are $L^\pm$, and the green curves connect the end-points of $\gamma_a'$ to those of $L^\pm$. The curve $\gamma_a$ is the union of the brown, green, and blue curves. It is obtained from $\partial\Omega_a$ by replacing the black curves by the union of the green and blue curves. The grey region indicates $\Omega_a\setminus\overline{V}_a$. These points escape $\Omega_a$ in finitely many steps. Right: The bounded complementary component of the red curve $\gamma_a$ is $V_a$. The domain $V_a$ subtends an angle $\frac{4\pi}{3}$ at the point $-2$. The Schwarz reflection map $\sigma_a$ is a two-to-one covering from the blue curve $\sigma_a^{-1}(\gamma_a)$ to the red curve $\gamma_a$.}
\label{schwarz_wedge_pic}
\end{figure}

We will now argue that $K_a\cap B(-2,\delta_0)$ is contained in $\overline{V_a}$; i.e., $K_a$ does not intersect $B(-2,\delta_0)\setminus\overline{V_a}$. To this end, let us introduce a change of coordinate $\eta_a(w):=\frac{3\sqrt{3}e^{i\theta_0}\vert a-1\vert}{2(\re(a)-4)}\frac{1}{\sqrt{w+2}}$ (where the branch of the square root sends the radial line at angle $2\theta_0$ to the radial line at angle $\theta_0$). Then, we have that $\pmb{\sigma}_a(z)= \overline{z}+\frac12+O(\frac{1}{\overline{z}})$ near $\infty$ (this follows from the asymptotics of $\sigma_a$ near $-2$ obtained in Subsection~\ref{cusp_asymp_sec_1}). The change of coordinate $\eta_a$ maps points of the form $(-2+\delta e^{2i\theta_0})$ to the negative real axis so that $\eta_a(-2+\delta e^{2i\theta_0})\to-\infty$ as $\delta\to 0^+$. 

Moreover, under the change of coordinate $\eta_a$, the line segments $L^\pm$ map to the infinite rays at angles $\pm\frac{2\pi}{3}$. Since $\pmb{\sigma}_a$ is approximately $\overline{z}+\frac12$ for $\vert\im(z)\vert$ large enough, it follows that points between $\eta_a(L^\pm)$ and $\eta_a(\partial\Omega_a)$ with sufficiently large imaginary part eventually escape $\eta_a(\Omega_a)$. Therefore, we can choose $\delta_0>0$ sufficiently small so that points in $B(-2,\delta_0)\setminus\overline{V_a}$ eventually escape $\Omega_a$. It now follows that $K_a$ is contained in $\overline{\sigma_a^{-1}(V_a)}$. Hence, we have that 
$$
K_a=\{z\in\overline{\sigma_a^{-1}(V_a)}: \sigma_a^{\circ n}(z)\in \overline{\sigma_a^{-1}(V_a)}\ \forall\ n\geq 0\}.
$$ 

It is now easy to see that $\pmb{\sigma}_a:\left(\overline{\mathbf{U}_a},\infty\right)\to\left(\overline{\mathbf{V}_a},\infty\right)$ is a pinched anti-quadratic-like map (in the sense of Definition~\ref{pinched_def}) with filled Julia set $\eta_a(K_a)$. The result now follows from Lemma~\ref{straightening_lemma}.

2) We now assume that $K_a$ is connected. Therefore, $K_a$ contains the critical point $c_a$ of $\sigma_a$. It now follows that the basin of attraction of the parabolic fixed point $\infty$ (of $R$) contains exactly one critical point of $R$, and hence the filled Julia set of $R$ is connected. Moreover, since the critical {\'E}calle height of $\mathfrak{q}$ is $0$, and since the hybrid conjugacy is conformal on the non-escaping set of $\sigma_a$, we conclude that the critical {\'E}calle height of $R$ (associated with the parabolic fixed point $\infty$) is also $0$. Therefore, up to M{\"o}bius conjugation, $R\in\cC(\mathfrak{L}_0)$. 

To finish the proof of the theorem, we need to prove uniqueness of $R$. The proof follows the standard argument for uniqueness of straightening of polynomial-like maps with connected Julia sets (see \cite[\S 1.5]{DH2}). 

Let us fix $a\in\cC(\mathcal{S})$, and suppose that there exist $(\alpha_i,A_i)\in\cC(\mathfrak{L}_0)$, and hybrid conjugacies $\Phi_{i}$ between $\mathbf{G}_a$ (where $\mathbf{G}_a$ is the quasiregular extension of the pinched anti-quadratic-like map $\pmb{\sigma}_a:\overline{\mathbf{U}_a}\to\overline{\mathbf{V}_a}$ constructed in the first part of this theorem and Lemma~\ref{straightening_lemma}) and the anti-rational maps $R_{\alpha_i,A_i}$ ($i=1,2$).

Then, $\widehat{\Phi}:=\Phi_{2}\circ\Phi_{1}^{-1}$ is a quasiconformal homeomorphism of the plane that is conformal on $\mathcal{K}_{\alpha_1,A_1}$ (where $\mathcal{K}_{\alpha_1,A_1}$ is the complement of the basin of attraction of the parabolic fixed point at $\infty$ of $R_{\alpha_1,A_1}$), and conjugates $R_{\alpha_1,A_1}:\mathcal{K}_{\alpha_1,A_1}\to\mathcal{K}_{\alpha_1,A_1}$ to $R_{\alpha_2,A_2}:\mathcal{K}_{\alpha_2,A_2}\to\mathcal{K}_{\alpha_2,A_2}$.

Note that by Appendix~\ref{anti_rational_parabolic}, the basin of attraction $\mathcal{B}_{\alpha_i,A_i}$ of the parabolic fixed point at infinity (of $R_{\alpha_i,A_i}$) is simply connected. Moreover by Proposition~\ref{unicritical_parabolic}, there exists a conformal isomorphism $\pmb{\psi}_{\alpha_i,A_i}:\mathcal{B}_{\alpha_i,A_i}\to\D$ that conjugates $R_{\alpha_i,A_i}$ to the (anti-)Blaschke product $B(z)=\frac{3\overline{z}^2+1}{3+\overline{z}^2}$. Hence, the map $\pmb{\psi}_{\alpha_1, A_1}^{\alpha_2, A_2}:=\pmb{\psi}_{\alpha_2, A_2}^{-1}\circ\pmb{\psi}_{\alpha_1, A_1}$ is a conformal isomorphism from $\mathcal{B}_{\alpha_1,A_1}$ onto $\mathcal{B}_{\alpha_2,A_2}$ that conjugates $R_{\alpha_1,A_1}$ to $R_{\alpha_2,A_2}$.

Since $R_{\alpha_1,A_1}:\mathcal{B}_{\alpha_1,A_1}\to\mathcal{B}_{\alpha_1,A_1}$ is conformally conjugate to the (anti-)Blaschke product $B:\D\to\D$ (and $B$ is topologically conjugate to $\overline{z}^2$ on $\mathbb{S}^1$), it follows that $R_{\alpha_1,A_1}$ has three fixed accesses to $\partial\mathcal{B}_{\alpha_1,A_1}$. As both maps $\widehat{\Phi}$ and $\pmb{\psi}_{\alpha_1, A_1}^{\alpha_2, A_2}$ send the parabolic fixed point at $\infty$ of $R_{\alpha_1,A_1}$ to the parabolic fixed point at $\infty$ of $R_{\alpha_2,A_2}$, the arguments of \cite[\S 1.5, Lemma~1]{DH2} imply that $\widehat{\Phi}$ and $\pmb{\psi}_{\alpha_1, A_1}^{\alpha_2, A_2}$ match continuously on $\partial\mathcal{K}_{\alpha_1,A_1}$. By the Bers-Rickman lemma \cite[\S 1.5, Lemma~2]{DH2}, this defines a quasiconformal homeomorphism 
\begin{equation}
H:=\left\{\begin{array}{ll}
                   \widehat{\Phi} & \mbox{on}\ \mathcal{K}_{\alpha_1,A_1}, \\
                   \pmb{\psi}_{\alpha_1, A_1}^{\alpha_2, A_2} & \mbox{on}\ \mathcal{B}_{\alpha_1,A_1},
                                          \end{array}\right.
\end{equation}
that conjugates $R_{\alpha_1,A_1}$ to $R_{\alpha_2,A_2}$. Moreover, $\overline{\partial}H=\overline{\partial}\widehat{\Phi}=0$ a.e. on $\mathcal{K}_{\alpha_1,A_1}$. By Weyl's lemma \cite[\S II.B, Corollary~2]{A3}, $H$ is conformal on $\widehat{\C}$ and hence a M{\"o}bius map. 

The fact that $H$ conjugates $R_{\alpha_1,A_1}$ to $R_{\alpha_2,A_2}$ implies that $H$ fixes the parabolic point $\infty$, the pre-parabolic point $0$, and the critical point $1$. Therefore, $H\equiv\mathrm{id}$, and $(\alpha_1,A_1)=(\alpha_2,A_2)$.
\end{proof}

\begin{corollary}\label{fatou_comp_classification}
\begin{enumerate}
\item Every Fatou component of $\sigma_a$ is eventually periodic.

\item Every periodic Fatou component of $\sigma_a$ is either the immediate basin of attraction of a (super-)attracting/parabolic periodic point or a Siegel disk.
\end{enumerate}
\end{corollary}
\begin{proof}
This follows from Theorem~\ref{straightening_schwarz} and classification of Fatou components for rational maps combined with the fact that maps in $\mathfrak{L}_0$ do not have Herman rings (this follows from the fact that for any $R_{\alpha,A}\in\mathfrak{L}_0$, the basin of attraction $\mathcal{B}_{\alpha,A}$ of the parabolic fixed point at infinity is connected, and hence every connected component of $\Int{\mathcal{K}_{\alpha,A}}$ is simply connected).
\end{proof}

\begin{corollary}\label{non_repelling_prop}
If $\sigma_a$ has a non-repelling cycle, then $a\in\cC(\mathcal{S})$. More precisely, the following statements hold true.
\begin{enumerate}
\item If $\sigma_a$ has an attracting or parabolic cycle, then the forward orbit of the critical value $2$ (of $\sigma_a$) converges to this cycle.

\item If $z_0$ is a parabolic periodic point of $\sigma_a$ of period $n$, and if there is an attracting direction to $z_0$ that is invariant under $\sigma_a^{\circ n}$, then it is the only attracting direction to $z_0$.

\item If $\sigma_a$ has a Siegel disk, then the boundary of the Siegel disk is contained in the closure of the forward orbit of $2$.

\item If $\sigma_a$ has a Cremer cycle, then this cycle lies in the closure of the forward orbit of $2$.
\end{enumerate}
\end{corollary}
\begin{proof}
This follows from Theorem~\ref{straightening_schwarz} and well-known relations between the non-repelling cycles and critical points of a rational map.
\end{proof}

Recall from Subsection~\ref{dyn_unif_tiling_sec} that $\mathcal{E}:\partial\mathcal{Q}\to\mathbb{S}^1$ is a homeomorphism that conjugates the external map $\rho$ (of the Schwarz family $\cC(\mathcal{S})$) to the anti-Blaschke product $B$ (which is the external map of the family $\cC(\mathfrak{L}_0)$), and sends $\omega, -1\in\partial\mathcal{Q}$ to $1,-1\in\mathbb{S}^1$ respectively.

\begin{proposition}\label{hybrid_preserves_symbol}
Let $a\in\cC(\mathcal{S})$, and $w\in\partial K_a$ be a (pre-)periodic point. Then, $\mathcal{E}$ maps the angles of the dynamical rays (of $\sigma_a$) landing at $w$ onto the angles of the dynamical rays (of $R_{\chi(a)}$) landing at $\Phi_a\circ\eta_a(w)\in\mathcal{J}_{\chi(a)}=\partial\mathcal{B}_{\chi(a)}$, where $\Phi_a\circ\eta_a$ is the hybrid equivalence between $\sigma_a$ and $R_{\chi(a)}$.
\end{proposition}

\begin{figure}[ht!]
\begin{tikzpicture}
  \node[anchor=south west,inner sep=0] at (-1,0) {\includegraphics[width=1\textwidth]{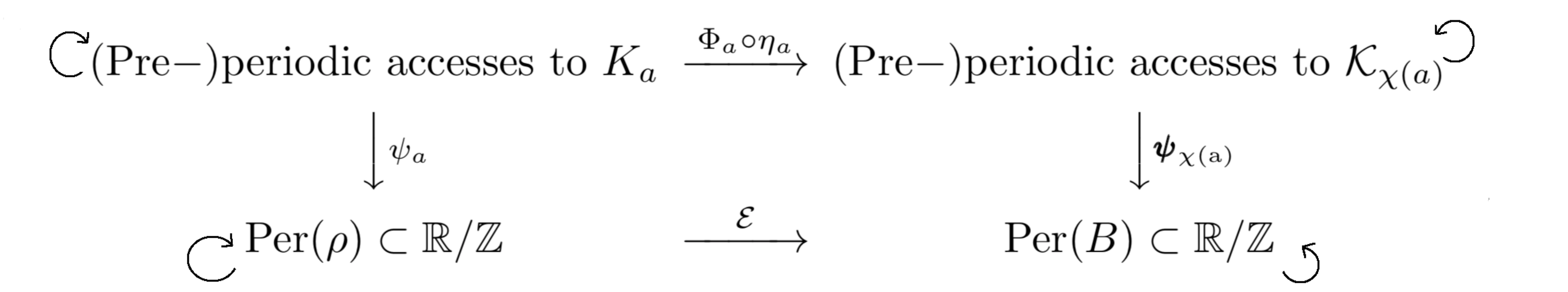}};
  \node at (-0.8,2.2) {$\sigma_a$};
  \node at (0.28,0) {$\rho$};
  \node at (10,0) {$B$};
  \node at (11.4,2.32) {$R_{\chi(a)}$};
\end{tikzpicture}
\caption{The commutative diagram shows the external straightening map $\mathcal{E}$ between the external map $\rho$ of the pinched anti-quadratic-like restriction of $\sigma_a$ and the external map $B$ of $R_{\chi(a)}$.}
\label{action_access}
\end{figure}

\begin{proof}
Recall that the action of $\sigma_a$ on angles of (pre-)periodic accesses to $K_a$ is conjugated to $\rho:\mathrm{Per}(\rho)\subset\R/\Z\to\mathrm{Per}(\rho)$ via the external conjugacy $\psi_a$. Similarly, the action of $R_{\chi(a)}$ on angles of (pre-)periodic accesses to $\mathcal{K}_{\chi(a)}$ is conjugated to $B:\mathrm{Per}(B)\to\mathrm{Per}(B)$ via the Riemann map $\pmb{\psi}_{\chi(a)}$ of $\mathcal{B}_{\chi(a)}$. Finally, $\Phi_a\circ\eta_a$ conjugates the action of $\sigma_a$ on angles of accesses to $K_a$ to the action of $R_{\chi(a)}$ on angles of accesses to $\mathcal{K}_{\chi(a)}$, and sends the accesses to $K_a$ at angles $\frac13,\frac12$ to the accesses to $\mathcal{K}_{\chi(a)}$ at angles $0,\frac12$ respectively.

Thus, we get a conjugacy between $\rho:\mathrm{Per}(\rho)\subset\R/\Z\to\mathrm{Per}(\rho)$ and $B:\mathrm{Per}(B)\to\mathrm{Per}(B)$ that respects the corresponding Markov partitions. It follows that if $\theta$ is the angle of a (pre-)periodic access to $w\in\partial K_a$ and $\theta'$ is the angle of the image access to $\Phi_a\circ\eta_a(w)\in\partial \mathcal{K}_{\chi(a)}$, then the $\rho$-orbit of $\theta$ and the $B$-orbit of $\theta'$ have the same symbolic representation with respect to the Markov partitions described in Subsection~\ref{dyn_unif_tiling_sec}. But this implies that $\theta'=\mathcal{E}(\theta)$. Since $\Phi_a\circ\eta_a(w)$ is a (pre-)periodic point on $\mathcal{J}_{\chi(a)}$, it follows that the dynamical ray of $R_{\chi(a)}$ at angle $\mathcal{E}(\theta)$ lands at $\Phi_a\circ\eta_a(w)$ on $\mathcal{J}_{\chi(a)}$. 

Finally, since $\sigma_a$ and $R_{\chi(a)}$ are topologically conjugate around their non-escaping set and filled Julia set (respectively), it follows that the number of accesses to $w\in\partial K_a$ is equal to the number of accesses to $\Phi_a\circ\eta_a(w)\in\mathcal{J}_{\chi(a)}$. This shows that the angles of the dynamical rays landing at $w\in\partial K_a$ are mapped onto the angles of the dynamical rays landing at $\Phi_a\circ\eta_a(w)\in\mathcal{J}_{\chi(a)}$ by $\mathcal{E}$.
\end{proof}

\begin{definition}[Straightening Map]\label{chi_def}
The \emph{straightening map} 
$$
\chi:\cC(\mathcal{S})\to\cC(\mathfrak{L}_0)
$$ 
is defined as $\chi(a):=(\alpha,A)$, where $\sigma_a:\overline{\sigma_a^{-1}(V_a)}\to\overline{V_a}$ is hybrid conjugate to the quadratic anti-rational map $R_{\alpha,A}\in\cC(\mathfrak{L}_0)$.
\end{definition}

\begin{proposition}[Injectivity of Straightening]\label{chi_injective_prop}
The map $\chi:\cC(\mathcal{S})\to\cC(\mathfrak{L}_0)$ is injective.
\end{proposition}
\begin{proof}
Let us assume that $\chi(a_1)=\chi(a_2)=\left(\alpha,A\right)$.

We choose (the homeomorphic extension of) a conformal isomorphism $\kappa$ between $\widehat{\C}\setminus V_{a_1}$ and $\widehat{\C}\setminus V_{a_2}$ with $\kappa(-2)=-2$. Since both of these domains make an angle $\frac{2\pi}{3}$ at $-2$, it follows that $\kappa$ is asymptotically linear near $-2$.

Again, $\Phi:=\eta_{a_2}^{-1}\circ \Phi_{a_2}^{-1}\circ\Phi_{a_1}\circ\eta_{a_1}:\overline{\sigma_{a_1}^{-1}(V_{a_1})}\to\overline{\sigma_{a_2}^{-1}(V_{a_2})}$ is a hybrid conjugacy between $\sigma_{a_1}$ and $\sigma_{a_2}$. Since the conformal map $\Phi_{a_2}^{-1}\circ\ \Phi_{a_1}:\widehat{\C}\setminus\overline{\eta_{a_1}(V_{a_1})}\to\widehat{\C}\setminus\overline{\eta_{a_2}(V_{a_2})}$ is asymptotically linear near $\infty$, the same is true for $\Phi_{a_2}^{-1}\circ\Phi_{a_1}:\eta_{a_1}(\sigma_{a_1}^{-1}(\gamma_{a_1}))\to\eta_{a_2}(\sigma_{a_2}^{-1}(\gamma_{a_2}))$. It follows that $\Phi:\sigma_{a_1}^{-1}(\gamma_{a_1})\to\sigma_{a_2}^{-1}(\gamma_{a_2})$ is also asymptotically linear near $-2$.

Following the arguments of Lemma~\ref{straightening_lemma}, we can now interpolate between $\Phi$ and $\kappa$ to obtain a quasiconformal homeomorphism of the sphere as an extension of $\Phi$. 

Since $a_1,a_2\in\cC(\mathcal{S})$, their dynamics on the tiling sets are conjugate to the reflection map $\rho$ via the conformal maps $\psi_{a_1}$ and $\psi_{a_2}$ respectively. Then, $\psi_{a_1}^{a_2}:=\psi_{a_2}^{-1}\circ\psi_{a_1}$ is a conformal conjugacy between $\sigma_{a_1}\vert_{T_{a_1}^\infty}$ and $\sigma_{a_2}\vert_{T_{a_2}^{\infty}}$, which matches continuously with $\Phi$ on $\partial K_{a_1}$.

Finally, we define a map on the sphere as follows

\begin{equation}
H:=\left\{\begin{array}{ll}
                   \Phi & \mbox{on}\ K_{a_1}, \\
                   \psi_{a_1}^{a_2} & \mbox{on}\ T_{a_1}^\infty.
                                          \end{array}\right.
\end{equation}

By the Bers-Rickman lemma, $H$ is a quasiconformal homeomorphism of the sphere that conjugates $\sigma_{a_1}$ to $\sigma_{a_1}$. Moreover, $\overline{\partial}H=\overline{\partial}\Phi=0$ a.e. on $K_{a_1}$. By Weyl's lemma, $H$ is conformal on $\widehat{\C}$ and hence a M{\"o}bius map. Since such an $H$ must fix $-2$, $2$, and $\infty$, we have that $H\equiv\mathrm{id}$. Hence, $a_1=a_2$.
\end{proof}

\section{Hyperbolic Components in $\mathcal{S}$, and Their Boundaries}\label{hyp_comp_sec}

We say that a parameter $a\in\mathcal{S}$ is \emph{hyperbolic} if $\sigma_a$ has an attracting cycle. By Corollary~\ref{non_repelling_prop}, a hyperbolic parameter of $\mathcal{S}$ belongs to $\cC(\mathcal{S})$.

We now discuss the structure of the closures of hyperbolic components in $\cC(\mathcal{S})$. Since the results (and the proof techniques) of this section are similar to those for the family of Schwarz reflection maps with respect to a circle and a cardioid which was considered in \cite{LLMM1,LLMM2}, we only sketch the proofs. 

\subsection{Uniformization of Hyperbolic Components}\label{hyp_comp_subsec}
Since $\sigma_a$ depends real-analy-tically on $a$, a straightforward application of the implicit function theorem shows that attracting periodic points can be locally continued as real-analytic functions of $a$. Hence, the set of all hyperbolic parameters is an open set. A connected component of the set of all hyperbolic parameters is called a \emph{hyperbolic component}. It is easy to see that every hyperbolic component $H$ has an associated positive integer $k$ such that each parameter in $H$ has an attracting cycle of period $k$. We refer to such a component as a hyperbolic component of period $k$.

A \emph{center} of a hyperbolic component is a parameter $a$ for which $\sigma_a$ has a super-attracting periodic cycle; i.e., the critical value $2$ is periodic. 

If $\sigma_a$ has an attracting cycle, then this attracting cycle must lie in $K_a$. By Corollary~\ref{non_repelling_prop}, the attracting cycle of $\sigma_a$ attracts the free critical point, and $a\in\cC(\mathcal{S})$. 

Moreover, we can associate a dynamically defined conformal invariant to every hyperbolic map $\sigma_a$; namely multiplier if the attracting cycle (of $\sigma_a$) has even period, and Koenigs ratio if the attracting cycle (of $\sigma_a$) has odd period (see \cite[\S 2.1.1]{LLMM2} for the corresponding definitions for anti-polynomials, since the definitions are local, they apply to any anti-holomorphic map). 

Let us now fix a hyperbolic component $H$ of odd (respectively, even) period $k$ in $\cC(\mathcal{S})$. For $a\in H$, the restriction of $\sigma_{a}^{\circ k}$ to the connected component $U_a$ of $\textrm{int}(K_a)$ containing $c_a$ is a degree $2$ proper anti-holomorphic (respectively, holomorphic) map. Moreover, $\sigma_{a}^{\circ k}$ has exactly three fixed points (respectively, has a unique fixed point) on $\partial U_a$. Exactly one of them is a cut point of $\partial K_a$, this point is called the \emph{dynamical root point} of $\sigma_a$ on $\partial U_a$ (when $k=1$, all these fixed points are non-cut points of $\partial K_a$; in this case, we call the cusp point $-2$ the \emph{root point}). Choosing a Riemann map of $U_a$ that maps the attracting periodic point to $0$ and the dynamical root point to $1$, we obtain a conjugacy between $\sigma_a^{\circ k}\vert_{U_a}$ and an anti-holomorphic (respectively, holomorphic) Blaschke product of degree $2$ on $\D$. By construction, such a Blaschke product must be of the form 
$$
B^{-}_{a,\lambda}(z)=\lambda\overline{z}\frac{(\overline{z}-a)}{(1-\overline{az})},\quad \textrm{or}\quad B_{a,\lambda}^{+}(z)=\lambda z\frac{(z-a)}{(1-\overline{a}z)},
$$
with $a\in\mathbb{D}$ and $\lambda\in\mathbb{S}^1$ such that $z=1$ is fixed by $B_{a,\lambda}^{\pm}$. The unique such Blaschke product with a super-attracting fixed point is $B_{0,1}^{\pm}$.

Let $\mathcal{B}^{\pm}$ be the space of all (anti-)holomorphic Blaschke products $B_{a,\lambda}^{\pm}$ with $a\in\mathbb{D}$ and $\lambda\in\mathbb{S}^1$ such that $z=1$ is fixed by $B_{a,\lambda}^{\pm}$. The following proposition describes the topology and dynamical uniformizations of hyperbolic components in $\cC(\mathcal{S})$.

\begin{figure}[ht!]
\centering
\includegraphics[scale=0.6]{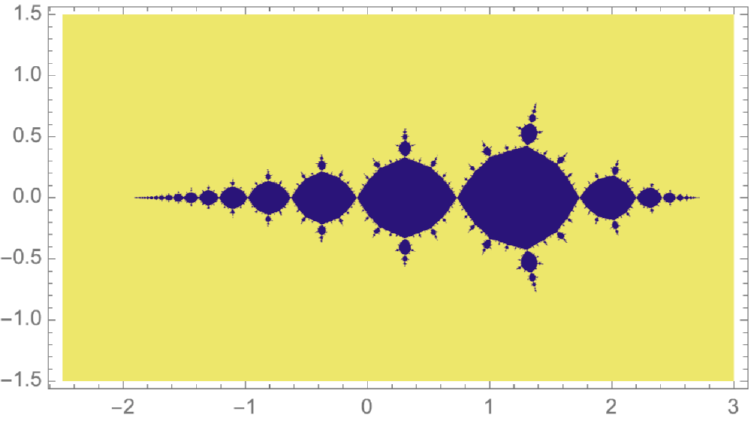}\ \includegraphics[scale=0.4]{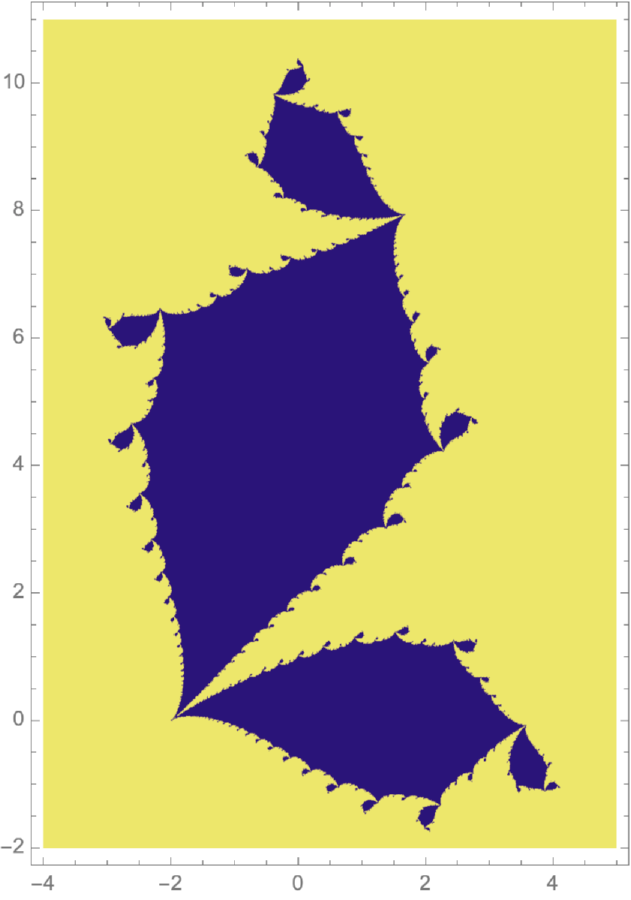}
\caption{Left: The non-escaping set of the center $\frac{5+\sqrt{33}}{4}$ of the unique period two hyperbolic component intersecting the real line. Right: The non-escaping set of the center $\frac72+ \frac{i\sqrt{3}}{2}$ of the unique period two hyperbolic component contained in the upper half-plane.}
\label{fig:dyn_period_two}
\end{figure}

\begin{proposition}[Dynamical Uniformization of Hyperbolic Components]\label{unif_hyp_schwarz}
Let $H$ be a hyperbolic component in $\cC(\mathcal{S})$.
\begin{enumerate}
\item If $H$ is of odd period, then there exists a homeomorphism $\widetilde{\eta}_H:H\to\mathcal{B}^-$ that respects the Koenigs ratio of the attracting cycle. In particular, the Koenigs ratio map is a real-analytic $3$-fold branched covering from $H$ onto the open unit disk, ramified only over the origin.

\item If $H$ is of even period, then there exists a homeomorphism $\widetilde{\eta}_H:H\to\mathcal{B}^+$ that respects the multiplier of the attracting cycle. In particular, the multiplier map is a real-analytic diffeomorphism from $H$ onto the open unit disk.
\end{enumerate}
In both cases, $H$ is simply connected and has a unique center.
\end{proposition}
\begin{proof}
See \cite[Theorem~5.6, Theorem~5.9]{NS} for a proof of the corresponding facts for unicritical anti-polynomials. It is straightforward to adapt the proof in our case. The main idea is to change the conformal dynamics of the first return map of a periodic Fatou component. More precisely, one can glue any Blaschke product belonging to the family $\mathcal{B}^{\pm}$ in the connected component of $\textrm{int}(K_a)$ containing $c_a$ by quasiconformal surgery. This gives the required homeomorphism between $H$ and $\mathcal{B}^{\pm}$.

However, there is an important detail here. Since the original dynamics $\sigma_a$ is modified only in a part of the connected component of $\textrm{int}(K_a)$ containing $c_a$ (this is precisely where an iterate of $\sigma_a$ is replaced by a Blaschke product), the resulting quasiregular modification $\mathbf{G}_a$ shares some of the mapping properties of $\sigma_a$. In particular, $\mathbf{G}_a$ sends $\mathbf{G}_a^{-1}(\Omega_a)$ onto $\Omega_a$ as a two-to-one branched covering, and $\mathbf{G}_a^{-1}(T_a^0)$ onto $T_a^0$ as a three-to-one branched covering. Hence, we can adapt the proof of Proposition~\ref{schwarz_qcdef} to show that $\mathbf{G}_a$ is quasiconformally conjugate to some map $\sigma_b$ in our family $\mathcal{S}$.  
\end{proof}

\begin{remark}\label{center_compute}
$a=3$ is the only parameter for which the critical point $c_a$ is equal to $2$. Hence, $3$ is the center of the unique hyperbolic component of period one of $\cC(\mathcal{S})$. On the other hand, the centers of the hyperbolic components of period two are $\frac{5+\sqrt{33}}{4}$, and $\frac72\pm \frac{i\sqrt{3}}{2}$ (see Figure~\ref{fig:dyn_period_two} for the non-escaping sets of the centers of two period two components). 
\end{remark}

\subsection{Boundaries of Hyperbolic Components}\label{hyp_comp_para_subsec}

We will start this subsection with a brief description of neutral parameters and boundaries of hyperbolic components of \emph{even} period of $\cC(\mathcal{S})$.

A parameter $a\in\cC(\mathcal{S})$ is called a \emph{parabolic} parameter if $\sigma_a$ has a periodic cycle with multiplier a root of unity. The following proposition states that every neutral (in particular, parabolic) parameter lies on the boundary of a hyperbolic component of the same period.

\begin{proposition}[Neutral Parameters on Boundary]\label{ThmIndiffBdyHyp_schwarz} 
If $\sigma_{a_0}$ has a neutral periodic point of period $k$, then every neighborhood of $a_0$ in $S$ contains parameters with attracting periodic points of period $k$, so the parameter $a_0$ is on the boundary of a hyperbolic component of period $k$ of $\cC(\mathcal{S})$. 
\end{proposition}
\begin{proof}
See \cite[Theorem~2.1]{MNS} for a proof in the Tricorn family. Since the proof given there only uses local dynamical properties of anti-holomorphic maps near neutral periodic points, it applies to the family $\mathcal{S}$ as well.
\end{proof}

The next result describes the bifurcation structure of even period hyperbolic components of $\cC(\mathcal{S})$. Once again, its proof in the Tricorn family is given in \cite[Theorem~1.1]{MNS}, which can be easily adapted for our setting.

\begin{proposition}[Bifurcations From Even Period Hyperbolic Components]\label{ThmEvenBif_schwarz} 
If $\sigma_a$ has a $2k$-periodic cycle with multiplier 
$e^{2\pi ip/q}$ with $\mathrm{gcd}(p,q)=1$, then the parameter $a$ sits on the boundary of a hyperbolic component of 
period $2kq$ (and is the root thereof) of $\cC(\mathcal{S})$. 
\end{proposition}

By Theorem~\ref{unif_hyp_schwarz}, each hyperbolic component of period $k$ in $\cC(\mathcal{S})$ contains a unique parameter with a superattracting cycle of period $k$ (i.e., a unique \emph{center}). According to Remark~\ref{center_compute}, there is a unique hyperbolic component of period one in $\cC(\mathcal{S})$ and the center of this component is $3$. We will denote this hyperbolic component by $\widetilde{H}$.

\begin{proposition}[Neutral Dynamics of Odd Period]\label{hyp_odd_parabolic}  
1) The boundary of a hyperbolic component of odd period $k>1$ of $\cC(\mathcal{S})$ is contained in $S$, and consists entirely of parameters having a parabolic orbit of exact period $k$. In suitable local conformal coordinates, the $2k$-th iterate of such a map has the form 
$z\mapsto z+z^{q+1}+\ldots$ with $q\in\{1,2\}$. 

2) Every parameter on the boundary of the hyperbolic component $\widetilde{H}$ of period one is either contained in $\mathcal{I}$ or has a parabolic fixed point (with local power series as above).
\end{proposition}
\begin{proof} 
1) Let $H$ be a hyperbolic component of period $k>1$. By Proposition~\ref{limit_point_conn_locus}, if $H$ has a boundary point $a'$ outside $S$, then $a'$ must be contained in $\mathcal{I}$. But then $a'\in\widehat{S}$, and the corresponding Schwarz reflection map $\sigma_{a'}$ is well-defined. Hence, the arguments of \cite[Lemma 2.5]{MNS} show that $\sigma_{a'}$ must have a parabolic cycle of period $k>1$ that attracts the forward orbit of the critical point $c_{a'}$. However, this is impossible as the forward orbit of the critical point $c_{a'}$ (under $\sigma_{a'}$) converges to the fixed point $-2$, for all parameters $a'\in\mathcal{I}$ (see Proposition~\ref{special_neutral_critical}). This proves that $H$ has no boundary point outside $S$. Finally, \cite[Lemma 2.5]{MNS} combined with the fact that the Schwarz reflection maps under consideration have unique critical points also show that for every parameter on the boundary of $H$, the $k$-cycle to which the critical orbit converges must be parabolic with the desired local Taylor series expansion.

2) The proof is similar to that of the previous part. The only difference is that the boundary of $\widetilde{H}$ may contain points on the interval $\mathcal{I}$.
\end{proof}

This leads to the following classification of odd periodic parabolic points.

\begin{definition}[Parabolic Cusps]\label{DefCusp_schwarz}
A parameter $a$ will be called a {\em parabolic cusp} if it has a parabolic 
periodic point of odd period such that $q=2$ in the previous proposition. Otherwise, it is called a \emph{simple} parabolic parameter.
\end{definition}

Let us now fix a hyperbolic component $H$ of odd period $k$, and let $a\in H$. Note that the first return map $\sigma_a^{\circ k}$ of a $k$-periodic Fatou component of $\sigma_a$ has precisely three fixed points (necessarily repelling when $k>1$) on the boundary of the component. As $a$ tends to a simple parabolic parameter on the boundary $\partial H$, the unique attracting periodic point of this Fatou component tends to merge with one of the repelling periodic points on its boundary. Similarly, as $a$ tends to a parabolic cusp on the boundary $\partial H$, the unique attracting periodic point of this Fatou component and two boundary repelling periodic points merge together.

Now let $a\in\mathcal{I}^0$ or $a$ be a simple parabolic parameter of odd (parabolic) period $k$. Consider an attracting petal of $\sigma_a$ that contains the critical value $2$ (for $a\in\mathcal{I}^0$, the existence of such petals follows from the proof of Proposition~\ref{special_neutral_critical}). The arguments of \cite[Lemma~3.1]{MNS} (also see \cite[\S 7.1]{LLMM1}) imply that there exists a \emph{Fatou coordinate} defined on the attracting petal which conjugates the first anti-holomorphic return map $\sigma_a^{\circ k}$ (of the attracting petal) to the map $\zeta\mapsto\overline{\zeta}+1/2$ on a right half-plane. This coordinate is unique up to addition of a real constant. The pre-image of the real line (which is invariant under $\zeta\mapsto\overline{\zeta}+1/2$) under this Fatou coordinate is called the \emph{attracting equator}. By construction, the attracting equator is invariant under the dynamics $\sigma_a^{\circ k}$.

The imaginary part of the critical value $2$ (whose forward orbit converges to the parabolic cycle or to the fixed point $-2$ when $a\in\mathcal{I}^0$) under this special Fatou coordinate is called the \emph{critical {\'E}calle height} of $\sigma_a$ (since this Fatou coordinate is unique up to addition of a real constant, the critical {\'E}calle height is well-defined). It is easy to see that critical {\'E}calle height is a conformal conjugacy invariant of simple parabolic parameters of odd period (respectively, of parameters on $\mathcal{I}^0$). One can change the critical {\'E}calle height of simple parabolic parameters by a quasiconformal deformation argument to obtain real-analytic arcs of parabolic parameters on the boundaries of odd period hyperbolic components. An analogous deformation can be performed for parameters on $\mathcal{I}^0$ as well.

\begin{proposition}[Parabolic Arcs]\label{parabolic_arcs_schwarz}
1) Let $\widetilde{a}$ be a simple parabolic parameter of odd period. Then $\widetilde{a}$ is on a parabolic arc in the following sense: there exists a real-analytic arc of simple parabolic parameters $a(h)$ (for $h\in\mathbb{R}$) with quasiconformally equivalent but conformally distinct dynamics of which $\widetilde{a}$ is an interior point, and the {\'E}calle height of the critical value $2$ of $\sigma_{a(h)}$ is $h$. This arc is called a \emph{parabolic arc}.

2) All maps $\sigma_a$ with $a\in\mathcal{I}^0\subset\widehat{S}$ are quasiconformally conjugate.
\end{proposition}
\begin{proof} 
1) See \cite[Theorem 3.2]{MNS} for a proof in the case of unicritical anti-polynomials (also compare \cite[Figure~2.6]{HS} for an illustration of the quasiconformal deformation used to change the critical {\'E}calle height). One uses the same quasiconformal deformation in the attracting petal at $-2$ for the map $\sigma_{\widetilde{a}}$, and Proposition~\ref{schwarz_qcdef} guarantees that the quasiconformal deformations of $\sigma_{\widetilde{a}}$ also lie in the family $\mathcal{S}$.

2) Let us start with the parameter $\widetilde{a}:=4\in\mathcal{I}^0$. Since $\sigma_{\widetilde{a}}$ commutes with complex conjugation, the {\'E}calle height of the critical value $2$ is $0$ for this map; i.e., the critical {\'E}calle height of $\sigma_{\widetilde{a}}$ is $0$. As in the previous case, one can change the critical {\'E}calle height by a quasiconformal deformation argument as in \cite[Theorem~3.2]{MNS} (and invoking Proposition~\ref{schwarz_qcdef}), and prove the existence of a real-analytic arc $\widehat{I}\subset\widehat{S}$ such that as $a$ runs over $\widehat{I}$, the critical {\'E}calle height of $\sigma_a$ varies from $-\infty$ to $+\infty$. Moreover, since all maps on $\widehat{I}$ are quasiconformally conjugate, it follows that for every parameter $a\in\widehat{I}$, the map $\sigma_a$ has a unique attracting direction at $-2$. Hence, by Subsection~\ref{dyn_near_cusp}, the arc $\widehat{I}$ is contained in $\mathcal{I}^0$. Therefore, the closure of $\widehat{I}$ is contained in $\overline{\mathcal{I}^0}=\mathcal{I}\subset\widehat{S}$. We claim that the limit points of this arc, as the critical {\'E}calle height goes to $\pm\infty$, must lie outside $\mathcal{I}^0$. Indeed, for each $a\in\mathcal{I}^0$, there is a unique attracting direction at $-2$ (under $\sigma_a$), and hence all such maps have finite critical {\'E}calle heights. Hence, no $a\in\mathcal{I}^0$ can be an accumulation point of a sequence of parameters on $\widehat{I}$ with critical {\'E}calle height diverging to $\pm\infty$. Thus, the limit points of $\widehat{I}$, as the critical {\'E}calle height goes to $\pm\infty$, must lie in $\mathcal{I}\setminus\mathcal{I}^0=\{4\pm i\sqrt{3}\}$; i.e., $\widehat{I}=\mathcal{I}^0$.
\end{proof}

Let us fix a parabolic arc $\cC\subset\cC(\mathcal{S})$ of period $k>1$, and its critical {\'E}calle height parametrization $a:\R\to\cC$. By Propositions~\ref{ThmIndiffBdyHyp_schwarz}, the arc $\cC$ must lie on the boundary of a single hyperbolic component of period $k$. By Proposition~\ref{hyp_odd_parabolic}, every accumulation point of $\cC$ is contained in $S$. Arguing as in \cite[Lemma 3.3]{MNS}, we see that such an accumulation point is a parabolic cusp of period $k$. In particular, $\overline{\cC}$ is a compact connected set in $S$. Moreover, $\chi$ maps parabolic cusps in $\cC(\mathcal{S})$ to parabolic cusps of the same period in $\cC(\mathfrak{L}_0)$. Since there are only finitely many cusps of a given period in $\cC(\mathfrak{L}_0)$ and $\chi$ is injective, it follows that there are only finitely many cusps of a given period in $\cC(\mathcal{S})$. Hence, $\cC$ limits at parabolic cusp points on both ends. This allows one to adapt the arguments of \cite[Proposition 3.7]{HS} for our setting to prove the following result.

\begin{proposition}[Fixed Point Index on Parabolic Arc]\label{index_to_infinity_schwarz}
Along any parabolic arc of odd period greater than one, the holomorphic fixed point index of the parabolic cycle is a real valued real-analytic function that tends to $+\infty$ at both ends.
\end{proposition}

It now follows by arguments similar to the ones used in \cite[Theorem~3.8, Corollary~3.9]{HS} that:

\begin{proposition}[Bifurcations Along Arcs]\label{ThmBifArc_schwarz}
Every parabolic arc of odd period $k>1$ intersects the boundary of a hyperbolic component of period $2k$ along an arc consisting of the set of parameters where the parabolic fixed point index is at least $1$. In particular, every parabolic arc has, at both ends, an interval of positive length at which bifurcation from a hyperbolic component of odd period $k$ to a hyperbolic component of period $2k$ occurs.
\end{proposition}

\begin{proposition}\label{index_increasing}
Let $H$ be a hyperbolic component of odd period $k$ in $\cC(\mathcal{S})$, $\mathcal{C}$ be a parabolic arc on $\partial H$, $a:\mathbb{R}\to\mathcal{C}$ be the critical {\'E}calle height parametrization of $\mathcal{C}$, and let $H'$ be a hyperbolic component of period $2k$ bifurcating from $H$ across $\mathcal{C}$. Then there exists some $h_0>0$ such that 
$$
\cC\cap\partial H'=a[h_0,+\infty).
$$
Moreover, the function 
\begin{center}
$\begin{array}{rccc}
  \ind_{\mathcal{C}}: & [h_0,+\infty) & \to & [1,+\infty) \\
  & h & \mapsto & \ind_{\cC}(\sigma_{a(h)}^{\circ 2})\\
   \end{array}$
\end{center}
is strictly increasing, and hence a bijection (where $\ind_{\cC}(\sigma_{a(h)}^{\circ 2})$ stands for the holomorphic fixed point index of the $k$-periodic parabolic cycle of $\sigma_{a(h)}^{\circ 2}$).
 \end{proposition}
 \begin{proof}
The proof of \cite[Lemma~2.13, Corollary~2.21]{IM2} can be applied mutatis mutandis to our setting.
 \end{proof}

Recall that there are exactly three distinct combinatorial ways in which parabolic arcs (respectively parabolic cusps) are formed on the boundary $\partial H$. One can now argue as in \cite[\S~7]{LLMM2} to prove the following structure theorem for boundaries of odd period hyperbolic components of $\cC(\mathcal{S})$ (see Figure~\ref{fig:period_three_hyp}).

\begin{figure}[ht!]
\centering
\includegraphics[scale=0.4]{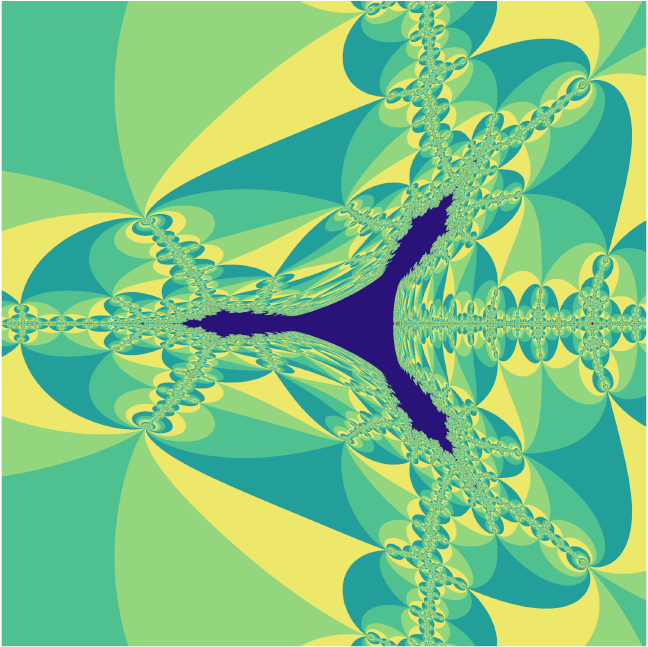}
\caption{A period three hyperbolic component in $\cC(\mathcal{S})$.}
\label{fig:period_three_hyp}
\end{figure}

\begin{proposition}[Boundaries Of Odd Period Hyperbolic Components]\label{Exactly 3}
The boundary of every hyperbolic component of odd period $k>1$ of $\cC(\mathcal{S})$ is a topological triangle having parabolic cusps as vertices and parabolic arcs as sides.
\end{proposition}

The situation for the hyperbolic component $\widetilde{H}$ of period one is slightly different.

\begin{proposition}\label{hyp_one_not_proper}
The boundary of $\widetilde{H}$ consists of two parabolic arcs, a parabolic cusp, and the arc $\mathcal{I}$. In particular, $\partial\widetilde{H}$ is not contained in $S$.
\end{proposition}
\begin{proof}
Note that for any $a\in\widetilde{H}$, the map $\sigma_a$ has precisely three fixed points on the boundary of the unique Fatou component. Two of these are repelling, and the other one is $-2$. As $a$ tends to the boundary of $\partial\widetilde{H}$, the unique attracting fixed point of $\sigma_a$ tends to merge with these boundary fixed points. The merger of the attracting fixed point with one of the (two) repelling fixed points happens on a parabolic arc on $\partial\widetilde{H}$. This accounts for the two parabolic arcs on $\partial\widetilde{H}$, and they meet at a parabolic cusp such that in the corresponding dynamical plane, the two repelling fixed points and the attracting fixed point meet to produce a double parabolic fixed point.

On the other hand, the merger of the attracting fixed point with the fixed point $-2$ happens precisely along the arc $\mathcal{I}$ which is the set of parameters for which $-2$ attracts the forward orbit of the free critical point. Moreover, the two parabolic arcs on $\partial\widetilde{H}$ (described above) meet $\mathcal{I}$ at parameters for which $-2$ has two attracting directions; hence these parameters are $4\pm i\sqrt{3}$.  
\end{proof}

\section{Tessellation of The Escape Locus}\label{parameter_tessellation}

The goal of this section is to prove a uniformization theorem for the escape locus of the family $\mathcal{S}$. The uniformizing map will be defined in terms of the conformal position of the critical value $2$ (of $\sigma_a$) under $\psi_a$ (see Proposition~\ref{schwarz_group}).

\begin{theorem}[Uniformization of The Escape Locus]\label{escape_unif_thm}
The map 
$$
\pmb{\Psi}:S\setminus\cC(\mathcal{S})\to\D_2,
$$ 
$$
a\mapsto\psi_a(2)
$$ is a homeomorphism.
\end{theorem}
\begin{proof}
The proof is analogous to that of \cite[Theorem~1.3]{LLMM2}. We only indicate the key differences.

Note that for all $a\in S$, the critical value $2$ of $\sigma_a$ lies in $\Omega_a$; i.e., $2\notin T_a^0$. It now follows from the definition of $\psi_a$ that $\psi_a(2)\in\D_2$ for each $a\in S\setminus\cC(\mathcal{S})$.

The map $\pmb{\Psi}$ is easily seen to be continuous. We will show that $\pmb{\Psi}$ is proper, and locally invertible. This will imply that $\pmb{\Psi}$ is a covering map from $S\setminus\cC(\mathcal{S})$ onto the simply connected domain $\D_2$, and hence a homeomorphism from each connected component of $S\setminus\cC(\mathcal{S})$ onto $\D_2$. However, $a_0=2$ is the only parameter in $S\setminus\cC(\mathcal{S})$ satisfying $\pmb{\Psi}(a_0)=\rho_2(0)$. So, $S\setminus\cC(\mathcal{S})$ must be connected; i.e., $\pmb{\Psi}$ is a homeomorphism. 

Local invertibility follows from a quasiconformal deformation/surgery argument as in \cite[Theorem~1.3]{LLMM2}.

We need to consider several cases to show that $\pmb{\Psi}$ is proper. Let us first assume that $\{a_k\}_k$ is a sequence in $S\setminus\cC(\mathcal{S})$ such that $\re(a_k)\to\frac32$. It follows from Subsection~\ref{setup} that the critical point $c_{a_k}$ tends to $\partial\sigma_{a_k}^{-1}(\Omega_{a_k})$, and hence $d_{\mathrm{sph}}\left(\sigma_a\left(c_{a_k}\right),\partial\Omega_{a_k}\right)=d_{\mathrm{sph}}\left(2,\partial\Omega_{a_k}\right)$ tends to $0$ as $k\to+\infty$. Therefore, $\pmb{\Psi}(a_k)=\psi_{a_k}(2)$ accumulates on $\widetilde{C}_2\subset\partial\D_2$.

Now suppose that $\{a_k\}_k\subset S\setminus\cC(\mathcal{S})$ is a sequence with $\{a_k\}_k\to a'\in\mathfrak{T}^\pm$ with $\re(a')\in\left(\frac32,4\right)$. Then, $f_{a'}\vert_{\D}$ is univalent and $f_{a'}(\mathbb{S}^1)$ is a closed curve with a point of tangential self-intersection (compare the proof of Proposition~\ref{limit_point_conn_locus} and Figure~\ref{fig:double_3_2}). It also follows from the proof of Lemma~\ref{arc_escape_time_lemma} that the critical value $2$ of $\sigma_{a'}$ lands in the bounded component $^{b}T_{a'}^0$ of the corresponding desingularized droplet (where the desingularized droplet is the set obtained by removing the cusp $-2$ and the point of tangential self-intersection from the droplet $T_{a'}=\widehat{\C}\setminus\Omega_{a'}$) under exactly $n'$ iterates of $\sigma_{a'}$ (for some $n'\equiv n'(a')\geq 1$). Note that for $k$ sufficiently large, $\sigma_{a_k}$ is a small perturbation of $\sigma_{a'}$. We set $U_k:= \Int{(T_{a_k}^0\cup\sigma_{a_k}^{-1} (T_{a_k}^0))}$, where $T_{a_k}^0$ is the desingularized droplet. Then, for $k$ large enough, the critical value $2$ of $\sigma_{a_k}$ lands in $T_{a_k}^0$ under $\sigma_{a_k}^{\circ n''}$ (where $n''\in\{n',n'+1\}$), and the hyperbolic geodesic in $U_k$ connecting $\sigma_{a_k}^{\circ n''}(2)$ and $\infty$ passes through an extremely narrow channel formed by the splitting of the double point on $\partial T_{a'}$ (the thickness of this channel decreases as $k\to+\infty$, and gets pinched in the limit). Since this part of the geodesic lies extremely close to the boundary of $U_k$, the hyperbolic distance between $\sigma_{a_k}^{\circ n''}(2)$ and $\infty$ (in $U_k$) tends to $\infty$ as $k$ increases. Furthermore, as $\psi_{a_k}$ is a conformal isomorphism from $U_k$ onto $\mathcal{Q}_1\cup\rho^{-1}(\mathcal{Q}_1)$, we have that the hyperbolic distance between $\psi_{a_k}(\sigma_{a_k}^{\circ n''}(2))$ and $0$ (in $\mathcal{Q}_1\cup\rho^{-1}(\mathcal{Q}_1)$) tends to $\infty$ as $k$ increases. Consequently, $\{\psi_{a_k}(\sigma_{a_k}^{\circ n''}(2))\}_k$ escapes to the boundary of $\mathcal{Q}_1\cup\rho^{-1}(\mathcal{Q}_1)$ as $k\to+\infty$. But the sequence $\{\psi_{a_k}(\sigma_{a_k}^{\circ n''}(2))\}_k$ is contained in $\mathcal{Q}_1$, and hence, $\{\psi_{a_k}(\sigma_{a_k}^{\circ n''}(2))\}_k$ must converge to $\omega\in\partial\mathcal{Q}_1$. In fact, the dynamical properties of $\sigma_{a_k}$ and the geometry of $T_{a_k}^0$ now imply that for $k$ sufficiently large, each $\psi_{a_k}(\sigma_{a_k}^{\circ j}(2))$ ($0\leq j\leq n''$) is close to $\omega\in\partial\mathbb{D}_2$, and hence $\pmb{\Psi}(a_k)=\psi_{a_k}(2)$ converges to $\omega\in\partial\mathbb{D}_2$ as $k\to+\infty$.

Finally let $\{a_k\}_k\subset S\setminus\cC(\mathcal{S})$ be a sequence accumulating on $\cC(\mathcal{S})$. Suppose that $\{\pmb{\Psi}(a_k)\}_k$ converges to some $u\in\D_2$. Then, $\{\psi_{a_k}(2)\}_k$ is contained in a compact subset $\mathcal{X}$ of $\D_2$. After passing to a subsequence, we can assume that $\mathcal{X}$ is contained in a single tile of $\D_2$. But this implies that each $a_k$ has a common depth $n_0$ (see Definition~\ref{def_depth}), and $\psi_{a_k}(\sigma_{a_k}^{\circ n_0}(2))$ is contained in the compact set $\rho^{\circ n_0}(\mathcal{X})\subset\mathcal{\mathcal{Q}}_1$ for each $k$. Note that the map $\sigma_a$, the fundamental tile $T_a^0$ as well as (the continuous extension of) the Riemann map $\psi_a^{-1}:\mathcal{Q}_1\to T_a^0$ change continuously with the parameter as $a$ runs over $S$. Therefore, for every accumulation point $a'$ of $\{a_k\}_k$, the point $\sigma_{a'}^{\circ n_0}(2)$ belongs to the compact set $\psi_{a'}^{-1}(\rho^{\circ n_0}(\mathcal{X}))$. In particular, the critical value $2$ of $\sigma_{a'}$ lies in the tiling set $T_{a'}^\infty$. This contradicts the assumption that $\{a_k\}_k$ accumulates on $\cC(\mathcal{S})$, and proves that $\{\pmb{\Psi}(a_k)\}_k$ must accumulate on the boundary of $\D_2$.
\end{proof}

\begin{definition}[Parameter Rays of $\mathcal{S}$]\label{para_ray_schwarz}
The pre-image of a ray at angle $\theta\in(\frac13,\frac23)$ in $\D_2$ under the map $\pmb{\Psi}$ is called a $\theta$-parameter ray of $\mathcal{S}$.
\end{definition}

\section{Properties of The Straightening Map}\label{chi_prop_sec}

In this section, we will study continuity and surjectivity properties of the straightening map $\chi$. The results of this section will allow us to show that $\chi$ is ``almost" a homeomorphism between $\cC(\mathcal{S})$ and the parabolic Tricorn $\cC(\mathfrak{L}_0)$.

\subsection{Continuity Properties}\label{continuity}
The goal of this subsection is to demonstrate that the straightening map is continuous at most parameters. Let us first discuss some basic topological properties of $\chi$.

\subsubsection{Properness of $\chi$}\label{chi_proper_subsec}

There are three period $2$ hyperbolic components bifurcating from $\widetilde{H}$ with centers at $\frac54+\frac{\sqrt{33}}{4}$, and $\frac72\pm \frac{i\sqrt{3}}{2}$. Let $a$ be a parameter on the boundary of any of these period two hyperbolic components (say, $H$) with a $2$-periodic cycle of multiplier $e^{2\pi ip/q}$, where $\mathrm{gcd}(p,q)=1$. By Proposition~\ref{ThmEvenBif_schwarz}, $a$ is the root of a hyperbolic component of period $2q$ that bifurcates from $H$. In particular, $\cC(\mathcal{S})\setminus\{a\}$ consists of two distinct connected components. We define the $p/q$-limb of $H$ as the closure of the connected component of $\cC(\mathcal{S})\setminus\{a\}$ not containing $H$.
 
\begin{proposition}\label{limbs_compact}
The topological closure (in $\C$) of every limb of the period two hyperbolic components of $\cC(\mathcal{S})$ is contained in $S$.
\end{proposition}
\begin{proof}
It follows from Propositions~\ref{limit_point_conn_locus} and~\ref{hyp_one_not_proper} that if a limit point of a limb of a period two hyperbolic component of $\cC(\mathcal{S})$ lies outside $S$, then such a limit point must be $4\pm i\sqrt{3}$. Since $\mathcal{I}\subset\partial\widetilde{H}$, it follows that if the closure (in $\C$) of such a limb contains $4\pm i\sqrt{3}$, then $S\setminus\cC(\mathcal{S})$ must have at least two connected components. But this contradicts Theorem~\ref{escape_unif_thm}. Hence, the topological closure of every limb of a period two hyperbolic component of $\cC(\mathcal{S})$ is contained in $S$.
\end{proof}

\begin{lemma}[Properness of of $\chi$ and $\chi^{-1}$]\label{chi_proper}
Let $\{a_n\}_n\subset\cC(\mathcal{S})$.
\begin{enumerate}
\item If $a_n\to a\in\overline{\cC(\mathcal{S})}\setminus\cC(\mathcal{S})$, then every accumulation point of $\{\chi(a_n)\}_n$ lies in $\overline{\cC(\mathfrak{L}_0)}\setminus\cC(\mathfrak{L}_0)$.

\item If $\chi(a_n)\to(\alpha,A)\in\overline{\cC(\mathfrak{L}_0)}\setminus\cC(\mathfrak{L}_0)$, then every accumulation point of $\{a_n\}_n$ lies in $\overline{\cC(\mathcal{S})}\setminus\cC(\mathcal{S})$.
\end{enumerate}
\end{lemma}
\begin{proof}
1) It follows by Propositions~\ref{conn_locus_almost_closed} and~\ref{limit_point_conn_locus} that if a sequence $\{a_n\}_n\subset\cC(\mathcal{S})$ has a limit point $a'\in\overline{\cC(\mathcal{S})}\setminus\cC(\mathcal{S})$, then $a'\in\mathcal{I}$. Moreover, the proof of Proposition~\ref{limbs_compact} implies that possibly after passing to a subsequence, $\{a_n\}_n$ is either contained in (the closures of) the hyperbolic components of period one or two; or each $a_n$ belongs to some $\frac{p_n}{q_n}$--limb of a period two hyperbolic component with $q_n\to+\infty$ as $n\to+\infty$. Since $\chi$ maps the $\frac{p_n}{q_n}$--limb of a period two hyperbolic component of $\cC(\mathcal{S})$ to the $\frac{p_n}{q_n}$--limb of a period two hyperbolic component of $\cC(\mathfrak{L}_0)$, it follows that for every finite accumulation point of $\{\chi(a_n)\}_n$, the point at $\infty$ is a multiple parabolic fixed point. But such parameters do not lie in $\cC(\mathfrak{L}_0)$; i.e., every accumulation point of $\{\chi(a_n)\}_n$ belongs to $\overline{\cC(\mathfrak{L}_0)}\setminus\cC(\mathfrak{L}_0)$. 

2) Let $\{a_n\}_n\subset\cC(\mathfrak{L}_0)$ be a sequence such that $\chi(a_n)\to(\alpha,A)\in\overline{\cC(\mathfrak{L}_0)}\setminus\cC(\mathfrak{L}_0)$. It follows that $R_{\alpha,A}$ has a multiple parabolic fixed point at $\infty$, and possibly passing to a subsequence, $\{\chi(a_n)\}_n$ is either contained in (the closures of) the hyperbolic components of period one or two; or each $\chi(a_n)$ belongs to some $\frac{p_n}{q_n}$--limb of a period two hyperbolic component of $\cC(\mathfrak{L}_0)$ with $q_n\to+\infty$ as $n\to+\infty$. Therefore, every accumulation point of $\{a_n\}_n$ lies on $\mathcal{I}=\overline{\cC(\mathcal{S})}\setminus\cC(\mathcal{S})$.
\end{proof}

\subsubsection{Continuity on Rigid and Hyperbolic parameters}\label{chi_cont_subsec}

A parameter $a\in\cC(\mathcal{S})$ is called \emph{qc rigid} if no map in $\cC(\mathcal{S})\setminus\{a\}$ is quasiconformally conjugate to $\sigma_a$.

\begin{lemma}\label{rigid_equiv}
Let $a\in\cC(\mathcal{S})$. Then, $a$ is qc rigid in $\cC(\mathcal{S})$ if and only if $\chi(a)$ is qc rigid in $\cC(\mathfrak{L}_0)$.
\end{lemma}
\begin{proof}
Note that $a$ (respectively, $\chi(a)$) is not qc rigid in $\cC(\mathcal{S})$ (respectively, in $\cC(\mathfrak{L}_0)$) if and only if there exists a non-trivial $\sigma_a$-invariant (respectively, $R_{\chi(a)}$-invariant) Beltrami coefficient supported on $K_a$ (respectively, on $\mathcal{K}_{\chi(a)}$). Moreover, such a non-trivial Beltrami coefficient can be pulled back by the hybrid conjugacy between the pinched anti-quadratic-like restrictions of $\sigma_a$ and $R_{\chi(a)}$ (or its inverse), and the resulting Beltrami coefficient will be non-trivial, invariant under the dynamics, and supported on the non-escaping set (or the filled Julia set). The result follows.
\end{proof}

The next proposition shows that the straightening map $\chi:\cC(\mathcal{S})\to\cC(\mathfrak{L}_0)$ is continuous at qc rigid parameters. 

\begin{proposition}[Continuity at Rigid Parameters]\label{continuity_rigid}
Let $\widetilde{a}\in\cC(\mathcal{S})$ be a qc rigid parameter in $\cC(\mathcal{S})$. Then $\chi$ is continuous at $\widetilde{a}$.
\end{proposition}
\begin{proof}
Let $\widetilde{a}\in\cC(\mathcal{S})$ be a qc rigid parameter. By Lemma~\ref{rigid_equiv}, the map $R_{\chi(\widetilde{a})}$ admits no non-trivial quasiconformal deformation. Let us pick a sequence $\{a_n\}$ in $\cC(\mathcal{S})$ converging to $\widetilde{a}\in\cC(\mathcal{S})$. By Lemma~\ref{chi_proper}, we can extract a convergent subsequence $\{\chi(a_{n_k})\}\to(\alpha,A)\in\cC(\mathfrak{L}_0)$. By our construction of the anti-quasiregular extension $\mathbf{G}_{a_{n_k}}$ of the pinched anti-quadratic-like map $\eta_{a_{n_k}}\circ\sigma_{a_{n_k}}\circ\eta_{a_{n_k}}^{-1}$, it follows that the dilatation ratios of the quasiconformal conjugacies $\Phi_{a_{n_k}}$ between $\mathbf{G}_{a_{n_k}}$ and $R_{\chi(a_{n_k})}$ stay uniformly bounded. By compactness of families of quasiconformal maps with uniformly bounded dilatation, we conclude that there is a subsequential limit of $\{\Phi_{a_{n_k}}\}$, that quasiconformally conjugates $\mathbf{G}_{\widetilde{a}}$ to $R_{\alpha,A}$ (compare \cite[\S II.7, Lemma, Page 313]{DH2}). On the other hand, by definition of $\chi$, the maps $\mathbf{G}_{\widetilde{a}}$ and $R_{\chi(\widetilde{a})}$ are quasiconformally conjugate. Therefore, $R_{\alpha,A}$ and $R_{\chi(\widetilde{a})}$ are quasiconformally conjugate. By the rigidity assumption on $\chi(\widetilde{a})$, it follows that $\chi(\widetilde{a})=(\alpha,A)$. Therefore, for every sequence $\{a_n\}$ converging to $\widetilde{a}$, there exists a convergent subsequence $\{\chi(a_{n_k})\}$ converging to $\chi(\widetilde{a})$. Hence, $\chi(a_n)\to\chi(\widetilde{a})$, whenever $a_n\to \widetilde{a}$. This proves that $\chi$ is continuous at $\widetilde{a}$.
\end{proof} 

Although hyperbolic parameters are not qc rigid (except the centers) in $\cC(\mathcal{S})$, we can prove continuity of $\chi$ on hyperbolic components thanks to the dynamical parametrization of hyperbolic components in $\cC(\mathcal{S})$ and $\cC(\mathfrak{L}_0)$.

\begin{proposition}[Continuity at Hyperbolic Parameters]\label{continuity_hyperbolic}
Let $H$ be a hyperbolic component in $\cC(\mathcal{S})$. Then the straightening map $\chi$ is a real-analytic diffeomorphism from $H$ onto a hyperbolic component in $\cC(\mathfrak{L}_0)$.
\end{proposition}
\begin{proof}
Let $a$ be the center of $H$, and $(\alpha,A):=\chi(a)$. Then $(\alpha,A)$ is the center of some hyperbolic component $H^{\sharp}$ of $\cC(\mathfrak{L}_0)$. Let $\widetilde{\eta}_H:H\to\mathcal{B}^{\pm}$ and $\eta_{H^{\sharp}}:H^{\sharp}\to\mathcal{B}^{\pm}$ be the dynamical uniformizations of the hyperbolic components $H$ and $H^{\sharp}$ (see Propositions~\ref{unif_hyp_schwarz} and~\ref{unif_hyp_rat}). It is now easy to see that $\chi=\eta_{H^{\sharp}}^{-1}\circ\widetilde{\eta}_H$ on $H$. It follows that $\chi$ is a real-analytic diffeomorphism from $H$ onto the hyperbolic component $H^{\sharp}$ in $\cC(\mathfrak{L}_0)$.
\end{proof}

\subsubsection{Discontinuity of $\chi$ on Parabolic Arcs}\label{chi_discont_subsec} 

Since parabolic parameters of even period (respectively, parabolic cusps of odd period) are qc rigid in $\cC(\mathcal{S})$, Proposition~\ref{continuity_rigid} implies that $\chi$ is continuous at all parabolic parameters of even period (respectively, parabolic cusps of odd period) of $\cC(\mathcal{S})$. However, note that the simple parabolic parameters of odd period in $\cC(\mathcal{S})$ admit non-trivial quasiconformal deformations, and hence the above results do not say anything about continuity of $\chi$ at the simple parabolic parameters of odd period of $\cC(\mathcal{S})$. Recall that by Proposition~\ref{ThmBifArc_schwarz}, every parabolic arc has, at both ends, an interval of positive length at which bifurcation from a hyperbolic component of odd period $k$ to a hyperbolic component of period $2k$ occurs. We will now carry out a finer analysis of continuity properties of $\chi$ near these bifurcating sub-arcs.

Let $H$ be a hyperbolic component of odd period $k$ in $\cC(\mathcal{S})$, $\mathcal{C}$ be a parabolic arc of $\partial H$, $a:\mathbb{R}\to\mathcal{C}$ be the critical {\'E}calle height parametrization of $\mathcal{C}$, and $H'$ be a hyperbolic component of period $2k$ bifurcating from $H$ across $\mathcal{C}$. Let us start with an easy observation.

\begin{lemma}\label{arc_to_arc}
Let $\cC$ be a parabolic arc in $\cC(\mathcal{S})$. Then the following hold true.

1) $\chi$ is a homeomorphism from $\cC$ onto a parabolic arc in $\cC(\mathfrak{L}_0)$.

2) If $\{a_n\}\subset\cC(\mathcal{S})$ converges to $a\in\cC$, then every accumulation point of $\{\chi(a_n)\}$ lies on the parabolic arc $\chi(\cC)$ in $\cC(\mathfrak{L}_0)$.
\end{lemma}
\begin{proof}
1) Let $a=a(0)$ be the parameter on $\cC$ with critical {\'E}calle height $0$, and $(\alpha,A):=\chi(a)$. Then $(\alpha,A)$ lies on some parabolic arc $\cC^{\sharp}$ of $\cC(\mathfrak{L}_0)$, and has critical {\'E}calle height $0$. Since critical {\'E}calle heights are preserved under hybrid equivalences, it is now easy to see that for each $a(h)\in\cC$ (for $h\in(-\infty,+\infty)$, its image under $\chi$ is the unique parameter on $\cC^{\sharp}$ with critical {\'E}calle height $h$. It follows that $\chi$ is a homeomorphism from $\cC$ onto the parabolic arc $\cC^{\sharp}$ in $\cC(\mathfrak{L}_0)$.

2) Let $\{a_n\}\subset\cC(\mathcal{S})$ be a sequence converging to $a\in\cC$. By Lemma~\ref{chi_proper}, we can extract a convergent subsequence $\{\chi(a_{n_k})\}\to(\alpha,A)\in\cC(\mathfrak{L}_0)$. By our construction of the anti-quasiregular extension $\mathbf{G}_{a_{n_k}}$ of the pinched anti-quadratic-like map $\eta_{a_{n_k}}\circ\sigma_{a_{n_k}}\circ\eta_{a_{n_k}}^{-1}$, it follows that the dilatation ratios of the quasiconformal conjugacies $\Phi_{a_{n_k}}$ between $\mathbf{G}_{a_{n_k}}$ and $R_{\chi(a_{n_k})}$ are uniformly bounded. By compactness of families of quasiconformal maps with uniformly bounded dilatation, we conclude that there is a subsequential limit of $\{\Phi_{a_{n_k}}\}$, that quasiconformally conjugates $\mathbf{G}_{a}$ to $R_{\alpha,A}$ (compare \cite[\S II.7, Lemma, Page 313]{DH2}). On the other hand, by definition of $\chi$, the maps $\mathbf{G}_{a}$ and $R_{\chi(a)}$ are quasiconformally conjugate. Therefore, $R_{\alpha,A}$ and $R_{\chi(a)}$ are quasiconformally conjugate.

Let $\widehat{\Phi}$ be a quasiconformal map conjugating $R_{\chi(a)}$ to $R_{\alpha,A}$. We set $\mu:=\partial_{\overline{z}}\widehat{\Phi}/\partial_z\widehat{\Phi}$, and define the Beltrami path $\{t\mu:t\in[0,1]\}$. By construction, each $t\mu$ is $R_{\chi(a)}$-invariant. By the parametric version of the measurable Riemann mapping theorem, there exists a continuous family of quasiconformal homeomorphisms $\{\widehat{\Phi}_t\}_{t\in[0,1]}$ such that $\partial_{\overline{z}}\widehat{\Phi}_t/\partial_z\widehat{\Phi}_t=t\mu$, and $\widehat{\Phi}_1=\widehat{\Phi}$. It follows from the $R_{\chi(a)}$-invariance of the Beltrami coefficients that $\widehat{\Phi}_t\circ R_{\chi(a)}\circ\widehat{\Phi}_t^{-1}=R_{\alpha(t),A(t)}\in\cC(\mathfrak{L}_0)$, for $t\in[0,1]$ with $R_{\alpha(1),A(1)}=R_{\alpha,A}$, and $t\mapsto (\alpha(t),A(t))$ is continuous. Since $\widehat{\Phi}_0$ is conformal; i.e., a M{\"o}bius map, and no two distinct maps in $\cC(\mathfrak{L}_0)$ are M{\"o}bius conjugate, we conclude that $R_{\alpha(0),A(0)}=R_{\chi(a)}$. Therefore, $R_{\alpha,A}$ and $R_{\chi(a)}$ are connected by a path of quasiconformally conjugate parameters in $\cC(\mathfrak{L}_0)$. Since $\chi(a)$ lies in the parabolic arc $\chi(\cC)\subset\cC(\mathfrak{L}_0)$ (by part 1), it follows from the definition of parabolic arcs that $(\alpha,A)\in\chi(\cC)$. Therefore, every subsequential limit of $\{\chi(a_n)\}$ lies on the parabolic arc $\chi(\cC)\subset\cC(\mathfrak{L}_0)$.
\end{proof}

We will denote the critical {\'E}calle height parametrization of the parabolic arc $\chi(\cC)$ by $c$. For any $h$ in $\mathbb{R}$, let us denote the fixed point index of the unique parabolic cycle of $\sigma_{a(h)}^{\circ 2}$ (respectively, the unique parabolic cycle of $R_{c(h)}^{\circ 2}$ other than the fixed point at $\infty$) by $\ind_{\mathcal{C}}(\sigma_{a(h)}^{\circ 2})$ (respectively, by $\ind_{\chi(\mathcal{C})}(R_{c(h)}^{\circ 2})$). This defines a pair of real-analytic functions (which we will refer to as index functions)

\begin{center}
$\begin{array}{rccc}
  \ind_{\mathcal{C}}: & \mathbb{R} & \to & \mathbb{R} \\
  & h & \mapsto & \ind_{\mathcal{C}}(\sigma_{a(h)}^{\circ 2})\\
   \end{array}$
\end{center}

and

\begin{center}
$\begin{array}{rccc}
  \ind_{\mathcal{\chi(C)}}: & \mathbb{R} & \to & \mathbb{R} \\
  & h & \mapsto & \ind_{\chi(\mathcal{C)}}(R_{c(h)}^{\circ 2}).\\
   \end{array}$
\end{center}

Our next goal is to use Proposition \ref{index_increasing} to show that $\chi\vert_{H'}$ has a dynamically defined continuous extension to $\cC\cap\partial H'$. By Propositions~\ref{index_increasing} and~\ref{index_increasing_1}, we can define a map $\xi : \overline{\mathcal{C}}\cap\partial H'\to \overline{\chi(\mathcal{C})}\cap\partial \chi(H')$ by sending the parabolic cusp on $\overline{\mathcal{C}}\cap\partial H'$ to the parabolic cusp on $\overline{\chi(\mathcal{C})}\cap\partial \chi(H')$, and the unique parameter on $\mathcal{C}\cap\partial H'$ with a parabolic cycle of index $\tau$ to the unique parameter on $\chi(\mathcal{C})\cap\partial \chi(H')$ with a parabolic cycle of index $\tau$ (compare \cite[\S 8]{IM2}). 

\begin{proposition}\label{cont_ext}
$\xi$ extends $\chi\vert_{H'}$ continuously to $\overline{\cC}\cap\partial H'$ preserving the index of the parabolic cycle. 
\end{proposition}
\begin{proof}
Let us pick the unique parameter $a$ on $\cC\cap\partial H'$ having parabolic fixed point index $\tau$. Consider a sequence $\{a_n\} \in H'$ with $a_n\to a$. Suppose that $\{\chi(a_n)\}\subset\chi(H')$ converges to some $(\alpha,A)\in\mathfrak{L}_0$. By Proposition~\ref{arc_to_arc} (part 2), we have that $(\alpha,A)\in\chi(\cC)$. Furthermore, since the sequence $\{\chi(a_n)\}$ is contained in $\chi(H')$, its limit $(\alpha,A)$ must lie in $\overline{\chi(H')}$. It follows that $(\alpha,A)\in\chi(\cC)\cap\partial\chi(H')$ (see Figure~\ref{bifurcation_point_2}).

For any $n$, the map $\sigma_{a_n}^{\circ 2}$ has two distinct $k$-periodic attracting cycles (which are born out of the parabolic cycle) with multipliers $\lambda_{a_n}$ and $\overline{\lambda}_{a_n}$. Since $a_n$ converges to $a$, we have that

\begin{equation}\label{multiplier_index}
\frac{1}{1-\lambda_{a_n}} + \frac{1}{1-\overline{\lambda}_{a_n}} \longrightarrow \tau
\end{equation}
as $n\to\infty$. 

\begin{figure}[ht!]
\begin{center}
\includegraphics[scale=0.16]{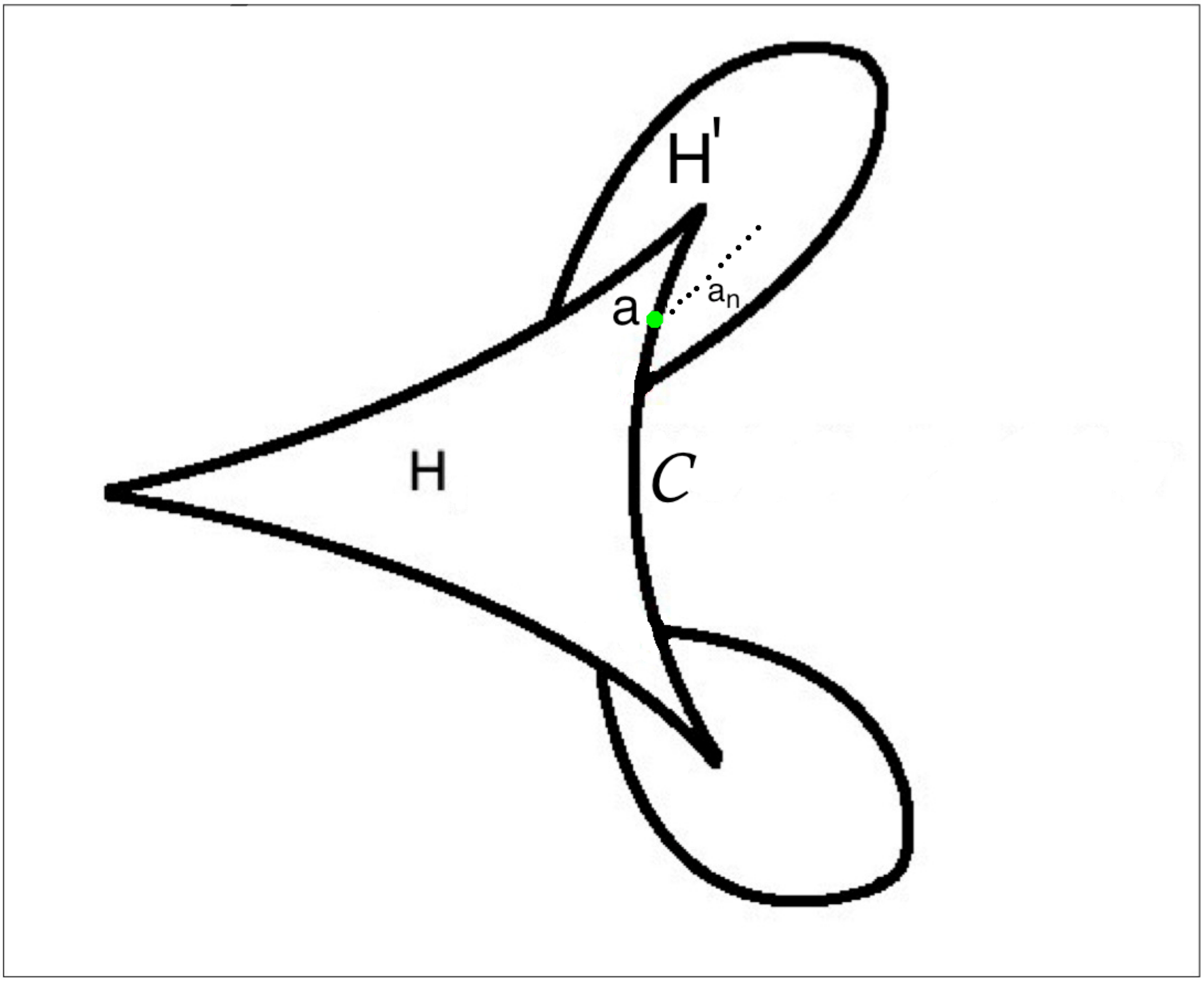}\ \includegraphics[scale=0.16]{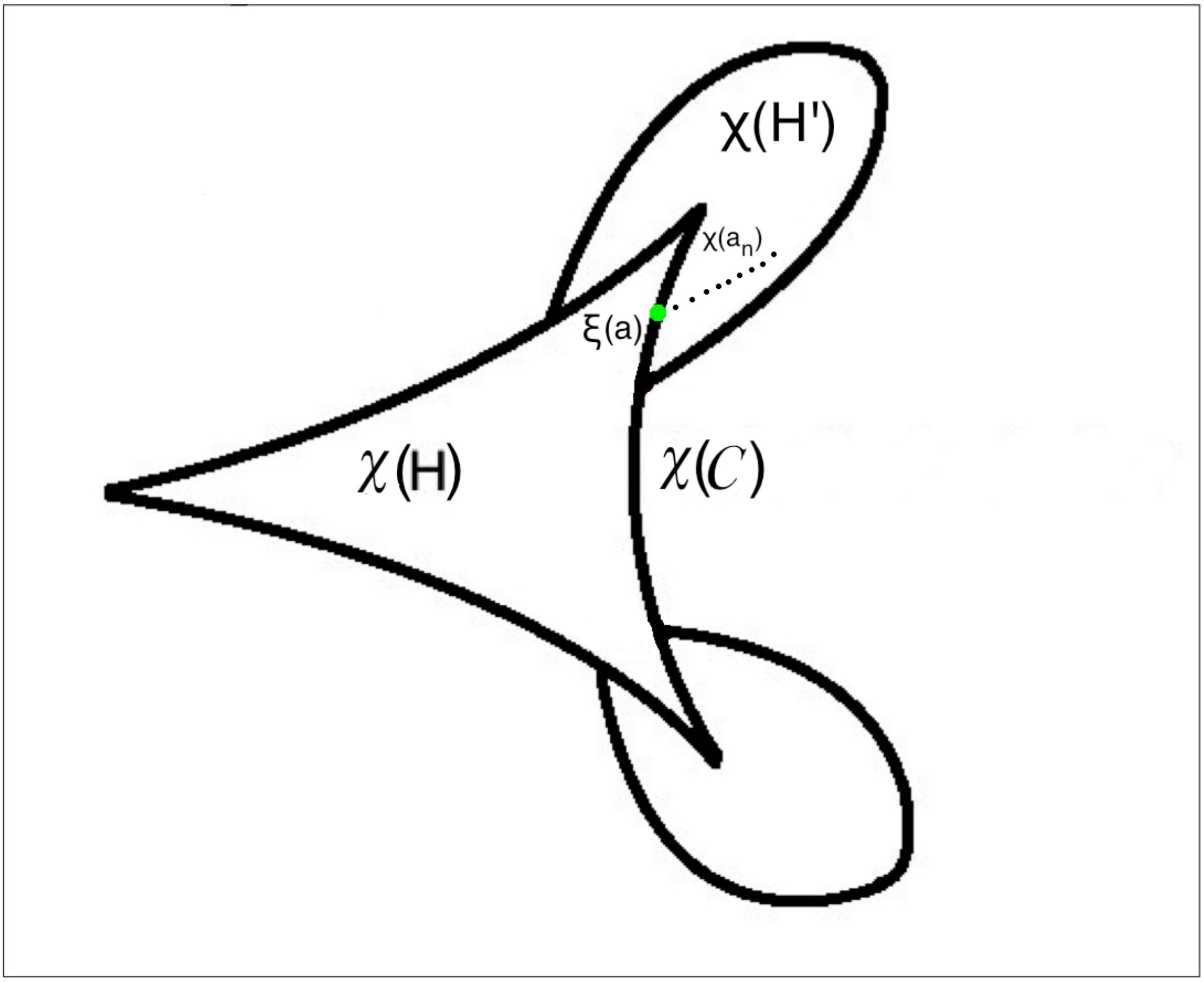}
\end{center}
\caption{The straightening map $\chi$, restricted to $\overline{H}$, is a homeomorphism. On the other hand, a continuous extension $\xi$ of $\chi\vert_{H'}$ to $\cC\cap\partial H'$ must respect the fixed point indices of the parabolic cycles of $\sigma_{a}$ and $R_{\xi(a)}$.}
\label{bifurcation_point_2}
\end{figure}

Since the multipliers of attracting periodic orbits are preserved by $\chi$, it follows that $R_{\chi(a_n)}^{\circ 2}$ has two distinct $k$-periodic attracting cycles with multipliers $\lambda_{a_n}$ and $\overline{\lambda}_{a_n}$. As $\{\chi(a_n)\}$ converges to the odd period parabolic parameter $(\alpha,A)$, the same limiting relation (\ref{multiplier_index}) holds for the fixed point index of the parabolic cycle of $R_{\alpha,A}^{\circ 2}$ as well. In particular, the parabolic fixed point index of $R_{\alpha,A}^{\circ 2}$ is also $\tau$. Therefore, $(\alpha,A)$ must be the unique parameter on $\chi(\mathcal{C})\cap\partial \chi(H')$ with a parabolic cycle of index $\tau$. 

On the other hand, parabolic cusps are qc rigid in $\cC(\mathfrak{L}_0)$. Hence, $\chi$ sends the parabolic cusp on $\overline{\mathcal{C}}\cap\partial H'$ to the parabolic cusp on $\overline{\chi(\mathcal{C})}\cap\partial \chi(H')$, and is continuous at the cusp.

It now follows that $\xi$ is the required continuous extension of $\chi\vert_{H'}$ to $\overline{\cC}\cap\partial H'$. 
\end{proof}

As a complementary result, let us mention that $\chi$ is a homeomorphism restricted to the closure $\overline{H}$ of every hyperbolic component $H$ of odd period. We denote the Koenigs ratio map of the hyperbolic component $H$ (respectively, $\chi(H)$) by $\widetilde{\zeta}_H$ (respectively, $\zeta_{\chi(H)}$). 

\begin{proposition}\label{homeo_odd_closure}
As $a$ in $H$ (respectively, $c$ is $\chi(H)$) approaches a simple parabolic parameter with critical {\'E}calle height $h$ on the boundary of $H$ (respectively, $\chi(H)$), the quantity $\frac{1-\widetilde{\zeta}_H(a)}{1-\vert\widetilde{\zeta}_H(a)\vert^2}$ (respectively, $\frac{1-\zeta_{\chi(H)}(c)}{1-\vert\zeta_{\chi(H)}(c)\vert^2}$) converges to $\frac{1}{2}-2ih$. Consequently, $\chi$ maps the closure $\overline{H}$ of the hyperbolic component $H$ homeomorphically onto the closure $\overline{\chi(H)}$ of the hyperbolic component $\chi(H)$.
\end{proposition}
\begin{proof}
The proof of the first statement is similar to that of \cite[Lemma~6.3]{IM2}. Since $\chi$ preserves Koenigs ratio (of parameters in $H$) and critical {\'E}calle height (of simple parabolic parameters on $\partial H$), it follows that $\chi$ extends continuously to $\partial H$. By Proposition \ref{continuity_hyperbolic} and Lemma \ref{arc_to_arc}, $\chi(\overline{H})$ is indeed the closure of the hyperbolic component $\chi(H)$. Since $\chi$ is injective, it is a homeomorphism from $\overline{H}$ onto $\overline{\chi(H)}$.
\end{proof}

We are now in a position to show that continuity of $\chi$ on $\cC$ imposes a severe restriction on the index functions $\ind_{\mathcal{C}}$ and $\ind_{\mathcal{\chi(C)}}$.

\begin{lemma}[Uniform Height-Index Relation]\label{uniform}
Let $\cC$ be a parabolic arc in $\cC(\mathcal{S})$. If $\chi$ is continuous at every parameter on $\cC$, then the functions $ \ind_{\mathcal{C}}$ and $\ind_{\mathcal{\chi(C)}}$ are identically equal.
\end{lemma}
\begin{proof}
Recall that $\chi$ preserves critical {\'E}calle height of simple parabolic parameters. By Proposition \ref{cont_ext}, continuity of $\chi$ on $\cC$ would imply that $\chi$ also preserves the index of simple parabolic cycles (for parameters in $\overline{\cC}\cap\partial H'$). But this means that the indices of the parabolic cycles of $\sigma_{a(h)}^{\circ 2}$ and $R_{c(h)}^{\circ 2}$ are equal for all values of $h$ in an unbounded interval. Since the index functions $\ind_{\mathcal{C}}$ and $\ind_{\mathcal{\chi(C)}}$ are real-analytic, we conclude that the indices of the parabolic cycles of $\sigma_{a(h)}^{\circ 2}$ and $R_{c(h)}^{\circ 2}$ are equal for all real values of $h$; i.e., $\ind_{\mathcal{C}}\equiv\ind_{\mathcal{\chi(C)}}$.
\end{proof}

\begin{remark}
The above condition on index functions seems unlikely to hold in general. This criterion can be used to prove discontinuity of $\chi$ on certain low period parabolic arcs of $\cC(\mathcal{S})$. 

For an analogous discussion of discontinuity of straightening maps for the Tricorn, see \cite[Proposition~8.1]{IM2}.
\end{remark}

\subsubsection{Possible Discontinuity on QC Non-rigid Parameters}\label{chi_discont_beltrami_arcs}
We have seen so far that $\chi$ is continuous at all qc rigid parameters of $\cC(\mathcal{S})$. This includes parameters with an even-periodic neutral cycle, and parabolic cusps of odd period. It was also shown that $\chi$ is continuous at hyperbolic parameters (which are not qc rigid, except for the centers). On the other hand, we analyzed the behavior of $\chi$ at odd period simple parabolic parameters (which are also not qc rigid in $\cC(\mathcal{S})$), and concluded that $\chi$ is not necessarily continuous at such parameters. We now turn our attention to the remaining qc non-rigid parameters in $\cC(\mathcal{S})$.

Let us assume that $\sigma_a\in\cC(\mathcal{S})$ admits a real one-dimensional quasiconformal deformation space in $\cC(\mathcal{S})$, and $\mu$ is a non-trivial $\sigma_a$-invariant Beltrami coefficient which is supported on $\partial K_a$. Set $m:=\vert\vert\mu\vert\vert_\infty\in\left(0,1\right)$. Then, $t\mu$ is also a non-trivial $\sigma_a$-invariant Beltrami coefficient supported on $\partial K_a$, for $t\in\left(-\frac1m,\frac1m\right)$. By the measurable Riemann mapping theorem with parameters, we obtain quasiconformal maps $\{h_t\}$ (where $t\in\left(-\frac1m,\frac1m\right)$) with associated Beltrami coefficients $t\mu$ such that $h_t$ fixes $\pm 2$ and $\infty$, and $\{h_t\}$ depends real-analytically on $t$. According to Proposition~\ref{schwarz_qcdef}, $h_t\circ\sigma_a\circ h_t^{-1}\in\cC(\mathcal{S})$ for all $t\in\left(-\frac1m,\frac1m\right)$. This produces a real-analytic arc of quasiconformally conjugate parameters in $\cC(\mathcal{S})$ which contains $a$ in its interior. This arc, which is the full quasiconformal deformation space of $\sigma_a$ in $\cC(\mathcal{S})$, is called the \emph{queer Beltrami arc} containing $a$.

The following proposition can be proven following the arguments of Lemma~\ref{arc_to_arc}.

\begin{lemma}\label{beltrami_arc_discont}
Let $\Gamma$ be a queer Beltrami arc in $\cC(\mathcal{S})$. Then,

1) $\chi$ is a homeomorphism from $\Gamma$ onto some queer Beltrami arc in $\cC(\mathfrak{L}_0)$, and

2) if $\{a_n\}\subset\cC(\mathcal{S})$ converges to $a\in\Gamma$, then every accumulation point of $\{\chi(a_n)\}$ lies on the queer Beltrami arc $\chi(\Gamma)$ in $\cC(\mathfrak{L}_0)$.
\end{lemma}

Finally, let $\sigma_a\in\cC(\mathcal{S})$ admit a real two-dimensional quasiconformal deformation space $\cC(\mathcal{S})$, and $\mu_1$, $\mu_2$ be $\R$-independent non-trivial $\sigma_a$-invariant Beltrami coefficients supported on $\partial K_a$. Then, $\left(t_1\mu_1+t_2\mu_2\right)$ is also a non-trivial $\sigma_a$-invariant Beltrami coefficient supported on $\partial K_a$ whenever $t_1, t_2\in\R$, and $\vert\vert t_1\mu_1+t_2\mu_2\vert\vert_\infty<1$. As in the previous case, this produces an open set of quasiconformally conjugate parameters in $\cC(\mathcal{S})$ containing $a$. This open set, which is the full quasiconformal deformation space of $\sigma_a$ in $\cC(\mathcal{S})$, is called the \emph{queer component} containing $a$.

\begin{lemma}\label{cont_queer}
$\chi$ is a homeomorphism from a queer component of $\cC(\mathcal{S})$ onto some queer component of $\cC(\mathfrak{L}_0)$.
\end{lemma}

\subsection{Almost Surjectivity}\label{almost_surjective}

We will now describe the image of the straightening map $\chi$. Let us start with some preliminary results.

\begin{lemma}\label{image_closed}
$\chi(\cC(\mathcal{S}))$ is closed in $\cC(\mathfrak{L}_0)$.
\end{lemma}
\begin{proof}
Let $\left(\alpha_0,A_0\right)\in\cC(\mathfrak{L}_0)$, and $\{(\alpha_n,A_n)\}_n$ be a sequence in $\chi(\cC(\mathcal{S}))$ converging to $(\alpha_0,A_0)$. By assumption, there exists $\{a_n\}_n\in\cC(\mathcal{S})$ such that $\chi(a_n)=(\alpha_n,A_n)$, for all $n\in\mathbb{N}$. By Lemma~\ref{chi_proper}, we can assume (possibly after passing to a subsequence) that $a_n\to a\in\cC(\mathcal{S})$. It now follows by the arguments used in the proof of Proposition~\ref{continuity_rigid} that $R_{\chi(a)}$ and $R_{\alpha_0,A_0}$ are quasiconformally conjugate. 

Let $\upsilon$ be a quasiconformal conjugacy between $R_{\chi(a)}$ and $R_{\alpha_0,A_0}$, and $\mu$ be the Beltrami form defined by $\upsilon$. Since $R_{\chi(a)}$ has a connected Julia set, the Beltrami form $\mu$ can be selected so that it is supported on $\mathcal{K}_{\chi(a)}$ and is $R_{\chi(a)}$-invariant. Pulling $\mu$ back by a hybrid equivalence $\Phi_a\circ\eta_a$ between a pinched anti-quadratic-like restriction of $\sigma_a$ and $R_{\chi(a)}$, we obtain an $\sigma_a$-invariant Beltrami form $\mu_0$ supported on $K_a$ such that $\mu$ and $\mu_0$ have the same dilatation at corresponding points. By Proposition~\ref{schwarz_qcdef}, there exists a quasiconformal homeomorphism $\Upsilon$ integrating $\mu_0$ and conjugating $\sigma_a$ to some map $\sigma_{\widetilde{a}}$ in the family $\mathcal{S}$. Clearly, $\widetilde{a}\in\cC(\mathcal{S})$. It is now easy to see that $\upsilon\circ\left(\Phi_a\circ\eta_a\right)\circ\Upsilon^{-1}\circ\left(\Phi_{\widetilde{a}}\circ\eta_{\widetilde{a}}\right)^{-1}$ is a hybrid equivalence between $R_{\chi(\widetilde{a})}$ and $R_{\alpha_0,A_0}$. Since each hybrid equivalence class in $\cC(\mathfrak{L}_0)$ is a singleton (compare Remark~\ref{hybrid_class}), we have that $(\alpha_0,A_0)=\chi(\widetilde{a})\in\chi(\cC(\mathcal{S}))$. This completes the proof.
\end{proof}

Here is an important corollary of the proof of the above lemma.

\begin{corollary}\label{qc_closed}
Let
$(\alpha _{1},A_{1}), (\alpha _{2},A_{2})\in \mathcal{C}(
\mathfrak{L}_{0})$ be such that $R_{\alpha _{1},A_{1}}$ and
$R_{\alpha _{2},A_{2}}$ are quasiconformally conjugate. If $(\alpha_1,A_1)\in\chi(\cC(\mathcal{S}))$, then $(\alpha_2,A_2)\in\chi(\cC(\mathcal{S}))$ as well.
\end{corollary}

Recall that in Subsection~\ref{parameter_tessellation}, we described a uniformization $\pmb{\Psi}:S\setminus\cC(\mathcal{S})\to\D_2$ of the exterior of the connectedness locus $\cC(\mathcal{S})$ (in the parameter space). Using the map $\pmb{\Psi}$, we defined the parameter rays of $\mathcal{S}$ (see Definition \ref{para_ray_schwarz}). As the first step towards a description of the image of $\chi$, we will now use these parameter rays to show that $\chi(\cC(\mathcal{S}))$ contains all dyadic tips of $\cC(\mathfrak{L}_0)$ (see Definition \ref{def_dyadic_tip_rat} for the definition of dyadic tips).

Similar to Definition~\ref{def_dyadic_tip_rat}, we will call a parameter $a\in\cC(\mathcal{S})$ a dyadic tip of pre-period $k$ if the critical value $2$ (of $\sigma_a$) maps to the cusp point $-2$ in exactly $k$ steps; i.e., $k\geq 1$ is the smallest integer such that $\sigma_a^{\circ k}(2)=-2$.

\begin{remark}
The only dyadic tip of pre-period $1$ in $\cC(\mathcal{S})$ is $a=\frac52$.
\end{remark} 

We say that a parameter $a\in\cC(\mathcal{S})$ is \emph{critically pre-periodic} if the critical point $c_a$ (or equivalently, the critical value $2$) of $\sigma_a$ is strictly pre-periodic.

\begin{lemma}\label{misi_rays}

\noindent\begin{enumerate}\upshape

\item If $\theta\in\left(\frac13,\frac23\right)$ is strictly pre-periodic under $\rho$, then the parameter ray of $\mathcal{S}$ at angle $\theta$ lands at a critically pre-periodic parameter such that in the corresponding dynamical plane, the dynamical ray at angle $\theta$ lands at the critical value $2$.

\item For every critically pre-periodic parameter $a_0$ of $\cC(\mathcal{S})$, the arguments of the parameter rays (at pre-periodic angles) of $\mathcal{S}$ landing at $a_0$ are exactly the arguments of the dynamical rays that land at the critical value $2$ in the dynamical plane of $\sigma_{a_0}$.
\end{enumerate}
\end{lemma}
\begin{proof}
1) Let $a_0$ be an accumulation point of the parameter ray of $\mathcal{S}$ at angle $\theta$. Standard arguments from polynomial dynamics (for instance, see \cite[Theorem~37.35]{L6}) can be used to show that $a_0$ is a  critically pre-periodic parameter such that in the dynamical plane of $\sigma_{a_0}$, the dynamical ray at angle $\theta$ lands at the critical value $2$. We include the details for completeness.

By Proposition~\ref{per_rays_land}, the dynamical ray of $\sigma_{a_0}$ at angle $\theta$ lands at some repelling or parabolic pre-periodic point $w$ (as $\theta$ is strictly pre-periodic under $\rho$, the landing point cannot be periodic). Let us suppose that $\sigma_{a_0}$ is a parabolic map. Note that as the landing point of the dynamical $\theta$-ray of $\sigma_{a_0}$ is not periodic, the ray does not land on the parabolic periodic point on the boundary of the Fatou component of $\sigma_{a_0}$ containing the critical value $2$. Since $\partial K_{a_0}$ is locally connected and repelling periodic points are dense on the limit set (these properties hold true for the straightened map $R_{\chi(a_0)}$, and are preserved under hybrid equivalences), it follows that there exists a cut-line through repelling periodic points on $\partial K_{a_0}$ separating the $\theta$-dynamical ray from the critical value $2$. But such cut-lines remain stable under small perturbation. Therefore, for parameters sufficiently close to $a_0$, the $\theta$-dynamical ray stays away from the critical value $2$. However, this is impossible as there are parameters near $a_0$ on the $\theta$-parameter ray for which the critical value $2$ lies on the $\theta$-dynamical ray. This contradiction shows that $\sigma_{a_0}$ is not parabolic; i.e., the dynamical ray of $\sigma_{a_0}$ at angle $\theta$ lands at some repelling pre-periodic point $w$.

We suppose that $w$ is not the critical value $2$ of $\sigma_{a_0}$, and will arrive at a contradiction. If $w$ is not a pre-critical point either, then for nearby parameters, the $\theta$-dynamical ray would land at the real-analytic continuation of the repelling pre-periodic point $w$, and would stay away from the critical value $2$. But there are parameters near $a_0$ on the $\theta$-parameter ray for which the critical value $2$ lies on the $\theta$-dynamical ray, a contradiction. Hence $w$ must be a pre-critical point implying that the critical value $2$ of $\sigma_{a_0}$ is strictly pre-periodic. So $a_0$ is a critically pre-periodic parameter. This implies that $\partial K_{a_0}$ is a dendrite and repelling periodic points are dense on it (once again, these properties hold true for the straightened map $R_{\chi(a_0)}$, and are preserved under hybrid equivalences). Hence, the dynamical ray at angle $\theta$ landing at $w$ can be separated from the critical value $2$ by a pair of dynamical rays landing at a common repelling periodic point. Once again, this separation line remains stable under perturbation, contradicting the existence of parameters near $a_0$ on the $\theta$-parameter ray. Hence, $w$ must be the critical value $2$ of $\sigma_{a_0}$.

We claim that $a_0$ is the unique parameter in $\cC(\mathcal{S})$ with the property that the dynamical ray at angle $\theta$ lands at the critical value $2$. Since the limit set of a ray is connected, this will prove that the parameter ray at angle $\theta$ indeed lands at $a_0$.

To prove the claim, let us assume that there exists another parameter $a_1$ with the same property. Clearly, for each of the maps $R_{\chi(a_0)}$ and $R_{\chi(a_1)}$, the critical value $-1$ is strictly pre-periodic. Moreover, by Proposition~\ref{hybrid_preserves_symbol}, the $\mathcal{E}(\theta)$-dynamical rays of $R_{\chi(a_0)}$ and $R_{\chi(a_1)}$ land at the corresponding critical values $R_{\chi(a_0)}(-1)$ and $R_{\chi(a_1)}(-1)$. It now follows that the external parameter ray of the parabolic Tricorn at angle $\mathcal{E}(\theta)$ lands both at $\chi(a_0)$ and $\chi(a_1)$ implying that $\chi(a_0)=\chi(a_1)$. Since $\chi$ is injective (Proposition~\ref{chi_injective_prop}), we conclude that $a_0=a_1$. This completes the proof.

2) Let $\mathcal{A}\subset(1/3,2/3)$ be the set of angles of the dynamical rays of $\sigma_{a_0}$ that land at the critical value $2$. By the first part of this proposition, the angles of the parameter rays (at pre-periodic angles) of $\mathcal{S}$ landing at $a_0$ are contained in $\mathcal{A}$.

Now pick $\theta\in\mathcal{A}$, and let $a_1$ be the landing point of the parameter ray of $\mathcal{S}$ at angle $\theta$. Then, the dynamical ray of $\sigma_{a_1}$ at angle $\theta$ lands at the critical value $2$. By the proof of the first part, we know that there can be at most one parameter in $\cC(\mathcal{S})$ whose dynamical $\theta$-ray lands at the critical value $2$. Therefore, $a_0=a_1$, i.e., the parameter ray of $\mathcal{S}$ at angle $\theta$ lands at $a_0$. As $\theta$ was an arbitrary element of $\mathcal{A}$, it follows that all parameter rays at angles in $\mathcal{A}$ land at $a_0$. The proof is now complete.
\end{proof}

\begin{lemma}[Counting Dyadic Tips]\label{dyadic_tips_schwarz}
The number of dyadic tips of $\cC(\mathcal{S})$ of pre-period $k$ is equal to the number of dyadic tips of $\cC(\mathfrak{L}_0)$ of pre-period $k$.
\end{lemma}
\begin{proof}
Note that by Proposition~\ref{per_rays_land}, in the dynamical plane of every dyadic tip $a_0$ of pre-periodic $k\geq1$, there exists a unique dynamical ray at angle $\theta$ landing at the critical value $2$ such that $\rho^{\circ k}(\theta)=\frac13$. It now follows from Lemma~\ref{misi_rays} that the number of dyadic tips of $\cC(\mathcal{S})$ of pre-period $k\geq1$ is equal to 
$$
\#\lbrace\theta\in(\frac13,\frac23):\rho^{\circ k}(\theta)=\frac13,\ \rho^{\circ (k-1)}(\theta)\neq\frac13\rbrace.
$$

A completely analogous argument shows that the number of dyadic tips of $\cC(\mathfrak{L}_0)$ of pre-period $k\geq1$ is equal to 
$$
\#\lbrace\theta\in\R/\Z:B^{\circ k}(\theta)=0,\ B^{\circ (k-1)}(\theta)\neq0\rbrace.
$$

Since $\rho:\partial\mathcal{Q}\to\partial\mathcal{Q}$ and $B:\R/\Z\to\R/\Z$ are topologically conjugate, the cardinalities of the above two sets are equal. The conclusion follows.
\end{proof}

\begin{proposition}[Onto Dyadic Tips of $\mathfrak{L}_0$]\label{onto_dyadic}
$\chi(\cC(\mathcal{S}))$ contains all dyadic tips of $\mathfrak{L}_0$.
\end{proposition}
\begin{proof}
Let us fix $k\geq1$. Evidently, $\chi$ maps a dyadic tip of pre-period $k$ of $\cC(\mathcal{S})$ to a dyadic tip of the same pre-period of $\cC(\mathfrak{L}_0)$. By Proposition~\ref{chi_injective_prop} and Lemma~\ref{dyadic_tips_schwarz}, $\chi$ is a bijection between the dyadic tips of pre-period $k$ of $\cC(\mathcal{S})$ and the dyadic tips of the same pre-period of $\cC(\mathfrak{L}_0)$. This completes the proof.
\end{proof}

The importance of the previous lemma stems from the fact that the closure of the dyadic tips of $\cC(\mathfrak{L}_0)$ contains ``most" parameters on $\partial\cC(\mathfrak{L}_0)$. This will allow us to show that the image of $\chi$ is sufficiently large.

\begin{proposition}\label{onto_para}
$\chi(\cC(\mathcal{S}))$ contains all parabolic parameters of $\cC(\mathfrak{L}_0)$.
\end{proposition}
\begin{proof}
Let us first assume that $(\alpha,A)\in\cC(\mathfrak{L}_0)$ is a parabolic parameter of even period. It follows by a straightforward parabolic perturbation argument that $(\alpha,A)$ lies in the closure of the dyadic tips of $\cC(\mathfrak{L}_0)$. Indeed, the iterated pre-images of $\infty$ are dense on the Julia set of $R_{\alpha,A}$, and these can be followed continuously when the parameter $(\alpha,A)$ is slightly perturbed. Using \cite[Proposition~2.2]{Lei1}, one can slightly perturb $(\alpha,A)$ to ensure that for such a perturbed map, the parabolic cycle splits into repelling cycles, and the critical point $-1$ eventually escapes through the `gates' formed by these repelling periodic points and lands on an iterated pre-image of $\infty$. Clearly, this produces dyadic tips of $\cC(\mathfrak{L}_0)$ arbitrarily close to $(\alpha,A)$. Proposition~\ref{onto_dyadic} and Lemma~\ref{image_closed} now imply that $(\alpha,A)\in\chi(\cC(\mathcal{S}))$.

Now let $(\alpha,A)\in\cC(\mathfrak{L}_0)$ be a simple parabolic parameter of odd period. Then, $(\alpha,A)$ lies on some parabolic arc $\cC$. Using the parabolic perturbation arguments of \cite[Theorem~7.3]{HS} and \cite[Proposition~2.2]{Lei1}, one can show that some parameter $(\alpha',A')\in\cC$ (lying on the non-bifurcating sub-arc of $\cC$) belongs to the closure of dyadic tips. Once again, it follows by Proposition~\ref{onto_dyadic} and Lemma~\ref{image_closed} that $(\alpha',A')\in\chi(\cC(\mathcal{S}))$. Finally, since $R_{\alpha,A}$ and $R_{\alpha',A'}$ are quasiconformally conjugate, Corollary~\ref{qc_closed} implies that $(\alpha,A)\in\chi(\cC(\mathcal{S}))$.

Since parabolic cusps lie on the boundary of parabolic arcs, and every parabolic arc of $\cC(\mathfrak{L}_0)$ is contained in $\chi(\cC(\mathcal{S}))$, it follows from Lemma~\ref{image_closed} that the parabolic cusps of $\cC(\mathfrak{L}_0)$ are also contained in the image of $\chi$.
\end{proof}

\begin{proposition}[Onto Hyperbolic Components]\label{onto_hyp}
$\chi(\cC(\mathcal{S}))$ contains every hyperbolic component of $\cC(\mathfrak{L}_0)$.
\end{proposition}
\begin{proof}
Let us first consider a hyperbolic component $H$ of odd period $k$ of $\cC(\mathfrak{L}_0)$. Then there is a parabolic cusp $(\alpha,A)$ of period $k$ on $\partial H$. By Proposition~\ref{onto_para}, there exists $a\in\cC(\mathcal{S})$ with $\chi(a)=(\alpha,A)$. Evidently, $a$ is a parabolic cusp of period $k$ of $\cC(\mathcal{S})$. By Proposition~\ref{ThmIndiffBdyHyp_schwarz}, $a$ lies on the boundary of a hyperbolic component $H_1$ of period $k$ of $\cC(\mathcal{S})$. Choose a sequence $\{a_n\}_n\in H_1$ converging to $a$. Since $\chi(a)$ is quasiconformally rigid in $\cC(\mathfrak{L}_0)$, it follows by Proposition~\ref{continuity_rigid} that $\{\chi(a_n)\}_n$ converges to $\chi(a)=(\alpha,A)$. But $\{\chi(a_n)\}_n$ is contained in the hyperbolic component $\chi(H_1)$ of period $k$, and hence, $(\alpha,A)\in\partial\chi(H_1)$. However, a parabolic cusp lies on the boundary of a unique $k$-periodic hyperbolic component. Hence we must have $\chi(H_1)=H$. 

We now consider a hyperbolic component $H$ of period $2k$ for some odd integer $k$ (of $\cC(\mathfrak{L}_0)$) such that $H$ bifurcates from a hyperbolic component of period $k$. By Propositions~\ref{ThmBifArc_para} and~\ref{index_increasing_1}, there is a parabolic cusp $(\alpha,A)$ of period $k$ on $\partial H$. By Proposition~\ref{onto_para}, there exists a parabolic cusp $a\in\cC(\mathcal{S})$ of period $k$ with $\chi(a)=(\alpha,A)$. By Proposition~\ref{ThmIndiffBdyHyp_schwarz}, $a$ lies on the boundary of a hyperbolic component $H_1$ of period $k$ of $\cC(\mathcal{S})$, and by Proposition~\ref{ThmBifArc_schwarz}, there is a hyperbolic component $H_1'$ of period $2k$ of $\cC(\mathcal{S})$ bifurcating from $H_1$ across $a$. An argument similar to the one used in the previous case now shows that $\chi(H_1')=H$. 

Finally, let $H$ be a hyperbolic component of even period $k$ (of $\cC(\mathfrak{L}_0)$) not bifurcating from any hyperbolic component of odd period. In this case, $H$ has a unique root point $(\alpha,A)$, which is an even-periodic parabolic parameter where the multiplier map of $H$ takes the value $+1$. By Proposition~\ref{onto_para}, there exists $a\in\cC(\mathcal{S})$ with $\chi(a)=(\alpha,A)$. By Proposition~\ref{ThmIndiffBdyHyp_schwarz} and Proposition~\ref{ThmEvenBif_schwarz}, $a$ lies on the boundary of a hyperbolic component $H_1$ of period $k$ of $\cC(\mathcal{S})$. Since $(\alpha,A)$ is quasiconformally rigid in $\cC(\mathfrak{L}_0)$ and $(\alpha,A)$ lies on the boundary of a unique $k$-periodic hyperbolic component, it follows by the same line of arguments used in the previous cases that $\chi(H_1)=H$. 
\end{proof}

We summarize the above results in the following corollary.

\begin{corollary}\label{onto_hyperbolic_closure}
$\chi(\cC(\mathcal{S}))$ contains the closure of all hyperbolic parameters in the parabolic Tricorn $\cC(\mathfrak{L}_0)$.
\end{corollary}
\begin{proof}
This follows from Proposition~\ref{onto_hyp} and Lemma~\ref{image_closed}
\end{proof}

\begin{remark}
Conjecturally, hyperbolic parameters are dense in $\cC(\mathfrak{L}_0)$. If this conjecture were true, the straightening map $\chi$ would be a bijection from $\cC(\mathcal{S})$ onto $\cC(\mathfrak{L}_0)$.
\end{remark}

\begin{corollary}\label{onto_ray_impression}
For each $\theta\in\left(\frac13,\frac23\right)$, there exists some $a\in\cC(\mathcal{S})$ such that $\chi(a)$ lies in the impression of the parameter ray (of $\mathfrak{L}_0$) at angle $\theta$.
\end{corollary}
\begin{proof}
This follows from Lemma~\ref{image_closed} and the fact that the impression of every parameter ray intersects the closure of all dyadic tips of $\cC(\mathfrak{L}_0)$.
\end{proof}

\section{Homeomorphism between Model Spaces}\label{model_homeo}

The goal of this section is to show that the straightening map $\chi$ can be slightly modified so that it induces a homeomorphism between the ``abstract connectedness locus'' $\widetilde{\cC(\mathcal{S})}$ (defined below) and the abstract parabolic Tricorn $\widetilde{\cC(\mathfrak{L}_0)}$ (see Appendix~\ref{anti_rational_parabolic}).

\subsection{The Abstract Connectedness Locus $\widetilde{\cC(\mathcal{S})}$}\label{abstract_conn_locus_schwarz}

Let $a\in\cC(\mathcal{S})$. Recall that by Proposition~\ref{per_rays_land}, all dynamical rays of $\sigma_a$ at angles in $\mathrm{Per}(\rho)$ (i.e., at pre-periodic angles) land on $\partial K_a$.

\begin{definition}[Pre-periodic Laminations, and Combinatorial Classes]\label{schwarz_lami_comb_class_def}

\noindent\begin{enumerate}\upshape
\item For $a\in\cC(\mathcal{S})$, the \emph{pre-periodic lamination} of $\sigma_a$ is defined as the equivalence relation on $\mathrm{Per}(\rho)\subset\partial\mathcal{Q}\cong\faktor{\left[\frac13,\frac23\right]}{\{\frac13\sim\frac23\}}$ such that $\theta, \theta'\in\mathrm{Per}(\rho)$ are related if and only if the corresponding dynamical rays land at the same point of $\partial K_a$.

\item Two parameters $a$ and $a'$ in $\cC(\mathcal{S})$ are said to be \emph{combinatorially equivalent} if they have the same pre-periodic lamination.  

\item The combinatorial class $\mathrm{Comb}(a)$ of $a\in\cC(\mathcal{S})$ is defined as the set of all parameters in $\cC(\mathcal{S})$ that are combinatorially equivalent to $a$.

\item A combinatorial class $\mathrm{Comb}(a)$ is called \emph{periodically repelling} if for every $a'\in\mathrm{Comb}(a)$, each periodic orbit of $\sigma_a$ is repelling.
\end{enumerate}
\end{definition} 

\begin{proposition}\label{straightening_preserves_comb_class} 
For $a\in\cC(\mathcal{S})$, the homeomorphism $\mathcal{E}:\partial\mathcal{Q}\to\R/\Z$ maps the pre-periodic lamination of $\sigma_a$ onto the pre-periodic lamination of $R_{\chi(a)}$. As a consequence, two parameters $a$ and $a'$ in $\cC(\mathcal{S})$ are combinatorially equivalent if and only if $\chi(a),\chi(a')\in\cC(\mathfrak{L}_0)$ are so.
\end{proposition}
\begin{proof}
This follows directly from Proposition~\ref{hybrid_preserves_symbol}.
\end{proof}

We are now ready to give a complete classification of the non-repelling combinatorial classes of $\cC(\mathcal{S})$.

\begin{proposition}[Classification of Combinatorial Classes]\label{comb_class_schwarz}
Every combinatorial class $\mathrm{Comb}(a)$ of $\cC(\mathcal{S})$ is of one of the following four types.
\begin{enumerate}
\item $\mathrm{Comb}(a)$ consists of an even period hyperbolic component that does not bifurcate from an odd period hyperbolic component, its root point, and the irrationally neutral parameters on its boundary.

\item $\mathrm{Comb}(a)$ consists of an even period hyperbolic component that bifurcates from an odd period hyperbolic component, the unique parabolic cusp and the irrationally neutral parameters on its boundary.

\item $\mathrm{Comb}(a)$ consists of an odd period hyperbolic component and the parabolic arcs on its boundary.

\item $\mathrm{Comb}(a)$ is periodically repelling.
\end{enumerate}
\end{proposition}
\begin{proof}
Let us assume that $\mathrm{Comb}(a)$ is not periodically repelling. Then there exists $\widetilde{a}\in\mathrm{Comb}(a)$ such that $\sigma_{\widetilde{a}}$ has a non-repelling periodic orbit. But then, $R_{\chi(\widetilde{a})}$ has a non-repelling periodic orbit. Therefore, the combinatorial class $\mathrm{Comb}(\chi(\widetilde{a}))=\mathrm{Comb}(\chi(a))$ is not periodically repelling, and hence is of one of the first three types described in Proposition \ref{comb_class_para}.

It now follows by Corollary~\ref{onto_hyperbolic_closure} that $\mathrm{Comb}(\chi(a))$ is contained in the image of $\chi$. By Proposition~\ref{straightening_preserves_comb_class}, $\chi^{-1}(\mathrm{Comb}(\chi(a)))$ is a single combinatorial class containing $a$; i.e., $\chi^{-1}(\mathrm{Comb}(\chi(a)))=\mathrm{Comb}(a)$. Since $\chi$ is injective (Proposition~\ref{chi_injective_prop}), it follows that $\chi$ is a bijection from $\mathrm{Comb}(a)$ onto $\mathrm{Comb}(\chi(a))$. Finally, since $\chi$ preserves hyperbolic, parabolic and irrationally neutral parameters, the result follows.
\end{proof}

\begin{remark}
We conjecture that every periodically repelling combinatorial class of $\cC(\mathcal{S})$ is a point. In light of Corollary~\ref{straightening_preserves_comb_class}, this will follow if the corresponding conjecture for $\cC(\mathfrak{L}_0)$ holds true.
\end{remark}

We are now in a position to define an abstract topological model for $\cC(\mathcal{S})$. We put an equivalence relation $\sim$ on $\mathbb{S}^2$ by
\begin{enumerate}
\item identifying all points in the closure of each periodically repelling combinatorial class of $\cC(\mathcal{S})$,

\item identifying all points in the closure of the non-bifurcating sub-arc of each parabolic arc of $\cC(\mathcal{S})$, and

\item identifying all points in $\mathcal{I}$.
\end{enumerate}

This is a non-trivial closed equivalence relation on the sphere such that all equivalence classes are connected and non-separating. By Moore's theorem, the quotient of the $2$-sphere by $\sim$ is again a $2$-sphere. The image of $\cC(\mathcal{S})$ under this quotient map is non-compact, but adding the class of $\mathcal{I}$ turns it into a full compact subset of the $2$-sphere.

\begin{definition}[Abstract Connectedness Locus of $\mathcal{S}$]\label{abstract_model_schwarz}
The \emph{abstract connectedness locus} of the family $\mathcal{S}$ is defined as the union of the image of $\cC(\mathcal{S})$ under the quotient map (defined by the above equivalence relation) and the class of $\mathcal{I}$. It is denoted by $\widetilde{\cC(\mathcal{S})}$.
\end{definition}

\begin{remark}
It is instructive to mention that the identifications above are designed to ``tame'' the straightening map $\chi$; i.e., to make it surjective and continuous.
\end{remark}

\subsection{Constructing The Homeomorphism}\label{homeo_construct}

In this subsection, we put together all the ingredients developed so far to construct a homeomorphism between the abstract connectedness loci of the families $\mathcal{S}$ and $\mathfrak{L}_0$.

\begin{proof}[Proof of Theorem~\ref{abstract_homeo}]
We will first define a map $\mathfrak{X}$ on $\cC(\mathcal{S})$. We begin by setting it equal to $\chi$ on all of $\cC(\mathcal{S})$ except on the closures of odd period hyperbolic components. 

Now let $H$ be a hyperbolic component of odd period $k$, $\cC$ be a parabolic arc on $\partial H$, and $H'$ be a hyperbolic component of period $2k$ bifurcating from $H$ across $\cC$. By Proposition~\ref{cont_ext}, $\chi\vert_{H'}$ has a continuous extension $\xi$ to $\overline{\cC}\cap\partial H'$. This allows us to extend $\mathfrak{X}$ to the bifurcating arcs of $\partial H$ so that it maps homeomorphically onto the bifurcating arcs of $\partial\chi(H)$. We now extend this map to the rest of $\overline{H}$ so that $\overline{H}$ maps homeomorphically onto $\overline{\chi(H)}$. 

This defines a map $\mathfrak{X}$ (possibly discontinuous on the non-bifurcating sub-arcs of parabolic arcs and on the queer Beltrami arcs) on $\cC(\mathcal{S})$. Note that $\mathfrak{X}$ agrees with $\chi$ on every periodically repelling combinatorial class. By Corollary~\ref{straightening_preserves_comb_class}, $\mathfrak{X}$ maps every periodically repelling combinatorial class of $\cC(\mathcal{S})$ to a periodically repelling combinatorial class of $\cC(\mathfrak{L}_0)$ (not necessarily surjectively). By construction, $\mathfrak{X}$ maps the non-bifurcating sub-arc of every parabolic arc of $\cC(\mathcal{S})$ onto the non-bifurcating sub-arc of the corresponding parabolic arc of $\cC(\mathfrak{L}_0)$. Therefore, $\mathfrak{X}$ descends to a map $\widetilde{\mathfrak{X}}:\widetilde{\cC(\mathcal{S})}\to\widetilde{\cC(\mathfrak{L}_0)}$ which sends the class of $\mathcal{I}=\overline{\cC(\mathcal{S})}\setminus\cC(\mathcal{S})$ to the class of $\overline{\cC(\mathfrak{L}_0)}\setminus\cC(\mathfrak{L}_0)$ (i.e., the class containing maps $R_{\alpha,A}$ with a multiple parabolic fixed point at $\infty$). Moreover, it follows by Propositions~\ref{chi_injective_prop},~\ref{comb_class_schwarz}, and Corollary~\ref{straightening_preserves_comb_class} that $\widetilde{\mathfrak{X}}$ is injective.

By Proposition~\ref{chi_proper}, the map $\widetilde{\mathfrak{X}}$ is continuous at the class of $\mathcal{I}$. By Proposition~\ref{continuity_rigid}, Lemma~\ref{cont_queer} and the above discussion, $\mathfrak{X}$ is continuous everywhere except possibly on the non-bifurcating sub-arcs of parabolic arcs and the queer Beltrami arcs. Since each queer Beltrami arc is contained in a periodically repelling combinatorial class (which are pinched to points in both $\widetilde{\cC(\mathcal{S})}$ and $\widetilde{\cC(\mathfrak{L}_0)}$), Lemma~\ref{beltrami_arc_discont} implies that possible discontinuities of $\mathfrak{X}$ on queer Beltrami arcs do not create discontinuity for the induced map $\widetilde{\mathfrak{X}}$. 

Therefore, to verify continuity of $\widetilde{\mathfrak{X}}$ on $\widetilde{\cC(\mathcal{S})}$, it suffices to look at the classes of the non-bifurcating sub-arcs of the parabolic arcs. To this end, let us fix a parabolic arc $\cC$ on the boundary of an odd period hyperbolic component $H$. Recall that discontinuity of $\mathfrak{X}$ on $\cC$ can only occur from the exterior of the union of $H$ and the hyperbolic components bifurcating from it. Now let $\lbrace a_n\rbrace_n$ be a sequence outside the union of $H$ and the hyperbolic components bifurcating from it such that $\lbrace a_n\rbrace_n$ converges to some point $a\in\cC$. It follows that $\{\mathfrak{X}(a_n)\}_n$ lies outside the union of $\mathfrak{X}(H)$ and the hyperbolic components bifurcating from it. Furthermore, any subsequential limit of $\{\mathfrak{X}(a_n)\}_n$ must lie on $\mathfrak{X}(\cC)$ (by Lemma~\ref{arc_to_arc}). So, any subsequential limit of $\{\mathfrak{X}(a_n)\}_n$ must lie on the non-bifurcating sub-arc of $\mathfrak{X}(\cC)$. Since the non-bifurcating sub-arc of every parabolic arc is collapsed to a point in $\widetilde{\cC(\mathcal{S})}$ and $\widetilde{\cC(\mathfrak{L}_0)}$, it now follows that $\widetilde{\mathfrak{X}}$ is continuous at the class of the non-bifurcating sub-arc of $\cC$. Thus, $\widetilde{\mathfrak{X}}$ is continuous everywhere on $\widetilde{\cC(\mathcal{S})}$.

Every periodically repelling combinatorial class of $\cC(\mathfrak{L}_0)$ contains at least one parameter ray impression. Thus, by Corollary~\ref{onto_ray_impression}, the image of $\mathfrak{X}$ hits every periodically repelling combinatorial class of $\cC(\mathfrak{L}_0)$. This, along with Corollary~\ref{onto_hyperbolic_closure}, implies that the image of $\widetilde{\mathfrak{X}}$ contains the boundary $\partial\widetilde{\cC(\mathfrak{L}_0)}$.

As $\widetilde{\cC(\mathcal{S})}$ is compact, it now follows that $\widetilde{\mathfrak{X}}$ is a homeomorphism from $\widetilde{\cC(\mathcal{S})}$ onto its image. Moreover, $\widetilde{\cC(\mathcal{S})}$ is a full subset of the sphere, and hence has a trivial Alexander-Kolmogorov cohomology. But this property is preserved by homeomorphisms. Thus, the image $\widetilde{\mathfrak{X}}(\widetilde{\cC(\mathcal{S})})$ also has a trivial Alexander-Kolmogorov cohomology, and hence is a full subset of the sphere. It follows that $\widetilde{\mathfrak{X}}(\widetilde{\cC(\mathcal{S})})$ also contains the interior of $\widetilde{\cC(\mathfrak{L}_0)}$.

This shows that the map $\widetilde{\mathfrak{X}}:\widetilde{\cC(\mathcal{S})}\to\widetilde{\cC(\mathfrak{L}_0)}$ is a homeomorphism.
\end{proof}

\section{The Associated Correspondences}\label{sec_corr}
In this section, we will define a correspondence on $\widehat{\C}$ by lifting $\sigma_a^{\pm 1}$ under $f_a$. We will then show that in a suitable sense, the correspondence is a ``mating" of the abstract modular group $\Z/2\Z\ast\Z/3\Z$ and a quadratic anti-rational map. 

\subsection{The Correspondence $\widetilde{\sigma_a}^*$}\label{corr_def}
We now begin with the construction of the correspondence $\widetilde{\sigma_a}\subset\widehat{\C}\times\widehat{\C}$.

Let us first consider $z\in\overline{\D}$. For such $z$, we have $\sigma_a(f_a(z))=f_a(\iota(z))$, where $\iota$ is the reflection in the unit circle. For $z\in\overline{\D}$, we say that 
$$
(z,w)\in\widetilde{\sigma_a} \iff f_a(w)=\sigma_a(f_a(z))=f_a(\iota(z)).
$$ 
Thus, the lifts of $\sigma_a$ under $f_a$ define a correspondence $\widetilde{\sigma_a}$ on $\overline{\D}\times\widehat{\C}$. 

We now turn our attention to $z\in\widehat{\C}\setminus\overline{\D}$. For such $z$, we have that $\sigma_a(f_a(\iota(z)))=f_a(z)$. Choosing a suitable branch of $\sigma_a^{-1}$, we can rewrite the previous relation as $f_a(\iota(z))=\sigma_a^{-1}(f_a(z))$. For $z\in\widehat{\C}\setminus\overline{\D}$, we say that 
$$
(z,w)\in\widetilde{\sigma_a} \iff f_a(w)=\sigma_a^{-1}(f_a(z))=f_a(\iota(z)).
$$ 
Thus, the lifts of (suitable inverse branches of) $\sigma_a^{-1}$ under $f_a$ define a correspondence $\widetilde{\sigma_a}$ on $\left(\widehat{\C}\setminus\overline{\D}\right)\times\widehat{\C}$.

Combining these two definitions of $\widetilde{\sigma_a}$ in the interior and exterior disk, we obtain the $3:3$ correspondence $\widetilde{\sigma_a}\subset\widehat{\C}\times\widehat{\C}$ defined as 
\begin{equation}
(z,w)\in\widetilde{\sigma_a} \iff f_a(w)-f_a(\iota(z))=0.
\label{corr_eqn_1}
\end{equation}

\begin{proposition}\label{inverse_lift_corr}
The correspondence $\widetilde{\sigma_a}$ defined by Equation~(\ref{corr_eqn_1}) contains all possible lifts of $\sigma_a$ (respectively, suitable inverse branches of $\sigma_a^{-1}$) when $z\in\overline{\D}$ (respectively, when $z\in\widehat{\C}\setminus\overline{\D}$) under $f_a$. More precisely,
\begin{itemize}
\item for $z\in\overline{\D}$, we have that $(z,w)\in\widetilde{\sigma_a} \iff f_a(w)=\sigma_a(f_a(z))$, and

\item for $z\in\widehat{\C}\setminus\overline{\D}$, we have that $(z,w)\in\widetilde{\sigma_a}\ \implies\ \sigma_a(f_a(w))=f_a(z)$.
\end{itemize}
\end{proposition}

Note that for all $z\in\widehat{\C}$, we have $(z,\iota(z))\in\widetilde{\sigma_a}$. Removing all pairs $(z,\iota(z))$ from the correspondence $\widetilde{\sigma_a}$, we obtain a $2:2$ correspondence $\widetilde{\sigma_a}^*\subset\widehat{\C}\times\widehat{\C}$ defined as 
\begin{equation}
(z,w)\in\widetilde{\sigma_a}^*\iff \frac{f_a(w)-f_a(\iota(z))}{w-\iota(z)}=0;
\label{corr_eqn_2}
\end{equation}
\begin{equation}
\mathrm{i.e.,}\ (z,w)\in\widetilde{\sigma_a}^*\iff  \mathrm{i.e.}\ (\iota_a(z'))^2+(\iota_a(z'))w'+w'^2=3,
\label{corr_eqn_3}
\end{equation}
where $z'=a+(1-a)z$, $w'=a+(1-a)w$, and $\iota_a$ is the reflection with respect to the circle $\partial B(a,\vert 1-a\vert)$.

\begin{remark}
The correspondence $\widetilde{\sigma_a}^*$ defined above is closely related to a holomorphic correspondence defined by Bullett and Penrose \cite{BP}. Indeed, the correspondence $\widetilde{\sigma_a}^*$ is a \emph{map of triples} in the sense of the diagram in Figure~\ref{triples_pic}, and has the form described in \cite[Lemma~2]{BP}. 
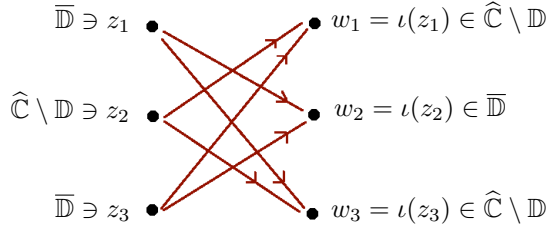
\begin{figure}[ht!]
\begin{tikzpicture}
  \node[anchor=south west,inner sep=0] at (0,0) {\includegraphics[width=0.2\textwidth]{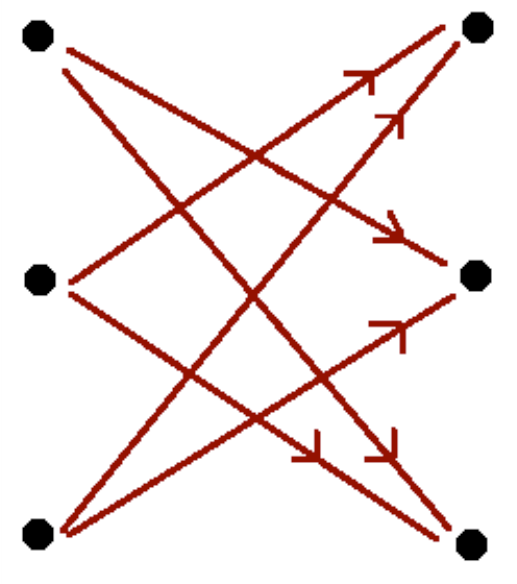}};
  \node at (-0.6,2.8) {$\overline{\D}\ni z_1$};
   \node at (4,2.8) {$w_1=\iota(z_1)\in\widehat{\C}\setminus\D$};
 \node at (-0.9,1.6) {$\widehat{\C}\setminus\D\ni z_2$};
   \node at (3.75,1.6) {$w_2=\iota(z_2)\in\overline{\D}$};
 \node at (-0.6,0.28) {$\overline{\D}\ni z_3$};
   \node at (4,0.28) {$w_3=\iota(z_3)\in\widehat{\C}\setminus\D$};
\end{tikzpicture}
\caption{If $f_a^{-1}(w)=\{w_1, w_2, w_3\}$ for some $w\in\overline{\Omega}_a$, and $z_i=\iota(w_i)$ for $i=1,2,3$, then the forward correspondence sends the triple $\left(z_1, z_2, z_3\right)$ to the triple $\left(w_1, w_2, w_3\right)$ as shown in the figure.}
\label{triples_pic}
\end{figure}
However, the key difference between their correspondence and ours is that the involution $\iota$ in our situation is anti-holomorphic (naturally arising from Schwarz reflections), while the involution in their setting is holomorphic.
\end{remark}

\subsection{Dynamics of $\widetilde{\sigma_a}^*$}\label{group_subsec}

The rest of this section will be devoted to giving a dynamical description of the correspondence $\widetilde{\sigma_a}^*$ for $a\in\cC(\mathcal{S})$. More precisely, we will partition the Riemann sphere into two $\widetilde{\sigma_a}^*$-invariant subsets such that on one of these sets, the dynamics of $\widetilde{\sigma_a}^*$ is equivalent to the action of the group $\Z/2\Z\ast\Z/3\Z$, and on the other, suitable branches of the correspondence are conjugate to a quadratic anti-rational map.

The coexistence of group structure and anti-rational map feature in the correspondence $\widetilde{\sigma_a}^*$ can be roughly explained as follows. Locally, the correspondence $\widetilde{\sigma_a}^*$ splits into two maps that are given by compositions of the reflection map $\iota$ with suitably chosen ``deck maps" of $f_a$ (compare Equation~\ref{corr_eqn_2}). In a punctured neighborhood of $\infty$, the map $f_a$ is a $3:1$ (unbranched) Galois cover with deck transformation group isomorphic to $\Z/3\Z$. This allows us to define a deck transformation of $f_a$ in a neighborhood of $\infty$ which permutes the three points in every fiber of $f_a$ in a fixed point free manner (and fixes $\infty$). The grand orbits of $\widetilde{\sigma_a}^*$ in a neighborhood of $\infty$ are then generated by the action of this order three deck transformation and the involution $\iota$ resulting in the desired group structure. On the other hand, the finite critical points of $f_a$ are obstructions to extending such a deck transformation to the entire plane. To define analogues of deck transformations of $f_a$ on a domain that contains a finite critical point (of $f_a$), we use a specific covering property of $f_a$. Recall that $f_a^{-1}(\Omega_a)=\D\sqcup V_a$ (where $V_a$ is a simply connected domain) such that each point $w$ in $\Omega_a$, except the critical value $2$, has two pre-images in $V_a$ and a unique pre-image in $\D$ (see Figure~\ref{corr_tiling_pic}). One can now define a deck transformation of $f_a$ on $V_a$ that permutes the two pre-images of $w$ in $V_a$, and another ``ramified deck map" that sends both the pre-images of $w$ in $V_a$ to the \emph{unique} pre-image (of $w$) in $\D$. The rational map feature of the correspondence comes directly from the existence of this ramified deck map on $V_a$.

\subsubsection{Dynamical Partition} 

Let us set 
$$
\widetilde{T_a^\infty}:=f_a^{-1}(T_a^\infty),\ \mathrm{and}\ \widetilde{K_a}:=f_a^{-1}(K_a).
$$  

We define tiles of rank $n$ in $\widetilde{T_a^\infty}$ as $f_a$-pre-images of tiles of rank $n$ in $T_a^\infty$. There is a unique rank $0$ tile in $\widetilde{T_a^\infty}$ (which maps as a three-to-one branched cover onto $T_a^0$, branched only at $\infty$). For $n\geq1$, every rank $n$ tile in $T_a^\infty$ lifts to three rank $n$ tiles in $\widetilde{T_a^\infty}$.

Since $K_a$ is connected, it follows that the critical value $2$ of $\sigma_a$ (which is also a critical value of $f_a$) lies in $K_a$. Hence, the $f_a$-pre-images of the finite critical values of $f_a$ lie in $\widetilde{K_a}$; i.e., $f_a^{-1}(\{2,-2\})\subset\widetilde{K_a}$. It is easy to see that $\widetilde{K_a}\cap\mathbb{S}^1=\{1\}$. Moreover, $\widetilde{K_a}\cap\overline{\D}$ is mapped univalently onto $K_a$, and $\widetilde{K_a}\setminus\D$ is mapped as a two-to-one ramified covering onto $K_a$ (ramified only at $\frac{a+1}{a-1}$) by $f_a$. Furthermore, $f_a$ maps $\widetilde{T_a^\infty}$ as a three-to-one branched cover (branched only at $\infty$) onto $T_a^\infty$. It follows that $\widetilde{K_a}$ is a connected, full, compact subset of the plane, and $\widetilde{T_a^\infty}$ is a simply connected domain (see Figure~\ref{schwarz_lift_pic}).

\begin{figure}[ht!]
\begin{center}
\includegraphics[scale=0.28]{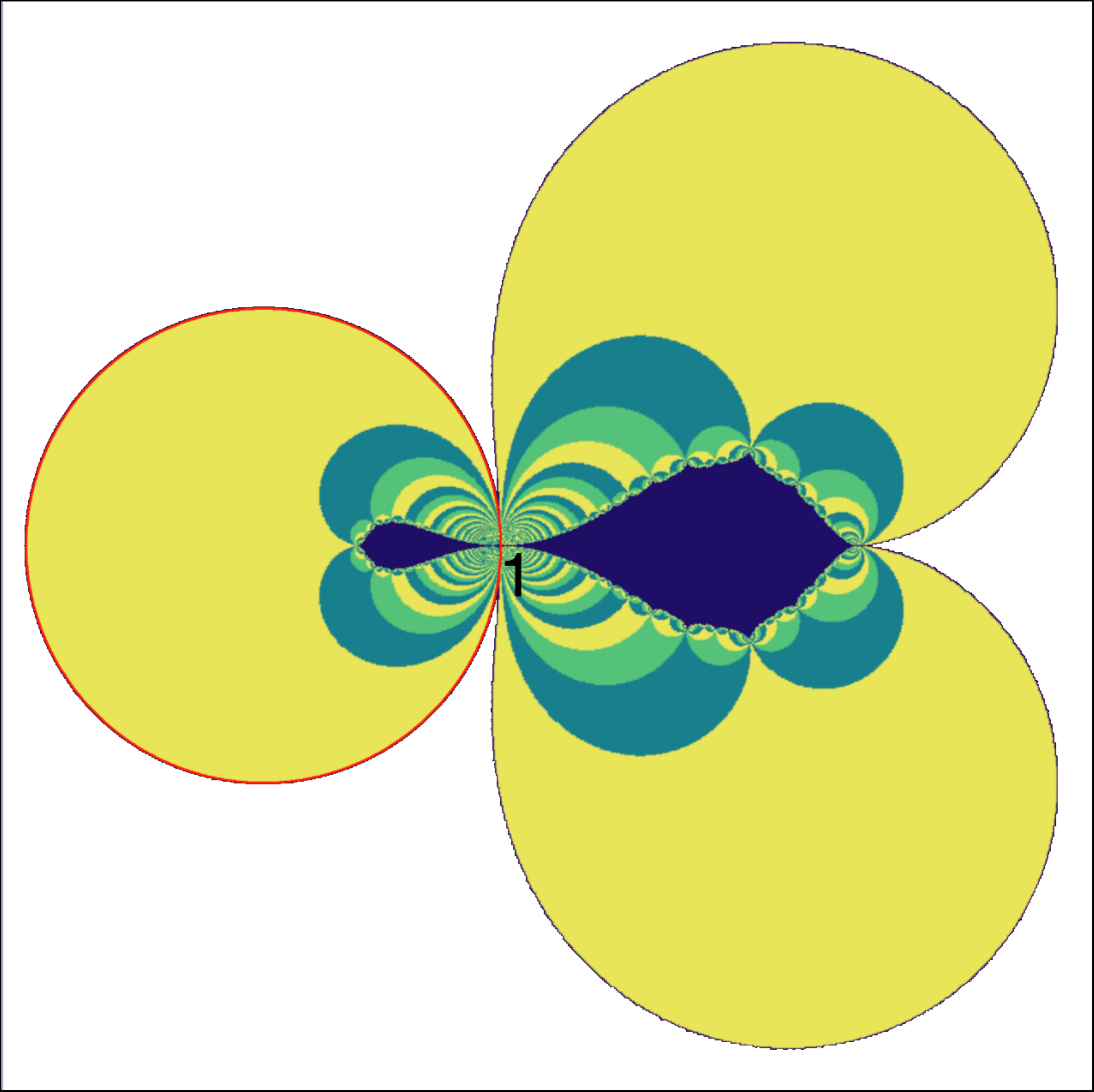}
\caption{The unit circle is marked in red. The dark blue region is $\widetilde{K_a}$, which contains $f_a^{-1}(\{2,-2\})$. Its complement is $\widetilde{T_a^\infty}$, which is a simply connected domain. $\widetilde{T_a^\infty}$ is mapped by $f_a$ as a three-to-one branched cover (branched only at $\infty$) onto $T_a^\infty$. The white region is the rank $0$ tile in $\widetilde{T_a^\infty}$. This is the only tile on which $f_a$ is ramified.}
\label{schwarz_lift_pic}
\end{center}
\end{figure}

\begin{proposition}\label{correspondence_partition}
1) Each of the sets $\widetilde{T_a^\infty}$ and $\widetilde{K_a}$ is completely invariant under the correspondence $\widetilde{\sigma_a}^*$. More precisely, if $(z,w)\in\widetilde{\sigma_a}^*$, then 
$$
z\in\widetilde{T_a^\infty}\iff w\in\widetilde{T_a^\infty},
$$ 
and 
$$z\in\widetilde{K_a}\iff w\in\widetilde{K_a}.
$$

2) $\iota(\widetilde{T_a^\infty})=\widetilde{T_a^\infty}$, and $\iota(\widetilde{K_a})=\widetilde{K_a}$.

\end{proposition}

\subsubsection{Group Structure of $\widetilde{\sigma_a}^*$ on $\widetilde{T_a^\infty}$}\label{group_subsec}

We will now analyze the structure of grand orbits of the correspondence $\widetilde{\sigma_a}^*$ in $\widetilde{T_a^\infty}$. To this end, we need to discuss the deck transformations of $f_a:\widetilde{T_a^\infty}\to T_a^\infty$.

\begin{lemma}\label{deck_group_lemma}
Let $a\in\cC(\mathcal{S})$. Then, $f_a: \widetilde{T_a^\infty}\setminus\{\infty\}\to T_a^\infty\setminus\{\infty\}$ is a regular three-to-one cover with deck transformation group isomorphic to $\Z/3\Z$.
\end{lemma}
\begin{proof}
Note that $\pi_1(\widetilde{T_a^\infty}\setminus\{\infty\})=\Z$, and 
$$
(f_a)_*(\pi_1(\widetilde{T_a^\infty}\setminus\{\infty\}))=3\Z\ \unlhd\ \Z=\pi_1(T_a^\infty\setminus\{\infty\}).
$$ 
This shows that the three-to-one cover $f_a: \widetilde{T_a^\infty}\setminus\{\infty\}\to T_a^\infty\setminus\{\infty\}$ is regular, and the associated deck transformation group is 
$$
\faktor{\pi_1(T_a^\infty\setminus\{\infty\})}{(f_a)_*(\pi_1(\widetilde{T_a^\infty}\setminus\{\infty\}))}\cong\Z/3\Z.
$$ 
\end{proof}

\begin{lemma}\label{deck_on_tiling_lift}
Let $a\in\cC(\mathcal{S})$. Then there exists a biholomorphism $\tau_a: \widetilde{T_a^\infty}\to\widetilde{T_a^\infty}$ such that the following hold true.
\begin{itemize}
\item $f_a\circ\tau_a=f_a$ on $\widetilde{T_a^\infty}$,

\item $\tau_a^{\circ 3}=\mathrm{id}$, and

\item $f_a^{-1}(f_a(z))=\{z,\tau_a(z),\tau_a^{\circ 2}(z)\}$, for $z\in \widetilde{T_a^\infty}$.
\end{itemize}
\end{lemma}
\begin{proof}
Let $\tau_a$ be a generator of the deck transformation group of $f_a: \widetilde{T_a^\infty}\setminus\{\infty\}\to T_a^\infty\setminus\{\infty\}$. Then, $\tau_a:\widetilde{T_a^\infty}\setminus\{\infty\}\to\widetilde{T_a^\infty}\setminus\{\infty\}$ is a biholomorphism such that $\tau_a(z)\to\infty$ as $z\to\infty$. Setting $\tau_a(\infty)=\infty$ yields a biholomorphism $\tau_a: \widetilde{T_a^\infty}\to\widetilde{T_a^\infty}$, and the required properties follow from the definition of $\tau_a$ and Lemma~\ref{deck_group_lemma}.
\end{proof}

Figure~\ref{corr_tiling_pic} shows the action of the correspondence $\widetilde{\sigma_a}^\ast$ on the lifted tiling set via the deck transformations $\tau_a, \tau_a^{\circ 2}$, and the reflection map $\iota$.

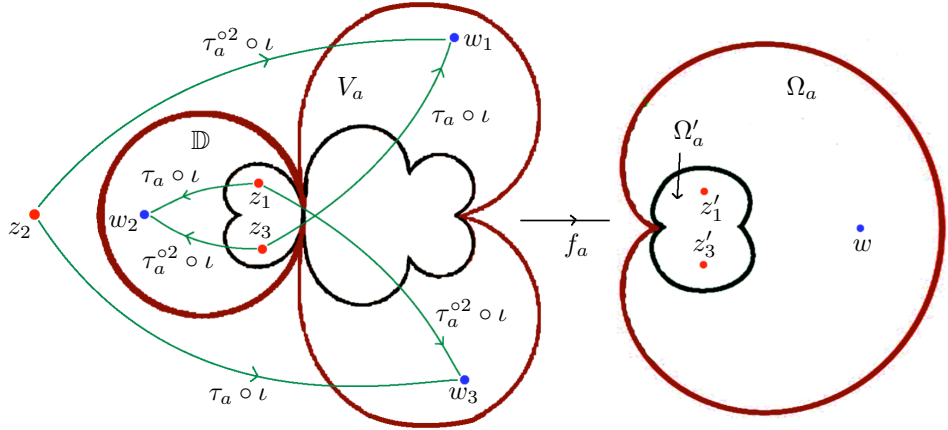
\begin{figure}[ht!]
\begin{tikzpicture}
  \node[anchor=south west,inner sep=0] at (0,0) {\includegraphics[width=1\textwidth]{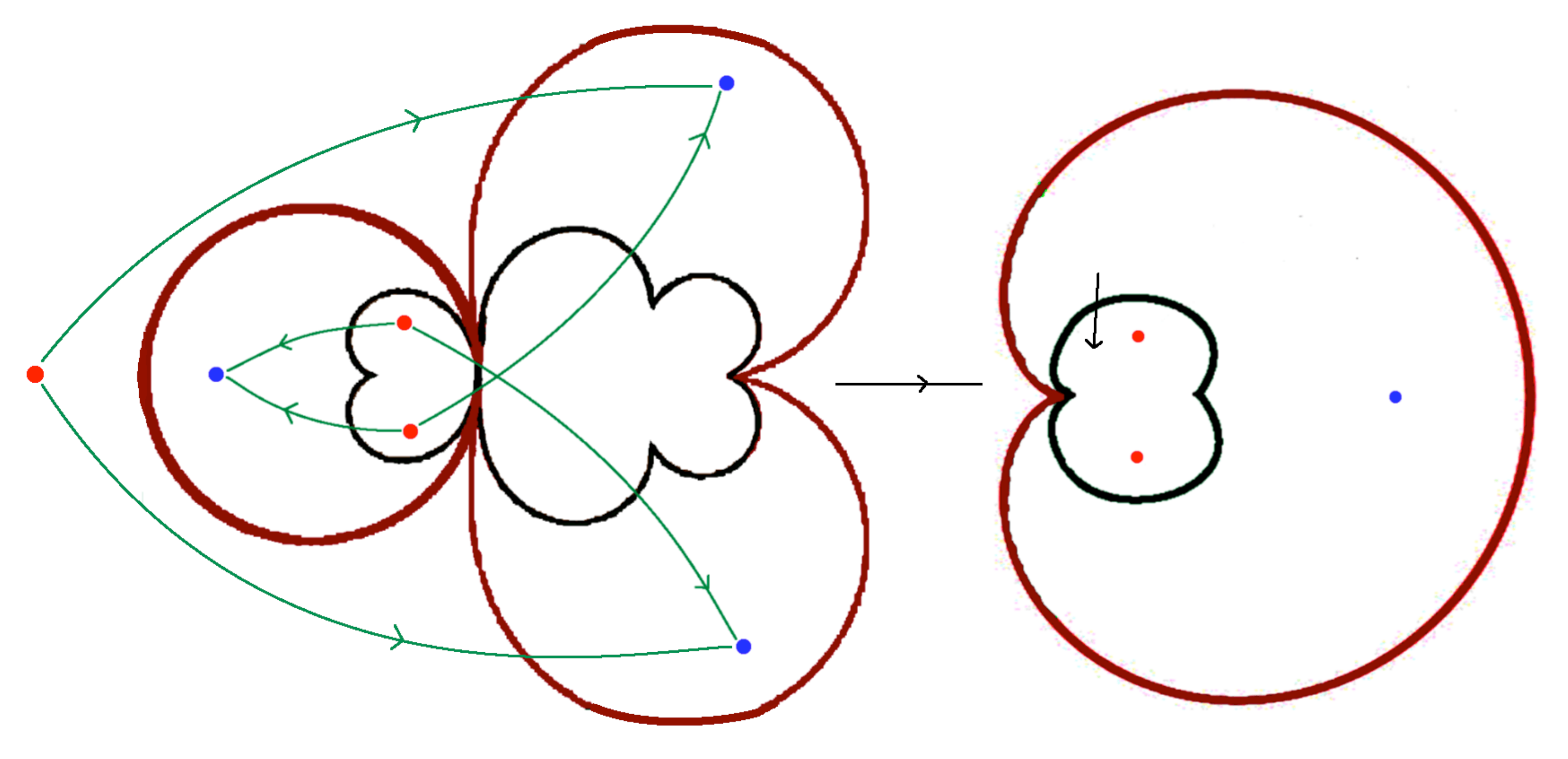}};
    \node at (0.1,2.9) {$z_2$};
   \node at (3.3,3.3) {$z_1$};
   \node at (3.3,2.9) {$z_3$};
    \node at (6.2,5.4) {$w_1$};
\node at (5.96,4.4) {$\tau_a\circ\iota$};
     \node at (1.48,3) {$w_2$};
      \node at (2.5,4.1) {$\D$};
      \node at (4.5,4.8) {$V_a$};
      \node at (2.1,3.6) {$\tau_a\circ\iota$};
      \node at (6,0.7) {$w_3$};
      \node at (6.1,1.8) {$\tau_a^{\circ 2}\circ\iota$};
      \node at (2.2,2.5) {$\tau_a^{\circ 2}\circ\iota$};
  \node at (11.3,2.7) {$w$};
   \node at (9.3,3.2) {$z_1'$};
    \node at (9.2,2.72) {$z_3'$};
  \node at (10.5,4.8) {$\Omega_a$};
 \node at (7.5,2.7) {$f_a$};
 \node at (9,4.2) {$\Omega_a'$};
  \node at (3,0.7) {$\tau_a\circ\iota$};
\node at (3,5.4) {$\tau_a^{\circ 2}\circ\iota$};
\end{tikzpicture}
\caption{The points $w_1, w_2$, and $w_3$ lie in the $f_a$-fiber of some $w\in\Omega_a\cap T_a^\infty$. One of them (namely, $w_2$) lies in $\D$, and the other two (namely, $w_1$ and $w_3$) lie in $V_a:=f_a^{-1}(\Omega_a)\setminus\D$. The points $z_i$ are the reflections of $w_i$ with respect to the unit circle (i.e., $w_i=\iota(z_i)$ for $i=1,2,3$), and $z_i'=f_a(z_i)$ for $i=1,3$. The Schwarz reflection map $\sigma_a$ sends $z_1'$ and $z_3'$ to $w$. Since the deck transformation $\tau_a$ is of order three, we can assume that $\tau_a$ sends $w_1, w_2, w_3$ to $w_2, w_3, w_1$ respectively. The forward correspondence $\widetilde{\sigma_a}^\ast$ splits into two maps $\tau_a\circ\iota$ and $\tau_a^{\circ 2}\circ\iota$ (on $\widetilde{T_a^\infty}$), and the actions of these maps on $z_1, z_2, z_3$ are shown with green arrows.}
\label{corr_tiling_pic}
\end{figure}

Since $\iota$ is an antiholomorphic involution preserving $\widetilde{T_a^\infty}$, it follows that both $\tau_a\circ\iota$ and $\tau_a^{\circ 2}\circ\iota$ are anti-conformal automorphisms of $\widetilde{T_a^\infty}$.

\begin{proposition}\label{grand_orbit_group}
For $a\in\cC(\mathcal{S})$, the grand orbits of the correspondence $\widetilde{\sigma_a}^*$ on $\widetilde{T_a^\infty}$ are equal to the orbits of $\langle\iota\rangle\ast\langle\tau_a\rangle$. Hence, the dynamics of $\widetilde{\sigma_a}^*$ on $\widetilde{T_a^\infty}$ is equivalent to the action of $\Z/2\Z\ast\Z/3\Z$.
\end{proposition}
\begin{proof}
It follows from the definition of the correspondence $\widetilde{\sigma_a}^*$ (see Equation~\ref{corr_eqn_2}) and Lemma~\ref{deck_on_tiling_lift} that the forward correspondence $\widetilde{\sigma_a}^\ast$ splits into two maps $\tau_a\circ\iota$ and $\tau_a^{\circ 2}\circ\iota$ on $\widetilde{T_a^\infty}$.

Also note that $\tau_a=(\tau_a^{\circ 2}\circ\iota)\circ(\tau_a\circ\iota)^{-1}$, and hence $\langle\tau_a\circ\iota,\tau_a^{\circ 2}\circ\iota\rangle=\langle\iota,\tau_a\rangle$. To complete the proof, we only need to show that $\langle\iota,\tau_a\rangle$ is the free product of $\langle\iota\rangle$ and $\langle\tau_a\rangle$.

To this end, we first observe that any relation in $\langle\iota,\tau_a\rangle$ other than $\iota^{\circ 2}=\mathrm{id}$ and $\tau_a^{\circ 3}=\mathrm{id}$ can be reduced to one of the form 
\begin{equation}
(\tau_a^{\circ k_1}\circ\iota)\circ\cdots\circ(\tau_a^{\circ k_r}\circ\iota)=\mathrm{id}
\label{group_relation_1}
\end{equation} 
or 
\begin{equation}
(\tau_a^{\circ k_1}\circ\iota)\circ\cdots\circ(\tau_a^{\circ k_r}\circ\iota)=\iota,
\label{group_relation_2}
\end{equation}
where $k_1,\cdots,k_r\in\{1,2\}$.

\noindent\textbf{Case 1:} Let us first assume that there exists a relation of the form (\ref{group_relation_1}) in $\langle\iota,\tau_a\rangle$. Each $(\tau_a^{\circ k_p}\circ\iota)$ maps the interior of a tile of rank $n$ in $\widetilde{T_a^\infty}\setminus\overline{\D}$ to a tile of rank $(n+1)$ in $\widetilde{T_a^\infty}\setminus\overline{\D}$. Hence, the group element on the left of Relation~(\ref{group_relation_1}) maps the tile of rank $0$ to a tile of rank $r$. Clearly, such an element cannot be the identity map proving that there is no relation of the form (\ref{group_relation_1}) in $\langle\iota,\tau_a\rangle$.

\noindent\textbf{Case 2:} Let us now assume that there exists a relation of the form (\ref{group_relation_2}) in $\langle\iota,\tau_a\rangle$. Each $(\tau_a^{\circ k_p}\circ\iota)$ maps $\widetilde{T_a^\infty}\setminus\overline{\D}$ to itself. Hence, the group element on the left of Relation~(\ref{group_relation_2}) maps $\widetilde{T_a^\infty}\setminus\overline{\D}$ to itself, while $\iota$ maps $\widetilde{T_a^\infty}\setminus\overline{\D}$ to $\widetilde{T_a^\infty}\cap\D$. This shows that there cannot exist a relation of the form (\ref{group_relation_2}) in $\langle\iota,\tau_a\rangle$.

We conclude that $\iota^{\circ 2}=\mathrm{id}$ and $\tau_a^{\circ 3}=\mathrm{id}$ are the only relations in $\langle\iota,\tau_a\rangle$, and hence $\langle\iota,\tau_a\rangle=\langle\iota\rangle\ast\langle\tau_a\rangle\cong\Z/2\Z\ast\Z/3\Z$. 
\end{proof}

\begin{remark}\label{free_rem}
$\langle\iota,\tau_a\rangle$ is not a free product of the subgroups $\langle\tau_a\circ\iota\rangle$ and $\langle\tau_a^{\circ 2}\circ\iota\rangle$ as these generators satisfy a relation $(\tau_a\circ\iota)^{\circ 2}(\tau_a^{\circ 2}\circ\iota)^{-1}(\tau_a\circ\iota)^{\circ 2}=\tau_a^{\circ 2}\circ\iota$.
\end{remark}

\begin{remark}\label{corr_thrice_punc_sphere_rem}
According to Remark~\ref{thrice_punc_sphere_rem}, the Schwarz reflection map $\sigma_a$ induces an anti-conformal involution on a thrice punctured sphere which is obtained by taking a suitable quotient of the tiling set $T_a^\infty$ by a holomorphic dynamical system. This manifests in the action of $\Z/2\Z$ on the lifted tiling set $\widetilde{T_a^\infty}$. On the other hand, the action of a generator of $\Z/3\Z$ on the lifted tiling set (which acts by a deck transformations of $f_a$ on $\widetilde{T_a^\infty}$) corresponds to a conformal isomorphism of the thrice punctured sphere that permutes the three punctures transitively.
\end{remark}

\subsubsection{Action of $\widetilde{\sigma_a}^*$ on $\widetilde{K_a}$}\label{rat_subsec}

Let us now study the dynamics of the correspondence on the lifted non-escaping set $\widetilde{K_a}$. To do so, we will define two maps on $V_a:=f_a^{-1}(\Omega_a)\setminus\D$ that will play the role of deck transformations of $f_a$ in spite of the presence of a critical point of $f_a$ in $V_a$.

The first map $g_{1,a}=g_1:V_a\to\D$ is defined as the composition of $f_a:V_a\to\Omega_a$ and $\left(f_a\vert_{\D}\right)^{-1}:\Omega_a\to\D$. Clearly, $g_1$ is a two-to-one branched covering satisfying $f_a\circ g_1=f_a$ on $V_a$.

On the other hand, since $f_a:V_a\to\Omega_a$ is a two-to-one branched covering, there exists a biholomorphism $g_{2,a}=g_2:V_a\to V_a$ such that $f_a\circ g_2=f_a$ on $V_a$. The map $g_2$ simply permutes the two elements in each non-critical fiber of $f_a:V_a\to\Omega_a$, and fixes the critical point $\frac{a+1}{a-1}$ of $f_a$. It follows that $g_2^{\circ 2}=\mathrm{id}$ on $V_a$.

\begin{remark}
Neither of the maps $g_1$ and $g_2$ agree with $\tau_a$ on $\widetilde{T_a^\infty}\cap V_a$. Indeed, if $w_1,w_3$ are as in Figure~\ref{corr_tiling_pic}, then $g_1(w_1)=g_1(w_3)=w_2$, and $g_2(w_1)=w_3$, $g_2(w_3)=w_1$; whereas $\tau_a(w_1)=w_2$, $\tau_a(w_3)=w_1$. In fact, due to monodromy, none of the $g_i$ extends to a neighborhood of $\infty$ satisfying the functional equation $f_a\circ g_i=f_a$.

Thus, the splitting of the forward correspondence into the maps $g_1\circ\iota$ and $g_2\circ\iota$ on $\iota\left(V_a\right)=\left(f\vert_{\D}\right)^{-1}\left(\Omega_a'\right)$ is different from its splitting into the maps $\tau_a\circ\iota$ and $\tau_a^{\circ 2}\circ\iota$ on $\widetilde{T_a^\infty}\cap\iota\left(V_a\right)$. While univalence of the maps $\tau_a\circ\iota$ and $\tau_a^{\circ 2}\circ\iota$ leads to group structure in the dynamics of $\widetilde{\sigma_a}^*$ on $\widetilde{T_a^\infty}$, we will now see that the existence of a critical point of $g_1\circ\iota$ is responsible for ``rational map behavior" of suitable branches of the correspondence on $\widetilde{K_a}$.
\end{remark}

\begin{proposition}\label{corr_filled_prop}
Let $a\in\cC(\mathcal{S})$. Then, on $\widetilde{K_a}\cap\overline{\D}$, one branch of the forward correspondence is hybrid conjugate to $R_{\chi(a)}\vert_{\mathcal{K}_{\chi(a)}}$. The other branch maps $\widetilde{K_a}\cap\overline{\D}$ onto $\widetilde{K_a}\setminus\D$.

\noindent On the other hand, the forward correspondence preserves $\widetilde{K_a}\setminus\D$. Moreover, on $\widetilde{K_a}\setminus\D$, one branch of the backward correspondence is conjugate to $R_{\chi(a)}\vert_{\mathcal{K}_{\chi(a)}}$, and the remaining branch maps $\widetilde{K_a}\setminus\D$ onto $\widetilde{K_a}\cap\overline{\D}$.
\end{proposition}
\begin{proof}
It is easy to see that on $\widetilde{K_a}\cap\overline{\D}$, the forward correspondence splits into two maps $g_1\circ\iota:\widetilde{K_a}\cap\overline{\D}\to\widetilde{K_a}\cap\overline{\D}$ and $g_2\circ\iota:\widetilde{K_a}\cap\overline{\D}\to\widetilde{K_a}\setminus\D$. Moreover by Proposition~\ref{inverse_lift_corr}, the map $g_1\circ\iota$ is conjugate to $\sigma_a$ via the conformal map $f_a\vert_{\overline{\D}}$. By Theorem~\ref{straightening_schwarz}, this branch is hybrid conjugate to $R_{\chi(a)}\vert_{\mathcal{K}_{\chi(a)}}$. By definition of $g_2$, the other branch of the forward correspondence (i.e $g_2\circ\iota$) maps $\widetilde{K_a}\cap\overline{\D}$ univalently onto $\widetilde{K_a}\setminus\D$.

The fact that the forward correspondence preserves $\widetilde{K_a}\setminus\D$ follows from the covering properties of $f_a\vert_{\widetilde{K_a}}$ (more precisely, from the observation that $\widetilde{K_a}\cap\overline{\D}$ is mapped univalently by $f_a$ onto $K_a$, and $\widetilde{K_a}\setminus\D$ is mapped as a two-to-one ramified covering onto $K_a$). 

Note that the map $\iota\circ g_1: \widetilde{K_a}\setminus\D\to\widetilde{K_a}\setminus\D$ is a branch of the backward correspondence. Since $\iota$ is a topological conjugacy between this backward branch and the forward branch $(g_1\circ\iota)\vert_{\widetilde{K_a}\cap\overline{\D}}$, it follows that $f\vert_{\widetilde{K_a}\cap\overline{\D}}\circ\iota:\widetilde{K_a}\setminus\D\to K_a$ is a topological conjugacy between the backward branch $(\iota\circ g_1)\vert_{\widetilde{K_a}\setminus\D}$ and $\sigma_a\vert_{K_a}$. Invoking Theorem~\ref{straightening_schwarz}, we conclude that $(\iota\circ g_1)\vert_{\widetilde{K_a}\setminus\D}$ is topologically conjugate to $R_{\chi(a)}\vert_{\mathcal{K}_{\chi(a)}}$.

Finally, it is easy to see that the remaining branch of the backward correspondence on $\widetilde{K_a}\setminus\D$ is $\left(g_2\circ\iota\right)^{-1}$, which maps $\widetilde{K_a}\setminus\D$ onto $\widetilde{K_a}\cap\overline{\D}$.
\end{proof}

\subsubsection{$\widetilde{\sigma_a}^\ast$ as a Mating}\label{main_thm_corr_subsec}

Combining the results from the previous two subsections, we will now give a dynamical description of the correspondence $\widetilde{\sigma_a}^*$ on the whole Riemann sphere, and conclude that it is a mating of the group $\Z/2\Z\ast\Z/3\Z$ and the anti-rational map $R_{\chi(a)}$.

\begin{theorem}[Anti-holomorphic Correspondences as Matings]\label{group_rational_mating_thm}
Let $a\in\cC(\mathcal{S})$. Then the following statements hold true.
\begin{itemize}
\item Each of the sets $\widetilde{T_a^\infty}$ and $\widetilde{K_a}$ is completely invariant under the correspondence $\widetilde{\sigma_a}^*$.

\item On $\widetilde{T_a^\infty}$, the dynamics of the correspondence $\widetilde{\sigma_a}^*$ is equivalent to the action of $\Z/2\Z\ast\Z/3\Z$.

\item On $\widetilde{K_a}\cap\overline{\D}$, one branch of the forward correspondence is hybrid conjugate to $R_{\chi(a)}\vert_{\mathcal{K}_{\chi(a)}}$. The other branch maps $\widetilde{K_a}\cap\overline{\D}$ onto $\widetilde{K_a}\setminus\D$.

\noindent On the other hand, the forward correspondence preserves $\widetilde{K_a}\setminus\D$. Moreover, on $\widetilde{K_a}\setminus\D$, one branch of the backward correspondence is conjugate to $R_{\chi(a)}\vert_{\mathcal{K}_{\chi(a)}}$, and the remaining branch maps $\widetilde{K_a}\setminus\D$ onto $\widetilde{K_a}\cap\overline{\D}$.
\end{itemize}
\end{theorem}

\begin{proof}
The statements follow from Propositions~\ref{correspondence_partition},~\ref{grand_orbit_group}, and~\ref{corr_filled_prop}. 
\end{proof}

In light of Theorem~\ref{group_rational_mating_thm}, we say that the correspondence $\widetilde{\sigma_a}^*$ is a \emph{mating} of the rational map $R_{\chi(a)}$ and the group $\Z/2\Z\ast\Z/3\Z$. We are now ready to show that the family of correspondences $\{\widetilde{\sigma_a}^*: a\in\cC(\mathcal{S})\}$ contains matings of the abstract modular group $\Z/2\Z\ast\Z/3\Z$ with every anti-rational map in $\cC(\mathfrak{L}_0)$ that lies in the closure of hyperbolic parameters.

\begin{proof}[Proof of Theorem~\ref{almost_all_maps_mated}]
This follows from Proposition~\ref{chi_injective_prop}, Corollary~\ref{onto_hyperbolic_closure} and Theorem~\ref{group_rational_mating_thm}.
\end{proof}

\appendix

\section{A Family of Parabolic Anti-rational Maps}\label{anti_rational_parabolic}

In this appendix, we prepare some background on a certain family of quadratic anti-rational maps with a neutral fixed point. 

\subsection{The Family $\mathcal{F}$}\label{appendix_subsec_1}

Let $R$ be a quadratic anti-rational map with a neutral fixed point. Conjugating by a M{\"o}bius map, we can assume that $\infty$ is a neutral fixed point of $R$; i.e., $DR^{\circ 2}(\infty)=1$ (note that a neutral fixed point of an anti-holomorphic map is necessarily a parabolic fixed point of multiplier $1$ for the second iterate). We can also assume that $R(0)=\infty$, and $\frac{\partial R}{\partial\overline{z}}(1)=0$ (i.e., $1$ is a critical point of $R$).

Let us now describe the explicit form of such a rational map 
$$
R(z)=~\frac{p_1\overline{z}^2+q_1\overline{z}+r_1}{p_2\overline{z}^2+q_2\overline{z}+r_2}.
$$ 
The requirements $R(\infty)=\infty$ and $R(0)=\infty$ imply that $p_1\neq0$ and $p_2=r_2=0$. Therefore, we must have $q_2\neq0$. A simple computation shows that the condition $DR^{\circ 2}(\infty)=1$ translates to the relation $\vert p_1\vert=\vert q_2\vert$; i.e., $p_1=q_2e^{i\alpha}$, for some $\alpha\in\faktor{\R}{2\pi\Z}$. Finally, the condition $\frac{\partial R}{\partial\overline{z}}(1)=0$ implies that $r_1=q_2e^{i\alpha}$. Therefore, we have that 
$$
R(z)=R_{\alpha,A}=e^{i\alpha}\left(\overline{z}+\frac{1}{\overline{z}}\right)+A,
$$ 
where $\alpha\in\faktor{\R}{2\pi\Z}$, and $A=\frac{q_1}{q_2}\in\C$.

Note that $R_{\alpha,A}(-z)=-R_{\alpha,-A}(z)$ for $z\in\C$.

A direct computation shows that $\infty$ is a higher order parabolic fixed point of $R_{\alpha,A}^{\circ 2}$ when $A=0$ or $\arg{A}=\frac{\alpha}{2}\pm\frac{\pi}{2}$. Otherwise, $\infty$ is a simple parabolic fixed point of $R_{\alpha,a}^{\circ 2}$. We will only be concerned with the case when $\infty$ is a simple parabolic fixed point of $R_{\alpha,A}^{\circ 2}$. 

Therefore, up to M{\"o}bius conjugation, the family 
$$
\mathcal{F}:=\left\{R_{\alpha,A}=e^{i\alpha}\left(\overline{z}+\frac{1}{\overline{z}}\right)+A:\alpha\in\R/2\pi\Z,\ A\neq0,\ \arg{A}\in\left(\frac{\alpha}{2}-\frac{\pi}{2},\frac{\alpha}{2}+\frac{\pi}{2}\right)\right\}
$$ 
contains all quadratic anti-rational maps with a neutral fixed point which is a simple parabolic fixed point for the second iterate of the map. Moreover, a simple computation shows that no two distinct maps in $\mathcal{F}$ are M{\"o}bius conjugate.

For each map in this family, there is a unique attracting direction at $\infty$, and hence a unique basin of attraction of $\infty$ (i.e., the set of all points that converge to $\infty$ asymptotic to the unique attracting direction). Note that the immediate basin of attraction of $\infty$ is fixed under $R_{\alpha,A}$, and contains the critical point $1$ of $R_{\alpha,A}$. Since $\Deg{R_{\alpha,A}}=2$, it follows that the immediate basin of $\infty$ is completely invariant; i.e., the unique basin of attraction of $\infty$ (of $R_{\alpha,A}$) is connected. 

\begin{definition}[Basin of Infinity and Filled Julia set]\label{filled_julia_para_rat_def}
For $R_{\alpha,A}\in\mathcal{F}$, the basin of attraction of the neutral fixed point $\infty$ is denoted by $\mathcal{B}_{\alpha,A}$. The complement $\widehat{\C}\setminus\mathcal{B}_{\alpha,A}$ of the basin of attraction of $\infty$ is called the \emph{filled Julia set} of $R_{\alpha,A}$, and is denoted by $\mathcal{K}_{\alpha,A}$.
\end{definition}

Since the basin of infinity $\mathcal{B}_{\alpha,A}$ is connected, it follows that every connected component of $\Int{\mathcal{K}_{\alpha,A}}$ is simply connected. In particular, $R_{\alpha,A}$ has no Herman ring.

The basin of infinity $\mathcal{B}_{\alpha,A}$ is simply connected (equivalently, the filled Julia set $\mathcal{K}_{\alpha,A}$ is connected) if and only if $+1$ is the only critical point of $R_{\alpha,A}$ contained in $\mathcal{B}_{\alpha,A}$; or equivalently, $-1\in\mathcal{K}_{\alpha,A}$.

\subsection{Pinched Anti-quadratic-like restrictions of maps in $\mathcal{F}$}\label{appendix_subsec_2} In Definition~\ref{pinched_def}, we introduced pinched anti-quadratic-like maps. Let us observe that each member of the family $\mathcal{F}$ naturally admits a pinched anti-quadratic-like restriction. 

For any $R_{\alpha,A}\in\mathcal{F}$, let $\mathcal{P}_{\alpha,A}$ be an attracting petal at $\infty$ such that 
$$
\partial\mathcal{P}_{\alpha,A}\cap\{\vert z\vert\geq M\}=\{m e^{\pm\frac{2\pi i}{3}}: m\geq M\}
$$ 
(for some $M>0$), and the critical point $1$ lies on $\partial\mathcal{P}_{\alpha,A}$. We can also require that $\partial\mathcal{P}_{\alpha,A}$ is smooth except at $\infty$, and $R_{\alpha,A}^{-1}(\mathcal{P}_{\alpha,A})$ is simply connected. Then, $R_{\alpha,A}:R_{\alpha,A}^{-1}(\mathcal{P}_{\alpha,A})\to\mathcal{P}_{\alpha,A}$ is a degree $2$ branched covering. Setting 
$$
\mathbf{V}:=\widehat{\C}\setminus\overline{\mathcal{P}_{\alpha,A}},\ \textrm{and}\ \mathbf{U}:=\widehat{\C}\setminus\overline{R_{\alpha,A}^{-1}(\mathcal{P}_{\alpha,A})},
$$ 
it is straightforward to see that $R_{\alpha,A}:(\overline{\mathbf{U}},\infty)\to(\overline{\mathbf{V}},\infty)$ is a pinched anti-quadratic-like map.

\subsection{The Leaf $\mathfrak{L}_0$}\label{appendix_subsec_3}

We will attach a conformal invariant to the maps $R_{\alpha,A}\in\mathcal{F}$. The proof of \cite[Lemma~2.3]{HS} (which applies generally to \emph{odd} period parabolic basins of anti-holomorphic maps) provides us with a Fatou coordinate on an attracting petal in $\mathcal{B}_{\alpha,A}$ such that the Fatou coordinate conjugates $R_{\alpha,A}$ to the glide reflection $\zeta\mapsto\overline{\zeta}+\frac12$. The imaginary part of this Fatou coordinate is called the \emph{{\'E}calle height}. Since the chosen Fatou coordinate is unique up to translation by a real constant, the {\'E}calle height of a point in the petal is well-defined. In particular, since the petal on which this Fatou coordinate is defined contains the critical value $R_{\alpha,A}(1)$, the {\'E}calle height of this critical value is also well-defined. This quantity is called the \emph{critical {\'E}calle height} of $R_{\alpha,A}$ (to be more precise, the critical {\'E}calle height of $R_{\alpha,A}$ associated with the neutral fixed point $\infty$). It is a conformal conjugacy invariant of the map.

Changing the critical {\'E}calle height of $R_{\alpha,A}\in\mathcal{F}$ by a quasiconformal deformation supported on $\mathcal{B}_{\alpha,A}$ (as in \cite[Theorem~3.2]{MNS}), it is easy to see that every parameter in $\mathcal{F}$ lies in the interior of a real-analytic arc of quasiconformally equivalent parabolic parameters. The resulting arc is entirely contained in $\mathcal{F}$. Such an arc is called a \emph{parabolic arc}, and its parametrization by critical {\'E}calle height is called the \emph{critical {\'E}calle height parametrization} of the arc.

Let us denote the set of all parameters in $\mathcal{F}$ with critical {\'E}calle height $h$ by $\mathfrak{L}_h$. Then, 
$$
\mathcal{F}=\bigsqcup_{h\in\R}\mathfrak{L}_h.
$$ 
Thus, $\mathcal{F}$ is foliated by the leaves $\mathfrak{L}_h$. Each parabolic arc is transverse to the leaves. Letting a parameter $(\alpha,A)\in\mathfrak{L}_h$ flow along the parabolic arc passing through this parameter until it hits $\mathfrak{L}_{h'}$ defines a \emph{holonomy map} from the leaf $\mathfrak{L}_h$ to the leaf $\mathfrak{L}_{h'}$.

For definiteness, from now on we will work with the leaf $\mathfrak{L}_0$. Up to M{\"o}bius conjugacy, the leaf $\mathfrak{L}_0$ contains all quadratic anti-rational maps $R$ such that 
\begin{enumerate}
\item $R$ has a neutral fixed point such that this fixed point is a simple parabolic fixed point of $R^{\circ 2}$, and

\item the critical {\'E}calle height of $R$ (associated with the ``fast" critical point in the immediate basin of attraction of the simple parabolic fixed point) is $0$.
\end{enumerate}

\begin{definition}[The Parabolic Tricorn/Connectedness Locus of $\mathfrak{L}_0$]\label{conn_rat_def}
The connectedness locus of the family $\mathfrak{L}_0$ is defined as 
\begin{align*}
\cC(\mathfrak{L}_0)&=\left\{R_{\alpha,A}\in\mathfrak{L}_0: \mathcal{K}_{\alpha,A}\ \mathrm{is\ connected}\right\}\\
&=\left\{R_{\alpha,A}\in\mathfrak{L}_0: -1\in\mathcal{K}_{\alpha,A}\right\}\\
&=\left\{R_{\alpha,A}\in\mathfrak{L}_0: +1\ \textrm{is the only critical point of}\ R_{\alpha,A}\ \mathrm{in}\ \mathcal{B}_{\alpha,A}\right\}.
\end{align*} 
We call $\cC(\mathfrak{L}_0)$ the \emph{parabolic Tricorn}.
\end{definition}

\subsection{Dynamical and Parameter Rays for $\mathfrak{L}_0$}\label{appendix_subsec_4}

For $(\alpha,A)\in\cC(\mathfrak{L}_0)$, let us choose a conformal isomorphism $\pmb{\psi}_{\alpha,A}$ from $\mathcal{B}_{\alpha,A}$ onto $\D$ such that $\pmb{\psi}_{\alpha,A}(1)=0,\ \pmb{\psi}_{\alpha,A}(\infty)=1$.

\begin{proposition}\label{unicritical_parabolic}
$\pmb{\psi}_{\alpha,A}$ conjugates $R_{\alpha,A}$ to the anti-Blaschke product 
$$
B:\mathbb{D}\to\mathbb{D}, B(z)=\frac{3\overline{z}^2+1}{3+\overline{z}^2},
$$ 
which has a neutral fixed point at $1$.
\end{proposition}
\begin{proof}
Since the {\'E}calle height of the critical point $1$ (of $R_{\alpha,A}$) is $0$, the Riemann map $\pmb{\psi}_{\alpha,A}$ conjugates $R_{\alpha,A}$ to an anti-holomorphic Blaschke product of degree $2$ having a (unique) critical point at $0$ and a neutral fixed point at $1$ with critical {\'E}calle height $0$. The only such Blaschke product is $B$.
\end{proof}

\begin{remark}\label{hybrid_class}
Since any two maps in $\cC(\mathfrak{L}_0)$ are conformally conjugate on their basins of attraction of $\infty$, two maps in $\cC(\mathfrak{L}_0)$ are hybrid conjugate (i.e., quasiconformally conjugate in a neighborhood of the filled Julia set such that the conjugacy is conformal on the filled Julia set) if and only if they are the same.
\end{remark}

Let us now pick $(\alpha,A)\in\mathfrak{L}_0\setminus\cC(\mathfrak{L}_0)$, and choose an (extended) attracting Fatou coordinate $\psi_{\alpha,A}^{\mathrm{att}}:\mathcal{B}_{\alpha,A}\to\C$ that semi-conjugates $R_{\alpha,A}$ to $\zeta\mapsto\overline{\zeta}+\frac12$. We can further assume that $\psi_{\alpha,A}^{\mathrm{att}}(1)=0$. Let $\mathcal{P}_{\alpha,A}$ be a (maximal) attracting petal in the basin $\mathcal{B}_{\alpha,A}$ such that $\psi_{\alpha,A}^{\mathrm{att}}(\mathcal{P}_{\alpha,A})$ is the right half-plane $\{\re(z)>0\}$ and $+1\in\partial \mathcal{P}_{\alpha,A}$. Similarly, let us choose an (extended) attracting Fatou coordinate $\psi_B^{\mathrm{att}}:\D\to\C$ that semi-conjugates $B$ to $\zeta\mapsto\overline{\zeta}+\frac12$ with $\psi_{B}^{\mathrm{att}}(0)=0$. Furthermore, let $\mathcal{P}_{B}$ be a (maximal) attracting petal in the basin $\D$ (of $B$) such that $\psi_{B}^{\mathrm{att}}(\mathcal{P}_{B})$ is the right half-plane $\{\re(z)>0\}$ and $0\in\partial \mathcal{P}_{B}$. Then, $\pmb{\psi}_{\alpha,A}:=\left(\psi_{B}^{\mathrm{att}}\right)^{-1}\circ\psi_{\alpha,A}^{\mathrm{att}}:\mathcal{P}_{\alpha,A}\to \mathcal{P}_B$ is a conformal conjugacy between $R_{\alpha,A}$ and $B$. Since $(\alpha,A)\in\mathfrak{L}_0\setminus\cC(\mathfrak{L}_0)$, this conjugacy can be lifted until we hit the other critical point $-1$ of $R_{\alpha,A}$. Consequently, we get a conformal conjugacy $\pmb{\psi}_{\alpha,A}$ between $R_{\alpha,A}$ and $B$ (defined on a subset of the basin of $\infty$) such that the domain of $\pmb{\psi}_{\alpha,A}$ contains the ``slow" critical value $R_{\alpha,A}(-1)$. In fact, since $\Deg(R_{\alpha,A})=2$, we have that $\pmb{\psi}_{\alpha,A}(R_{\alpha,A}(-1))\in\D\setminus\overline{B(\mathcal{P}_B)}$.

\begin{figure}[ht!]
\begin{tikzpicture}
  \node[anchor=south west,inner sep=0] at (0,0) {\includegraphics[width=0.9\textwidth]{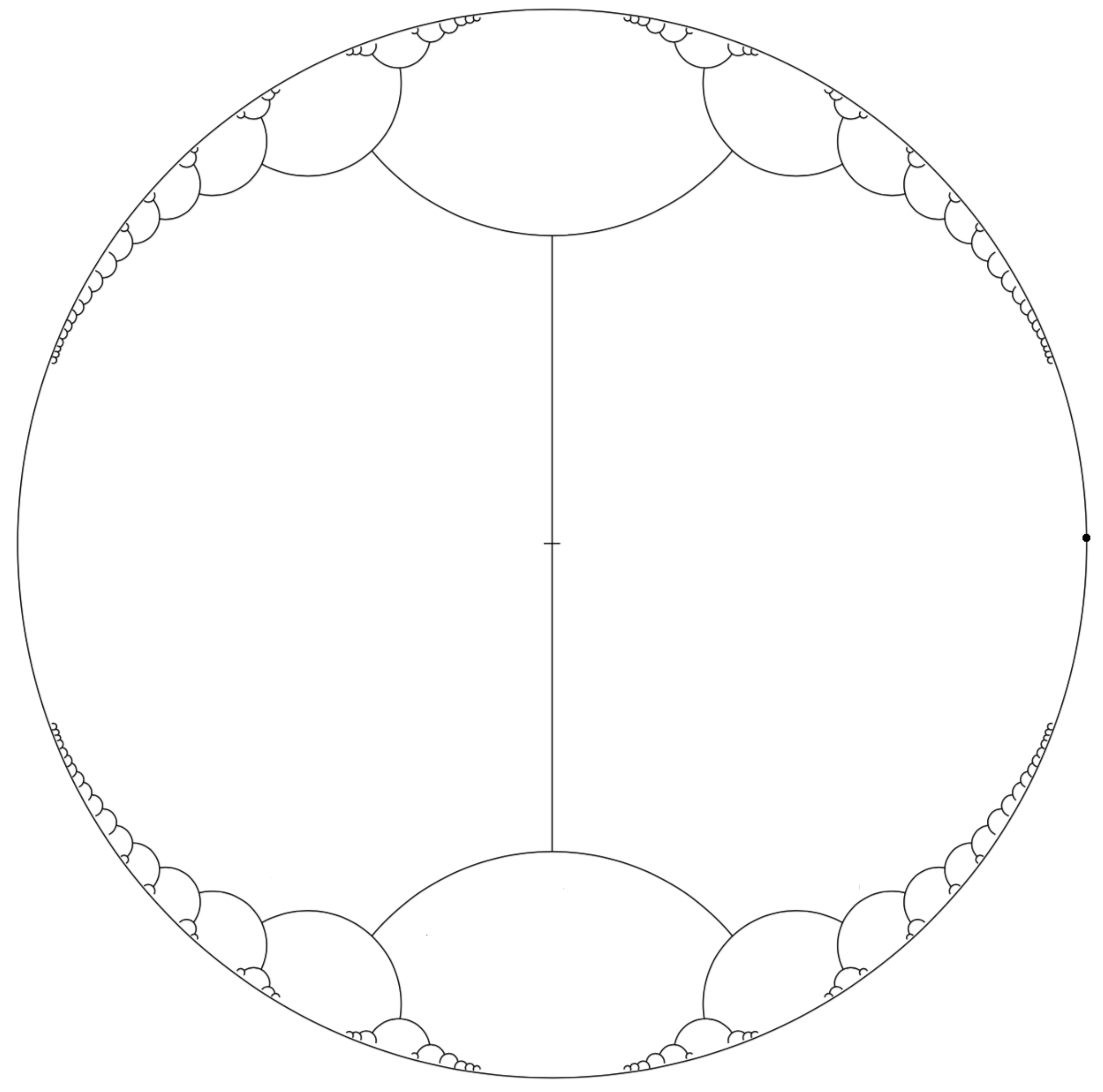}};
  \node at (5.75,2.2) {$z_1$};
  \node at (5.75,9.08) {$z_0$};
  \node at (3.84,9.4) {$z_{00}$};
  \node at (7.72,9.4) {$z_{01}$};
  \node at (7.7,2) {$z_{10}$};
  \node at (3.84,2) {$z_{11}$};
  \node at (8.56,2.2) {$z_{101}$};
  \node at (2.85,2.2) {$z_{110}$};
  \node at (2.85,9.15) {$z_{001}$};
  \node at (8.56,9.15) {$z_{010}$};
  \node at (6,5.72) {$0$};
  \node at (11,5.72) {$1$};
\end{tikzpicture}
\caption{The figure depicts rays in the parabolic basin $\D$ for the map $B$. The points $z_0$ and $z_1$ are the pre-images of the critical point $0$ under $B$. For a finite binary sequence $\left(\epsilon_1,\cdots,\epsilon_k\right)$, we have $B\left(z_{\epsilon_1,\cdots,\epsilon_k}\right)=z_{\epsilon_2,\cdots,\epsilon_k}$. For $\overline{\epsilon}=\left(\epsilon_1,\cdots,\epsilon_k,\cdots\right)\in\{0,1\}^{\N}$, the parabolic ray for $B$ with itinerary $\overline{\epsilon}$ is defined by connecting the points $0,z_{\epsilon_1},z_{\epsilon_1,\epsilon_2},\cdots$ consecutively by suitable iterated pre-images of the line segment $\left[0,\frac13\right]$. Every parabolic ray of $B$ lands at some point on $\mathbb{S}^1\cong\R/\Z$. If a parabolic ray of $B$ lands at some $\theta\in\R/\Z$, then the corresponding ray is called a parabolic ray of $B$ at angle $\theta$.}
\label{parabolic_ray_fig}
\end{figure}

Following \cite[\S 2.2]{PR1}, we can define dynamical rays for the map $B$ (see the description in Figure~\ref{parabolic_ray_fig}).

We can now use the map $\pmb{\psi}_{\alpha,A}$ to define dynamical rays for the maps $R_{\alpha,A}$.

\begin{definition}[Dynamical Rays of $R_{\alpha,A}$]\label{dyn_ray_para}
The pre-image of a parabolic ray of $B$ at angle $\theta\in\R/\Z$ under the map $\pmb{\psi}_{\alpha,A}$ is called a $\theta$-dynamical ray of $R_{\alpha,A}$.
\end{definition}

\begin{remark}\label{non_unique_ray}
Although a dynamical ray of $R_{\alpha,A}$ at an angle $\theta$ is not unique, it is easy to see that any two dynamical rays at a common angle define the same access to $\partial\mathcal{K}_{\alpha,A}$.
\end{remark}

The next result discusses landing properties of pre-periodic dynamical rays of $R_{\alpha,A}$ for $(\alpha,A)\in\cC(\mathfrak{L}_0)$. The proof is a straightforward adaption of the corresponding results for external rays of polynomials (see \cite[\S 18]{M1new}, also compare \cite[Theorems~2.4,~2.5]{PR1}, \cite[Proposition~6.34]{LLMM1}).

\begin{proposition}[Landing of Dynamical Rays]\label{rays_land_para}
Let $(\alpha,A)\in\cC(\mathfrak{L}_0)$. Then, every dynamical ray of $R_{\alpha,A}$ at a (pre-)periodic angle lands at a repelling or parabolic (pre-)periodic point on $\partial\mathcal{K}_{\alpha,A}$. Conversely, every repelling or parabolic (pre-)periodic point of $R_{\alpha,A}$ is the landing point of a finite non-zero number of (pre-)periodic dynamical rays.
\end{proposition}

As in the case for quadratic anti-polynomials (and for the Schwarz reflection family $\mathcal{S}$), the conformal position of the ``escaping" critical value provides us with a dynamically defined uniformization of the exterior of the connectedness locus.

\begin{proposition}[Uniformization of The Exterior of The Connectedness Locus]\label{unif_escape_locus_para}
The map $(\alpha,A)\mapsto\pmb{\psi}_{\alpha,A}(R_{\alpha,A}(-1))$ is a homeomorphism from $\mathfrak{L}_0\setminus\cC(\mathfrak{L}_0)$ onto $\D\setminus\overline{B(\mathcal{P}_B)}$.
\end{proposition}
\begin{proof}[Sketch of Proof]
This can be proved by adapting the arguments of \cite[Proposition~6.5]{KN} for our setting (also compare \cite[Theorem~1.3]{LLMM2}). One shows that the map under consideration is continuous, proper, and locally invertible. Finally, the only map $R_{\alpha,A}$ in $\mathfrak{L}_0\setminus\cC(\mathfrak{L}_0)$ with $\pmb{\psi}_{\alpha,A}(R_{\alpha,A}(-1))=0$ is $\overline{z}+\frac{1}{\overline{z}}+3$. Thus, the map is a degree one covering onto the simply connected domain $\D\setminus\overline{B(\mathcal{P}_B)}$, and hence a homeomorphism.
\end{proof}

\begin{definition}[Parameter Rays of $\mathfrak{L}_0$]\label{para_ray_para}
The pre-image of a parabolic ray at angle $\theta\in\R/\Z$ under the uniformization of Proposition~\ref{unif_escape_locus_para} is called a $\theta$-parameter ray of $\mathfrak{L}_0$.
\end{definition}

The proof of the next result is classical for parameter spaces of polynomials (see \cite{orsay,S1a}), and can be easily adapted for the family $\mathfrak{L}_0$.

\begin{proposition}\label{para_rays_land_para}
1) Let $\theta\in\R/\Z$ be periodic under $B$. Then every accumulation point of the parameter ray of $\mathfrak{L}_0$ at angle $\theta$ is a parabolic parameter $(\alpha,A)$ (i.e., $R_{\alpha,A}^{\circ 2}$ has a parabolic cycle other than the parabolic fixed point $\infty$) such that in the corresponding dynamical plane, the dynamical $\theta$-ray lands at the parabolic periodic point on the boundary of the Fatou component containing the critical value $R_{\alpha,A}(-1)$.

2) Let $\theta\in\R/\Z$ be strictly pre-periodic under $B$. Then the parameter ray of $\mathfrak{L}_0$ at angle $\theta$ lands at a Misiurewicz parameter $(\alpha,A)$ (i.e., the critical point $-1$ is strictly pre-periodic under $R_{\alpha,A}$) in $\cC(\mathfrak{L}_0)$ such that in the corresponding dynamical plane, the dynamical ray at angle $\theta$ lands at the critical value $R_{\alpha,A}(-1)$.
\end{proposition}

\subsection{Hyperbolic Components of $\cC(\mathfrak{L}_0)$}\label{appendix_subsec_5}

A parameter $(\alpha,A)\in\mathfrak{L}_0$ is called \emph{hyperbolic} if $R_{\alpha,A}$ has an attracting cycle.\footnote{This slightly differs from the standard definition of hyperbolic rational maps (see \cite[\S 19]{M1new}) since the map $R_{\alpha,A}$ necessarily has a parabolic fixed point at $\infty$. In fact, according to the definition of \cite[\S 19]{M1new}, no map in $\mathfrak{L}_0$ is hyperbolic. However, our definition, which only concerns the behavior of the forward orbit of the free critical point $-1$, is convenient for the current paper, and should not create any confusion.} For a hyperbolic parameter of $\mathfrak{L}_0$, the forward orbit of the critical point $-1$ converges to the unique attracting cycle, and hence $-1\in\mathcal{K}_{\alpha,A}$. In particular, $(\alpha,A)\in\cC(\mathfrak{L}_0)$.

The hyperbolic parameters in $\cC(\mathfrak{L}_0)$ form an open set. A connected component of the set of all hyperbolic parameters is called a \emph{hyperbolic component} of $\cC(\mathfrak{L}_0)$. It is easy to see that every hyperbolic component $H$ has an associated positive integer $k$ such that each parameter in $H$ has an attracting cycle of period $k$. We refer to such a component as a hyperbolic component of period $k$.

The following proposition gives a dynamical uniformization of the hyperbolic components in $\cC(\mathfrak{L}_0)$. For the definitions of the Blaschke product spaces and the Koenigs ratio/multiplier map appearing in the statement of the proposition, see Subsection~\ref{hyp_comp_subsec} and \cite[\S 2.1.1]{LLMM2}.

\begin{proposition}[Dynamical Uniformization of Hyperbolic Components]\label{unif_hyp_rat}
Let $H$ be a hyperbolic component in $\cC(\mathfrak{L}_0)$.
\begin{enumerate}
\item If $H$ is of odd period, then there exists a homeomorphism $\eta_H:H\to\mathcal{B}^-$ that respects the Koenigs ratio of the attracting cycle. In particular, the Koenigs ratio map is a real-analytic $3$-fold branched covering from $H$ onto the open unit disk, ramified only over the origin.

\item If $H$ is of even period, then there exists a homeomorphism $\eta_H:H\to\mathcal{B}^+$ that respects the multiplier of the attracting cycle. In particular, the multiplier map is a real-analytic diffeomorphism from $H$ onto the open unit disk.
\end{enumerate}
In both cases, $H$ is simply connected and has a unique center.
\end{proposition}
\begin{proof}
The proofs of \cite[Theorem~5.6, Theorem~5.9]{NS} apply verbatim to our situation.
\end{proof}

By Proposition~\ref{unif_hyp_rat}, each hyperbolic component has a unique parameter for which the critical point $-1$ of the corresponding map is periodic. This parameter is called the \emph{center} of the hyperbolic component. The unique hyperbolic component of period one of $\cC(\mathfrak{L}_0)$ has center $(0,1)$.

It is well-known that for every parameter $(\alpha,A)$ on the boundary of a hyperbolic component of period $k$, the corresponding map $R_{\alpha,A}$ has a $k$-periodic neutral cycle. If $k$ is odd, then such a neutral cycle is necessarily a parabolic cycle of multiplier $1$ for the second iterate $R_{\alpha,A}^{\circ 2}$ (compare \cite[Lemma~2.5]{MNS}).

\begin{proposition}[Neutral Dynamics of Odd Period]\label{hyp_odd_parabolic_rat}  
1) The boundary of a hyperbolic component of odd period $k>1$ of $\cC(\mathfrak{L}_0)$ consists entirely of parameters having a neutral cycle of exact period $k$. In suitable local conformal coordinates, the $2k$-th iterate of such a map has the form $z\mapsto z+z^{q+1}+\ldots$ with $q\in\{1,2\}$.

2) Every parameter on the boundary of the hyperbolic component of period one is either contained in $\overline{\cC(\mathfrak{L}_0)}\setminus\cC(\mathfrak{L}_0)$ (in which case the neutral fixed point is $\infty$, and it is a multiple parabolic fixed point of $R_{\alpha,A}^{\circ 2}$) or has a neutral fixed point other than $\infty$ (with local power series as above).
\end{proposition}

As in the case for quadratic anti-polynomials, this leads to the following classification of odd periodic parabolic points.

\begin{definition}[Parabolic Cusps]\label{DefCusp_schwarz}
Let $(\alpha,A)$ be a parameter with a neutral periodic point of odd period $k$ (other than $\infty$). It is called a {\em parabolic cusp} if the local power series expansion of $R_{\alpha,A}^{\circ 2k}$ at any of its neutral periodic points has $q=2$ (in the sense of the previous proposition), and a \emph{simple parabolic parameter} otherwise.
\end{definition}

Following \cite[Lemma~3.1]{MNS}, one can show that every odd period simple parabolic parameter $(\alpha,A)\in\cC(\mathfrak{L}_0)$ is contained in a real-analytic arc of quasiconformally conjugate parameters. Such a \emph{parabolic arc} is constructed by varying the \emph{{\'E}calle height} of the critical value $R_{\alpha,A}(-1)$ from $-\infty$ to $\infty$. Moreover, each parabolic arc of period greater than one in $\cC(\mathfrak{L}_0)$ limits at parabolic cusps on both ends. One can now mimic the proofs of \cite[Proposition~3.7, Theorem~3.8]{HS} to conclude the following statements.

\begin{proposition}[Bifurcations Along Arcs]\label{ThmBifArc_para}
1) Along any parabolic arc of odd period greater than one, the holomorphic fixed point index of the parabolic cycle is a real valued real-analytic function that tends to $+\infty$ at both ends.

2) Every parabolic arc of odd period $k>1$ intersects the boundary of a hyperbolic component of period $2k$ along an arc consisting of the set of parameters where the parabolic fixed point index is at least $1$. In particular, every parabolic arc has, at both ends, an interval of positive length at which bifurcation from a hyperbolic component of odd period $k$ to a hyperbolic component of period $2k$ occurs.
\end{proposition}

Using the arguments of \cite[Lemma~2.13, Corollary~2.21]{IM2}, one can make the following sharper statement about the fixed point index of the parabolic cycle on a parabolic arc.

\begin{proposition}\label{index_increasing_1}
Let $H$ be a hyperbolic component of odd period $k$ in $\cC(\mathfrak{L}_0)$, $\mathcal{C}$ be a parabolic arc on $\partial H$, $c:\mathbb{R}\to\mathcal{C}$ be the critical {\'E}calle height parametrization of $\mathcal{C}$, and let $H'$ be a hyperbolic component of period $2k$ bifurcating from $H$ across $\mathcal{C}$. Then there exists some $h_0>0$ such that 
$$
\cC\cap\partial H'=c[h_0,+\infty).
$$ 
Moreover, the fixed point index function 
\begin{center}
$\begin{array}{rccc}
  \ind_{\mathcal{C}}: & [h_0,+\infty) & \to & [1,+\infty) \\
  & h & \mapsto & \ind_{\mathcal{C}}(R_{c(h)}^{\circ 2})\\
   \end{array}$
\end{center}
is strictly increasing, and hence a bijection (where $\ind_{\cC}(R_{c(h)}^{\circ 2})$ stands for the holomorphic fixed point index of the $k$-periodic parabolic cycle of $R_{c(h)}^{\circ 2}$).
\end{proposition}
 
We can now adapt the arguments of \cite[\S 5]{MNS} for the current setting to prove the following result on the structure of the boundaries of odd period hyperbolic components of $\cC(\mathfrak{L}_0)$.
 
\begin{proposition}[Boundaries of Odd Period Hyperbolic Components]\label{hyp_comb_boundary_para}
1) The boundary of every hyperbolic component of odd period $k>1$ of $\cC(\mathfrak{L}_0)$ is a topological triangle having parabolic cusps as vertices and parabolic arcs as sides.

2) The boundary of the unique period one hyperbolic component of $\cC(\mathfrak{L}_0)$ consists of two parabolic arcs, a parabolic cusp, and $\overline{\cC(\mathfrak{L}_0)}\setminus\cC(\mathfrak{L}_0)$ (where the corresponding maps have a multiple parabolic fixed point at $\infty$).
\end{proposition}

\subsection{Beltrami Arcs and Queer Components}\label{appendix_subsec_6}

We will now briefly discuss qc non-rigid parameters in $\cC(\mathfrak{L}_0)$. Note that $(\alpha,A)\in\cC(\mathfrak{L}_0)$ is qc non-rigid if and only if $R_{\alpha,A}$ admits a non-trivial invariant Beltrami coefficient that is supported on its filled Julia set. By the classification of Fatou components of rational maps (and the fact that Siegel parameters, even period parabolic parameters, and odd period parabolic cusps are qc rigid in $\cC(\mathfrak{L}_0)$), it follows that if such a Beltrami coefficient is supported on $\Int{\mathcal{K}_{\alpha,A}}$, then the corresponding map lies either in a hyperbolic component or on a parabolic arc. It remains to discuss qc non-rigid parameters in $\cC(\mathfrak{L}_0)$ for which the corresponding invariant Beltrami coefficient is supported on its Julia set. We do this in the next couple of paragraphs.

Let us first assume that $R_{\alpha,A}\in\cC(\mathfrak{L}_0)$ admits a real one-dimensional quasiconformal deformation space in $\cC(\mathfrak{L}_0)$, and $\mu$ is a non-trivial $R_{\alpha,A}$-invariant Beltrami coefficient which is supported on $\partial\mathcal{K}_{\alpha,A}$. Set $m:=\vert\vert\mu\vert\vert_\infty\in\left(0,1\right)$. Then, $t\mu$ is also a non-trivial $R_{\alpha,A}$-invariant Beltrami coefficient supported on $\partial\mathcal{K}_{\alpha,A}$, for $t\in\left(-\frac1m,\frac1m\right)$. By the measurable Riemann mapping theorem with parameters, we obtain quasiconformal maps $\{h_t\}$ (where $t\in\left(-\frac1m,\frac1m\right)$) with associated Beltrami coefficients $t\mu$ such that $h_t$ fixes $0, 1$, and $\infty$, and $\{h_t\}$ depends real-analytically on $t$. Hence, $h_t\circ R_{\alpha,A}\circ h_t^{-1}\in\cC(\mathfrak{L}_0)$ for all $t\in\left(-\frac1m,\frac1m\right)$. This produces a real-analytic arc of quasiconformally conjugate parameters in $\cC(\mathfrak{L}_0)$ which contains $(\alpha,A)$ in its interior. This arc, which is the full quasiconformal deformation space of $R_{\alpha,A}$ in $\cC(\mathfrak{L}_0)$, is called the \emph{queer Beltrami arc} containing $(\alpha,A)$.

Finally, let $R_{\alpha,A}\in\cC(\mathfrak{L}_0)$ admit a real two-dimensional quasiconformal deformation space in $\cC(\mathfrak{L}_0)$, and $\mu_1$, $\mu_2$ be $\R$-independent non-trivial $R_{\alpha,A}$-invariant Beltrami coefficients which are supported on $\partial\mathcal{K}_{\alpha,A}$. Then, $\left(t_1\mu_1+t_2\mu_2\right)$ is also a non-trivial $R_{\alpha,A}$-invariant Beltrami coefficient supported on $\partial\mathcal{K}_{\alpha,A}$ whenever $t_1, t_2\in\R$, and $\vert\vert t_1\mu_1+t_2\mu_2\vert\vert_\infty<1$. As in the previous case, this produces an open set of quasiconformally conjugate parameters in $\cC(\mathfrak{L}_0)$ containing $(\alpha,A)$. This open set, which is the full quasiconformal deformation space of $R_{\alpha,A}$ in $\cC(\mathfrak{L}_0)$, is called the \emph{queer component} containing $(\alpha,A)$.

\subsection{Dyadic Tips of $\cC(\mathfrak{L}_0)$}\label{appendix_subsec_7}
Let us now define an important class of critically pre-periodic parameters in $\cC(\mathfrak{L}_0)$.

\begin{definition}[Dyadic Tips of $\cC(\mathfrak{L}_0)$]\label{def_dyadic_tip_rat}
We say that $(\alpha,A)\in\cC(\mathfrak{L}_0)$ is a \emph{dyadic tip of pre-period $k$} if the critical point $-1$ (of $R_{\alpha,A}$) maps to the parabolic fixed point $\infty$ in exactly $k$ steps; i.e., $k\geq 1$ is the smallest integer such that $R_{\alpha,A}^{\circ k}(-1)=\infty$.
\end{definition}

Note that only the $0$-ray of $R_{\alpha,A}$ lands at the parabolic point $\infty$, and hence in the dynamical plane of a dyadic tip of pre-period $k$ of $\cC(\mathfrak{L}_0)$, the critical value is the landing point of only one dynamical ray. This ray has an angle $\theta$ such that $B^{\circ k}(\theta)=0$.

\subsection{Abstract Parabolic Tricorn}\label{appendix_subsec_8}

By Proposition~\ref{rays_land_para}, for $(\alpha,A)\in\cC(\mathfrak{L}_0)$, all dynamical rays of $R_{\alpha,A}$ at angles in $\mathrm{Per}(B)$ (i.e., at pre-periodic angles under $B$) land on $\partial\mathcal{K}_{\alpha,A}$. This leads to the following definitions.

\begin{definition}[Pre-periodic Laminations, and Combinatorial Classes]\label{para_lami_comb_class_def}
i) For $(\alpha,A)\in\cC(\mathfrak{L}_0)$, the \emph{pre-periodic lamination} of $R_{\alpha,A}$ is defined as the equivalence relation on $\mathrm{Per}(B)\subset\R/\Z$ such that $\theta, \theta'\in\mathrm{Per}(B)$ are related if and only if the corresponding dynamical rays of $R_{\alpha,A}$ land at a common point of $\partial\mathcal{K}_{\alpha,A}$.

ii) Two parameters $(\alpha,A)$ and $(\alpha',A')$ in $\cC(\mathfrak{L}_0)$ are said to be \emph{combinatorially equivalent} if the corresponding maps have the same pre-periodic lamination.  

iii) The combinatorial class $\mathrm{Comb}(\alpha,A)$ of $(\alpha,A)\in\cC(\mathfrak{L}_0)$ is defined as the set of all parameters in $\cC(\mathfrak{L}_0)$ that are combinatorially equivalent to $(\alpha,A)$.

iv) A combinatorial class $\mathrm{Comb}(\alpha,A)$ is called \emph{periodically repelling} if for every $(\alpha',A')\in\mathrm{Comb}(\alpha,A)$, each periodic orbit of $R_{\alpha,A}$ is repelling.
\end{definition} 

The following proposition gives a rough description of the combinatorial classes of $\cC(\mathfrak{L}_0)$. 

\begin{proposition}[Description of Combinatorial Classes]\label{comb_class_para}
Every combinatorial class $\mathrm{Comb}(\alpha,A)$ of $\cC(\mathfrak{L}_0)$ is of one of the following four types.
\begin{enumerate}
\item $\mathrm{Comb}(\alpha,A)$ consists of an even period hyperbolic component that does not bifurcate from an odd period hyperbolic component, its root point, and the irrationally neutral parameters on its boundary,

\item $\mathrm{Comb}(\alpha,A)$ consists of an even period hyperbolic component that bifurcates from an odd period hyperbolic component, the unique parabolic cusp and the irrationally neutral parameters on its boundary,

\item $\mathrm{Comb}(\alpha,A)$ consists of an odd period hyperbolic component and the parabolic arcs on its boundary,

\item $\mathrm{Comb}(\alpha,A)$ is periodically repelling.
\end{enumerate}

\end{proposition}

We are now ready to define an abstract topological model for $\cC(\mathfrak{L}_0)$. We put an equivalence relation $\sim$ on $\mathbb{S}^2$ by
\begin{enumerate}
\item identifying all parameters in the closure of each periodically repelling combinatorial class of $\cC(\mathfrak{L}_0)$,

\item identifying all parameters in the closure of the non-bifurcating sub-arc of each parabolic arc of $\cC(\mathfrak{L}_0)$, and

\item identifying all parameters in $\overline{\cC(\mathfrak{L}_0)}\setminus\cC(\mathfrak{L}_0)$ (the corresponding maps have a multiple parabolic fixed point at $\infty$).
\end{enumerate}

This is a non-trivial closed equivalence relation on the sphere such that all equivalence classes are connected and non-separating. By Moore's theorem, the quotient of the $2$-sphere by $\sim$ is again a $2$-sphere. The image of $\cC(\mathfrak{L}_0)$ under this quotient map is non-compact, but adding the class of $\overline{\cC(\mathfrak{L}_0)}\setminus\cC(\mathfrak{L}_0)$ turns it into a full, compact subset of the $2$-sphere.

\begin{definition}[The Abstract Parabolic Tricorn/Abstract Connectedness Locus of $\mathfrak{L}_0$]\label{abstract_model_para}
The \emph{abstract parabolic Tricorn} (or, the abstract connectedness locus of the family $\mathfrak{L}_0$) is defined as the union of the image of $\cC(\mathfrak{L}_0)$ under the quotient map (defined by the above equivalence relation) and the class of $\overline{\cC(\mathfrak{L}_0)}\setminus\cC(\mathfrak{L}_0)$. It is denoted by $\widetilde{\cC(\mathfrak{L}_0)}$.
\end{definition}

\begin{remark}
The abstract parabolic Tricorn $\widetilde{\cC(\mathfrak{L}_0)}$ can also be described as the quotient of the unit disk under a suitable lamination. 

Let us identify the angles of all parameter rays of $\mathfrak{L}_0$ at pre-periodic angles (under $B$) that land at a common (parabolic or Misiurewicz) parameter or accumulate on a common parabolic arc of $\mathfrak{L}_0$. This defines an equivalence relation on $\mathrm{Per}(B)\cap\partial\D$. We then consider the smallest closed equivalence relation on $\partial\D$ generated by the above relation. Taking the hyperbolic convex hull of each of these equivalence classes in $\overline{\D}$ yields a geodesic lamination of $\D$. Finally, we consider the quotient of $\overline{\D}$ by collapsing each hyperbolic convex hull obtained above to a single point. The resulting continuum coincides with the abstract connectedness locus $\widetilde{\cC(\mathfrak{L}_0)}$ (see \cite[\S 9.4.2]{L6} for a general discussion on the construction of pinched disk models of planar continua). 
\end{remark}

\section{Mating $\overline{z}^d$ with The Abstract Hecke Groups}\label{deltoid_corr}

In this appendix, we will construct for each integer $d\geq 2$, an anti-holomorphic correspondence that is a mating of the anti-polynomial $\overline{z}^d$ with the abstract Hecke group $\Z/2\Z\ast\Z/(d+1)\Z$. The correspondence will arise from a univalent restriction of a suitable degree $d+1$ rational map.

\begin{proposition}\label{hypocycloid_quadrature} 
The map $f(z):=z+1/dz^d$ is injective on $\D^*:=\widehat{\mathbb{C}}\setminus\overline{\D}$. Moreover, $f(\mathbb{S}^1)$ is a Jordan curve.
\end{proposition}

\begin{proof} 
Note that $f$ has a $(d-1)$-fold critical point at the origin, and $(d+1)$ simple critical points on $\mathbb{S}^1$. In particular, $f$ has no critical point in $\D^*$. So it suffices to show that $f(\mathbb{S}^1)$ is a Jordan curve (indeed, this implies that $f(\D^*)$ is simply connected, and hence uniqueness of analytic continuation yields an inverse branch of $f$ that maps $f(\D^*)$ onto $\D^*$).

We now prove injectivity of $f\vert_{\mathbb{S}^1}$. To this end, pick $z,w\in\mathbb{S}^1$ with $z\neq w$, and suppose that $f(z)=f(w)$. Note that:
\begin{align*}
f(z)=f(w) \implies \displaystyle\sum_{j=0}^{d-1} z^{d-1-j}w^j=d\cdot z^dw^d\implies \left\vert \displaystyle \sum_{j=0}^{d-1}z^{d-1-j}w^j\right\vert=d.
\end{align*}
By the triangle inequality and the fact that $z,w\in\mathbb{S}^1$, we now conclude that all the complex numbers $z^{d-1-j}w^j$ (for $j=0, \cdots, d-1$) have the same argument. But this implies that $z=w$, a contradiction. Therefore, $f$ is injective on $\mathbb{S}^1$, and hence $f(\mathbb{S}^1)$ is a Jordan curve.
\end{proof}

By Proposition~\ref{simp_conn_quad}, $\Omega:=f(\D^*)$ is a quadrature domain with associated Schwarz reflection map 
$$
\sigma:=f\circ\iota\circ\left(f\vert_{\overline{\D^*}}\right)^{-1}:\overline{\Omega}\to\widehat{\C},
$$ 
where $\iota$ is the reflection in the unit circle. The map $\sigma$ is anti-meromorphic on $\Omega$, and fixes $\partial\Omega$ pointwise. In fact, $\sigma$ has a $d$-fold pole at $\infty$, and no other critical point in $\Omega$. Moreover, $\sigma:\sigma^{-1}(\Omega)\to\Omega$ is a proper branched covering map of degree $d$ (branched only at $\infty$), and $\sigma:\sigma^{-1}(\Int{\Omega^c})\to\Int{\Omega^c}$ is a degree $(d+1)$ covering map. Since $f$ has $(d+1)$ critical points on $\mathbb{S}^1$, it follows that $\partial\Omega$ has $(d+1)$ singular points. We will denote the set of these $(d+1)$ singular points of $\partial\Omega$ by $S$.

We define $T\equiv T(\sigma):=\hat{\mathbb{C}}\setminus \Omega$. We further set $T^0\equiv T^0(\sigma)=T\setminus S$, and 
$$
T^\infty\equiv T^\infty(\sigma):=\bigcup_{n\geq0} \sigma^{-n}(T^0).
$$ 
We will call $T^\infty$ the \emph{tiling set} of $\sigma$. For any $n\geq0$, the connected components of $\sigma^{-n}(T^0)$ are called \emph{tiles} of rank $n$. Two distinct tiles have disjoint interior. The \emph{non-escaping set} of $\sigma$ is defined as 
$$
K\equiv K(\sigma):=\widehat{\C}\setminus T^\infty\subset\Omega\cup S.
$$

A direct generalization of \cite[Proposition~5.6]{LLMM1} shows the following.

\begin{proposition}\label{simp_conn_prop} 
$T^\infty$ is a simply connected domain, and hence $\overline{T^\infty}$ is a compact, connected set.
\end{proposition}

Furthermore, $\infty$ is a super-attracting fixed point of $\sigma$; more precisely, $\infty$ is a fixed critical point of $\sigma$ of multiplicity $(d-1)$. We denote the basin of attraction of $\infty$ by $\mathcal{B}_\infty\equiv \mathcal{B}_\infty(\sigma)$. The following result asserts that the tiling set and the basin of infinity of $\sigma$ are Jordan domains with a common boundary (see Figure~\ref{deltoid_reflection_julia_fig}).
 
\begin{proposition}\label{basin_inf_prop}
1) $K=\overline{\mathcal{B}_\infty}$.

2) $\mathcal{B}_\infty$ is a Jordan domain, and $\sigma:\overline{\mathcal{B}_\infty}\to\overline{\mathcal{B}_\infty}$ is topologically conjugate (conformally on the interior) to $\overline{z}^d:\overline{\D}\to\overline{\D}$.
\end{proposition} 
\begin{proof}
1) For a proof of this fact where $d=2$, see \cite[Theorem~5.11]{LLMM1}. The general case follows from \cite[Corollary~4.11]{LMM20}.

2) In the $d=2$ case, the result follows from \cite[Theorem~5.11, Proposition~5.26]{LLMM1}. While the arguments of \cite{LLMM1} directly generalize to the higher degree case, here is an alternative route to proving the assertion.

Since $\sigma$ has no critical point other than $\infty$ in $\mathcal{B}_\infty$, the proof of \cite[Theorem 9.3]{M1new} implies that there exists a conformal map conjugating $\overline{z}^d\vert_{\D}$ to $\sigma\vert_{\mathcal{B}_\infty}$. Since $\partial\mathcal{B}_\infty$ is locally connected \cite[Proposition~4.2]{LMM20}, this conformal isomorphism extends to a continuous semi-conjugacy between $\overline{z}^d\vert_{\mathbb{S}^1}$ and $\sigma\vert_{\partial\mathcal{B}_\infty}$. To complete the proof, it now suffices to argue that $\partial\mathcal{B}_\infty$ is a Jordan curve. To this end, we note that by \cite[Proposition~4.19]{LMM20}, every cut-point of $\partial\mathcal{B}_\infty$ must land on a double point of $f(\mathbb{S}^1)$ under some iterate of $\sigma$. But since $f(\mathbb{S}^1)$ is a Jordan curve, it does not have any double point. This rules out the existence of cut-points on $\partial\mathcal{B}_\infty$, and shows that $\partial\mathcal{B}_\infty$ is a Jordan curve.
\end{proof}

We define the $d:d$ anti-holomorphic correspondence $\mathfrak{C}\subset\widehat{\C}\times\widehat{\C}$ as

\begin{equation}
(z,w)\in\mathfrak{C}\iff \frac{f(w)-f(\iota(z))}{w-\iota(z)}=0.
\label{corr_eqn_4}
\end{equation}

The proof of Proposition~\ref{inverse_lift_corr} applies mutatis mutandis to the current setting, and yields the following description of the correspondence $\mathfrak{C}$ as lifts of forward and backward branches of $\sigma$ under $f$.

\begin{proposition}\label{inverse_lift_corr_1}
The correspondence $\mathfrak{C}$ defined by Equation~(\ref{corr_eqn_4}) contains all possible lifts of $\sigma$ (respectively, suitable inverse branches of $\sigma^{-1}$) when $z\in\overline{\D^*}$ (respectively, when $z\in\D$) under $f$. More precisely,
\begin{itemize}
\item for $z\in\overline{\D^*}$, we have that $(z,w)\in\mathfrak{C} \iff f(w)=\sigma(f(z))$, and

\item for $z\in\D$, we have that $(z,w)\in\mathfrak{C}\ \implies\ \sigma(f(w))=f(z)$.
\end{itemize}
\end{proposition}

Finally, we set 
$$
\widetilde{T^\infty}:=f^{-1}(T^\infty),\qquad \widetilde{K}:=f^{-1}(K),
$$ 
and call them the \emph{lifted tiling set} and \emph{lifted non-escaping set}, respectively. 

\begin{figure}[ht!]
{\includegraphics[width=0.45\linewidth]{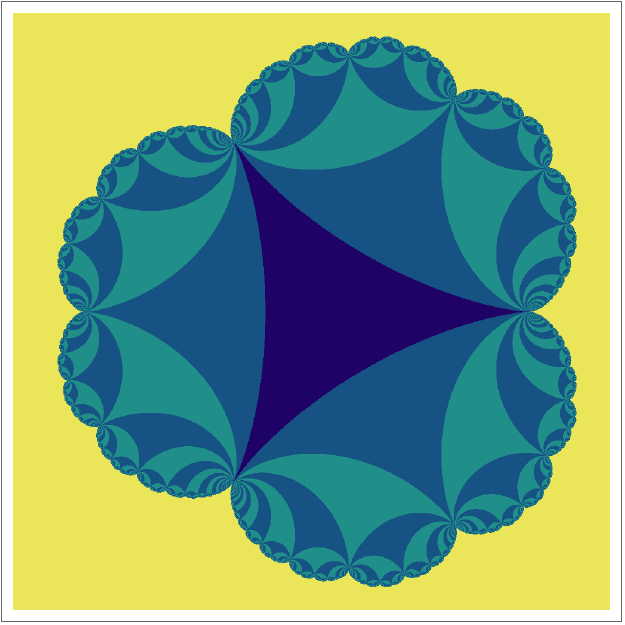}}\ {\includegraphics[width=0.45\linewidth]{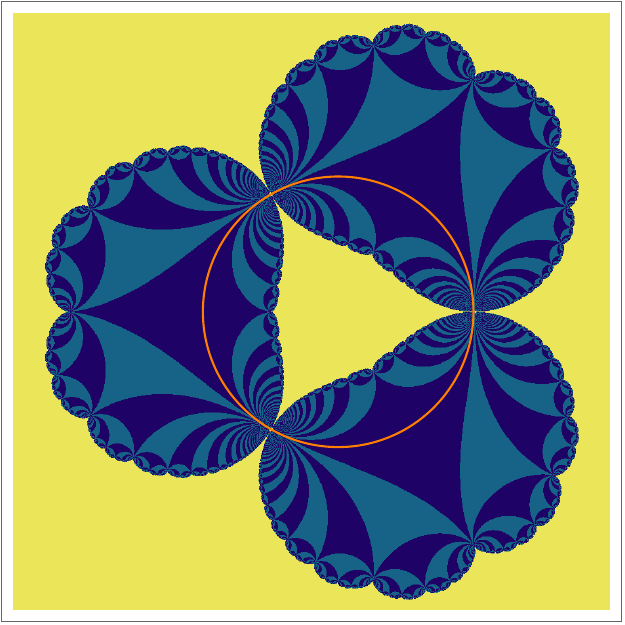}}
\caption{Left: The dynamical plane of the Schwarz reflection map associated with $f(z)=z+\frac{1}{2z^2}$ is shown. The `cauliflower' curve is the common boundary of the tiling set and the basin of infinity. Both these domains are Jordan. Right: The dynamical plane of the correspondence $\mathfrak{C}$ is shown. The blue/green region is the lifted tiling set, and the yellow region is the lifted non-escaping set.}
\label{deltoid_reflection_julia_fig}
\end{figure}

\begin{figure}[ht!]
{\includegraphics[width=0.45\linewidth]{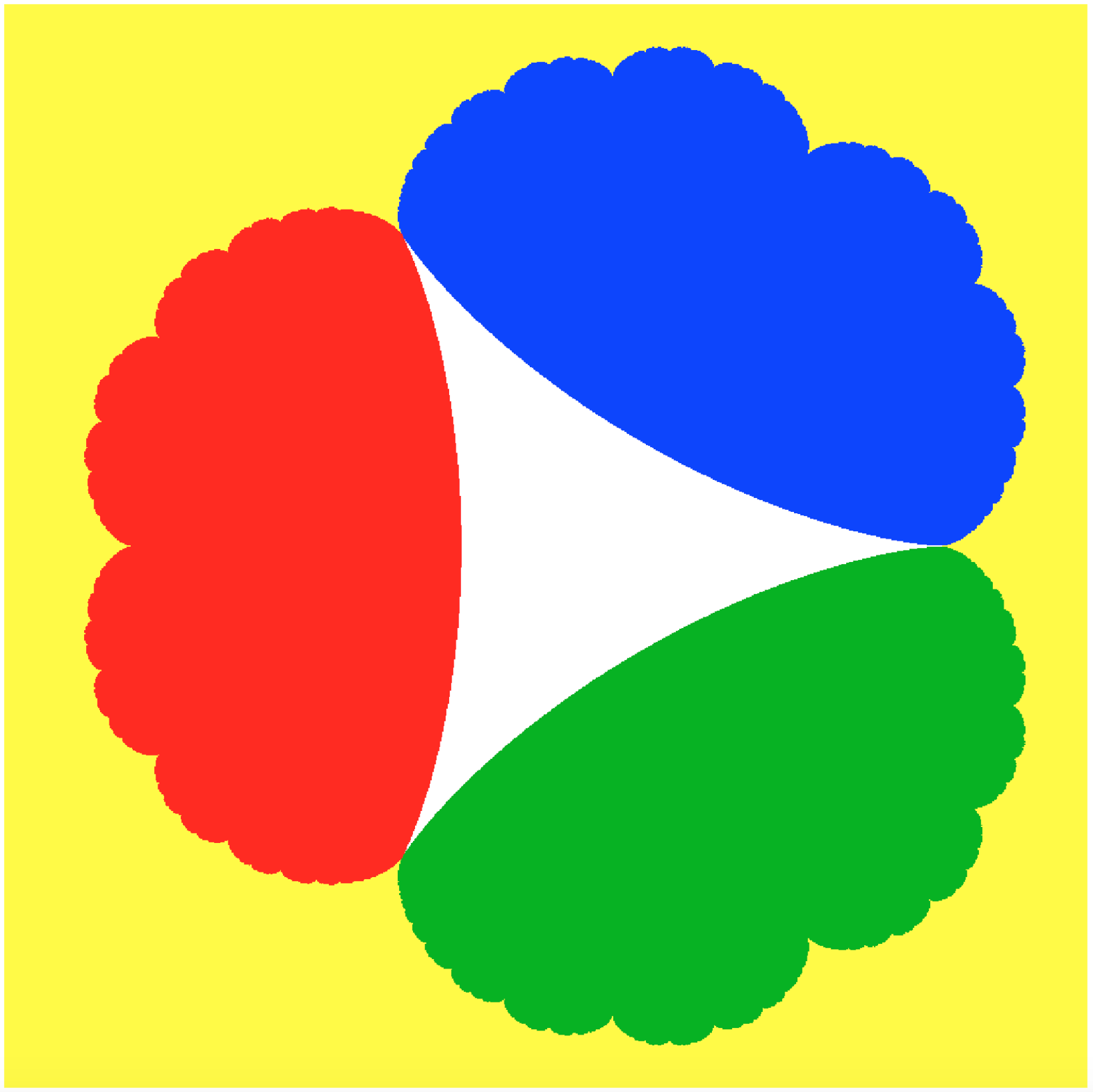}}\ {\includegraphics[width=0.45\linewidth]{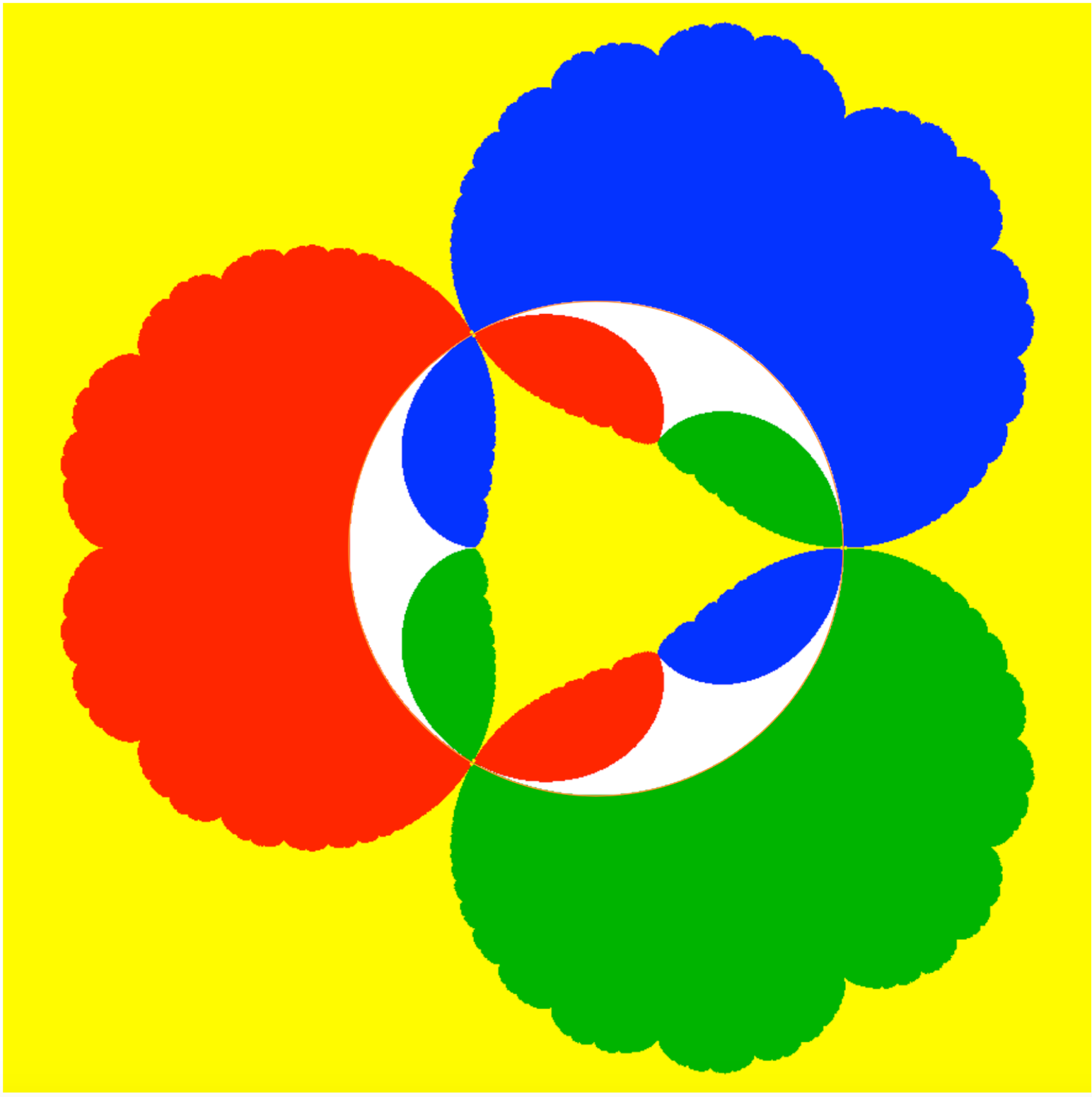}}
\caption{Left: The dynamical plane of the Schwarz reflection map arising from $f(z)=z+\frac{1}{2z^2}$ is shown. The basin of infinity $\mathcal{B}_\infty$ is the yellow region, and the tiling set (the interior of the `cauliflower' shaped region) is the union of the white, blue, red, and green regions. The white region stands for $T^0$. Right: The figure shows the dynamical plane of the correspondence arising from $f$. The two connected yellow regions comprise the pre-image of $\mathcal{B}_\infty$ under $f$. The interiors of the three `cauliflower' shaped regions stand for $\widetilde{T^\infty_j}$, for $j=1,2,3$. Each $\widetilde{T^\infty_j}$ is mapped univalently onto $T^\infty$ by $f$. The white region in each $\widetilde{T^\infty_j}$ is mapped to $T^0$, and the colored regions in each $\widetilde{T^\infty_j}$ are mapped to the corresponding colored regions in $T^\infty$ under $f$. The deck transformation $\tau$ carries the colored regions of $\widetilde{T^\infty_j}$ to the corresponding colored regions of $\widetilde{T^\infty_{j+1}}$.}
\label{deltoid_corr_fig}
\end{figure}

The next proposition follows directly from the definitions.

\begin{proposition}\label{correspondence_partition_1}
1) Each of the sets $\widetilde{T^\infty}$ and $\widetilde{K}$ is completely invariant under the correspondence $\mathfrak{C}$. More precisely, if $(z,w)\in\mathfrak{C}$, then 
$$
z\in\widetilde{T^\infty}\iff w\in\widetilde{T^\infty},
$$ 
and 
$$z\in\widetilde{K}\iff w\in\widetilde{K}.
$$

2) $\iota(\widetilde{T^\infty})=\widetilde{T^\infty}$, and $\iota(\widetilde{K})=\widetilde{K}$.
\end{proposition}

Note that $\mathcal{B}_\infty$ contains only one critical value (namely, $\infty$) of $f$, and this critical value has exactly two pre-images, namely $0$ and $\infty$. Moreover, $f$ is locally injective near $\infty$, and locally $d:1$ near $0$. Recall also that $f$ is univalent on $\overline{\D^*}$. As $f^{-1}(\mathcal{B}_\infty)$ is disjoint from $\mathbb{S}^1$, it now follows that $f^{-1}(\mathcal{B}_\infty)$ has exactly two connected components; one of them is contained in $\D^*$ and maps conformally onto $\mathcal{B}_\infty$ under $f$, and the other is contained in $\D$ and maps as a $d:1$ branched cover onto $\mathcal{B}_\infty$ under $f$ (see Figure~\ref{deltoid_corr_fig}). It also follows from the mapping properties of $f$ that $\widetilde{K}\cap\mathbb{S}^1$ consists precisely of the $(d+1)$-st roots of unity (in fact, these are the points where the two connected components of $f^{-1}(\mathcal{B}_\infty)$ touch), and $\widetilde{K}\setminus\D$ (respectively, $\widetilde{K}\cap\overline{\D}$) is mapped univalently (respectively, as a $d:1$ ramified covering) onto $K$ under $f$.

\begin{proposition}\label{dynamics_filled_julia_corr_prop}
1) $\widetilde{K}\cap\overline{\D}$ is forward invariant, and hence, $\widetilde{K}\setminus\D$ is backward invariant under $\mathfrak{C}$.

2) $\mathfrak{C}$ has a forward branch carrying $\widetilde{K}\setminus\D$ onto itself with degree $d$, and this branch is topologically conjugate to $\sigma:K\to K$. 

3) $\mathfrak{C}$ has a backward branch carrying $\widetilde{K}\cap\overline{\D}$ onto itself with degree $d$, and this branch is topologically conjugate to $\sigma:K\to K$. 
\end{proposition}
\begin{proof}
1) By Proposition~\ref{correspondence_partition_1}, we have that 
$$
\iota(\widetilde{K}\setminus\D)=\widetilde{K}\cap\overline{\D}.
$$ 
Moreover, the fact that the degree $(d+1)$ rational map $f$ sends $\overline{\D^*}$ homeomorphically onto $\overline{\Omega}$ implies that each $z\in K$ has exactly one pre-image in $\widetilde{K}\setminus\D$ and exactly $d$ pre-images (counted with multiplicity) in $\widetilde{K}\cap\overline{\D}$. The statement now follows from the above observations and the definition of $\mathfrak{C}$.

2) Let us set $V:=f^{-1}(\Omega)\cap\D$, and define $g:\overline{V}\to\overline{\D^*}$ as the composition of $f:\overline{V}\to\overline{\Omega}$ and $\left(f\vert_{\overline{\D^*}}\right)^{-1}:\overline{\Omega}\to\overline{\D^*}$. By definition, $g$ is a $d:1$ branched covering satisfying $f\circ g=f$ on $\overline{V}$. It follows that 
$$
g\circ\iota:\widetilde{K}\setminus\D\to\widetilde{K}\setminus\D
$$ 
is a $d:1$ forward branch of the correspondence. 

Clearly, the forward branch $(g\circ\iota)\vert_{\widetilde{K}\setminus\D}$ is topologically conjugate to $\sigma\vert_{K}$ via the univalent map $f:\widetilde{K}\setminus\D\to K$.

3) Note that the map 
$$
\iota\circ g=\iota\circ\left(f\vert_{\overline{\D^*}}\right)^{-1}\circ f: \widetilde{K}\cap\overline{\D}\to\widetilde{K}\cap\overline{\D}
$$ 
is a backward branch of the correspondence $\mathfrak{C}$ carrying $\widetilde{K}\cap\overline{\D}$ onto itself with degree $d$. 

Finally, $\iota$ is a topological conjugacy between the backward branch $(\iota\circ g)\vert_{\widetilde{K}\cap\overline{\D}}$ and the forward branch $(g\circ\iota)\vert_{\widetilde{K}\setminus\D}$, and hence $f\vert_{\widetilde{K}\setminus\D}\circ\iota:\widetilde{K}\cap\overline{\D}\to K$ is a topological conjugacy between the backward branch $(\iota\circ g)\vert_{\widetilde{K}\cap\overline{\D}}$ and $\sigma\vert_{K}$.
\end{proof}

We will now describe the group structure of the correspondence on the lifted tiling set. Since $T^\infty$ is a simply connected domain containing no critical value of $f$, it follows that $\widetilde{T^\infty}$ consists of exactly $(d+1)$ simply connected domains each of which maps conformally onto $T^\infty$ under $f$. Let us call them $\widetilde{T^\infty_1},\cdots,\widetilde{T^\infty_{d+1}}$, so 
$$
\widetilde{T^\infty}=\bigsqcup_{j=1}^{d+1}\widetilde{T^\infty_j}.
$$ 
To analyze the structure of grand orbits of the correspondence $\mathfrak{C}$ on $\widetilde{T^\infty}$, we need to discuss the deck transformations of $f:\widetilde{T^\infty}\to T^\infty$. As $\widetilde{T^\infty}$ consists of exactly $(d+1)$ simply connected domains each of which maps conformally onto $T^\infty$ under $f$, we can define a map 
$$
\tau:\widetilde{T^\infty}\to\widetilde{T^\infty}
$$ 
satisfying the conditions
\begin{enumerate}
\item $\tau(\widetilde{T^\infty_j})=\widetilde{T^\infty_{j+1}},\ j\in\Z/(d+1)\Z,$ and

\item $f\circ\tau=f$, for $z\in \widetilde{T^\infty}$,
\end{enumerate}
where the components $\widetilde{T^\infty_j}$ are labeled counter-clockwise. Clearly, $\tau$ is a conformal isomorphism of $\widetilde{T^\infty}$ satisfying 
$$
\tau^{\circ (d+1)}=\mathrm{id},\ \textrm{and}\ f^{-1}(f(z))=\{z,\tau(z),\cdots,\tau^{\circ d}(z)\}\ \forall\ z\in \widetilde{T^\infty}.$$ 
Since $\iota$ is an antiholomorphic involution preserving $\widetilde{T^\infty}$, it follows that each of $\tau\circ\iota,\cdots,\tau^{\circ d}\circ\iota$ is an anti-conformal automorphism of $\widetilde{T^\infty}$.

Here is an explicit description of $\tau$. We set $\omega:=e^{\frac{2\pi i}{d+1}}$, and denote the inverse branch of $f$ that sends $T^\infty$ onto $\widetilde{T^\infty_j}$ by $f_j$. Observe that the map $z\mapsto f_j(\omega^df(z))$ is a conformal rotation (of order $d+1$) of the simply connected domain $\widetilde{T^\infty_j}$ around the point $f_j(0)$.

\begin{proposition}[Description of Deck Transformations]\label{deck_description_prop}
1) $\omega\widetilde{T^\infty}=\widetilde{T^\infty}$; more precisely, $\omega \widetilde{T^\infty_j}=\widetilde{T^\infty_{j+1}}$, for $j\in\Z/(d+1)\Z$.

2) $\tau(z)=\omega f_j(\omega^df(z))$, for $z\in\widetilde{T^\infty_{j}},\ j\in\Z/(d+1)\Z$.
\end{proposition}
\begin{proof}
1) Note that both $f$ and $\iota$ commute with $z\mapsto\omega z$. Hence, $\sigma$ commutes with $z\mapsto\omega z$. It follows that $\omega T^\infty= T^\infty$. The fact that $f$ commutes with $z\mapsto\omega z$ now implies that $\omega\widetilde{T^\infty}=\widetilde{T^\infty}$. Clearly, $z\mapsto\omega z$ must permute the connected components of $\widetilde{T^\infty}$. Our labeling of the components $\widetilde{T^\infty_j}$ now guarantees that $\omega \widetilde{T^\infty_j}=\widetilde{T^\infty_{j+1}}$, for $j\in\Z/(d+1)\Z$.

2) As mentioned earlier, the map $z\mapsto f_j(\omega^df(z))$ is a conformal rotation of the simply connected domain $\widetilde{T^\infty_j}$ around the point $f_j(0)$ (for $j\in\Z/(d+1)\Z$). Thus, the map 
\begin{equation}
z\mapsto \omega f_j(\omega^df(z)),\ z\in\widetilde{T^\infty_j},
\label{deck_definition}
\end{equation}
is a conformal automorphism of $\widetilde{T^\infty}$ that carries $\widetilde{T^\infty_{j}}$ onto $\widetilde{T^\infty_{j+1}}$. Finally, the identity 
$$
f(\omega f_j(\omega^df(z)))=\omega f(f_j(\omega^df(z)))=f(z),\ z\in\widetilde{T^\infty_{j}},
$$ 
shows that the map defined by Equation~\ref{deck_definition} is a deck transformation for $f:\widetilde{T^\infty}\to T^\infty$. It follows that the map defined by Equation~\ref{deck_definition} coincides with $\tau$.
\end{proof}

We are now ready to conclude that the correspondence $\mathfrak{C}$ is a mating of the abstract Hecke group $\Z/2\Z\ast\Z/(d+1)\Z$ and the anti-polynomial $\overline{z}^d$.

\begin{theorem}[$\mathfrak{C}$ as a Mating]\label{group_poly_mating_thm_1}
The anti-holomorphic correspondence $\mathfrak{C}$ is a mating of $\Z/2\Z\ast\Z/(d+1)\Z$ and $\overline{z}^d$ in the following sense.
\begin{itemize}
\item On $\widetilde{T^\infty}$, the dynamics of the correspondence $\mathfrak{C}$ is equivalent to a $\Z/2\Z\ast\Z/(d+1)\Z$-action. More precisely, for each $z\in\widetilde{T^\infty}$, the grand orbit of $z$ under $\mathfrak{C}$ is equal to the $\langle\iota,\tau\rangle\cong\Z/2\Z\ast\Z/(d+1)\Z$-orbit of $z$.

\item The $d:1$ forward (respectively, backward) branch of $\mathfrak{C}$ carrying $\widetilde{K}\setminus\D$ (respectively, $\widetilde{K}\cap\overline{\D}$) onto itself is topologically conjugate to $\overline{z}^d\vert_{\overline{\D}}$ such that the conjugacy is conformal on the interior.
\end{itemize} 
\end{theorem}
\begin{proof}
By construction of $\tau$ and the definition of $\mathfrak{C}$, we have that if $z\in\widetilde{T^\infty}$, then $(z,w)\in\mathfrak{C}$ if and only if $w\in\{\tau\circ\iota(z), \cdots,\tau^{\circ d}\circ\iota(z)\}$. 
Also note that $\tau=(\tau^{\circ 2}\circ\iota)\circ(\tau\circ\iota)^{-1}$, and hence $\langle\tau\circ\iota,\tau^{\circ 2}\circ\iota,\cdots,\tau^{\circ d}\circ\iota\rangle=\langle\iota,\tau\rangle$ (considered as subgroups of the group of all conformal and anti-conformal automorphisms of $\widetilde{T^\infty}$). Finally, the fact that $\langle\iota,\tau\rangle$ is the free product of $\langle\iota\rangle$ and $\langle\tau\rangle$ easily follows by the arguments used in the proof of Proposition~\ref{grand_orbit_group}.

In light of Proposition~\ref{dynamics_filled_julia_corr_prop}, to complete the proof, it suffices to show that $\sigma\vert_{K}$ is topologically conjugate to $\overline{z}^d\vert_{\overline{\D}}$ such that the conjugacy is conformal on the interior. But this is precisely the content of Proposition~\ref{basin_inf_prop}.
\end{proof}


\end{document}